\documentclass[10pt]{amsart}
\usepackage[T1]{fontenc}
\usepackage[nohug,heads=littlevee]{diagrams_arxiv}
\usepackage{mathrsfs}
\usepackage{bm}
\usepackage{stmaryrd}
\usepackage[bbgreekl]{mathbbol}
\usepackage{amssymb}
\usepackage[latin1]{inputenc}
\usepackage{xr-hyper}
\usepackage[plainpages=false,pdfpagelabels]{hyperref}
\usepackage[left=2cm,top=2.5cm,right=2cm,bottom=1.4in,asymmetric]{geometry}
\usepackage{mathtools}
\usepackage[backend=bibtex]{biblatex}
\usepackage{enumitem}
\usepackage{color}
\usepackage[all]{xy}
\usepackage{tensor}
\usepackage{xcolor}
\usepackage{pifont}
\usepackage{tikz-cd}

   \DeclareFontFamily{U}{wncy}{}
    \DeclareFontShape{U}{wncy}{m}{n}{<->wncyr10}{}
    \DeclareSymbolFont{mcy}{U}{wncy}{m}{n}
    \DeclareMathSymbol{\Sha}{\mathord}{mcy}{"58}

\newcommand{\defnword}[1]{\textbf{#1}}
\newcommand{\comment}[1]{}
\newcommand{\on}[1]{\operatorname{#1}}

\DeclareSymbolFontAlphabet{\mathbbl}{bbold}
\DeclareSymbolFontAlphabet{\mathbb}{AMSb}

\numberwithin{equation}{subsubsection}
\newtheorem{introthm}{Theorem}

\newtheorem{proposition}[subsubsection]{Proposition}
\newtheorem{theorem}[subsubsection]{Theorem}

\newtheorem*{thm*}{Theorem}
\newtheorem{lemma}[subsubsection]{Lemma}
\newtheorem*{lem*}{Lemma}
\newtheorem{corollary}[subsubsection]{Corollary}

\theoremstyle{definition}
\newtheorem{definition}[subsubsection]{Definition}

\newtheorem{example}[subsubsection]{Example}

\newtheorem{remark}[subsubsection]{Remark}
\newtheorem{notation}[subsubsection]{Notation}

\newtheorem{construction}[subsubsection]{Construction}

\newtheorem*{assump*}{Assumption}

\newarrow{Equals}{=}{=}{}{=}{}
\newarrow{Onto}{-}{-}{-}{-}{>>}

\newcommand{\Int}{\mathbb{Z}}

\newcommand{\Aff}{\mathbb{A}}

\newcommand{\Field}{\mathbb{F}}
\newcommand{\Rat}{\mathbb{Q}}
\newcommand{\Lie}{\on{Lie}}
\newcommand{\Dieu}{\mathbb{D}}

\newcommand{\Map}{\on{Map}}

\newcommand{\BT}[2][G,\mu]{\mathrm{BT}^{#1}_{#2}}

\newcommand{\Prism}{\mathbbl{\Delta}}
\newcommand{\alphap}{\mbox{$\hspace{0.12em}\shortmid\hspace{-0.62em}\alpha_p$}}
\newcommand{\mup}[1][p]{\mbox{$\raisebox{-0.59ex}
  {$l$}\hspace{-0.18em}\mu\hspace{-0.88em}\raisebox{-0.98ex}{\scalebox{2}
  {$\color{white}.$}}\hspace{-0.416em}\raisebox{+0.88ex}
  {$\color{white}.$}\hspace{0.46em}_{#1}$}{}}
\newcommand{\Rees}{\mathcal{R}}
\newcommand{\defn}{\overset{\mathrm{defn}}{=}}

\newcommand{\colim}{\on{colim}}

\newcommand{\cofib}{\on{cofib}}
\newcommand{\fib}{\on{fib}}

\newcommand{\gr}{\on{gr}}

\newcommand{\Rg}{\mathscr{O}}
\newcommand{\Reg}[1]{\Rg_{#1}}

\newcommand{\Hom}{\on{Hom}}

\newcommand{\Ext}{\on{Ext}}

\newcommand{\Gm}{\mathbb{G}_m}

\newcommand{\Ga}{\mathbb{G}_a}

\newcommand{\dR}{\on{dR}}

\newcommand{\Sig}{\mathfrak{S}}
\newcommand{\pow}[1]{[\vert#1\vert]}

\newcommand{\Mod}[2][ ]{\on{Mod}_{#2}^{#1}}

\newcommand{\Spec}{\on{Spec}}
\newcommand{\Spf}{\on{Spf}}

\newcommand{\et}{\on{\acute{e}t}}

\newcommand{\Fzip}{F\mathrm{Zip}}

\newcommand{\mathhyph}{\text{-}}

\newcommand{\Fil}{\on{Fil}}

\newcommand{\mx}{\mathfrak{m}}

\newcommand{\ad}{\on{ad}}

\newcommand{\GL}{\on{GL}}

\bibliography{fflat}

\begin{document}
\title{Perfect $F$-gauges and finite flat group schemes}

\author{Keerthi Madapusi}
\address{Keerthi Madapusi\\
Department of Mathematics\\
Maloney Hall\\
Boston College\\
Chestnut Hill, MA 02467\\
USA}
\email{madapusi@bc.edu}

\author{Shubhodip Mondal}
\address{Shubodip Mondal\\
Department of Mathematics\\Purdue University\\ 150 N. University Street\\ West Lafayette, IN 47907, USA
}
\email{mondalsh@purdue.edu}

\begin{abstract}
We show an equivalence of categories, over general $p$-adic bases, between finite locally $p^n$-torsion commutative group schemes and $\Int/p^n\Int$-modules in perfect $F$-gauges of Tor amplitude $[-1,0]$ with Hodge-Tate weights $0,1$. By relating fppf cohomology of group schemes and syntomic cohomology of $F$-gauges, we deduce some consequences: These include the representability of relative fppf cohomology of finite flat group schemes under proper smooth maps of $p$-adic formal schemes, as well as a reproof of a purity result of \v{C}esnavi\v{c}ius-Scholze. We also give a general criterion for a classification in terms of objects closely related to Zink's windows over frames and Lau's divided Dieudonn\'e crystals, and we use this to recover several known classifications, and also give some new ones.
\end{abstract}

\maketitle

\tableofcontents

\section{Introduction}

The goal of this article is to give a classification of finite flat group schemes over general $p$-adic formal bases, as well as to give some applications of this classification.

\subsection{Classification by perfect $F$-gauges}
\label{sub:classification_by_perfect_f_gauges}

To state our first main result, we need the $p$-adic cohomological stacks that arose in recent work of Bhatt-Lurie~\cites{bhatt2022absolute,bhatt2022prismatization,bhatt_lectures} and Drinfeld~\cite{drinfeld2022prismatization}. These authors have shown that one can associate with every $p$-adic formal scheme $X$ a $p$-adic formal stack\footnote{This is actually a \emph{derived} formal stack that is in general not a classical object. We will attempt to ignore this fact in this introduction.} $X^{\mathrm{syn}}$, its \emph{syntomification}, whose coherent cohomology computes the $p$-adic syntomic cohomology of $X$. If $X = \Spf R$ is affine, we will also denote this by $R^{\mathrm{syn}}$.\footnote{We have adopted this notation from the lecture notes of Bhatt~\cite{bhatt_lectures} and it is also employed in~\cite{gmm}.} Perfect complexes on this stack  and its mod-$p^n$ fibers---which are examples of objects known as $F$-\emph{gauges over $X$}---have a naturally defined subset of integers associated with them, which are called the \emph{Hodge-Tate weights}. We prove:

\begin{introthm}
 [Theorem~\ref{thm:main}]  \label{introthm:main} Suppose that $X$ is a $p$-adic formal scheme.
Let $\mathrm{FFG}(X)$ be the category of finite locally free $p$-power torsion commutative group schemes over $X$, and let $\mathsf{P}^{\mathrm{syn}}_{\{0,1\}}(X)$ be the category of perfect \footnote{Here perfect means dualizable object in the category of $F$-gauges.} $F$-gauges over $X$ with Hodge-Tate weights in $\{0,1\}$, with Tor amplitude in $[-1,0]$, and with cohomology sheaves killed by a power of $p$.\footnote{See \S\ref{sec:ffg} in the body of the paper for the precise definition of this category.} Then there is a canonical (covariant) \emph{exact} equivalence of categories
\[
\mathcal{G}:\mathsf{P}^{\mathrm{syn}}_{\{0,1\}}(X)\xrightarrow{\simeq}\mathrm{FFG}(X)
\]
compatible with Cartier duality.\footnote{In the body of the paper, we only state results for $p$-adic formal affine schemes $X = \Spf R$. However, since our constructions and statements will satisfy fpqc and in particular Zariski descent, they globalize immediately to the case of $p$-adic formal schemes, and in fact to $p$-adic formal algebraic stacks.}
\end{introthm}

\begin{remark}
The category $\mathsf{P}^{\mathrm{syn}}_{\{0,1\}}(X)$ admits an alternate description that does not require the introduction of the syntomification of $p$-adic formal schemes, but only the absolute prismatization as described in~\cite{bhatt2022absolute}, or even just the absolute prismatic cohomology of semiperfectoid rings. The reader interested mainly in this and more concrete classifications can skip ahead to \S~\ref{subsec:classification_by_divided_dieudonne_complexes} and \S~\ref{subsec:some_concrete_classifications} below, where they will also find a discussion of some already known classifications and their connection with the work here.
\end{remark}

\begin{remark}
The compatibility with Cartier duality takes the following shape (cf.~\cite[Proposition~3.87]{Mondal2024-cy}): There is a canonical object $\mathcal{O}^{\mathrm{syn}}\{1\}$, called the \emph{Breuil-Kisin twist}, that is a line bundle $F$-gauge over $X$. By our conventions in the current paper, the Hodge-Tate weight of $\mathcal{O}^{\mathrm{syn}}\{1\}$ is $1$. We can tensor $\mathcal{O}^{\mathrm{syn}}\{1\}$ with any perfect complex $\mathcal{M}$ over $X^{\mathrm{syn}}$ to obtain the twist $\mathcal{M}\{1\}${}. If $\mathcal{M}$ belongs to $\mathsf{P}^{\mathrm{syn}}_{\{0,1\}}(X)$, then so does $\mathcal{M}^\vee[1]\{1\}$, and we now have a canonical isomorphism of finite flat group schemes
\[
\mathcal{G}(\mathcal{M}^\vee[1]\{1\})\xrightarrow{\simeq}\mathcal{G}(\mathcal{M})^*,
\]
where the right hand side is the Cartier dual of $\mathcal{G}(\mathcal{M})$.
\end{remark}

\begin{remark}
The functor $\mathcal{G}$ is given by truncated syntomic cohomology: Given $\mathcal{M}\in \mathsf{P}^{\mathrm{syn}}_{\{0,1\}}(X)$ and any map $\Spec R\to X$ with $R$ $p$-nilpotent, the values of $\mathcal{G}(\mathcal{M})$ on $R$ are given by $\tau^{\le 0}R\Gamma(R^{\mathrm{syn}},\mathcal{M}\vert_{R^{\mathrm{syn}}})$. As such, it is completely canonical, compatible with arbitrary base-change and satisfies quasisyntomic (and in fact fpqc) descent.
\end{remark}

\begin{remark}\label{compatiblity1}
When $X$ is $p$-quasisyntomic, the inverse of the functor $\mathcal{G}$ is given by the functor that sends $G \in \mathrm{FFG}(X)$ to $\mathcal{M}(G^*)\left \{1 \right\}$, where the $F$-gauge $\mathcal{M}$ is constructed using Nygaard filtered prismatic cohomology of the (higher) classifying stack $B^2G$ as in \cite{Mondal2024-cy}. Note that $\mathcal{M}$ is equivalent to the classical contravariant Dieudonn\'e module of $G$ over a perfect field (cf.~\cite{Mondal2020-cy}). While we do not give an explicit description of the inverse of $\mathcal{G}$ over a general base $X$, it could be seen as a moduli theoretic extension of the functor $\mathcal{M}$ (see the discussion before Theorem \ref{introthmK}).
\end{remark}{}

\begin{remark}
The theorem above generalizes and refines our previous results in~\cite{gmm} and~\cite{Mondal2024-cy}, in different directions. The first cited reference proves Theorem~\ref{introthm:main}, but only for $n$-truncated Barsotti-Tate groups (for some $n\ge 1$) over $X$, matching them up with mod-$p^n$ vector bundle $F$-gauges in $\mathsf{P}^{\mathrm{syn}}_{\{0,1\}}(X)$.\footnote{The exactness of the equivalence was not addressed in \emph{loc. cit.} however. This is the content of \S~\ref{sec:exactness_of_the_syntomic_dieudonne_functor} of this article.} In the second reference, one finds an equivalence of $\mathrm{FFG}(X)$ with a certain full subcategory of $\mathsf{P}^{\mathrm{syn}}_{\{0,1\}}(X)$, but only for $X$ that are $p$-quasisyntomic.
\end{remark}

We also show that the classification in Theorem~\ref{introthm:main} is compatible with various cohomological realizations on either side. As a particular instance, we have
\begin{introthm}
   [Syntomic cohomology and fppf cohomology, Proposition~\ref{prop:syntomic_and_fppf}]
\label{introthm:fppf_and_syntomic}
Suppose that we have $\mathcal{M}\in \mathsf{P}^{\mathrm{syn}}_{\{0,1\}}(X)$ with $G = \mathcal{G}(\mathcal{M})$. Then there is a canonical isomorphism
\[
   R\Gamma_{\mathrm{fppf}}(X,G)\xrightarrow{\simeq}R\Gamma(X^{\mathrm{syn}},\mathcal{M}).
\]
\end{introthm}
A variant of the above theorem for cohomology of $G$ in the quasisyntomic topology was proven in \cite[Proposition~3.84]{Mondal2024-cy}.
As a consequence of Theorem~\ref{introthm:fppf_and_syntomic}, and general representability results from~\cite{gmm} for the cohomology of $F$-gauges with Hodge-Tate weights bounded by $1$, we obtain the following generalization of a result of Bragg-Olsson~\cite[Theorem 1.8]{bragg2021representability}:
\begin{introthm}
 [Representability of relative fppf cohomology, Theorem~\ref{thm:bragg_olsson_smooth}]
\label{introthm:bragg_olsson}  
Let $\pi:X\to S$ be a proper smooth map of $p$-adic formal algebraic spaces, and that $G\in \mathrm{FFG}(X)$. If $R^i_{\mathrm{fppf}}\pi_*G$ is represented by a flat formal algebraic space over $S$ for all $i<n$, then $R^n_{\mathrm{fppf}}\pi_*G$ is also represented by a finitely presented formal algebraic space over $S$.
\end{introthm}

\begin{remark}
   Suppose that $G = \mup$. In this special case, the cited result of Bragg-Olsson already proves the theorem when $X$ and $S$ are schemes of finite type over a field. But in the generality here, even this case appears to be new. Here, it says that $R^1\pi_*\mup$ is represented by a finitely presented formal algebraic space over $S$ (which can also be deduced from the theory of the Picard scheme), and that, when this algebraic space is also flat over $S$, then $R^2\pi_*\mup$ is also representable.
\end{remark}

We can also use Theorem~\ref{introthm:fppf_and_syntomic} and an idea of Bhatt-Lurie in ~\cite[Corollary~8.5.7]{bhatt2022absolute} to recover the following (special case of a) result of \v{C}esnavi\v{c}ius-Scholze~\cite[Theorem 7.1.2]{Cesnavicius2019-pn}:
\begin{introthm}
   [\v{C}esnavi\v{c}ius-Scholze purity theorem, Corollary~\ref{cor:cesnavicius_scholze}]
\label{introthm:cesnavicius-scholze}
Let $X$ be a qcqs scheme and let $Z\subset X_{(p=0)}$ be a constructible closed subset such that, for all $z\in Z$, $\Reg{X,z}$ is a local complete intersection ring of dimension $\ge d$. Then, if $G$ is a finite locally free $p$-power torsion group scheme over $X$, the map
\[
   H^i_{\mathrm{fppf}}(X,G)\to H^{i}_{\mathrm{fppf}}(X\backslash Z,G)
\]
is injective for $i<d$ and an isomorphism for $i<d-1$.
\end{introthm}

\subsection{Classification by divided Dieudonn\'e complexes}
\label{subsec:classification_by_divided_dieudonne_complexes}

The classification in terms of $F$-gauges, while powerful in its generality, might appear somewhat abstract. In Section~\ref{sec:divided_dieudonne}, we deduce several more explicit classifications that recover known ones and also yield some new examples. 

First, under some weak finiteness conditions, we extract from Theorem~\ref{introthm:main} a somewhat less abstract classification that is very close to those found in~\cite{MR4530092} and~\cite{lau2018divided}. To explain this, let $\Prism_{-,\mathrm{cl}}$ be the sheaf of rings $S\mapsto H^0(\Prism_{S})_{\mathrm{sep}}$ on the site $R_{\mathrm{qsyn}}$ of quasisyntomic semiperfectoid $R$-algebras $S$, equipped with the quasisyntomic topology. This assigns to each $S$ the classically complete quotient of the classical truncation of its initial prism. For instance, if $S$ is a semiperfect $\Field_p$-algebra, then $\Prism_{S,\mathrm{cl}}$ is simply the classical $p$-adically completed divided power envelope of the surjection $W(S^\flat)\to S$, usually denoted $A_{\mathrm{crys}}(S)$. There is a natural map $\Prism_{-,\mathrm{cl}}\to \mathcal{O}$, where $\mathcal{O}$ is the structure sheaf $S\mapsto S$. There is also a lift of the mod-$p$ Frobenius $\varphi:\Prism_{-,\mathrm{cl}}\to \Prism_{-,\mathrm{cl}}$. Moreover, $\Prism_{-,\mathrm{cl}}$ admits a canonical generalized Cartier divisor $I_{-,\mathrm{cl}}\to \Prism_{-,\mathrm{cl}}$ arising from the prism structure with (derived) quotient $\overline{\Prism}_{-,\mathrm{cl}}$. The precomposition of the quotient map with the Frobenius lift factors through a map $\mathcal{O}\to \overline{\Prism}_{-,\mathrm{cl}}$

The following objects can be used to classify $p$-divisible groups under some conditions:
\begin{definition}
   [Prismatic divided Dieudonn\'e modules]
Let $\mathrm{DDC}_{\Prism_{\mathrm{cl}}}^{[0,0]}(R)$ be the category of tuples 
\[
(\mathcal{M}_{-}\xrightarrow{\Psi_{\mathcal{M}}}\varphi^*\mathcal{M}_{-},\Fil^0_{\mathrm{Hdg}}M\subset M,\xi) 
\]
where:
\begin{enumerate}
   \item $\mathcal{M}_{-}$ is a (quasicoherent) vector bundle over $\Prism_{-,\mathrm{cl}}$ of finite rank;
   \item $M$ is the finite locally free $R$-module corresponding to the base-change of $\mathcal{M}_{-}$ along $\Prism_{-,\mathrm{cl}}\to \mathcal{O}$ and $\Fil^0_{\mathrm{Hdg}}M\subset M$ is a local direct summand;
   \item $\Psi_{\mathcal{M}}:\mathcal{M}_{-}\to \varphi^* \mathcal{M}_{-}$ is a map of $\Prism_{-,\mathrm{cl}}$-modules whose cofiber is equipped with an isomorphism $\xi$ to $ \overline{\Prism}_{-,\mathrm{cl}}\otimes_R\gr^{-1}_{\mathrm{Hdg}}M$ where $\gr^{-1}_{\mathrm{Hdg}}M = M/\Fil^0_{\mathrm{Hdg}}M$.
\end{enumerate}
\end{definition}

\begin{remark}
\label{remark:alternate_windows}
   One can alternately define the category $\mathrm{DDC}^{[0,0]}_{\Prism_{\mathrm{cl}}}(R)$ using the approach taken in~\cite{lau2018divided}, as a category of windows over a sheaf theoretic frame with underlying sheaf of rings $\Prism_{-,\mathrm{cl}}$. The reader can extrapolate this from Remark~\ref{rem:windows}. We have chosen the presentation here because it is the one that generalizes most easily when one wants to classify finite flat group schemes.
\end{remark}

\begin{remark}
   The usual convention would have been to consider \emph{effective} objects equipped with $\varphi$-semilinear endomorphisms rather than a map $\Psi_{\mathcal{M}}$ as above. The reason for our choice here is that it is the sort of structure that arises naturally when considering $F$-gauges of Hodge-Tate weights $0,1$. It might also be worth emphasizing once again here that the equivalences of categories we construct in this paper are naturally \emph{covariant}.
\end{remark}

To classify finite flat group schemes, we need to consider perfect complexes.
\begin{definition}
   [Prismatic divided Dieudonn\'e complexes]
Let $\Mod{\Int/p^n\Int}(\mathrm{DDC}_{\Prism_{\mathrm{cl}}}^{[-1,0]}(R))$ be the category of tuples 
\[
(\mathcal{M}_{-}\xrightarrow{\Psi_{\mathcal{M}}}\varphi^* \mathcal{M}_{-},\Fil^0_{\mathrm{Hdg}}M\to M,\xi)
\]
where:
\begin{enumerate}
   \item $\mathcal{M}_{-}$ is a (quasicoherent) perfect complex over $\Prism_{-,\mathrm{cl}}$ with Tor amplitude $[-1,0]$ and with cohomology killed by $p^n$;
   \item $M$ is the perfect complex over $R$ corresponding to the base-change of $\mathcal{M}_{-}$ along $\Prism_{-,\mathrm{cl}}\to \mathcal{O}$ and $\Fil^0_{\mathrm{Hdg}}M\to M$ is a map of perfect complexes with Tor amplitude in $[-1,0]$ and whose cofiber $\gr^{-1}_{\mathrm{Hdg}}M$ also has Tor amplitude in $[-1,0]$;
   \item $\Psi_{\mathcal{M}}:\mathcal{M}_{-}\to \varphi^* \mathcal{M}_{-}$ is a map of complexes over $\Prism_{-,\mathrm{cl}}$ whose cofiber is equipped with an isomorphism $\xi$ to $ \overline{\Prism}_{-,\mathrm{cl}}\otimes_R\gr^{-1}_{\mathrm{Hdg}}M$.
\end{enumerate}
\end{definition}

\begin{introthm}
   [Classification in terms of divided Dieudonn\'e complexes, Theorem~\ref{thm:classical_prismatization_enough}]
\label{introthm:classical_enough}
Suppose that one of the following holds: 
\begin{enumerate}
   \item $R$ is $p$-quasisyntomic.
   \item $R/pR$ is $F$-finite and $F$-nilpotent~\cite{lau2018divided}.
\end{enumerate}
Then there is a canonical exact equivalence of categories
\[
   \mathrm{FFG}_n(R) \xrightarrow{\simeq}\Mod{\Int/p^n\Int}(\mathrm{DDC}_{\Prism_{\mathrm{cl}}}^{[-1,0]}(R))
\]
and also an equivalence of categories
\[
   \mathrm{BT}(R) \xrightarrow{\simeq}\mathrm{DDC}_{\Prism_{\mathrm{cl}}}^{[0,0]}(R).
\]
\end{introthm}

\begin{remark}
   Via the translation to the language of windows and frames, as explained in Remark~\ref{remark:alternate_windows}, the presentation can be brought closer to that found in~\cite{lau2018divided} or~\cite{MR4530092}. Indeed, in case (1), the second assertion of the theorem is simply a reinterpretation of the main result of~\cite{MR4530092}; see Remark~\ref{rem:anschutz_lebras_comp}. In case (2), we are obtaining a generalization of the main result of~\cite{lau2018divided}, which deals with the case where $R$ is an $\Field_p$-algebra.
\end{remark}

\begin{remark}
   The proof of the theorem involves first proving a more general version that works for all $R$ but uses the untruncated derived absolute prismatic cohomology sheaf $\Prism_{-}$, in the guise of the (derived) prismatization $R^\Prism$. We then make crucial use of a construction of Lau from~\cite{lau2018divided} to show that the classical prismatization already does the job under further finiteness hypotheses.
\end{remark}

\subsection{Some concrete classifications}
\label{subsec:some_concrete_classifications}
A substantial chunk of this paper is devoted to the proof of results leading up to Theorem~\ref{introthm:classical_enough} as well as some more explicit classifications. To give some further instances of the latter, we need another definition. This expands on the notion of Breuil windows, which was introduced in~\cite{vasiu_zink}, following the work of Breuil~\cite{breuil:pdiv} and Kisin~\cite{kisin:f_crystals}. 
\begin{definition}
   [A generalization of Breuil windows, Definition~\ref{def:bk_windows}]
Suppose that $(A,I')$ is a $p$-torsion-free prism with $R = A/I'$. Set $I = \varphi^*I'$ and $\overline{A} = A/I$. For each $n\ge 1$, let $\mathrm{BK}_{\underline{A},n}(R)$ be the category of tuples $(\mathsf{N},F_{\mathsf{N}},V_{\mathsf{N}})$ where $\mathsf{N}$ is a $p^n$-torsion $A$-module of projective dimension $1$ and 
\[
   F_{\mathsf{N}}:\varphi^* \mathsf{N}\to \mathsf{N}\;;\; V_{\mathsf{N}}:I'\otimes_A \mathsf{N}\to \varphi^* \mathsf{N}
\]
are $A$-module maps such that $V_{\mathsf{N}}\circ (1\otimes F_{\mathsf{N}})$ and $F_{\mathsf{N}}\circ V_{\mathsf{N}}$ are the canonical maps $I'\otimes \varphi^* \mathsf{N}\to \varphi^*\mathsf{N}$ and $I'\otimes_A \mathsf{N}\to \mathsf{N}$, respectively.
\end{definition}

\begin{remark}
\label{remark:ffg_to_bk}
   In this situation, there is a canonical functor $ \mathrm{FFG}_n(R)\to \mathrm{BK}_{\underline{A},n}(R)$. See Remark~\ref{rem:bk_syntomic}.
\end{remark}

\begin{definition}
   [Nilpotence conditions]
Suppose that $\mathfrak{c}\subset B\defn (R/pR)_{\mathrm{red}}$ is a finitely generated ideal. An object $(\mathsf{N},F_{\mathsf{N}},V_{\mathsf{N}})$ in $\mathrm{BK}_{\underline{A},n}(R)$ is $\mathfrak{c}$-\defnword{nilpotent} if $F_{\mathsf{N}^*}:\varphi^*\mathsf{N}^*\to \mathsf{N}^*$\footnote{Here, $*$ denotes Cartier duality; see Remark~\ref{rem:bk_cartier_duality}.} is nilpotent after base-change to $B/\mathfrak{c}$. Write $\mathrm{BK}^{\mathfrak{c}\mathhyph\mathrm{nilp}}_{\underline{A},n}(R)$ for the subcategory of $\mathrm{BK}_{\underline{A},n}(R)$ spanned by such objects.
\end{definition}

\begin{introthm}
   [Classification of connected group schemes, Corollary~\ref{cor:nilpotent_connected}]
Suppose that $B$ is $\mathfrak{c}$-adically derived complete. Then the functor from Remark~\ref{remark:ffg_to_bk} restricts to an exact equivalence of categories
\[
  \mathrm{FFG}^{\mathfrak{c}\mathhyph\mathrm{conn}}_n(R)\xrightarrow{\simeq} \mathrm{BK}^{\mathfrak{c}\mathhyph\mathrm{nilp}}_{\underline{A},n}(R)
\]
where the left hand side is the subcategory of $\mathrm{FFG}_n(R)$ spanned by the objects whose restriction to $\Spec B/\mathfrak{c}$ is connected.
\end{introthm}

\begin{remark}
When $R$ is an $\Field_p$-algebra (so that $I' = pA$), this is essentially due to Zink and Lau; see Remark~\ref{rem:lau_zink_lifts}. The special case where $R$ is a regular complete local ring is also already known by work of Lau~\cite{MR2679066}. 
\end{remark}

\begin{remark}
    The methods used for the proof here also recover the classification, due to Zink and Lau, of connected group schemes in terms of Witt vector displays; see Remark~\ref{rem:lau_zink_connected}.
\end{remark}

\begin{introthm}
   [Classification for semiperfectoid rings, Corollary~\ref{cor:semiperf_bk_case}, Remark~\ref{rem:cotangent_complex_semiperfectoid_deg_0}]
Suppose that $R$ is semiperfectoid with $R/pR$ $F$-nilpotent and that $R_0\to R$ is a surjection from a perfectoid ring. Set $K = \ker(A_{\mathrm{inf}}(R_0)\to A)$, and suppose that the semilinear operator $\overline{\delta}$ on $K/K^2$ induced from the $\delta$-structures on $A_{\mathrm{inf}}(R_0)$ and $A$ is topologically locally nilpotent. Then the functor from Remark~\ref{remark:ffg_to_bk} is an exact equivalence.
\end{introthm}

\begin{remark}
This theorem recovers the following already known classifications:
\begin{enumerate}
   \item (Gabber, Lau) Over perfect rings $R$ in terms of $p$-power torsion Dieudonn\'e modules over $W(R)$ (Remark~\ref{rem:perfect_case}).
   \item (Scholze-Weinstein, Ansch\"utz-LeBras, Lau) Over perfectoid rings $R$ in terms of $p$-power torsion Breuil-Kisin modules over $A_{\mathrm{inf}}(R)$ (Example~\ref{ex:perfect_prisms}).
   \item (Lau) Over $F$-nilpotent semiperfect rings, in terms of torsion Dieudonn\'e modules~\cite[\S 10.3]{MR3867290}. Lau's result applies more generally, however.
\end{enumerate}
\end{remark}

We also obtain the following generalization of the main Theorem from~\cite{MR1235021}, where it is shown when the residue field is perfect (Example~\ref{ex:complete_local_dejong}).
\begin{introthm}
   [Classification over complete local $\Field_p$-algebras, Remark~\ref{rem:complete_local_rings}]
\label{introthm:dejong}
Suppose that $R$ is a complete local Noetherian ring with maximal ideal $\mx$ whose residue field admits a finite $p$-basis. Suppose also that the divided Frobenius operator $`d\varphi/p'$ on $\widehat{\Omega}^1_{A}$ induces a nilpotent endomorphism of $\Omega^1_{R/\Field_p}/\mx\Omega^1_{R/\Field_p}$, and that a certain operator $\gamma_{\underline{A}}$ on $\alphap(R) = \{a\in R:\; a^p=0\}$ is nilpotent mod-$\mx$. Then the functor from Remark~\ref{remark:ffg_to_bk} is an exact equivalence.
\end{introthm}

\begin{remark}
Some of the results flowing into the proof of Theorem~\ref{introthm:dejong} can also be used to recover results of Lau from~\cite{MR2679066} for regular local rings, which have their antecedents in~\cite{kisin:f_crystals} and related works. See Example~\ref{ex:regular_complete_local}.
\end{remark}

\subsection{Strategy of the proof of Theorem~\ref{introthm:main}}
\label{subsec:strategy_of_proof}
Starting from the recent work in \cite{MR4530092}, \cite{gmm} and \cite{Mondal2024-cy}, 
the basic idea behind the proof of Theorem~\ref{introthm:main} is quite simple, and one that appears quite often in the subject: Reduce the classification of finite flat group schemes to that of $p$-divisible groups using Raynaud's theorem~\cite[Th\'eor\`eme 3.1.1]{bbm:cris_ii}, which tells us that every such group scheme is Zariski locally the kernel of an isogeny of $p$-divisible groups. However, there are subtleties arising from the lack of functoriality of these choices of $p$-divisible groups. Using our functor $\mathcal{G}$ partially resolves this issue. Indeed, using Raynaud's theorem and~\cite[Theorem A]{gmm}, one finds that our functor $\mathcal{G}$ is Zariski locally essentially surjective. But, in order to see that it is fully faithful and (hence) essentially surjective on the nose, we find ourself needing to \emph{a priori} show an $F$-gauge analogue of Raynaud's theorem (as well as certain exactness properties of our constructions):

\begin{introthm}
   [Raynaud for $F$-gauges, Theorem~\ref{thm:syntomic_raynaud}]
\label{introthm:raynaud}
Given $\mathcal{M}\in \mathsf{P}^{\mathrm{syn}}_{\{0,1\}}(R)$, there exists an affine $p$-completely pro-\'etale cover $\Spf R'\to \Spf R$, and a cofiber sequence of perfect $F$-gauges over $R'$
\[
   \mathcal{V}_1\to \mathcal{V}_2\to \mathcal{M}\vert_{R^{',\mathrm{syn}}}
\]
where $\mathcal{V}_1$ and $\mathcal{V}_2$  are vector bundle $F$-gauges over $R'$ of Hodge-Tate weights $\{0,1\}$ 
\end{introthm}

In the process of proving Theorem~\ref{introthm:raynaud}, we also obtain the following linear algebraic reinterpretation of a classical fact about $p$-divisible groups.
\begin{introthm}
   [Bartling-Hoff for $F$-gauges, Corollary~\ref{cor:bartling_hoff}]
Suppose that $\mathcal{V}$ is a vector bundle $F$-gauge over $R$ with Hodge-Tate weights $0,1$. Consider the formal prestack $\mathcal{X}_V(n)$ over $\Spf R$ parameterizing maps $f:\mathcal{V}\to \mathcal{V}'$ of vector bundle $F$-gauges such that there exists $\hat{f}:\mathcal{V}'\to \mathcal{V}$ with $f\circ \hat{f}$ and $\hat{f}\circ f$ both equal to multiplication by $p^n$. Then $\mathcal{X}_V(n)$ is represented by a finitely presented formal scheme over $\Spf R$.
\end{introthm}

\begin{remark}
   Via Theorem~\ref{introthm:main}---or more directly, its antecedent in~\cite{gmm} for truncated Barsotti-Tate groups---$\mathcal{V}$ corresponds to a $p$-divisible group $\mathcal{G}$ over $R$, and the above theorem translates to the assertion that isogenies $\mathcal{G}\to \mathcal{G}'$ of degree $\leq p^n$ are represented by a finitely presented formal scheme. This is of course a consequence of the fact that finite flat subgroups of $\mathcal{G}[p^n]$ are parameterized by a projective formal scheme over $\Spf R$. The point here is that we are able to give a `linear algebraic' proof of this in the setting of $F$-gauges, which generalizes to other contexts where such a translation is not possible. This is partly inspired by a result of Bartling and Hoff~\cite{bartling_hoff}, who did the same, but in the context of the Witt vector displays of Lau and Zink.
\end{remark}

With Theorem~\ref{introthm:raynaud} in hand, the remaining key point for establishing full faithfulness turns out to be the exactness of the equivalence established in~\cite{gmm}. For quasisyntomic rings, this exactness is shown in~\cite{Mondal2024-cy}. Our strategy is to reduce to this case (see Remark \ref{compatiblity1}) by proving the next result, which was anticipated by Grothendieck~\cite{GrothLetter}. This is combined with the $F$-gauge analogue of this smoothness, shown in~\cite{gmm}, to complete the reduction.
\begin{introthm}
   [Smoothness of Ext stacks, Theorem~\ref{thm:smoothness_of_extensions}]\label{introthmK}
Suppose that we are given two $n$-truncated Barsotti-Tate group schemes $G(n),H(n)$ over $\Spf R$ with $p^{n-1}$-torsion subgroups $G(n-1),H(n-1)$. Then the formal prestack $\underline{\Ext}(G(n),H(n))$ over $\Spf R$ parameterizing $n$-truncated Barsotti-Tate groups exhibited as extensions of $G(n)$ by $H(n)$ is $p$-completely smooth over $\Spf R$. Moreover, the natural map
\[
   \underline{\Ext}(G(n),H(n))\to \underline{\Ext}(G(n-1),H(n-1))
\]
is smooth and surjective.
\end{introthm}

\section*{Acknowledgments}
K. M. was partially supported by NSF grant DMS-2200804. S.M. was supported by the Institute for Advanced Study, University of British Columbia and a start up grant from Purdue University. 

We thank Luc Illusie for sharing with us Grothendieck's letter~\cite{GrothLetter}. We also thank Bhargav Bhatt, K\k{e}stutis {\v{C}}esnavi\v{c}ius, Brian Conrad, Hai Long Dao, Eike Lau and Shizhang Li for helpful comments and conversations.

\section{Conventions}

Our conventions and notational choices will be as in~\cite{gmm}. In particular, we will freely use $\infty$-categories and all our constructions will be derived unless otherwise noted. Since we will be working mainly with (derived) $p$-complete objects, all tensor products and cotangent complexes appearing here are the $p$-complete versions, again, unless otherwise noted. 

\section{Cohomological stacks and $F$-gauges}
\label{sec:cohomological_stacks_and_f_gauges}

We will be using the formal stacks defined by Drinfeld and Bhatt-Lurie~\cites{bhatt2022absolute,bhatt2022prismatization,bhatt_lectures,drinfeld2022prismatization} geometrizing $p$-adic cohomology theories. For a quick summary and a refresher on the notation, the reader is referred to~\cite[\S 6]{gmm}. Here, we recall what is needed in this paper.

\subsection{Prismatization}

In this subsection, we review the story of the prismatization of $p$-complete animated commutative rings from~\cite{bhatt2022prismatization}. Most of this material will not be used until Section~\ref{sec:divided_dieudonne}. Unless otherwise specified, $R$ will always denote a derived $p$-complete animated commutative ring. We assume that the reader is familiar with animated $\delta$-rings and prisms; see for instance~\cite[\S 2]{bhatt2022prismatization} or~\cite[\S 5.3]{gmm}. We will also assume familiarity with the theory of absolute prismatic cohomology from~\cite{bhatt2022absolute}.

\begin{definition}
   [Cartier-Witt divisors]
A \defnword{Cartier-Witt divisor} on $R$ is a surjective map $\pi:W(R)\twoheadrightarrow \overline{W(R)}$ of animated commutative rings such that two properties hold:
\begin{itemize}
   \item $I = \fib(\pi)$ is a locally free $W(R)$-module of rank $1$;
   \item The map $\pi_0(I) \simeq I\otimes_{W(R)}W(\pi_0(R))\to W(\pi_0(R))$ is a Cartier-Witt divisor in the sense of~\cite[\S 3.1.1]{bhatt2022absolute}. 
\end{itemize}
The second condition means that, Zariski-locally on $\Spec R$, we have a $W(\pi_0(R))$-linear isomorphism $\pi_0(I)\simeq W(\pi_0(R))$ such that the composition $W(\pi_0(R))\simeq \pi_0(I)\to W(\pi_0(R))$ is given by multiplication by a \defnword{distinguished element} $d\in W_{\mathrm{dist}}(\pi_0(R))$, given in Witt coordinates by $(d_0,d_1,\ldots)$ with $d_0\in \pi_0(R)$ nilpotent mod-$p$ and with $d_1\in \pi_0(R)^\times$. We will usually denote the Cartier-Witt divisor by the map $I\to W(R)$.
\end{definition}

\begin{definition}
[Prismatizations]
The $p$-adic formal prestack $\Int_p^{\Prism}$ (also known as the \defnword{Cartier-Witt stack} $\mathrm{WCart}$) is the fpqc sheaf on $p$-complete animated commutative rings whose values on $R$ are given by the $\infty$-groupoid of Cartier-Witt divisors on $R$. Over $\Int_p^{\Prism}$ we have the ring stack $\Ga^{\Prism}$ associating with any Cartier-Witt divisor $I\to W(R)$ the quotient $\overline{W(R)}$. For any $p$-complete animated commutative ring $C$, we can now use the process of \emph{transmutation} to get its \defnword{prismatization} $C^\Prism$, which is the formal (derived) prestack over $\Int_p^\Prism$ parameterizing maps $C\to \overline{W(R)} = \Ga^\Prism(R)$ of $p$-complete animated commutative rings. 
\end{definition}

\begin{remark}
   The prismatization of $\Int_p$ is the Cartier-Witt stack. We have $\Ga^\Prism = (\Int_p[x]^{\wedge}_p)^{\Prism}$.
\end{remark}

\begin{remark}
   The assignment $\Spf C\to C^{\Prism}$ is an \'etale sheaf that preserves all limits and takes $p$-completely \'etale maps to $p$-completely \'etale maps of formal $p$-adic stacks. In particular, one can extend it to a functor $X\to X^{\Prism}$ on $p$-adic formal (derived) algebraic spaces such that $(\Spf C)^{\Prism} = C^{\Prism}$.
\end{remark}

\begin{remark}
   [$\delta$-structure on $C^\Prism$]
\label{rem:delta_structure_prismatization}
There is a canonical endomorphism $\varphi:C^\Prism\to C^\Prism$ lifting the Frobenius endomorphism of the $\Field_p$-stack $C^\Prism\otimes\Field_p$: This takes a Cartier-Witt divisor $I\to W(R)$ to its pullback $F^*I\to W(R)$ under the endomorphism $F:W(R)\to W(R)$ and replaces $C\to \overline{W(R)}$ with its composition with the map $\overline{W(R)}\to W(R)/{}^{\mathbb{L}}F^*I$ induced by $F$. 
\end{remark}

\begin{definition}
[The Hodge-Tate locus]
The \defnword{Hodge-Tate locus} $\Int_p^{\mathrm{HT}}\to \Int_p^\Prism$ is the locus where the map $I\otimes_{W(R)}R\to R$ vanishes. For any $R$, we set $R^{\mathrm{HT}} = R^\Prism\times_{\Int_p^\Prism}\Int_p^{\mathrm{HT}}$, and refer to it as the Hodge-Tate locus of $R^\Prism$.
\end{definition}

\begin{remark}
\label{rem:hodge-tate_locus}
There is a canonical map $\Spf \Int_p\to \Int_p^{\mathrm{HT}}$ given by the Cartier-Witt divisor $W(\Int_p)\xrightarrow{V(1)}W(\Int_p)$. This presents $\Int_p^{\mathrm{HT}}$ as the classifying stack $B\Gm^\sharp$. See~\cite[Theorem 3.4.13]{bhatt2022absolute}.
\end{remark}

\begin{construction}
   [The de Rham point]
\label{const:de_rham_point}
There is a canonical map $x_{\dR}:\Spf C\to C^\Prism$ obtained as follows: The underlying Cartier-Witt divisor is $W(C)\xrightarrow{p}W(C)$, and the map $C\to W(C)/{}^{\mathbb{L}}p$ is obtained as the one canonically factoring the composition $W(C)\xrightarrow{F}W(C)\to W(C)/{}^{\mathbb{L}}p$.
\end{construction}

\begin{remark}
[Prisms and the prismatization]
\label{rem:prisms_and_prismatization}
Suppose that we have $(A,I,R\to \overline{A})$ in the (animated) absolute prismatic site for $R$. Endow $A$ with the $(p,I)$-adic topology. As in~\cite[Construction 3.10]{bhatt2022prismatization}, we find a canonical map $\iota_{(A,I)}:\Spf A\to R^{\Prism}$ classifying the Cartier-Witt divisor $I\otimes_AW(A)\to W(A)$ obtained from the prism structure on $A$, along with the structure map 
\[
R\xrightarrow{\overline{\varphi}}\overline{A}\to \overline{A}\otimes_AW(R)=W(R)/^{\mathbb{L}}(I\otimes_AW(R)). 
\]
\end{remark}

The next result follows from~\cite[Corollary 7.18]{bhatt2022prismatization}. See also~\cite[Theorem 3.3.7]{holeman2023derived}.
\begin{theorem}
[Prismatizations via prismatic cohomology I]
\label{thm:bl_holeman}
Suppose that $R$ is semiperfectoid. Then the absolute prismatic cohomology $\Prism_R$ underlies the initial (animated) prism $(\Prism_R,I_R,R\to \overline{\Prism}_R)$ for $R$, and the associated map $\Spf \Prism_R \to R^\Prism$ is an isomorphism.
\end{theorem}

\begin{definition}
   A map $R\to S$ in $\mathrm{CRing}^{p\text{-comp}}$ is \defnword{quasisyntomic} if it is $p$-completely flat (that is, $S/{}^{\mathbb{L}}p$ is flat over $R/{}^{\mathbb{L}}p$), and if $\mathbb{L}_{S/R}$ has $p$-complete Tor amplitude $[-1,0]$: that is, $\mathbb{L}_{S/R}/{}^{\mathbb{L}}p$ has Tor amplitude $[-1,0]$ over $S/{}^{\mathbb{L}}p$. The map $R\to S$ is a \defnword{quasisyntomic cover} if it is quasisyntomic and $S/{}^{\mathbb{L}}p$ is faithfully flat over $R/{}^{\mathbb{L}}p$.
\end{definition}

\begin{proposition}
[Prismatizations via prismatic cohomology II]
\label{prop:qsyn_to_flat_covers}
For any $R\in \mathrm{CRing}^{p\text{-comp}}$, let $R\to R_{\infty}$ be a quasisyntomic cover with $R_\infty$ semiperfectoid. Then the map $\Spf \Prism_{R_\infty} \simeq R_{\infty}^\Prism\to R^\Prism$ is a flat cover. If $R$ is also semiperfectoid, then the map is in fact $(p,I)$-completely faithfully flat.
\end{proposition}
\begin{proof}
   See~\cite[Lemma 6.3]{bhatt2022prismatization} for the first statement and~\cite[Corollary 6.12.7]{gmm} for the second.
\end{proof}

\begin{remark}
[Relative affineness of prismatizations and base-change for prisms]
\label{rem:base_change_prisms}
Suppose that $S$ is an $R$-algebra such that $S^\Prism\to R^\Prism$ is relatively formally affine. If we have $(A,I,R\to\overline{A})$ in the absolute prismatic site for $R$, then the base-change
\[
\Spf A\times_{R^\Prism}S^\Prism\to S^\Prism
\]
is of the form $\Spf B$ for some $(p,I)$-complete $A$-algebra $B$. Moreover, $B$ is in fact a(n animated) $\delta$-ring: The Frobenius lift is obtained from that on $A$ and the endomorphism $\varphi$ of $S^{\Prism}$. Therefore, $B$ underlies a prism $(B,J)$ over $(A,I)$ equipped with a map $S\to \overline{B}$. 
\end{remark}

\begin{remark}
[Semiperfectoid base-change for prisms]
\label{rem:base_change_prisms_semiperfectoid}
If $S$ is a semiperfectoid $R$-algebra, then $S^\Prism\to R^\Prism$ is relatively formally affine, and so Remark~\ref{rem:base_change_prisms} applies. Indeed, it suffices to check this after $p$-completely faithfully flat base-change. Here, we can choose a quasisyntomic cover $R\to R_\infty$ such that $R_\infty$ is semiperfectoid. In this case, $R_\infty^\Prism\to R^\Prism$ is a flat cover by Proposition~\ref{prop:qsyn_to_flat_covers}, and
\[
R_{\infty}^\Prism\times_{R^\Prism}S^\Prism\simeq \Spf\Prism_{R_\infty}\times_{R^\Prism}\Spf\Prism_S\to S^\Prism
\]
is affine by~\cite[Corollary 3.2.9]{bhatt2022absolute}.
\end{remark}

\begin{remark}
[Flat covers of prisms from quasisyntomic maps]
\label{rem:quasisyntomic_base_change}
Choose a quasisyntomic cover $R\to R_\infty$ with $R^{\otimes_R m}_{\infty}$ semiperfectoid for all $m\ge 1$.  Then by Remark~\ref{rem:base_change_prisms_semiperfectoid}, we have
\[
\Spf A\times_{R^\Prism}(R^{\otimes_R \bullet + 1}_\infty)^{\Prism}\simeq \Spf(A^{(\bullet)}_\infty)
\]
for a cosimplicial prism $(A^{(\bullet)}_\infty,I^{(\bullet)}_\infty)$ such that $A^{(m)}_\infty \simeq A^{\otimes_A (m+1)}_{\infty}$. Here we have set $(A_\infty,I_\infty) \defn (A^{(0)}_\infty,I^{(0)}_\infty)$: this is a prism over $(\Prism_{R_\infty},I_{R_\infty})$. Moreover $\Spf A_\infty\to \Spf A$ is $(p,I)$-completely faithfully flat. All these assertions follow from Proposition~\ref{prop:qsyn_to_flat_covers}.
\end{remark}

\begin{remark}
[Base-change for prisms along closed immersions]
\label{rem:base_change_prisms_closed_immersions}
Suppose that $R\to S$ is a surjective map. Then $S^\Prism\to R^\Prism$ is relatively formally affine and so Remark~\ref{rem:base_change_prisms} applies. To see this, we can use Proposition~\ref{prop:qsyn_to_flat_covers} to reduce to the case where $R$, and therefore $S$, are semiperfectoid, and here the result is clear from Theorem~\ref{thm:bl_holeman}.
\end{remark}

\subsection{Syntomification}

Here, we review some facts about the Nygaard filtered prismatization and syntomification of $p$-complete rings. Since we will not need many details about these stacks in what follows, we will be somewhat terse, and direct the reader to the references given above for precise definitions and explanation of the notation used.

\begin{remark}
   [Filtered prismatization of $\Int_p$]
To begin we have the \defnword{filtered prismatization} $\Int_p^{\mathcal{N}}$, which is a $p$-adic formal prestack over $\Aff^1/\Gm\times \Int_p^\Prism$. For its definition on classical inputs, see~\cite[\S 5.3]{bhatt_lectures} and for its values on animated inputs, see~\cite[Definition 6.4.4]{gmm}. Over this prestack we have a \defnword{filtered Cartier-Witt divisor}, which is a map $M\xrightarrow{d} W$ of $W$-module schemes that is the fiber of a map $W\to W/_dM$ of animated $W$-algebras (these are all sheaves in the flat topology over $\Int_p^{\mathcal{N}}$). This map sits in a commutative diagram of $W$-modules with exact rows
\begin{align}\label{eqn:filt_cartier-witt_diagram}
\begin{diagram}
    L\otimes\Ga^\sharp&\rTo&M&\rTo&F_*M'\\
    \dTo^{t^\sharp}&&\dTo^d&&\dTo^{F_*d'}\\
    \Ga^\sharp&\rTo&W&\rTo^F&F_*W
\end{diagram}
\end{align}
where $t:L\to \Ga$ is the tautological line bundle with cosection over $\Aff^1/\Gm$ and $M'\to W$ is obtained by sheafifying the tautological Cartier-Witt divisor over $\Int_p^\Prism$ into a map of $W$-modules.
\end{remark}

\begin{remark}
   [The de Rham and Hodge-Tate embeddings]
There are two physically disjoint open immersions $j_{\dR},j_{\mathrm{HT}}:\Int_p^{\Prism}\to \Int_p^{\mathcal{N}}$. The first is the pre-image of the open substack $\Gm/\Gm\subset \Aff^1/\Gm$, and is the locus where the right square in~\eqref{eqn:filt_cartier-witt_diagram} is Cartesian. The second is the locus where $M\xrightarrow{d}W$ is obtained from a Cartier-Witt divisor and $d' = F^*d$. 
\end{remark}

\begin{definition}
   [Filtered prismatization of $p$-complete rings]
The sheaf of animated commutative rings $W/_dM$ is represented by a ring stack $\Ga^{\mathcal{N}}\to \Int_p^{\mathcal{N}}$. We can therefore use transmutation to associate with each $C\in \mathrm{CRing}^{p\text{-comp}}$ its \defnword{filtered prismatization} $C^{\mathcal{N}}\to \Int_p^{\mathcal{N}}$ which parameterizes maps of $p$-complete animated commutative rings $C\to \Ga^{\mathcal{N}}(R)$. 
\end{definition}

\begin{construction}
   [The filtered de Rham point]
\label{const:filtered_de_rham}
There is a canonical map $x^{\mathcal{N}}_{\dR}:\Aff^1/\Gm\times \Spf C\to C^{\mathcal{N}}$ whose restriction over the open $\Gm/\Gm\times\Spf C$ is the de Rham point $x_{\dR}$ from Construction~\ref{const:de_rham_point}. The underlying filtered Cartier-Witt divisor associates with every cosection $t:L\to R$ over a $p$-nilpotent $C$-algebra $R$ the map 
\[
F_*W\oplus (L\otimes \Ga^\sharp)\xrightarrow{(V, \mathrm{can}\circ t^\sharp)}W
\]
where $\mathrm{can}:\Ga^\sharp = W[F]\to W$ is the canonical map. The quotient by this map is also a quotient of $\Ga$, and so its $R$-points are naturally equipped with the structure of a $C$-algebra.
\end{construction}

\begin{definition}
   [Syntomification of $p$-complete rings]
The de Rham and Hodge-Tate embeddings for $\Int^{\mathcal{N}}$ pullback to disjoint open immersions $j_{\dR},j_{\mathrm{HT}}:C^{\Prism}\to C^{\mathcal{N}}$. The \defnword{syntomification} $C^{\mathrm{syn}}$ is the coequalizer in $p$-adic formal stacks. of these open immersions.
\end{definition}

\begin{remark}
   [Nygaard filtered absolute prismatic cohomology]
In~\cite[\S 5.5]{bhatt2022absolute}, Bhatt and Lurie construct the \defnword{Nygaard filtration} $\Fil^\bullet_{\mathcal{N}}\Prism_R$ on absolute prismatic cohomology. When $R$ is quasiregular semiperfectoid (qrsp), we have
\[
\Fil^i_{\mathcal{N}}\Prism_R = \{x\in \Prism_R:\;\varphi(x)\in \Fil^i_{I_R}\Prism_R\},
\]
where $\Fil^\bullet_{I_R}\Prism_R$ is the $I_R$-adic filtration. When $R$ is perfectoid, then $\Fil^\bullet_{\mathcal{N}}\Prism_R$ has the structure of a $(p,I_R)$-complete filtered animated commutative ring; see~\cite[Lemma 6.11.6]{gmm}. In general, for any semiperfectoid ring $R$, the Frobenius lift $\varphi:\Prism_R\to \Prism_R$ lifts to a map
\[
   \Phi:\Fil^\bullet_{\mathcal{N}}\Prism_R\to \Fil^\bullet_{I_R}\Prism_R
\]
of $(p,I_R)$-complete filtered animated commutative rings.
\end{remark}

\begin{remark}
   [Rees stacks]
\label{rem:nygaard_rees_stack}
We refer the reader to~\cite[\S 4.3,5.2]{gmm} for the conventions on filtered animated commutative rings and Rees stacks used here. For semiperfectoid $R$, associated with $\Fil^\bullet_{\mathcal{N}}\Prism_R$ is the Rees stack $\Rees(\Fil^\bullet_{\mathcal{N}}\Prism_R)$. This is a formal stack over $\Aff^1/\Gm$, and is equipped with two open immersions
\[
   \sigma,\tau:R^\Prism\simeq \Spf \Prism_R\to \Rees(\Fil^\bullet_{\mathcal{N}}\Prism_R),
\]
where $\tau$ is the pullback of $\Gm/\Gm\to \Aff^1/\Gm$, while $\sigma$ is obtained as the composition
\[
   \Spf \Prism_R\xrightarrow{\simeq}\Rees(\Fil^\bullet_{I_R,\pm}\Prism_R) \hookrightarrow \Rees(\Fil^\bullet_{I_R}\Prism_R) \xrightarrow{\Rees(\Phi)}\Rees(\Fil^\bullet_{\mathcal{N}}\Prism_R).
\]
Here, $\Fil^\bullet_{I_R,\pm}\Prism_R$ is the two-sided $I_R$-adic filtration on $\Prism_R$ given by $\Fil^m_{I_R,\pm}\Prism_R = I_R^{\otimes m}$ for all integers $m$.
\end{remark}

\begin{remark}
\label{rem:sigma_tau_pullbacks}
Quasicoherent sheaves over $\Rees(\Fil^\bullet_{\mathcal{N}}\Prism_R)$ are equivalent as a symmetric monoidal $\infty$-category to that of $(p,I_R)$-complete filtered complexes over $\Fil^\bullet_{\mathcal{N}}\Prism_R$. Pullback along $\tau$ amounts to forgetting the filtration, and pullback along $\sigma$ amounts to filtered base-change to $\Fil^\bullet_{I_R,\pm}\Prism_R$ followed by taking $\Fil^0$. Symbolically one can write this as
\[
   \Fil^\bullet \mathcal{M}\mapsto \mathcal{M}_\sigma = \sum_mI^{-\otimes m}\varphi^*\Fil^m \mathcal{M}.
\]
\end{remark}

\begin{theorem}
[Filtered prismatization of semiperfectoid rings]
   \label{thm:semiperfectoid_nygaard}
Suppose that $R$ is a semiperfectoid ring. Then there is an isomorphism of $\Aff^1/\Gm$-stacks
\[
   \Rees(\Fil^\bullet_{\mathcal{N}}\Prism_R)\xrightarrow{\simeq}R^{\mathcal{N}}
\]
identifying $\tau$ with $j_{\dR}$ and $\sigma$ with $j_{\mathrm{HT}}$
\end{theorem}
\begin{proof}
   See~\cite[Theorem 6.11.7]{gmm}.
\end{proof}

\begin{proposition}
   \label{prop:nygaard_qsynt_descent}
If $R\to S$ is a quasisyntomic cover, then $S^{\mathcal{N}}\to R^{\mathcal{N}}$ is surjective in the $p$-completely flat topology. In fact, if $R$ and $S$ are semiperfectoid, then this map is faithfully flat. Moreover, there exists a quasisyntomic cover $R\to R_\infty$ such that $R_\infty^{\otimes_R m}$ is semiperfectoid for all $m$.
\end{proposition}
\begin{proof}
   See~\cite[Corollaries 6.12.3, 6.12.5]{gmm}.
\end{proof}

\subsection{$F$-gauges}
\label{subsec:f_gauges}

\begin{definition}
   An \defnword{$F$-gauge} over $R$ is a quasicoherent sheaf $\mathcal{M}$ over $R^{\mathrm{syn}}$ (over $R^{\mathrm{syn}}\otimes\Int/p^n\Int)$. It is \defnword{perfect} (resp. \defnword{a vector bundle $F$-gauge}) if it is a perfect complex (resp. a vector bundle) over $R^{\mathrm{syn}}$. One obtains corresponding notions for $F$-gauges of \defnword{level $n$} by replacing $R^{\mathrm{syn}}$ with $R^{\mathrm{syn}}\otimes\Int/p^n\Int$ in the definitions.
\end{definition}

\begin{remark}
   When $R = \kappa$ is a perfect field, this notion is an instance of the \emph{Frobenius gauges} or \emph{$\varphi$-gauges} introduced by Fontaine and Jannsen~\cite{Fontaine2021-pr}. The general definition here was introduced by Bhatt-Lurie under the term \emph{prismatic $F$-gauges}~\cite[Definition 6.1.1]{bhatt_lectures}. For economy of language, we have dropped the adjective `prismatic' here.
\end{remark}

\begin{construction}
\label{const:fil_hdg}
   Given an $F$-gauge $\mathcal{M}$ over $R$, its pullback along the map $x^{\mathcal{N}}_{\dR}$ from Construction~\ref{const:filtered_de_rham} yields a quasicoherent sheaf over $\Aff^1/\Gm\times\Spf R$, which is equivalent to a $p$-complete filtered complex $\Fil^\bullet_{\mathrm{Hdg}}M$ over $R$.
\end{construction}

\begin{example}
   [Breuil-Kisin twist]
Over $\Int_p^{\mathrm{syn}}$ we have a canonical line bundle, the \defnword{Breuil-Kisin twist} $\mathcal{O}\{1\}$; see~\cite[\S 6.6]{gmm} for a quick summary of its construction and properties. We will denote its pullback over $R^{\mathrm{syn}}$ by the same symbol. For any $F$-gauge $\mathcal{M}$ over $R$, we will set $\mathcal{M}\{1\} = \mathcal{M}\otimes \mathcal{O}\{1\}$.
\end{example}

\begin{definition}
   The \defnword{Hodge-Tate weights} of an $F$-gauge $\mathcal{M}$ are the integers $m$ such that $\gr^{-m}_{\mathrm{Hdg}}M$ is not nullhomotopic.
\end{definition}

\begin{remark}
   With this convention, the Breuil-Kisin twist has Hodge-Tate weight $1$.
\end{remark}

\begin{remark}
   [$F$-gauges over a semiperfectoid ring]
Using Theorem~\ref{thm:semiperfectoid_nygaard} and Remark~\ref{rem:sigma_tau_pullbacks}, we see that giving an $F$-gauge over a semiperfectoid ring $R$ is equivalent to giving a $(p,I_R)$-complete filtered complex $\Fil^\bullet \mathcal{M}$ over $\Fil^\bullet_{\mathcal{N}}\Prism_R$ equipped with an isomorphism $\mathcal{M}_\sigma\simeq \mathcal{M}$. 
\end{remark}

\begin{remark}
   [Quasisyntomic descent]
By Proposition~\ref{prop:nygaard_qsynt_descent}, $F$-gauges satisfy quasisyntomic descent, and one can use this descent to reduce many questions to the situation of semiperfectoid rings, which is addressed by the previous remark.
\end{remark}

\begin{remark}
[Cohomology in the quasisyntomic site]
\label{rem:qsynt_site}
The assignment
\[
\Fil^\bullet_{ \mathcal{N}}\Prism_{-}:C\mapsto \Fil^\bullet_{\mathcal{N}}\Prism_C
\]
is a sheaf of animated filtered commutative rings over the site $R_{\mathrm{qsyn}}$ formed by semiperfectoid algebras $C$ that are quasisyntomic over $R$, equipped with the $p$-quasisyntomic topology. An $F$-gauge $\mathcal{M}$ can be viewed as a sheaf of filtered modules 
\[
\Fil^\bullet \mathcal{M}_{-}: C\mapsto \Fil^\bullet \mathcal{M}_{C}
\]
over this sheaf. Moreover the structure of an $F$-gauge on $\mathcal{M}$ yields maps
\[
   \varphi_i:\Fil^i \mathcal{M}_{-} \to \mathcal{I}^{\otimes i}\otimes_{\Prism_{-}}\mathcal{M}_{-}
\]
where $\mathcal{I}$ is the sheafification of the assignment $C\mapsto I_C$ on semiperfectoid $R$-algebras $C$. Unwinding definitions, one now finds a canonical isomorphism
\[
R\Gamma(R^{\mathrm{syn}},\mathcal{M})\xrightarrow{\simeq}\fib\left(R\Gamma_{\mathrm{qsyn}}(\Spf R,\Fil^0 \mathcal{M}_{-})\xrightarrow{\varphi_0 - \mathrm{can}}R\Gamma_{\mathrm{qsyn}}(\Spf R,\mathcal{M}_{-})\right),
\]
where $\mathrm{can}:\Fil^0 \mathcal{M}_{-}\to \colim_{n\mapsto -\infty}\Fil^n \mathcal{M}_{-}\simeq \mathcal{M}_{-}$ is the natural map.
\end{remark}

\section{Representability and Dieudonn\'e theory}
\label{sec:recollection_of_some_representability_results}

The purpose of this section is to recall some results from~\cite{gmm} regarding the representability of stacks associated with perfect $F$-gauges and the classification of truncated Barsotti-Tate group schemes in terms of $F$-gauges.

\subsection{Stacks of perfect $F$-gauges of Hodge-Tate weights $0,1$}
For integers $a\leq b$, consider the $p$-adic formal prestack $\mathrm{Perf}^{\mathrm{syn},[a,b]}_{n,\{0,1\}}$: This associates with every $R\in \mathrm{CRing}^{p\text{-nilp}}$ the $\infty$-groupoid $\mathrm{Perf}^{[a,b]}_{\{0,1\}}(R^{\mathrm{syn}}\otimes\Int/p^n\Int)_{\simeq}$ of perfect $F$-gauges of level $n$ over $R$ with Hodge-Tate weights $0,1$ and with Tor amplitude in $[a,b]$. 

Over this prestack we have a canonical filtered perfect complex $\Fil^\bullet_{\mathrm{Hdg}}M_{\mathrm{taut}}$ obtained via Construction~\ref{const:fil_hdg}.

\begin{theorem}
\label{thm:HTwts01_representable}
The prestack $\mathrm{Perf}^{\mathrm{syn},[a,b]}_{n,\{0,1\}}$ is represented by a $p$-adic formal locally finitely presented derived Artin stack over $\Int_p$ with cotangent complex $(\gr^{-1}_{\mathrm{Hdg}}M_{\mathrm{taut}})^\vee\otimes \Fil^0_{\mathrm{Hdg}}M_{\mathrm{taut}}$. Moreover, if $(C'\twoheadrightarrow C,\gamma)$ is a nilpotent divided power thickening of $p$-complete algebras, then we have a Cartesian square
\[
\begin{diagram}
\mathrm{Perf}^{\mathrm{syn},[a,b]}_{n,\{0,1\}}(C')&\rTo&\mathrm{Perf}^{[a,b]}_{\{0,1\}}(\Aff^1/\Gm\times\Spec C'/{}^{\mathbb{L}}p^n)\\
\dTo&&\dTo\\
\mathrm{Perf}^{\mathrm{syn},[a,b]}_{n,\{0,1\}}(C)&\rTo&\mathrm{Perf}^{[a,b]}_{\{0,1\}}(\Aff^1/\Gm\times\Spec C/{}^{\mathbb{L}}p^n)\times_{\mathrm{Perf}^{[a,b]}(C/{}^{\mathbb{L}}p^n)}\mathrm{Perf}^{[a,b]}(C'/{}^{\mathbb{L}}p^n).
\end{diagram}
\]
\end{theorem}
\begin{proof}
This follows from~\cite[Theorem 8.13.1]{gmm} and its proof. Here, $\mathrm{Perf}^{[a,b]}_{\{0,1\}}(\Aff^1/\Gm\times\Spec B)$ is the $\infty$-groupoid of filtered perfect complexes over $B$ with Tor amplitude in $[a,b]$, and with their associated graded complexes supported in graded degrees $0,-1$. 
\end{proof}

\subsection{Sections of $F$-gauges of Hodge-Tate weights $\le 1$}

\begin{theorem}
\label{thm:perfect_f-gauges_repble}
Let $\mathcal{M}$ be a perfect $F$-gauge of level $n$ over $R\in \mathrm{CRing}^{p\text{-comp}}$ with Hodge-Tate weights bounded by $1$; then the $p$-adic formal prestack given on $\mathrm{CRing}^{p\text{-nilp}}_{R/}$ given by
\[
C\mapsto \tau^{\le 0}R\Gamma\left(C^{\mathrm{syn}}\otimes\Int/p^n\Int,\mathcal{M}\vert_{C^{\mathrm{syn}}\otimes\Int/p^n\Int}\right)
\]
is represented by a finitely presented derived $p$-adic formal Artin stack $\Gamma_{\mathrm{syn}}(\mathcal{M})$ over $\Spf R$. Moreover, if $(C'\twoheadrightarrow C,\gamma)$ is a nilpotent divided power thickening of in $\mathrm{CRing}^{p\text{-comp}}_{R/}$, then we have a Cartesian square
\[
\begin{diagram}
\Gamma_{\mathrm{syn}}(\mathcal{M})(C')&\rTo&C'\otimes_R\Fil^0_{\mathrm{Hdg}}M_n\\
\dTo&&\dTo\\
\Gamma_{\mathrm{syn}}(\mathcal{M})(C)&\rTo&(C\otimes_R\Fil^0_{\mathrm{Hdg}}M_n)\times_{C\otimes_RM_n}(C'\otimes_RM_n).
\end{diagram}
\]
Here, $\Fil^\bullet_{\mathrm{Hdg}}M_n$ is the filtered perfect complex over $R$ obtained by viewing $\mathcal{M}$ as a perfect complex over $R^{\mathrm{syn}}$ and pulling back along $x^{\mathcal{N}}_{\dR,R}$. In particular, $\Gamma_{\mathrm{syn}}(\mathcal{M})$ admits a perfect tangent complex over $R$ given by the pullback of $\gr^{-1}_{\mathrm{Hdg}}M_n[-1]$. 
\end{theorem}
\begin{proof}
See ~\cite[Theorem 8.12.1]{gmm}.
\end{proof}

\begin{corollary}
[Deformation theory for syntomic cohomology]
\label{cor:def_theory_f_gauges}
With the notation as above, let $\iota_n:R^{\mathrm{syn}}\otimes\Int/p^n\Int\to R^{\mathrm{syn}}$ be the tautological closed immersion, and set $I = \fib(C'\twoheadrightarrow C)$. Then we have a canonical fiber sequence
\[
R\Gamma(C^{',\mathrm{syn}},\iota_{n,*}\mathcal{M}\vert_{C^{',\mathrm{syn}}}) \to R\Gamma(C^{\mathrm{syn}},\iota_{n,*}\mathcal{M}\vert_{C^{\mathrm{syn}}})\to I\otimes_R\gr^{-1}_{\mathrm{Hdg}}M_n.
\]
\end{corollary}  
\begin{proof}
Apply the theorem to $\mathcal{M}[j]$ for $j\ge 1$ to get canonical fiber sequences of $(-j)$-connective animated abelian groups
\[
\tau^{\le j}R\Gamma(C^{',\mathrm{syn}},\iota_{n,*}\mathcal{M}\vert_{C^{',\mathrm{syn}}})\to \tau^{\le j}R\Gamma(C^{\mathrm{syn}},\iota_{n,*}\mathcal{M}\vert_{C^{\mathrm{syn}}})\to I\otimes_R\tau^{\leq j}\gr^{-1}_{\mathrm{Hdg}}M_n.
\]
Now, one finishes by taking the colimit over $j$.
\end{proof}

\begin{remark}
[Shape of stacks given by syntomic cohomology]
\label{rem:f-gauges_repble_close_look}
The proof of Theorem~\ref{thm:perfect_f-gauges_repble} in~\cite{gmm} yields more information. When $R$ is an $\Field_p$-algebra and $n=1$, then we have a Cartesian diagram (see the proof of~\cite[Corollary 8.7.6]{gmm})
\[
\begin{diagram}
\Gamma_{\mathrm{syn}}(\mathcal{F})&\rTo&\Gamma_{\Fzip}(\mathsf{F})\\
\dTo&&\dTo\\
\Gamma_{\Fzip}(\mathsf{F})&\rTo&\mathsf{S}_{(N,\psi)[1]},
\end{diagram} 
\]
where the other objects involved are as follows:
\begin{enumerate}
   \item $\Gamma_{\Fzip}(\mathsf{M})$ is the \emph{$F$-zip cohomology} associated with an $F$-zip as defined in~\cite{Pink2015-ye}. Explicitly, its values on an $R$-algebra $C$ are given by 
   \[
    \Gamma_{\Fzip}(\mathsf{F})(C) \simeq \tau^{\le 0}\fib\left(C\otimes_R\left(\Fil^{\mathrm{conj}}_0 F \times_F \Fil^0_{\mathrm{Hdg}}F\right)\to C\otimes_R\gr^{\mathrm{conj}}_0\mathsf{F}\right)
   \]
   where $\Fil^\bullet_{\mathrm{Hdg}}F$ and $\Fil^{\mathrm{conj}}_0F$ are filtered perfect complexes over $R$ with the same underlying perfect complex $F$.

   \item $\mathsf{S}_{(N,\psi)[1]}$ is a derived Artin stack over $\Gamma_{\Fzip}(\mathsf{F})$, whose classical truncation is in the category of (classical) $\Gamma_{\Fzip}(\mathsf{F})_{\mathrm{cl}}$-stacks obtained by taking the subcategory spanned by height $1$ finite flat $p$-torsion group schemes and their iterated classifying stacks, and then closing this up under pullbacks; see the proof of~\cite[Theorem 7.1.5]{gmm}. This stack was first systematically considered by Bragg-Olsson~\cite{bragg2021representability}. Explicitly, it is the cofiber of a map
   \[
   N\otimes_RZ^1_{\Prism} \to N\otimes_RH^1_{\Prism},
   \]
   where $N$ is a perfect complex over $R$, $H^1_{\Prism}$ is the quasisyntomic sheaf of $R$-modules $C\mapsto \mathbb{L}_{C/\Field_p}$, and $Z^1_{\Prism}$ is another quasisyntomic sheaf of $R$-modules sitting in a fiber sequence
   \[
    \fib(\Ga\xrightarrow{\varphi}\Ga)\to Z^1_{\Prism}\to H^1_{\Prism}.
   \]
\end{enumerate}
\end{remark}

\begin{remark}
[Fpqc descent for syntomic cohomology]
\label{rem:fpqc_descent_for_syntomic_cohomology}
The stack $\Gamma_{\mathrm{syn}}(\mathcal{F})$ satisfies $p$-complete fpqc descent. To see this, one can reduce using derived descent to the case of $\Field_p$-algebras, and then by d\'evissage to the case where $n=1$. Here, the description in Remark~\ref{rem:f-gauges_repble_close_look} shows that it suffices to establish fpqc descent for $\Gamma_{\Fzip}(\mathsf{F})$ and for $\mathsf{S}_{(N,\psi)}[1]$ separately. Given the explicit descriptions of these sheaves in the remark, we are reduced to checking that perfect complexes and the cotangent complex satisfy fpqc descent, and this is well-known. 

Alternatively, one can also follow the proof of \cite[Proposition 7.4.7]{bhatt2022absolute} and prove the more general result that $\Gamma_{\mathrm{syn}}(\mathcal{F})$ satisfies $p$-complete fpqc descent for \emph{any} $F$-gauge $\mathcal{F}$, with bounded above Hodge-Tate weights. To do this, one uses the description of syntomic cohomology given in Remark~\ref{rem:qsynt_site}, but shows in addition that one can use the Nygaard \emph{completed} sheaves instead. With this, the desired descent statement can be ultimately reduced as in the argument of Bhatt-Lurie to the fact that the $p$-complete cotangent complex of an animated commutative ring tensored with a module satisfies $p$-complete fpqc descent (see \cite[Proposition~3.2]{lm24}).
\end{remark}

\subsection{Dieudonn\'e theory and $F$-gauges}

The following theorem is a generalization of work of Ansch\"utz-Le Bras~\cite{MR4530092}:
\begin{theorem}
\label{thm:dieudonne}
Suppose that $R$ is a classical $p$-complete ring in $\mathrm{CRing}^{p\text{-comp}}$.  Set 
\[
\mathrm{Vect}_{\{0,1\}}(R^{\mathrm{syn}}\otimes\Int/p^n\Int) = \mathrm{Perf}^{[0,0]}_{\{0,1\}}(R^{\mathrm{syn}}\otimes\Int/p^n\Int).
\]
Then the functor $\mathcal{M}\mapsto \Gamma_{\mathrm{syn}}(\mathcal{M})$ yields a canonical equivalence of ($\infty$-)categories
\[
\mathrm{Vect}_{\{0,1\}}(R^{\mathrm{syn}}\otimes\Int/p^n\Int)\xrightarrow{\simeq}\BT[]{n}(R),
\]
where the right hand side is category of $n$-truncated Barsotti-Tate groups over $R$. This equivalence is compatible with arbitrary base-change and satisfies fpqc descent. Moreover, it is compatible with Cartier duality: There is a canonical non-degenerate pairing
\[
\Gamma_{\mathrm{syn}}(\mathcal{M}^\vee\{1\})\times\Gamma_{\mathrm{syn}}(\mathcal{M})\to \mu_{p^n}
\]
yielding an isomorphism $\Gamma_{\mathrm{syn}}(\mathcal{M}^\vee\{1\})\xrightarrow{\simeq}\Gamma_{\mathrm{syn}}(\mathcal{M})^*$.
\end{theorem}
\begin{proof}
This is~\cite[Theorem 11.1.4]{gmm}. The only point to be explained perhaps is the compatibility with arbitrary base-change and fpqc descent. This is a by-product of the fact that the equivalence is obtained via an isomorphism of $p$-adic formal (smooth) algebraic stacks $\mathrm{Vect}^{\mathrm{syn}}_{n,\{0,1\}}\xrightarrow{\simeq}\BT[]{n}$, where the source is the stack $\mathrm{Perf}^{\mathrm{syn},[0,0]}_{n,\{0,1\}}$ from Theorem~\ref{thm:HTwts01_representable} with a different name, and where the target is the stack of $n$-truncated Barsotti-Tate group schemes.
\end{proof}

We now present some complements to the theorem, beginning with the observation that fppf cohomology of $n$-truncated Barsotti-Tate group schemes can be computed using syntomic cohomology.
\begin{proposition}
[Syntomic cohomology and fppf cohomology]
\label{prop:syntomic_and_fppf}
Suppose that $R$ is $p$-nilpotent. If $G$ is an $n$-truncated Barsotti-Tate group scheme over $R$ presented as $G = \Gamma_{\mathrm{syn}}(\mathcal{M})$ for $\mathcal{M}\in \mathrm{Vect}_{\{0,1\}}(R^{\mathrm{syn}}\otimes\Int/p^n\Int)$, then there is a canonical isomorphism
\[
R\Gamma_{\mathrm{fppf}}(\Spec R,G)\xrightarrow{\simeq}R\Gamma(R^{\mathrm{syn}},\iota_{n,*}\mathcal{M}) \simeq R\Gamma(R^{\mathrm{syn}}\otimes\Int/p^n\Int,\mathcal{M}).
\]
In fact, for every $m\ge 1$ there is a canonical isomorphism of smooth Artin $m$-stacks
\begin{align}\label{eqn:higher_classifying_stacks}
B^mG \xrightarrow{\simeq}\Gamma_{\mathrm{syn}}(\mathcal{M}[m]).
\end{align}
\end{proposition}
\begin{proof}
The second assertion implies the first: Indeed, evaluating both sides of~\eqref{eqn:higher_classifying_stacks} on $R$ and shifting gives isomorphisms
\[
\tau^{\leq m}R\Gamma_{\mathrm{fppf}}(\Spec R,G)\xrightarrow{\simeq}\tau^{\le m}R\Gamma(R^{\mathrm{syn}},\iota_{n,*}\mathcal{M}) 
\]
whose colimit over $m$ yields the isomorphism in the first assertion.

Now, note that $B^mG$ is smooth over $R$ (by the flatness of $G$), and the description of the tangent complex of $G = \Gamma_{\mathrm{syn}}(\mathcal{M})$ in Theorem~\ref{thm:perfect_f-gauges_repble} shows that the tangent complex of $B^mG$ is given by the pullback of $\gr^{-1}_{\mathrm{Hdg}}M[-m+1]$, a perfect complex with Tor amplitude in $[-m,-m+1]$. The same theorem also shows that $\Gamma_{\mathrm{syn}}(\mathcal{M}[m])$ is a locally finitely presented derived algebraic stack with the same tangent complex. In particular, this implies that $\Gamma_{\mathrm{syn}}(\mathcal{M}[m])$ is in fact a \emph{smooth} (and hence classical) Artin $m$-stack over $R$.

Now, $B^mG$ is the fppf sheafification of the $\Mod[cn]{\Int/p^n\Int}$-valued presheaf 
\[
C\mapsto \mathcal{G}(C)[m] \simeq \left(\tau^{\leq 0}R\Gamma(C^{\mathrm{syn}},\mathcal{M})\right)[m],
\]
on $R$-algebras. The canonical map
\[
\left(\tau^{\le 0}R\Gamma(C^{\mathrm{syn}},\mathcal{M})\right)[m]\to (\tau^{\le m}R\Gamma(C^{\mathrm{syn}},\mathcal{M}))[m]
\]
sheafifies to a natural map of smooth Artin stacks $B^mG\to \Gamma_{\mathrm{syn}}(\mathcal{M}[m])$, which the description of tangent complexes above shows is actually an \'etale map. To finish, it is enough to know that, when $C =\kappa$ is an algebraically closed field, $H^i(\kappa^{\mathrm{syn}},\mathcal{M})$ vanishes for $i>0$. By d\'evissage, we can reduce what remains to be shown to the case $n=1$. Now, giving $\mathcal{M}$ is equivalent to giving maps of finite dimensional $\kappa$-vector spaces
\[
M \xrightarrow{V}\varphi^*M\xrightarrow{F}M
\]
such that the sequence is exact in the middle. Furthermore, we have
\[
R\Gamma(\kappa^{\mathrm{syn}},\mathcal{M})\simeq \fib(M \xrightarrow{V - \varphi^*}\varphi^*M).
\]
Therefore, we have to know the surjectivity of $V-\varphi^*$, which reduces to the classical fact that, for any $\varphi$-semilinear endomorphism $T:\kappa^h\to \kappa^h$, the map
\[
T-\mathrm{id}:\kappa^h\to \kappa^h
\]
is surjective.
\end{proof}

\begin{remark}
\label{rem:crystalline_realization_bt}
   By~\cite[Remark 6.8.2]{gmm}, every vector bundle $F$-gauge over $R^{\mathrm{syn}}\otimes\Int/p^n\Int$ gives rise to a crystal of vector bundles $\Dieu(\mathcal{M})$ over the mod-$p^n$ crystalline site of $\Spf R$ relative to $\Int_p$ with its $p$-adic divided powers. Moreover, this crystal is \emph{filtered} in the sense that its evaluation $\Dieu(\mathcal{M})(C')$ on any (animated) divided power thickening of $p$-nilpotent $R$-algebras $C'\twoheadrightarrow C$ admits a canonical Hodge filtration $\Fil^\bullet_{\mathrm{Hdg}}\Dieu(\mathcal{M})(C')$ compatible with the (animated) divided power filtration on $C'$. Concretely, this is obtained via quasisyntomic descent from the universal divided power thickening $A_{\mathrm{crys}}(C) \simeq \Prism_C\twoheadrightarrow C$ of semiperfect $R$-algebras $C$, where the filtration arises from the filtered $\Fil^\bullet_{\mathcal{N}}\Prism_C$-module underlying the restriction of $\mathcal{M}$ to $C^{\mathrm{syn}}\otimes\Int/p^n\Int$. On the other hand, if $\mathcal{M}$ is associated with an $n$-truncated Barsotti-Tate group scheme $G$ by Theorem~\ref{thm:dieudonne}, then we also have the Dieudonn\'e crystal $\Dieu(G^*)$ defined in~\cite[\S 3.1, Th\'eor\`eme 3.3.10]{bbm:cris_ii}, given as the internal Ext sheaf $\underline{\Ext}^1(G^*,\mathcal{O}^{\mathrm{crys}})$ in the big crystalline topos of $\Spf R$ equipped with the fppf topology with $\mathcal{O}^{\mathrm{crys}}$ its structure sheaf. This has a filtration by a subsheaf $\underline{\Ext}^1(G^*,\mathcal{I}^{\mathrm{crys}})$, where $\mathcal{I}^{\mathrm{crys}}\subset \mathcal{O}^{\mathrm{crys}}$ is the tautological divided power ideal sheaf in the crystalline topos, whose evaluation on $C'\twoheadrightarrow C$ gives a canonical submodule $\Fil^0_{\mathrm{Hdg}}\Dieu(G^*)(C')\subset \Dieu(G^*)(C')$.
\end{remark}

\begin{proposition}
[The filtered Dieudonn\'e crystal via $F$-gauges]
\label{prop:lie_complex}In the above set up, there is a canonical isomorphism of crystals
\[
   \Dieu(\mathcal{M})\xrightarrow{\simeq}\Dieu(G^*).
\]
mapping $\Fil^0_{\mathrm{Hdg}}\Dieu(\mathcal{M})(C')$ isomorphicaly onto $\Fil^0_{\mathrm{Hdg}}\Dieu(G^*)(C')$ for all divided power thickenings $C'\twoheadrightarrow C$ of $p$-nilpotent $R$-algebras. In particular, if $\Lie(G)\defn e^*\mathbb{L}^\vee_{G/R}$ is the Lie complex of $G$ over $R$ (where $e\in G(R)$ is the identity section), then there is a canonical isomorphism of perfect complexes
\[
\gr^{-1}_{\mathrm{Hdg}}M_n\xrightarrow{\simeq}\Lie(G)[1].
\]
\end{proposition}
\begin{proof}
This follows from the argument used for~\cite[Proposition 11.7.4]{gmm}, and reduces in the end to the comparison between prismatic and crystalline Dieudonn\'e theory appearing in~\cite[\S 4.3]{MR4530092} or \cite{Mondal2024-cy} (cf.~Remark \ref{compatiblity1}).
\end{proof}

The next result---which is immediate from Propositions~\ref{prop:syntomic_and_fppf} and~\ref{prop:lie_complex} and Corollary~\ref{cor:def_theory_f_gauges}---should be compared with~\cite[Theorem 5.2.8]{Cesnavicius2019-pn}. In \emph{loc. cit.}, one finds the statement for arbitrary finite locally free commutative group schemes, but only for square-zero extensions. The proof there is via the B\'egueri resolution by smooth affine commutative group schemes. We will recover the more general assertion in Corollary~\ref{cor:mazur_roberts_ffg} below using our main classification theorem.
\begin{corollary}
[The Mazur-Roberts carpet]
\label{cor:deformation_theory_bt_ns}
Suppose that $(C'\twoheadrightarrow C,\gamma)$ is a nilpotent divided power extension of discrete $p$-nilpotent $R$-algebras with $I = \ker(C'\twoheadrightarrow C)$. Then we have a canonical fiber sequence
\[
R\Gamma_{\mathrm{fppf}}(\Spec C',G)\to R\Gamma_{\mathrm{fppf}}(\Spec C,G)\to \Lie(G)\otimes_RI[1]
\]
\end{corollary}

\begin{remark}
[\'Etale realization of $F$-gauges]
\label{rem:etale_realization_bt}
Via~\cite[Constructions 6.3.1, 6.3.2]{bhatt_lectures}, we find a canonical functor
\[
T_{\et}: \mathrm{Perf}(R^{\mathrm{syn}})\to D^b_{\mathrm{lisse}}(\Spf(R)^{\ad}_\eta,\Int_p)
\]
where the right hand side is the bounded derived category of lisse $\Int_p$-sheaves on the adic generic fiber of $\Spf R$. Since both sides satisy quasisyntomic descent, it suffices to specify the equivalence for semiperfectoid $R$. This is given as
\[
T_{\et}(\mathcal{M}) = \left(\mathcal{M}[1/I_R]^{\wedge}\right)^{\varphi = \mathrm{id}}.
\]
Here, we are viewing $\mathcal{M}$ as a perfect complex of $\Prism_R$-modules, where $(\Prism_R,I_R)$ is the initial prism for $R$, and $\mathcal{M}[1/I_R]^{\wedge}$ is the $p$-adic completion of $\mathcal{M}[1/I]$. The $F$-gauge structure on $\mathcal{M}$ yields a canonical $\varphi$-semilinear isomorphism $\mathcal{M}[1/I_R]^{\wedge}\xrightarrow{\simeq} \mathcal{M}[1/I_R]^{\wedge}$, and $T_{\et}(\mathcal{M})$ is the module of invariants for this operator. Note that the $p$-adic completion is not required if $\mathcal{M}$ has cohomology sheaves supported along the $p=0$ locus.
\end{remark}  

\begin{proposition}
[Compatibility with \'etale realization]
 \label{prop:etale_realization_bt}
With the notation of Proposition~\ref{prop:syntomic_and_fppf}, suppose that $R$ is instead $p$-complete. Then there is a canonical isomorphism $T_{\et}(\mathcal{M})\simeq G^{\ad}_\eta$, where $G^{\ad}_\eta$ is the adic generic fiber of $G$, viewed as a perfect complex of lisse $\Int_p$-sheaves over $\Spf(R)^{\ad}_\eta$.
 \end{proposition}
 \begin{proof}
   It is enough to consider the case where $R$ is quasisyntomic (since the moduli stack of $n$-truncated BTs is $p$-completely smooth), where this follows from~\cite[Proposition 3.99]{Mondal2024-cy}.
 \end{proof}  

\begin{remark}
[Compatibility with \'etale cohomology]
\label{rem:etale_realization_cohom}
The assignment 
\[
\mathcal{M}_{-}[1/\mathcal{I}]:C\mapsto \mathcal{M}_C[1/I_C]
\]
is a sheaf on $R_{\mathrm{qsyn}}$. We now have a canonical isomorphism~\cite[Remark 3.100]{Mondal2024-cy}
\[
R\Gamma_{\mathrm{qsyn}}(\Spf R,\mathcal{M}_{-}[1/\mathcal{I}])^{\varphi = \mathrm{id}}\xrightarrow{\simeq}R\Gamma_{\et}(\Spf(R)^{\ad}_\eta,G^{\ad}_\eta).
\]
Note that the $p$-adic completion is unnecessary here since we are dealing with $\Int/p^n\Int$-modules.
\end{remark}    

\section{Exactness}
\label{sec:exactness_of_the_syntomic_dieudonne_functor}

\subsection{From extensions of $F$-gauges to extensions of BTs}

Our goal for this section is to prove the following result:

\begin{theorem}
[Exactness]
\label{thm:exactness}
The functor $\mathcal{M}\mapsto \Gamma_{\mathrm{syn}}(\mathcal{M})$ from Theorem~\ref{thm:dieudonne} is an \emph{exact} equivalence. More precisely, a sequence
\[
\mathcal{M}_1\to \mathcal{M}_2\to \mathcal{M}_3
\]
in $\mathrm{Vect}^{\mathrm{syn}}_{n,\{0,1\}}(R)$ is a fiber sequence if and only if 
\[
0\to \Gamma_{\mathrm{syn}}(\mathcal{M}_1)\to \Gamma_{\mathrm{syn}}(\mathcal{M}_2)\to \Gamma_{\mathrm{syn}}(\mathcal{M}_3)\to 0
\]
is an exact sequence of $n$-truncated Barsotti-Tate group schemes over $R$.
\end{theorem}

\begin{remark}
\label{rem:exts_to_ext}
Suppose that $\mathcal{M}_1\to \mathcal{M}_2\to \mathcal{M}_3$ is a fiber sequence: It corresponds to a map $\mathcal{M}_3\to \mathcal{M}_1[1]$, which gives rise via Proposition~\ref{prop:syntomic_and_fppf} to a map
\[
\Gamma_{\mathrm{syn}}(\mathcal{M}_3)\to \Gamma_{\mathrm{syn}}(\mathcal{M}_1[1])\xleftarrow{\simeq}B\Gamma_{\mathrm{syn}}(\mathcal{M}_1)
\]
classifying the exact sequence
\[
0\to \Gamma_{\mathrm{syn}}(\mathcal{M}_1)\to \Gamma_{\mathrm{syn}}(\mathcal{M}_2)\to \Gamma_{\mathrm{syn}}(\mathcal{M}_3)\to 0
\]
of fppf sheaves. The difficult part of the theorem now is to show that exactness of syntomic cohomology implies that the original sequence is a fiber sequence.
\end{remark}

\begin{remark}
[The case of qrsp rings]
\label{rem:exactness_qrsp}
If $R = \kappa$ is an algebraically closed (or even perfect) field, then the theorem holds. More generally, when $R$ is a qrsp ring, we have an explicit inverse to $\Gamma_{\mathrm{syn}}$, described in~\cite[\S 3]{Mondal2024-cy} (see also~\cite[\S 11.5]{gmm}), which carries exact sequences to fiber sequences~\cite[Remark 3.83]{Mondal2024-cy}. Therefore, the theorem holds in this case as well. Our plan of action now is to reduce to this situation by studying the geometry of the stack of extensions of truncated Barsotti-Tate groups.
\end{remark}

\begin{remark}
[Translation to a geometric statement]
\label{rem:geometric_statement}
Set $\mathcal{N} = \mathcal{M}_3^\vee\otimes \mathcal{M}_1[1]$: this is a perfect $F$-gauge over $R$ with Hodge-Tate weights bounded by $1$. Therefore, by Theorem~\ref{thm:perfect_f-gauges_repble}, $\Gamma_{\mathrm{syn}}(\mathcal{N})$ is represented by a $p$-adic formal locally finitely presented derived Artin stack over $R$ with tangent complex $\gr^{-1}_{\mathrm{Hdg}}N_n[-1]$. Since $\mathcal{N}[-1]$ is a vector bundle over $R^{\mathrm{syn}}\otimes\Int/p^n\Int$, we see that the tangent complex is in fact connective, and so $\Gamma_{\mathrm{syn}}(\mathcal{N})$ is a \emph{smooth}, hence classical, formal Artin stack over $R$. It parameterizes, for each $C\in \mathrm{CRing}^{p\text{-nilp},f}_{R/}$, maps of vector bundles 
\[
\mathcal{M}_3\vert_{C^{\mathrm{syn}}\otimes\Int/p^n\Int}\to \mathcal{M}_1[1]\vert_{C^{\mathrm{syn}}\otimes\Int/p^n\Int}.
\]
Giving such a map is precisely equivalent to giving a fiber sequence of vector bundles
\[
\mathcal{M}_1\vert_{C^{\mathrm{syn}}\otimes\Int/p^n\Int}\to \mathcal{M}_2 \to \mathcal{M}_3\vert_{C^{\mathrm{syn}}\otimes\Int/p^n\Int}.
\]

On the other hand, set $G_i = \Gamma_{\mathrm{syn}}(\mathcal{M}_i)$. Then we also have the classical formal algebraic stack $\underline{\Ext}(G_3,G_1)$ over $R$ parameterizing extensions of $G_3$ by $G_1$ in fppf sheaves of $\Int/p^n\Int$-modules. Remark~\ref{rem:exts_to_ext} can be interpreted as giving a map of formal classical algebraic stacks 
\begin{align}
\label{eqn:stack_map_extensions}
\Gamma_{\mathrm{syn}}(\mathcal{N})_{\mathrm{cl}}\to \underline{\Ext}(G_3,G_1).
\end{align}
Theorem~\ref{thm:exactness} would follow if we knew that this is an isomorphism.

By taking $\Spf R$ to be a formal affine cover of the smooth formal stack $\mathrm{Vect}^{\mathrm{syn}}_{[0,1],n}\times \mathrm{Vect}^{\mathrm{syn}}_{[0,1],n}$, we can assume that $R$ is $p$-completely smooth. Further, Remark~\ref{rem:exactness_qrsp}, combined with quasisyntomic descent, tells us that
\[
\Gamma_{\mathrm{syn}}(\mathcal{N})(C)\to \underline{\Ext}(G_3,G_1)(C)
\]
is an isomorphism for $p$-quasisyntomic $R$-algebras $C$. If we knew that $\underline{\Ext}(G_3,G_1)$ is also smooth, then it would follow that~\eqref{eqn:stack_map_extensions} is an isomorphism.
\end{remark}

\subsection{Smoothness of the stack of extensions}

\begin{remark}
[Truncating extensions of BT group schemes]
Suppose that $G(n)$ and $H(n)$ are $n$-truncated Barsotti-Tate groups, and suppose that we have a short exact sequence $0\to H(n)\to T(n)\to G(n)\to 0$ of  fppf sheaves of $\Int/p^n\Int$-modules over $R$. Then one has the following observations:
\begin{enumerate}
   \item $T(n)$ is an $n$-truncated Barsotti-Tate group.
   \item If $n>k\ge 1$ and we use $(k)$ to denote the $p^k$-torsion subgroups, then the sequence $0\to H(k)\to T(k)\to G(k)\to 0$ is a short exact sequence of $k$-truncated Barsotti-Tate groups.
\end{enumerate}
See~\cite[Lemme 3.3.9]{bbm:cris_ii}.
\end{remark}

By Remark~\ref{rem:geometric_statement}, the following theorem implies Theorem~\ref{thm:exactness}.
\begin{theorem}
\label{thm:smoothness_of_extensions}
Suppose that $G(n),H(n)$ are two $n$-truncated Barsotti-Tate group schemes over $R$. Then:
\begin{enumerate}
   \item  $\underline{\Ext}(G(n),H(n))$ is $p$-completely smooth over $\Spf R$.
   \item If $n\ge 2$, the natural map
   \[
    \underline{\Ext}(G(n),H(n))\xrightarrow{T(n)\mapsto T(n-1)} \underline{\Ext}(G(n-1),H(n-1))
   \]
   is smooth and surjective.
\end{enumerate}
\end{theorem}

\begin{remark}
The statement of the above theorem was anticipated by Grothendieck a long time ago: See the discussion on p. 10 of~\cite{IllusieChicago}, which references p. 15 of Grothendieck's letter~\cite{GrothLetter}. The method of proof we follow here is essentially Grothendieck's as explained in~\cite{MR0801922}: It was used there to prove (among other things) the smoothness of the stack of truncated Barsotti-Tate groups. The key point is to set up the obstruction theory correctly, which we do in Proposition~\ref{prop:obst_theory} below.
\end{remark}

\begin{lemma}
[Surjectivity on geometric points]
\label{lem:surj_geometric_points}
For any algebraically closed field $\kappa$ over $R$ the map
\[
\underline{\Ext}(G(n),H(n))(\kappa)\xrightarrow{T(n)\mapsto T(n-1)} \underline{\Ext}(G(n-1),H(n-1))(\kappa)
\]
is surjective
\end{lemma}
\begin{proof}
Without loss of generality we can assume $R = \kappa$. By the discussion in Remark~\ref{rem:geometric_statement}, this would follow from the following assertion: Suppose that $\mathcal{Q}$ is a vector bundle $F$-gauge of level $n$ over $\kappa$; then the map 
\[
\Gamma_{\mathrm{syn}}(\mathcal{Q}[1])(\kappa)\to \Gamma_{\mathrm{syn}}(\mathcal{Q}[1]\otimes_{\Int/p^n\Int}\Int/p^{n-1}\Int)(\kappa)
\]
is surjective. This reduces to knowing that we have
\[
H^1(\kappa^{\mathrm{syn}},\mathcal{Q}[1]\otimes_{\Int/p^n\Int}\Int/p\Int) = H^2(\kappa^{\mathrm{syn}},\mathcal{Q}\otimes_{\Int/p^n\Int}\Int/p\Int)=0,
\]
which is clear since $R\Gamma(\kappa^{\mathrm{syn}},\mathcal{Q}\otimes_{\Int/p^n\Int}\Int/p\Int)$ is computed by a two-term complex whose terms are in degrees $0,1$.
\end{proof}

\begin{proposition}
[Ext against coherent sheaves]
\label{prop:ext_coherent_sheaf}
Suppose that $M$ is an $R$-module and that $p^nR  = 0$. Viewing $M$ as a quasicoherent fppf sheaf over $\Spec R$:
\begin{enumerate}
   \item There is a canonical isomorphism
   \[
        \underline{\Ext}^1_{\Int}(G(n),M)\xrightarrow{\simeq}\underline{\Hom}_{\Int/p^n\Int}(G(n),M),
   \]

\item We have\footnote{This is the internal Ext sheaf computed in the category of fppf $\Int/p^n\Int$-modules over $\Spec R$.}
   \[
      \underline{\Ext}^1_{\Int/p^n\Int}(G(n),M) = 0.
   \]
   \item The natural map
   \[
      \underline{\Ext}^2_{\Int/p^n\Int}(G(n),M)\to \underline{\Ext}^2_{\Int}(G(n),M)
   \]
   is an isomorphism.
\end{enumerate}
\end{proposition}
\begin{proof}
   This follows from~\cite[Proposition 2.2.4]{MR0801922}. 
\end{proof}

\begin{notation}
   For any finite flat group scheme $H$ over a base $R$ set
   \[
       t_{H} = H^0(\Lie(H))\;;\; v_{H} = H^1(\Lie(H)).
   \]
\end{notation}

\begin{lemma}
[Tangent complexes of BTs]
\label{lem:tangent_complex_bts}
Consider the tangent complex $\Lie(H(n))$ over $R$, and let $m\ge 1$ be such that $p^m = 0\in R$ with $n\ge m$. Then:
\begin{enumerate}
   \item $\Lie(H(n))$ is perfect with Tor amplitude in $[0,1]$. In fact, $t_{H(n)}$ and $v_{H(n)}$ are both finite locally free over $R$ and the canonical fiber sequence
\[
   t_{H(n)}\to \Lie(H(n))\to v_{H(n)}[-1].
\]
is split. 

   \item For $m\le n'\le n$ (resp. $m\le n'\leq n-m)$, the map $t_{H(n')}\to t_{H(n)}$ (resp. $v_{H(n')}\to v_{H(n)}$) induced from the inclusion $H(n')\subset H(n)$ as the $p^{n'}$-torsion subgroup is an isomorphism (resp. identically zero).

   \item For $m\le n'\le n$ (resp. $m\leq n'\leq n-m$) the map $v_{H(n)}\to v_{H(n')}$ (resp. $t_{H(n)}\to t_{H(n')}$) induced by the multiplication-by-$p^{n-n'}$ map $H(n)\to H(n')$ is an isomorphism (resp. identically zero).
\end{enumerate}
\end{lemma}
\begin{proof}
   See~\cite[Proposition 2.2.1]{MR0801922}.
\end{proof}

\begin{lemma}
   \label{lem:ext_coherent_sheaf}
Suppose that we have a surjection of rings $R\to R_0$ with kernel $I$, and suppose that $n\ge m$. Then for $A = \Int,\Int/p^n\Int$, we have canonical short exact sequences
   \begin{align*}
  0\to \Ext^1_{A}(G(n),I\otimes_R t_{H(n)})\to \Ext^1_{A}(G(n),I\otimes_R \Lie(H(n)))\to \Hom_{A}(G(n),I\otimes_R v_{H(n)})\to 0\\
  0\to \Ext^2_{A}(G(n),I\otimes_R t_{H(n)})\to \Ext^2_{A}(G(n),I\otimes_R \Lie(H(n)))\to \Ext^1_{A}(G(n),I\otimes_R v_{H(n)})\to 0.
   \end{align*}
Moreover, for $M(n) = I\otimes_R t_{H(n)}$ or $M(n) = I\otimes_R v_{H(n)}$, the following hold:
\begin{enumerate}
   \item $\Ext^1_{\Int}(G(n),M)\xrightarrow{\simeq}\Hom_{\Int/p^n\Int}(G(n),M)$;
   \item $\Ext^1_{\Int/p^n\Int}(G(n),M) = 0$;
   \item $\Ext^2_{\Int/p^n\Int}(G(n),M)\xrightarrow{\simeq}\Ext^2_{\Int}(G(n),M)$.
\end{enumerate}
\end{lemma}
\begin{proof}
  The first part is clear from (1) of Lemma~\ref{lem:tangent_complex_bts}. For the numbered assertions, first note that $\underline{\Hom}_A(G(n),M)\simeq t_{G(n)^*}\otimes_RM$ is a quasicoherent sheaf over $\Spec R$; see~\cite[(2.1.6.1)]{MR0801922}. The lemma now follow from Proposition~\ref{prop:ext_coherent_sheaf} by taking global sections and observing that the quasicoherent sheaf $\underline{\Ext}^i_A(G(n),M)$ for $A = \Int,\Int/p^n\Int$ and $i=0,1$ has no higher cohomology.
\end{proof}

\begin{proposition}
[Obstruction theory for extensions of BTs]
\label{prop:obst_theory}
Suppose that $R$ is an $\Field_p$-algebra and that $\varpi:R\twoheadrightarrow R_0$ is a square zero extension with kernel $I$. Use a subscript $0$ to denote base-change of objects from $R$ to $R_0$. Suppose that we have $T_0(n)\in \underline{\Ext}(G_0(n),H_0(n))(R_0)$ corresponding to a map $f_0(n):G_0(n)\to H_0(n)[1]$ of complexes of fppf sheaves over $R_0$. Then there exists a canonical class
\[
\mathrm{ob}(f_0(n))\in \Ext^2_{\Int/p^n\Int}(G_0(n),I\otimes_R t_{H(n)})\simeq \Ext^2_{\Int}(G_0(n),I\otimes_R t_{H(n)})
\]
with the following properties:
\begin{enumerate}
   \item $T_0(n)$ lifts to $T(n)\in \underline{\Ext}(G(n),H(n))(R)$ if and only if $\mathrm{ob}(f_0(n))$ vanishes.
   
   \item  The class $\mathrm{ob}(f_0(n))$ maps to $\mathrm{ob}(f_0(n-1))$ under the restriction map
   \[
     \Ext^2_{\Int}(G_0(n),I\otimes_Rt_{H(n)})\to \Ext^2_{\Int}(G_0(n-1),I\otimes_Rt_{H(n)})\xleftarrow{\simeq}\Ext^2_{\Int}(G_0(n-1),I\otimes_R t_{H(n-1)})
   \]
   induced by the inclusion $G_0(n)\subset G_0(n-1)$. Here, we are using the following observation from Lemma~\ref{lem:tangent_complex_bts}: The map $t_{H(n-1)}\to t_{H(n)}$ induced from the inclusion $H(n-1)\subset H(n)$ is an isomorphism. 

   \item If $\mathrm{ob}(f_0(n))$ vanishes, then the groupoid of lifts $T(n)$ is a torsor under 
\begin{equation}\label{eqn:ext1_isom}
   \Ext^1_{\Int/p^n\Int}(G_0(n),I\otimes_R\Lie(H(n)))\simeq \Hom_{\Int/p^n\Int}(G_0(n),v_{H_0(n)})
\end{equation}
and the map from the space of lifts of $T_0(n)$ to those of $T_0(n-1)$ is equivariant for the isomorphism
\begin{align}\label{eqn:hom_isom}
   \Hom_{\Int/p^n\Int}(G_0(n),I\otimes_R v_{H_0(n)})&\xrightarrow{\simeq} \Hom_{\Int/p^{n-1}\Int}(G_0(n)\otimes^{\mathbb{L}}_{\Int/p^n\Int}\Int/p^{n-1}\Int,I\otimes_R v_{H_0(n)})\\ 
   &\xrightarrow{\simeq}\Hom_{\Int/p^{n-1}\Int}(G_0(n-1),I\otimes_R v_{H_0(n-1)})\nonumber
\end{align}
induced by the multiplication-by-$p$ map $H_0(n)\to H_0(n-1)$.
\end{enumerate}

\end{proposition} 
\begin{proof}
   Let $\iota:(\Spec R_0)_{\mathrm{fppf}}\to (\Spec R)_{\mathrm{fppf}}$ be the map of fppf sites associated with $\varpi$. From Corollary~\ref{cor:deformation_theory_bt_ns} or directly from~\cite[Theorem 5.2.8]{Cesnavicius2019-pn}, one finds a fiber sequence of complexes of fppf sheaves
   \[
      H(n) \to R\iota_*H_0(n)\to R\iota_*(I\otimes_R\Lie(H(n))[1]),
   \]
   where we are treating $I\otimes_R\Lie(H(n))[1]$ as a complex of quasi-coherent sheaves over $(\Spec R_0)_{\mathrm{fppf}}$.

   Applying $\mathrm{RHom}_{\Int/p^n\Int}(G(n),\_\_)$ to this fiber sequence, using the standard adjunction between $\iota^*$ and $R\iota_*$, and taking cohomology yields an exact sequence
   \[
       \Ext^1_{\Int/p^n\Int}(G(n),H(n))\to \Ext^1_{\Int/p^n\Int}(G_0(n),H_0(n))\to \Ext^2_{\Int/p^n\Int}(G_0(n),I\otimes_R\Lie(H(n)))
   \]
   The class $\mathrm{ob}(f_0(n))$ that we seek is exactly where the class of $T_0(n)$ in $\Ext^1_{\Int}(G_0(n),H_0(n))$ ends up on the right hand side via the isomorphism
   \[
      \Ext^2_{\Int/p^n\Int}(G_0(n),I\otimes_R\Lie(H(n)))\xleftarrow{\simeq}\Ext^2_{\Int/p^n\Int}(G_0(n),I\otimes_Rt_{H(n)})
   \]
   obtained from Lemma~\ref{lem:ext_coherent_sheaf}. It is clear that this class has property (1), and it's also not hard to check that it has property (2).

   As for property (3), first note that~\eqref{eqn:ext1_isom} is a consequence of Lemma~\ref{lem:ext_coherent_sheaf}. Moreover, the first isomorphism in~\eqref{eqn:hom_isom} follows from (3) of Lemma~\ref{lem:tangent_complex_bts} and the second from the flatness of $G(n)$ over $\Int/p^n\Int$. Now, we use the exact sequence 
  \begin{align*}
   \Hom(G(n),H(n))\to \Hom(G_0(n),H_0(n))\to \Ext^1_{\Int/p^n\Int}(G_0(n),I\otimes_R\Lie(H(n)))\to\cdots\\
   \cdots\to \Ext^1_{\Int/p^n\Int}(G(n),H(n))\to \Ext^1_{\Int/p^n\Int}(G_0(n),H_0(n)).
  \end{align*}
  Assuming the vanishing of the obstruction class, this shows that the set of isomorphism classes of lifts of $T_0(n)$ is a torsor under the image of $\Ext^1_{\Int}(G_0(n),I\otimes_R\Lie(H(n)))$, and one now upgrades to the statement about the groupoid of lifts in a standard way. For the stated equivariance, note that, for any lift $T(n)$, multiplication-by-$p$ yields a commutative diagram with exact rows where the top right and the bottom left squares are Cartesian:
  \[
     \begin{diagram}
        0&\rTo &H(n)&\rTo&T(n)&\rTo&G(n)&\rTo&0\\
        &&\dEquals&&\dTo&&\dTo_p\\
         0&\rTo &H(n)&\rTo&T'(n)&\rTo&G(n-1)&\rTo&0\\
         &&\dTo_p&&\dTo&&\dEquals\\
        0&\rTo &H(n-1)&\rTo&T(n-1)&\rTo&G(n-1)&\rTo&0
     \end{diagram}
  \]
\end{proof}

\begin{remark}
   [Construction of obstruction classes using higher stacks]
\label{rem:higher_stacks_ob_const}
There is an alternative to the use of the Mazur-Roberts carpet that proceeds by looking at the deformation theory of maps of higher classifying stacks. Here is a sketch:
\begin{enumerate}
   \item We begin by assuming that we have $p$-divisible groups $G,H$ over $R$ such that $G[p^n] = G(n)$ and $H[p^n] = H(n)$, and that we have an extension $T_0\in \underline{\Ext}(G,H)(R_0)$ such that $T_0[p^n] \simeq T_0(n)\in \underline{\Ext}(G(n),H(n))(R_0)$. 

   \item We now note that the obstruction to lifting the map of fppf sheaves of animated abelian groups $f_0(n):G_0(n)\to H_0(n)[1]$ classifying $T_0$ vanishes if and only if the obstruction to lifting the map $B^mG_0(n)\to B^{m+1}H_0(n)$ of higher classifying stacks vanishes for $m\ge 1$. 

   \item By standard deformation theory this latter obstruction class lives in 
   \[
      \Ext^1_{B^mG_0(n)}(f_0(n)^*\mathbb{L}_{B^{m+1}H_0(n)/R_0},\Reg{B^mG_0(n)}\otimes I) = H^{m+2}(B^mG_0(n),I\otimes_R\Lie(H(n))).
   \]

   \item Using classical facts about the cohomology of Eilenberg-MacLane spaces \cite{MR38072}, one sees that there is a short exact sequence
   \begin{equation}\label{eqn:M2_ext_sequence}
      0\to \Ext^2_{\Int}(G_0(n),I\otimes_R\Lie(H(n)))\to H^{m+2}(B^mG_0(n),I\otimes_R\Lie(H(n)))\to \Hom_{\Int}(M_2\otimes^{\mathbb{L}}_{\Int}G_0(n),I\otimes_R\Lie(H(n)))\to 0,
   \end{equation}
   where $M_2 = \pi_2(\Int\otimes_{\mathbb{S}}\Int)$\footnote{This is a smash product of spectra over the sphere spectrum.} is a finite abelian group. Moreover, we have
   \[
      \Ext^1_{\Int}(G_0(n),I\otimes_R\Lie(H(n)))\simeq H^{m+1}(B^mG_0(n),I\otimes_R\Lie(H(n))).
   \]

   \item Since $M_2$ is perfect with Tor amplitude in $[-1,0]$ as a $\Int$-module with perfect dual $M_2^\vee$ with Tor amplitude $[0,1]$, the right hand side of~\eqref{eqn:M2_ext_sequence} can be rewritten using Lemma~\ref{lem:tangent_complex_bts} as 
   \[
      \Hom_{\Int}(G_0(n),I\otimes_R\Lie(H(n))\otimes_{\Int}M_2^\vee)\xrightarrow{\simeq}\Hom_{\Int}(G_0(n),I\otimes_Rt_{H(n)}\otimes_{\Int}M_2^\vee).
   \]
   \item Now, one uses the functoriality of the obstruction class and the invariance of $t_{H(n)}$ under passing to the $p^k$-torsion subgroups $H(k)\subset H(n)$  to show that the image of the obstruction class in $\Hom_{\Int}(G_0(n),I\otimes_Rt_{H(n)}\otimes_{\Int}M_2^\vee)$vanishes, and so it actually lives in $\Ext^2_{\Int}(G_0(n),I\otimes_R\Lie(H(n)))$. See the argument in the proof of Theorem~\ref{thm:smoothness_of_extensions} below.

   \item A similar functoriality argument also shows that the image of the resulting class under the map
   \[
      \Ext^2_{\Int}(G_0(n),I\otimes_R\Lie(H(n)))\to \Ext^1_{\Int}(G_0(n),I\otimes_Rv_{H(n)})
   \]
   vanishes, finally yielding the obstruction class in
   \[
       \Ext^2_{\Int}(G_0(n),I\otimes_R t_{H(n)})\simeq \Ext^2_{\Int/p^n\Int}(G_0(n),I\otimes_Rt_{H(n)})
   \]
\end{enumerate}
\end{remark}

\begin{proof}
   [Proof of Theorem~\ref{thm:smoothness_of_extensions}]
Given Lemma~\ref{lem:surj_geometric_points},  standard arguments reduce us to showing the following statement: Suppose that we are in the situation of Remark~\ref{rem:higher_stacks_ob_const}. Then, for all $n\ge 2$, $T_0(n-1)$ can be lifted to a BT extension $T(n-1)$ of $H(n-1)$ by $G(n-1)$, and every such lift can further be lifted to a BT extension $T(n)$ of $H(n)$ by $G(n)$. By Proposition~\ref{prop:obst_theory}, for all $n\ge 1$, we have the obstruction class
\[
\mathrm{ob}(f_0(n)) \in \mathrm{Ext}^2_{\Int/p^n\Int}(G_0(n), I\otimes_R t_{H(n)}).
\] 
Given (3) of Proposition~\ref{prop:obst_theory}, it is enough now to show that $\mathrm{ob}(f_0(n)) = 0$ for all $n\ge 1$. This follows from (2) of Proposition~\ref{prop:obst_theory} and the fact that the map
\[
   \Ext^2_{\Int}(G_0(n),I\otimes_Rt_{H(n)})\to \Ext^2_{\Int}(G_0(n-1),I\otimes_Rt_{H(n)})
\]
induced via restriction along $G_0(n-1)\subset G_0(n)$ is identically zero; see~\cite[Corollaire 2.2.7]{MR0801922}.
\end{proof}

\begin{proof}[Proof of Theorem \ref{thm:exactness}]
 Follows from Theorem \ref{thm:smoothness_of_extensions} (see Remark \ref{rem:geometric_statement} and Remark \ref{rem:exactness_qrsp}).   
\end{proof}{}

\section{An analogue of a result of Raynaud}
\label{sec:raynaud}

\subsection{Isogenies}

For any stable $\Int$-linear $\infty$-category $\mathcal{C}$, write $\Mod{\Int/p^n\Int}(\mathcal{C})$ for the stable $\infty$-category of $\Int/p^n\Int$-module objects in $\mathcal{C}$. Effectively, we are looking at the $\infty$-category of objects $M\in \mathcal{C}$ equipped with a nullhomotopy for the endomorphism $p^n:M\to M$ given by multiplication by $p^n$. The following observation will be useful:
\begin{lemma}
\label{lem:isogenies_pn}
Suppose that we have a map $f:M_1\to M_2$ in $\mathcal{C}$. Then the following are equivalent:
\begin{enumerate}
   \item Giving $\cofib(f)$ the structure of an object in $\Mod{\Int/p^n\Int}(\mathcal{C})$.
   \item Giving a map $\hat{f}:M_2\to M_1$ equipped with an isomorphism $f\circ \hat{f} \simeq p^n$ (resp. $\hat{f}\circ f\simeq p^n$) as endomorphisms of $M_2$ (resp. of $M_1$).
\end{enumerate}
\end{lemma}

\begin{definition}
With $\mathcal{C}$ as above, a \defnword{height-$n$ isogeny} is the data of a pair of maps $M_1\xrightarrow{f}M_2\xrightarrow{\hat{f}}M_2$ along with homotopies $f\circ \hat{f}\simeq p^n$ and $\hat{f}\circ f \simeq p^n$. We will denote such an object simply by the pair $(f,\hat{f})$, and refer to it as a height-$n$ isogeny \defnword{between $M_1$ and $M_2$}.
\end{definition}

\subsection{Raynaud's theorem for $F$-gauges}

For any $R\in \mathrm{CRing}^{p\text{-nilp}}$, let $\mathsf{P}^{\mathrm{syn}}_{n,\{0,1\}}(R)$ be the $\infty$-subcategory of $\Mod{\Int/p^n\Int}(\mathrm{Perf}_{\{0,1\}}(R^{\mathrm{syn}}))$ spanned by the objects $\mathcal{F}$ with Tor amplitude in $[-1,0]$. Our goal in this section is to prove the following syntomic analogue of a theorem of Raynaud~\cite[Th\'eor\`eme 3.1.1]{bbm:cris_ii}:

\begin{theorem}
\label{thm:syntomic_raynaud}
Suppose that $R$ is discrete and that we have $\mathcal{M}\in \mathsf{P}^{\mathrm{syn}}_{n,\{0,1\}}(R)$. Then there exist an ind-\'etale cover $R\to \tilde{R}$, vector bundle $F$-gauges
   \[
    \mathcal{V}^{-1},\mathcal{V}^0\in \mathrm{Vect}_{\{0,1\}}(\tilde{R}^{\mathrm{syn}}),
   \]
   of Hodge-Tate weights $0,1$, a map $f:\mathcal{V}^{-1}\to \mathcal{V}^0$, and an equivalence
\[
   \mathcal{M}\vert_{\tilde{R}^{\mathrm{syn}}}\simeq \cofib(\mathcal{V}^{-1}\xrightarrow{f}\mathcal{V}^0)
\]
of perfect complexes over $\tilde{R}^{\mathrm{syn}}$.
\end{theorem}

\begin{remark}
Raynaud's theorem in the context of finite flat group schemes actually allows one to restrict to \emph{Zariski} localization to achieve the analogous result. \emph{A posteriori}, combined with Theorem~\ref{thm:main} below, this implies that Zariski localization is sufficient in Theorem~\ref{thm:syntomic_raynaud} as well. However, our methods, which are moduli theoretic, cannot yield this stronger assertion directly.
\end{remark}

\begin{remark}
   It might be possible to give a more streamlined proof of the theorem using the results of \S~\ref{sec:divided_dieudonne}. However, the more scenic route taken here yields additional information about the moduli of $F$-gauges, such as Corollary~\ref{cor:bartling_hoff} below, and also illustrates principles that will prove useful in other contexts beyond that of this paper.
\end{remark}

The proof will take some preparation. For now, here are some useful corollaries.
\begin{corollary}
\label{cor:bt_in_the_middle}
Suppose that $R$ is discrete and that $\mathcal{M}\in \mathsf{P}^{\mathrm{syn}}_{n,\{0,1\}}(R)$. Then, pro-\'etale locally on $\Spec R$, there exist $\mathcal{V}_n,\mathcal{W}_n\in \mathrm{Vect}^{\mathrm{syn}}_{n,\{0,1\}}(R)$ and cofiber sequences of perfect $F$-gauges
\[
\mathcal{M}\to \mathcal{V}_n\to \mathcal{M}'\;;\; \mathcal{M}'\to \mathcal{W}_n\to \mathcal{M}
\]
with $\mathcal{M}'$ in the image of $\mathsf{P}^{\mathrm{syn}}_{n,\{0,1\}}(R)$.
\end{corollary}
\begin{proof}
We can assume that we have a cofiber sequence
\[
\mathcal{V}^{-1}\xrightarrow{f} \mathcal{V}^0\to \mathcal{M}
\]
with $\mathcal{V}^{-1},\mathcal{V}^0$ in $\mathrm{Vect}^{\mathrm{syn}}_{\{0,1\}}(R)$. Then by Lemma~\ref{lem:isogenies_pn} we have $\hat{f}: \mathcal{V}^{-1}\to \mathcal{V}^0$ with $f\circ \hat{f}$ and $\hat{f}\circ f$ homotopic to multiplication by $p^n$. It is now easy to see that we now have cofiber sequences
\[
\cofib(\hat{f})\to \mathcal{V}^{0}/{}^{\mathbb{L}}p^n\to \mathcal{M}\;;\; \mathcal{M}\to \mathcal{V}^{-1}/{}^{\mathbb{L}}p^n\to \cofib(\hat{f})
\]
\end{proof}

\begin{corollary}
[Discreteness of the categories]
\label{cor:syntomic_raynaud}
Suppose that $R$ is discrete. Then: 
\begin{enumerate}
   \item $\mathsf{P}^{\mathrm{syn}}_{n,\{0,1\}}(R)$ is a classical category and the forgetful functor
\[
\mathsf{P}^{\mathrm{syn}}_{n,\{0,1\}}(R)\to \mathrm{Perf}^{[-1,0]}(R^{\mathrm{syn}})
\]
is fully faithful. 
\item $\mathsf{P}^{\mathrm{syn}}_{n,\{0,1\}}(R)$ has the structure of an exact additive category where the exact sequences are given by fiber sequences of perfect complexes. Moreover, $\mathrm{Vect}_{\{0,1\}}(R^{\mathrm{syn}}\otimes\Int/p^n\Int)$ is a thick subcategory\footnote{By this, we mean that it has the `two-out-of-three' property for exact sequences of objects.} of $\mathsf{P}^{\mathrm{syn}}_{n,\{0,1\}}(R)$.
\end{enumerate}
\end{corollary}
\begin{proof}
Assertion (1) will follow from knowing that the full subcategory of $\mathrm{Perf}^{\mathrm{syn},[-1,0]}_{\{0,1\}}(R)$ spanned by the image of $\mathsf{P}^{\mathrm{syn}}_{n,\{0,1\}}(R)$ has discrete mapping spaces between objects.

By pro-\'etale descent, and Theorem~\ref{thm:syntomic_raynaud}, we can assume that we have cofiber sequences
\[
\mathcal{V}^{-1}\xrightarrow{f} \mathcal{V}^0\to \mathcal{M}\;;\; \mathcal{W}^{-1}\xrightarrow{h} \mathcal{W}^0\to \mathcal{N}.
\]
We want to know that the mapping space $\Map(\mathcal{M},\mathcal{N})$ is $0$-truncated. For this, it is enough to know that the spaces $\Map(\mathcal{V}^i,\mathcal{N})$ are $0$-truncated for $i=0,-1$. This in turn would follow if we knew that the cofiber of
\[
\Map(\mathcal{V}^i,\mathcal{W}^{-1})\to \Map(\mathcal{V}^i,\mathcal{W}^0)
\]
is $0$-truncated. Since multiplication by $p^n$ factors through $h$, we see that there is a sequence of maps
\[
\Map(\mathcal{V}^i,\mathcal{W}^{-1})\to \Map(\mathcal{V}^i,\mathcal{W}^0)\to \Map(\mathcal{V}^i,\mathcal{W}^{-1})
\]
whose composition is $p^n$.

It is now enough to know that $\Map(\mathcal{V}^i,\mathcal{W}^j)$ is a flat $\Int_p$-module for $j=0,-1$. This can be checked `linear algebraically', but it is also immediate from Theorem~\ref{thm:dieudonne}, and the analogous assertion for homomorphisms between $p$-divisible groups.

As for the second assertion, the only part that requires proof is that the subcategory $\mathrm{Vect}^{\mathrm{syn}}_{n,\{0,1\}}(R)$ has the two-out-of-three property. Suppose that we have a fiber sequence
\[
\mathcal{M}_1\to \mathcal{M}_2\to \mathcal{M}_3
\]
in $\mathsf{P}^{\mathrm{syn}}_{n,\{0,1\}}(R)$. If $\mathcal{M}_1$ and $\mathcal{M}_3$ are vector bundles over $R^{\mathrm{syn}}\otimes\Int/p^n\Int$, then it is clear that so is $\mathcal{M}_2$. Similarly clear is the case where $\mathcal{M}_2$ and $\mathcal{M}_3$ are vector bundles. The remaining case now follows from Cartier duality.
\end{proof}

\subsection{Preparations for the proof}

We can now begin preparations for our proof of the theorem. $R$ will always be discrete in what follows.

\begin{lemma}
[The case of a geometric point]
\label{lem:checking_on_points}
The theorem holds when $R = \kappa$ is an algebraically closed field.
\end{lemma}
\begin{proof}
By~\cite[Remarks 3.4.6,4.2.8]{bhatt_lectures}, giving a perfect $F$-gauge of Hodge-Tate weights $0,1$ amounts to giving maps $M^0\xleftrightharpoons[t]{u}M^{-1}$ of perfect complexes over $W(\kappa)$ with $u\circ t = p\cdot\mathrm{id}_{M^0}$ and  $t\circ u = p\cdot\mathrm{id}_{M^{-1}}$, as well as a homotopy equivalence $\varphi^*M^0 \xrightarrow{\simeq}M^{-1}$. The lemma follows easily from this explicit description. 
\end{proof}

\begin{remark}
[$F$-gauges over perfect rings]
\label{rem:F_gauges_perfect}
In fact, the description of perfect $F$-gauges of Hodge-Tate weights $0,1$ used in the proof above is valid with $\kappa$ replaced by any perfect $\Field_p$-algebra $R$. We will see a similar, but somewhat more involved, description that holds for any $R$ in Section~\ref{sec:divided_dieudonne} below.
\end{remark}

\begin{lemma}
\label{lem:enough_to_check_mod_n}
Suppose that $\mathcal{M}\in \mathsf{P}^{\mathrm{syn}}_{[0,1],n}(R)$, and let $\mathcal{M}_n\in \mathrm{Perf}^{\mathrm{syn},[-1,0]}_{n,\{0,1\}}(R)$ be its image. Suppose that we have a height $n$ isogeny $(f,\hat{f})$ between two vector bundle $F$-gauges $\mathcal{V}^{-1}$ and $\mathcal{V}^0$ over $R$ such that the image of $\cofib(f)$ in $\mathrm{Perf}^{\mathrm{syn},[-1,0]}_{n,\{0,1\}}(R)$ is isomorphic to $\mathcal{M}_n$. Then there is an isomorphism $\cofib(f)\xrightarrow{\simeq}\mathcal{M}$ of perfect complexes over $R^{\mathrm{syn}}$.

\end{lemma}
\begin{proof}
Set $\mathcal{M}' = \cofib(f)$ and let $\mathcal{M}'_n$ be its image in $ \mathrm{Perf}^{\mathrm{syn},[-1,0]}_{n,\{0,1\}}(R)$. By hypothesis, we have an isomorphism $\mathcal{M}'_n \xrightarrow{\simeq} \mathcal{M}_n$.

Now, observe that the $\Int/p^n\Int$-module structures on $\mathcal{M}'$ and $\mathcal{M}$ (the former arising from Lemma~\ref{lem:isogenies_pn} and our hypothesis) yield isomorphisms
\[
\mathcal{M}'_n \xrightarrow{\simeq} \mathcal{M}'\oplus \mathcal{M}'[1]\;;\; \mathcal{M}_n \xrightarrow{\simeq} \mathcal{M}\oplus \mathcal{M}[1]
\]
of perfect complexes over $R^{\mathrm{syn}}$. Therefore, we obtain an isomorphism
\[
\mathcal{M}'\oplus \mathcal{M}'[1]\xrightarrow{\simeq}\mathcal{M}\oplus \mathcal{M}[1]
\]
of objects in $\mathrm{Perf}^{\mathrm{syn},[-2,0]}_{\{0,1\}}(R)$. We claim that the resulting map $\mathcal{M}'[1]\to \mathcal{M}$, which is the top-right corner of a `matrix' representation of the isomorphism, is nullhomotopic. This would then imply that the `diagonal' entries $\mathcal{M}'\to \mathcal{M}$ and $\mathcal{M}'[1]\to \mathcal{M}[1]$ are both isomorphisms in $\mathrm{Perf}^{\mathrm{syn},[-1,0]}_{\{0,1\}}(R)$.

It remains to prove the claim. For this, it is enough to know that any map $\mathcal{V}^0[1]\to \mathcal{M}$ is nullhomotopic. Write $\mathcal{F} = \mathcal{V}^{0,\vee}[-1]\otimes \mathcal{M}$: this is a perfect $F$-gauge of Hodge-Tate weights $-1,0,1$ and Tor amplitude $[0,1]$. For $m\ge 1$, let $\mathcal{F}_m$ be the image of $\mathcal{F}$ in $\mathrm{Perf}^{\mathrm{syn},[0,1]}_{m,\{0,1\}}(R)$.

By Theorem~\ref{thm:perfect_f-gauges_repble}, its space of sections 
\[
\Gamma_{\mathrm{syn}}(\mathcal{F}) = \varprojlim_m \Gamma_{\mathrm{syn}}(\mathcal{F}_m)
\]
is a quasisyntomic sheaf, and for every square zero extension $C'\twoheadrightarrow C$ of discrete $R$-algebras with kernel $I$, we have
\[
\Gamma_{\mathrm{syn}}(\mathcal{F})(C')\xrightarrow{\simeq}\tau^{\le 0}\fib\left(\Gamma_{\mathrm{syn}}(\mathcal{F})(C)\to I\otimes_R\gr^{-1}_{\mathrm{Hdg}}F[-1]\right),
\]
where $\gr^{-1}_{\mathrm{Hdg}}F$ is a perfect complex over $R$ with Tor amplitude $[0,1]$. In particular, this shows that we have $\Gamma_{\mathrm{syn}}(\mathcal{F})(C')\simeq 0$ whenever $\Gamma_{\mathrm{syn}}(\mathcal{F})(C)\simeq 0$.

To finish, we need to know that $\Gamma_{\mathrm{syn}}(\mathcal{F})(R)\simeq 0$. By the deformation theory explained in the previous paragraph, we can reduce to the case where $R$ is an $\Field_p$-algebra. Then, by quasisyntomic descent, we can reduce to the case where $R$ is semiperfect. By considering the direct limit perfection $R\to R_{\mathrm{perf}}$, which is a nilpotent thickening, we can use deformation theory once again to reduce to the case where $R$ is perfect. 

Here, $R^{\mathrm{syn}}$ is a classical stack, and the desired claim follows because $H^{-1}(\mathcal{M}) = 0$, as can be seen for instance from the explicit description of the objects involved provided by Remark~\ref{rem:F_gauges_perfect}.
\end{proof}

Next, we prove the following $m$-truncated version of the theorem

\begin{lemma}
[$m$-truncated analogue of Raynaud]
\label{lem:local_vect_bundle}
Suppose that we have $\mathcal{M}\in \mathsf{P}^{\mathrm{syn}}_{n,\{0,1\}}(R)$, and, for $m\ge 1$, let $\mathcal{M}_m\defn \mathcal{M}/{}^{\mathbb{L}}p^m\in \mathrm{Perf}^{\mathrm{syn},[-1,0]}_{m,\{0,1\}}(R)$ be its image. Then there exists an integer $h\ge 1$ such that, for any $m\ge 1$, there exist: 
\begin{itemize}
   \item An \'etale cover $R\to \tilde{R}$;
   \item Vector bundles
   \[
    \mathcal{V}^{-1}_m,\mathcal{V}^0_m\in  \mathrm{Vect}^{\mathrm{syn}}_{m,\{0,1\}}(R)
   \]
   of rank $h$;
   \item and a map $f:\mathcal{V}^{-1}_m\to \mathcal{V}^0_m$ of vector bundle $F$-gauges 
\end{itemize}
such that there is an equivalence
\[
   \mathcal{M}_m\vert_{\tilde{R}^{\mathrm{syn}}\otimes\Int/p^{m}\Int}\simeq \cofib(\mathcal{V}^{-1}_m\xrightarrow{f}\mathcal{V}^0_m).
\]
\end{lemma}
\begin{proof}
By Theorem~\ref{thm:HTwts01_representable}, the assignment
\begin{align*}
\mathcal{X}_m:R&\mapsto \mathrm{Perf}^{[-1,0]}_{\{0,1\}}(R^{\mathrm{syn}}\otimes\Int/p^m\Int)_{\simeq}
\end{align*}
is represented by a locally finitely presented derived $p$-adic formal Artin stack over $\Int_p$. Similarly, by Theorem~\ref{thm:perfect_f-gauges_repble}, we have a derived Artin stack $\mathcal{Y}_m$ over $\Int_p$ assigning to $R$ the $\infty$-groupoid of maps $\mathcal{V}^{-1}_m\xrightarrow{f}\mathcal{V}^0_m$ of objects in $\mathrm{Vect}{\{0,1\}}(R^{\mathrm{syn}}\otimes\Int/p^m\Int)$. We have the obvious `forgetful' map $\mathcal{Y}_m\to \mathcal{X}_m$, which remembers only the cofiber of the map $f$.

Let $\mathcal{Y}_{m,h}\subset \mathcal{Y}_m$ be the open and closed substack where the source and target of $f$ both have rank $h$. The lemma now comes down to the assertion that there exists $h\ge 1$ such that, for every $m\ge 1$, $\mathcal{M}_m\in \mathcal{X}_m(R)$ is \'etale locally in the image of $\mathcal{Y}_{m,h}$

We will see this as follows:
\begin{itemize}
   \item Some geometry: The map $\mathcal{Y}_m\to \mathcal{X}_m$ is \emph{smooth}.
   \item Given the previous point, it is now enough to check the lemma under the additional assumption that $R$ is an algebraically closed field, where Lemma~\ref{lem:checking_on_points} does the trick.
\end{itemize}

To show that $\mathcal{Y}_m$ is smooth over $\mathcal{X}_m$, let us set up some notation. Write $Y$ for the algebraic stack over $\Int_p$ parameterizing maps of vector bundles $V^{-1}\to V^0$, and $X$ for the (derived) algebraic stack over $\Int_p$ parameterizing perfect complexes of Tor amplitude $[-1,0]$. Over $Y$ (resp. $X$), we have the algebraic stack $Y^-$ (resp. $X^-$) parameterizing maps of filtered vector bundles $\Fil^\bullet V^{-1}\to \Fil^\bullet V^0$, whose associated gradeds are supported in degrees $0,-1$ (resp. filtered perfect complexes $\Fil^\bullet M$ whose associated gradeds are nullhomotopic outside of degrees $0,-1$). If $Z$ is any one of $Y,Y^-,X,X^-$, let $Z_m$ be the derived Weil restriction 
\[
Z_m:C\mapsto Z(C/{}^{\mathbb{L}}p^m).
\]
This is also a derived Artin stack over $\Int_p$.

Pullback along $x^{\mathcal{N}}_{\dR}$ gives canonical maps $\mathcal{Y}_m\to Y^-_m$ and $\mathcal{X}_m\to X^-_m$, and a quick analysis of the deformation theory explained in Theorems~\ref{thm:HTwts01_representable} and~\ref{thm:perfect_f-gauges_repble} now shows that, for any square zero thickening $C'\twoheadrightarrow C$, we have canonical Cartesian squares
\[
\begin{matrix}
\begin{diagram}
\mathcal{Y}_m(C')&\rTo&Y^-_m(C')\\
\dTo&&\dTo\\
\mathcal{Y}_m(C)&\rTo&Y^-_m(C)\times_{Y_m(C)}Y_m(C')
\end{diagram}&&&;&&&
\begin{diagram}
\mathcal{X}_m(C')&\rTo&X^-_m(C')\\
\dTo&&\dTo\\
\mathcal{X}_m(C)&\rTo&X^-_m(C)\times_{X_m(C)}X_m(C')
\end{diagram},
\end{matrix}
\]
which can be combined to obtain a Cartesian square
\[
\begin{diagram}
\mathcal{Y}_m(C')&\rTo& Y^-_m(C')\\
\dTo&&\dTo\\
\mathcal{Y}_m(C)\times_{\mathcal{X}_m(C)}\mathcal{X}_m(C')&\rTo&Y^-_m(C)\times_{Y_m(C)\times_{X_m(C)}X^-_m(C)}(Y_m(C')\times_{X_m(C')}X^-_m(C')).
\end{diagram}  
\]

Therefore, the smoothness of the map $\mathcal{Y}_n\to \mathcal{X}_n$ reduces to that of the map $Y^-_n\to Y_n\times_{X_n}X^-_n$, which in turn reduces to that of the map $Y^-\to Y\times_XX^-$. 

Explicitly, this means the following: Suppose that $\Fil^\bullet V^{-1}\xrightarrow{\Fil^\bullet \alpha} \Fil^\bullet V^0$ is a map of filtered vector bundles over $C$ (with associated gradeds supported in degrees $-1,0$). Suppose also that the following conditions hold:
\begin{itemize}
   \item The underlying map of vector bundles admits a lift $\alpha':V^{',-1}\to V^{',0}$, a map of vector bundles over $C'$;
   \item $M' = \cofib(\alpha')$ admits a lift to a filtered perfect complex $\Fil^\bullet M'$ over $C'$ whose base-change over $C$ is equipped with an isomorphism to $\cofib(\Fil^\bullet\alpha)$.
\end{itemize}
Then there exists a filtered lift $\Fil^\bullet \alpha':\Fil^\bullet V^{',-1}\to \Fil^\bullet V^{',0}$ of $\alpha'$ that is also a lift of $\Fil^\bullet \alpha$.

The condition on the filtration allows us to simplify this even further. The map $\Fil^\bullet \alpha$ is given simply by a commuting square
\[
\begin{diagram}
\Fil^0 V^{-1}&\rTo^{\Fil^0\alpha}&\Fil^0V^0\\
\dTo&&\dTo\\
V^{-1}&\rTo_{  \alpha = \Fil^{-1}\alpha}&V^0.
\end{diagram}
\]
Viewed in this way, it is enough to know that, for any vector bundle $\Fil^\bullet V^{',0}$ lifting $\Fil^\bullet V^0$, we can fit it into a commuting diagram
\[
\begin{diagram}
\Fil^0V^{',0}&\rTo&\Fil^0 M'\\
\dTo&&\dTo\\
V^{',0}&\rTo&M'
\end{diagram}
\]
lifting the one over $C$. This reduces to the fact that sections of the perfect complex 
\[
(\Fil^0V^{',0})^\vee \otimes(V^{',0}\times_{M'}\Fil^0M'), 
\]
which has Tor amplitude in $[-1,0]$, are parameterized by a smooth Artin stack over $C'$.

\end{proof}

\begin{proposition}
[Lifting truncated isogenies]
\label{prop:lifting_isogenies}
Fix integers $h,n\ge 1$. Then there exists an integer $M(h,n)\ge n$ with the following property: Given:
\begin{itemize}
   \item A discrete $\Int/p^r\Int$-algebra $R$;
   \item A vector bundle $F$-gauge over $R$, $\mathcal{V}\in \mathrm{Vect}^{\mathrm{syn}}_{\{0,1\}}(R)$ of rank $h$;
   \item An integer $N\ge M(h,n)+r$ and a height-$n$ isogeny $(f_N,\hat{f}_N)$ between $\mathcal{V}_N \defn \mathcal{V}/{}^{\mathbb{L}}p^N$ and another mod-$p^N$ vector bundle $F$-gauge $\mathcal{W}_N$;
\end{itemize}
there exists a height-$n$ isogeny $(f,\hat{f})$ between $\mathcal{V}$ and a vector bundle $F$-gauge $\mathcal{W}$ of height $h$, and an isomorphism between the mod-$p^{N-M(h,n)-r+1}$ reductions of $(f,\hat{f})$ and $(f_N,\hat{f}_N)$. 
\end{proposition} 

\begin{remark}
This proposition is a special case of some finite presentation results that can be found in \cite{lee_madapusi}, and is related to similar results appearing in the work of Bartling-Hoff~\cite{bartling_hoff} in the context of Witt vector displays. It admits a translation into the setting of truncated Barsotti-Tate groups via Theorem~\ref{thm:dieudonne}, and one can prove this translated result directly, yielding another proof of the proposition. However, since this seems to offer no substantial simplification, we have chosen to stick with the linear algebraic perspective here.
\end{remark}

The proof will be given in the next subsection. For now, let us return to the main result of this section.
  
\begin{proof}
[Proof of Theorem~\ref{thm:syntomic_raynaud}]
Suppose that $R$ is a $\Int/p^r\Int$-algebra. Let $h\ge 1$ be as in Lemma~\ref{lem:local_vect_bundle}, and let $M(h,n)$ be as in Proposition~\ref{prop:lifting_isogenies}.

Choose $m\ge M(h,n)+r+n$, and use Lemma~\ref{lem:local_vect_bundle} to find, after an \'etale localization, a map $f_m:\mathcal{V}^{-1}_m\to \mathcal{V}^0_m$ of mod-$p^m$ vector bundle $F$-gauges of rank $h$ such that $\cofib(f_m)\simeq \mathcal{M}_m$. By Lemma~\ref{lem:isogenies_pn}, we can complete $f_m$ to a height-$n$ isogeny $(f_m,\hat{f}_m)$ between $\mathcal{V}^{-1}_m$ and $\mathcal{V}^{0}_m$.

After a further pro-\'etale localization, we can assume that $\mathcal{V}^{-1}_m$ lifts to a vector bundle $F$-gauge $\mathcal{V}^{-1}\in \mathrm{Vect}(R^{\mathrm{syn}})$. Proposition~\ref{prop:lifting_isogenies} now tells us that the reduction-mod-$p^{m-M(h,n)-r+1}$ of $(f_m,\hat{f}_m)$ lifts to a height-$n$ isogeny $(f,\hat{f})$ between $\mathcal{V}^{-1}$ and  $\mathcal{V}^0$, and Lemma~\ref{lem:enough_to_check_mod_n} tells us that there is an isomorphism $\cofib(f)\xrightarrow{\simeq}\mathcal{M}$.
\end{proof}

\subsection{Lifting isogenies}

Here, we will give a proof of Proposition~\ref{prop:lifting_isogenies}.

\begin{construction}
[Some local models]
Suppose that we have a filtered vector bundle $\Fil^\bullet V$ over a ring $R$ such that the underlying vector bundle $V$ has rank $h\in \Int_{\ge 1}$. Suppose also that the filtration is supported in degrees $-1,0$ with $\gr^{-1}V$ of rank $d\in \Int_{\ge 0}$.

For any $n\ge 1$, we will define $X^-_{\Fil^\bullet V}(n)\to \Spec R$ to be the (derived) stack parameterizing height-$n$ isogenies $\Fil^\bullet V\xleftrightharpoons[f]{\hat{f}}\Fil^\bullet W$ of filtered vector bundles with
\[
\mathrm{rank}\gr^iW = \begin{cases}
d&\text{if $i=-1$}\\
h-d&\text{if $i=0$}\\
0&\text{otherwise}.
\end{cases}
\]

We will define $X_{V}(n)\to \Spec R$ to be the (derived) stack parameterizing height-$n$ isogenies $V\xleftrightharpoons[f]{\hat{f}} W$ with $W$ of rank $h$.

For $m\ge 1$, we will also need the Weil restricted versions of these stacks:
\[
X_{m,\Fil^\bullet V}(n): C\mapsto X_{\Fil^\bullet V}(n)(C/{}^{\mathbb{L}}p^m)\;;\; X_{m,V}(n): C\mapsto X_{V}(n)(C/{}^{\mathbb{L}}p^m)
\]
\end{construction}

\begin{construction}
[More local models]
Suppose that we have $R\in \mathrm{CRing}^{p\text{-nilp}}$ and a vector bundle $F$-gauge $\mathcal{V}$ over $R$ of rank $h$ and Hodge-Tate weights $0,1$. For $m\ge 1$, write $\mathcal{V}_m$ for the image of $\mathcal{V}$ in $\mathrm{Vect}^{\mathrm{syn}}_{m,\{0,1\}}(R)$.

For $n\ge 1$, write $\mathcal{X}_{\mathcal{V}}(n)\to \Spec R$ for the prestack parameterizing height-$n$ isogenies $\mathcal{V}\xleftrightharpoons[f]{\hat{f}} \mathcal{W}$ of vector bundle $F$-gauges with Hodge-Tate weights in  $\{0,1\}$.

Given another integer $m\ge 1$, write $\mathcal{X}_{m,\mathcal{V}}(n)\to \Spec R$ for the prestack parameterizing height-$n$ isogenies $\mathcal{V}_m\xleftrightharpoons[f_m]{\hat{f}_m} \mathcal{W}_m$ of vector bundle $F$-gauges of level $m$ with Hodge-Tate weights in  $\{0,1\}$.

Reduction mod-$p^m$ yields a canonical map of prestacks $\mathcal{X}_{\mathcal{V}}(n)\to \mathcal{X}_{m,\mathcal{V}}(n)$. Pullback along $x^{\mathcal{N}}_{\dR}$ yields a filtered vector bundle $\Fil^\bullet V$ over $R$ with associated graded supported in degrees $-1,0$, as well as canonical maps
\[
\mathcal{X}_{\mathcal{V}}(n) \to X^-_{\Fil^\bullet V}(n)\;;\; \mathcal{X}_{m,\mathcal{V}}(n)\to X^-_{m,\Fil^\bullet V}(n).
\]
\end{construction}

Proposition~\ref{prop:lifting_isogenies} will be implied by the following more refined assertion: 
\begin{proposition}
\label{prop:translate_isogenies_to_maps_of_stacks}
There exists $M(h,n)\ge n$, depending only on $h$ and $n$, such that, for any discrete $R$-algebra $C$ with $p^s = 0\in C$,  any $N\ge M(h,n)+s$, and any $m\leq N-M(h,n)-s+1$ we have a canonical commuting diagram
\begin{align}\label{eqn:retraction}
\begin{diagram}
\mathcal{X}_{\mathcal{V}}(n)(C)&\rTo&\mathcal{X}_{N,\mathcal{V}}(n)(C)&\rTo&\mathcal{X}_{\mathcal{V}}(n)(C)\\
&\rdTo&\dTo&\ldTo\\
&&\mathcal{X}_{m,\mathcal{V}}(n)(C)
\end{diagram}
\end{align}
where the composition of the horizontal arrows is the identity. Moreover, $\mathcal{X}_{\mathcal{V}}(n)(C)$ is discrete for any such $C$.
\end{proposition}

\begin{corollary}
\label{cor:bartling_hoff}
   The $p$-adic formal prestack $\mathcal{X}_{\mathcal{V}}(n)$, \emph{a priori} an inverse limit of finitely presented formal schemes over $\Spf R$ is in fact itself a finitely presented formal scheme over $\Spf R$.
\end{corollary}

\begin{remark}
As observed in the introduction, the corollary is related to results appearing in the proof of the main result of~\cite{bartling_hoff}. From the optic of $p$-divisible groups, this finite presentation can be easily deduced from that of the scheme parameterizing finite flat subgroups of the $p^n$-torsion of a fixed $p$-divisible group.
\end{remark}

The rest of this section is devoted to the proof of Proposition~\ref{prop:translate_isogenies_to_maps_of_stacks}.

\begin{proposition}
[Reduction to the case of perfect rings]
\label{prop:isogenies_to_perfect_case}
It suffices to find $M(h,n)\ge n$, depending only on $h$ and $n$, such that Proposition~\ref{prop:translate_isogenies_to_maps_of_stacks} holds for perfect $R$-algebras $C$ (with $s=1$). In particular, we can assume without loss of generality that $R$ is perfect.
\end{proposition}
\begin{proof}
We will prove this using deformation theory.

First, let us reduce to the case $s=1$. Suppose that we have a discrete $R/p^sR$-algebra $C$ and integers $N\ge M(h,n)+s$, $m\leq N-M(h,n)-s+1$ such that $\mathcal{X}_{\mathcal{V}}(n)(C)$ is discrete and that we have the commuting diagram~\eqref{eqn:retraction} where the composition of the horizontal arrows is the identity. 

Using Lemma~\ref{lem:Zn_groth_messing} below, we find that, for any square-zero extension $C'\twoheadrightarrow C$ of discrete $R$-algebras, we have for any $k\ge 1$,
\begin{align}\label{eqn:XkVn}
\mathcal{X}_{\mathcal{V}}(n)(C')&\simeq \mathcal{X}_{\mathcal{V}}(n)(C)\times_{X^-_{\Fil^\bullet V}(n)(C)\times_{X^-_V(n)(C)}X^-_V(n)(C')}X^-_{\Fil^\bullet V}(n)(C');\nonumber\\
\mathcal{X}_{k,\mathcal{V}}(n)(C')&\simeq \mathcal{X}_{k,\mathcal{V}}(n)(C)\times_{X^-_{\Fil^\bullet V}(n)(C/{}^{\mathbb{L}}p^k)\times_{X^-_V(n)(C/{}^{\mathbb{L}}p^k)}X^-_V(n)(C'/{}^{\mathbb{L}}p^k)}X^-_{\Fil^\bullet V}(n)(C'/{}^{\mathbb{L}}p^k).
\end{align}

The first isomorphism tells us that the fibers of $\mathcal{X}_{\mathcal{V}}(n)(C')\to \mathcal{X}_{\mathcal{V}}(n)(C)$ are discrete, and therefore, $\mathcal{X}_{\mathcal{V}}(n)(C')$ is itself discrete. 

Suppose now that $C'$ is a $\Int/p^{s'}\Int$-algebra, and that we have $1\leq m\leq N-M(h,n)-s'+1$. Then, since $N\ge N-m\ge s'$, the map
\[
C'\oplus C'[1]\xrightarrow{\simeq}C'/{}^{\mathbb{L}}p^N \to C'/{}^{\mathbb{L}}p^m
\]
factors through the natural surjection $C'\oplus C'[1]\twoheadrightarrow C'$. Therefore, the isomorphisms from~\eqref{eqn:XkVn} now tell us that the assumed factoring 
\[
\mathcal{X}_{N,\mathcal{V}}(C)\to \mathcal{X}_{\mathcal{V}}(C)\to \mathcal{X}_{m,\mathcal{V}}(C)
\]
lifts to a similar factoring over $C'$. 

Therefore, via induction on $s$, we are reduced to the case of $R/pR$-algebras. We need to find the integer $M(h,n)$ such that, for any $N\ge M(h,n)+1$, we have a canonical retraction $\mathcal{X}_{N,\mathcal{V}}(C)\to \mathcal{X}_{\mathcal{V}}(C)$ on discrete $R/pR$-algebras $C$. For this, using quasisyntomic descent, we can reduce to the case where $C$ is semiperfect.

Consider the map $C\to C_{\mathrm{perf}}$ to the direct limit perfection of $C$: this is surjective with locally nilpotent kernel $J$. Using~\cite[Corollary 1.5]{BHATT2017576} (and also the first part of Lemma~\ref{lem:Zn_groth_messing} below), we find that we have
\[
\mathcal{X}_{N,\mathcal{V}}(C)\xrightarrow{\simeq}\varprojlim_k\mathcal{X}_{N,\mathcal{V}}(C/J^k)\;;\; \mathcal{X}_{\mathcal{V}}(C)\xrightarrow{\simeq}\varprojlim_k\mathcal{X}_{\mathcal{V}}(C/J^k).
\]

Next, the deformation argument above shows that, if $I\subset C$ is a square zero ideal, then, given the desired retraction over $C/I$, we can lift it canonically to a retraction over $C$. Repeated application of this principle tells us that the conclusion is valid even if $I$ is only assumed to be nilpotent.

Combining the last three paragraphs now completes the reduction to the case of perfect $R$-algebras.
\end{proof}

\begin{lemma}
[Application of Grothendieck-Messing]
\label{lem:Zn_groth_messing}
For each $m\ge 1$, $\mathcal{X}_{m,\mathcal{V}}(n)$ is represented by a finitely presented derived Artin stack over $C$. Moreover, suppose that $(C'\twoheadrightarrow C,\gamma)$ is a nilpotent divided power thickening in $\mathrm{CRing}_{R/}$. Then we have canonical Cartesian squares
\begin{align*}
\begin{diagram}
\mathcal{X}_{\mathcal{V}}(n)(C') &\rTo&X^-_{\Fil^\bullet V}(n)(C')\\
\dTo&&\dTo\\
\mathcal{X}_{\mathcal{V}}(n)(C)&\rTo&X^-_{\Fil^\bullet V}(n)(C)\times_{X^-_V(n)(C)}X^-_V(n)(C')
\end{diagram};\\
\begin{diagram}
\mathcal{X}_{m,\mathcal{V}}(n)(C') &\rTo&X^-_{m,\Fil^\bullet V}(n)(C')\\
\dTo&&\dTo\\
\mathcal{X}_{m,\mathcal{V}}(n)(C)&\rTo&X^-_{m,\Fil^\bullet V}(n)(C)\times_{X^-_{m,V}(n)(C)}X^-_{m,V}(n)(C')
\end{diagram}.
\end{align*}
\end{lemma}
\begin{proof}
This follows from Theorems~\ref{thm:HTwts01_representable} and~\ref{thm:perfect_f-gauges_repble}.
\end{proof}

\begin{remark}
[Explication of problem in the perfect case]
\label{rem:perfect_rings_explicit}
We can use Remark~\ref{rem:F_gauges_perfect} to make the problem in the perfect case quite explicit. The data of the $F$-gauge $\mathcal{V}$ is equivalent to that of maps $V^0\xleftrightharpoons[u]{t}V^{-1}$ of finite locally free $W(R)$-modules of rank $h$ with $u\circ t$ and $t\circ u$ both equal to multiplication-by-$p$, along with an isomorphism $\xi:\varphi^*V^0\xrightarrow{\simeq}V^{-1}$. In other words, we are in the classical remit of Zink's theory of displays. Here, the height-$n$ isogeny $(f_N,\hat{f}_N)$ amounts to giving: 
\begin{itemize}
   \item Maps $V^{',0}_N\xleftrightharpoons[u'_N]{t'_N}V^{',-1}_N$ of finite locally free $W_N(R)$-modules of rank $h$ with $u'_N\circ t'_N = t'_N\circ u'_N = p$;
   \item An isomorphism $\xi'_N:\varphi^*V^{',-0}_N\xrightarrow{\simeq}V^{',-1}_N$
   \item Height-$n$ isogenies $V^i/p^NV^i = V^i_N\xleftrightharpoons[\hat{f}^{(i)}_N]{f^{(i)}_N} W^i_N$ of finite locally free $W_N(R)$-modules for $i=0,-1$ such that we have commuting diagrams
   \begin{align*}
    \begin{diagram}
    V^0_N&\rTo^{t_N = t(\mathrm{mod}~p^N)}&V^{-1}_N&\rTo^{u_N = u(\mathrm{mod}~p^N)}&V^0_N\\
    \dTo^{f^{(0)}_N}&&\dTo^{f^{(-1)}_N}&&\dTo^{f^{(0)}_N}\\
     V^{',0}_N&\rTo^{t'_N}&V^{',-1}_N&\rTo^{u'_N }&V^{',0}_N\\
     \dTo^{\hat{f}^{(0)}_N}&&\dTo^{\hat{f}^{(-1)}_N}&&\dTo^{\hat{f}^{(0)}_N}\\
    V^0_N&\rTo^{t_N}&V^{-1}_N&\rTo^{u_N}&V^0_N;
    \end{diagram}
   \;\;;&\;\;
   \begin{diagram}
   \varphi^*V^0_N&\rTo^{\xi_N = \xi(\mathrm{mod}~p^N)}&V^{-1}_N\\
   \dTo_{\varphi^*f^{(0)}_N}&&\dTo_{f^{(-1)}_N}\\
   \varphi^*V^{',0}_N&\rTo^{\xi'_N}&V^{',-1}_N\\
   \dTo_{\varphi^*\hat{f}^{(0)}_N}&&\dTo_{\hat{f}^{(-1)}_N}\\
   \varphi^*V^0_N&\rTo^{\xi_N }&V^{-1}_N
   \end{diagram}  
   \end{align*}

\end{itemize}

The problem we are now concerned with is the following: Find an integer $M(h,n)\ge 1$ such that, for $N\ge M(h,n)+1$, we can find a \emph{canonical} lift (up to isomorphism) of the mod-$p^{N-M(h,n)}$ reduction of the above data to a commuting diagrams of locally free $W(R)$-modules:
   \begin{align*}
    \begin{diagram}
    V^0&\rTo^{t }&V^{-1}&\rTo^{u}&V^0\\
    \dTo^{f^{(0)}}&&\dTo^{f^{(-1)}}&&\dTo^{f^{(0)}}\\
     V^{',0}&\rTo^{t'}&V^{',-1}&\rTo^{u' }&V^{',0}\\
     \dTo^{\hat{f}^{(0)}}&&\dTo^{\hat{f}^{(-1)}}&&\dTo^{\hat{f}^{(0)}}\\
    V^0&\rTo^{t }&V^{-1}&\rTo^{u }&V^0
    \end{diagram}
      \;\;;&\;\;
   \begin{diagram}
   \varphi^*V^0&\rTo^{\xi }&V^{-1}\\
   \dTo_{\varphi^*f^{(0)}}&&\dTo_{f^{(-1)}}\\
   \varphi^*V^{',0}&\rTo^{\xi'}&V^{',-1}\\
   \dTo_{\varphi^*\hat{f}^{(0)}}&&\dTo_{\hat{f}^{(-1)}}\\
   \varphi^*V^0&\rTo^{\xi }&V^{-1}.
   \end{diagram}  
   \end{align*}
such that, in the first diagram, the composition along all columns is multiplication-by-$p^n$, while that along all the rows is multiplication-by-$p$. 

The problem was addressed by Bartling-Hoff~\cite{bartling_hoff} in the more general context of Zink displays over arbitrary $p$-nilpotent rings, but these authors impose a nilpotency hypothesis in their arguments, and so we cannot use their results directly. For now, note that such commuting diagrams have no non-trivial automorphisms that restrict to the identity on $V^0$. This already shows that $\mathcal{X}_V(n)(R)$ is discrete when $R$ is perfect.
\end{remark}

\begin{remark}
[Banal $F$-gauges]
\label{rem:perfect_rings_bases}
We will say that the $F$-gauge $\mathcal{V}$ (resp. $\mathcal{V}_N$) is \defnword{banal} if  $V^0$ (resp. $V^{',0}_N$) is free over $W(R)$ (resp. over $W_N(R)$). In this case, we can assume that we have bases
\[
W(R)^h\xrightarrow{\simeq}V^0\;;\; W(R)^h\xrightarrow{\simeq}V^{-1}\;;\;W_N(R)^h\xrightarrow{\simeq}V^{',0}_N\;;\;W_N(R)^h\xrightarrow{\simeq}V^{',-1}_N
\]
such that with respect to these bases $t,t'_N$ are given by the matrix $J_1 = \begin{pmatrix}
p\cdot 1_{h-d}&0\\
0&1_d
\end{pmatrix}$, while $u,u'_N$ are given by the matrix  $J_2 = \begin{pmatrix}
1_{h-d}&0\\
0&p\cdot 1_d
\end{pmatrix}$. The identification $\varphi^*V^0\xrightarrow{\simeq}V^{-1}$ is now given by an element $g\in \GL_h(W(R))$, while $\varphi^*V^{',0}_N\xrightarrow{\simeq}V^{',-1}_N$ is given by an element $g'_N\in \GL_h(W_N(R))$. The maps $f^{(i)}_N,\hat{f}^{(i)}_N$ are given by matrices $A^{(i)}_N,B^{(i)}_N\in \mathrm{Mat}_{h\times h}(W_N(R))$ satisfying $A^{(i)}_NB^{(i)}_N = B^{(i)}_NA^{(i)}_N = p^n\cdot I_h$. Moreover, if $g_N\in \GL_h(W_N(R))$ is the image of $g$, we have the identities
\begin{align*}
A^{(0)}_N J_2=J_2A^{(-1)}_N\;;\; J_1A^{(0)}_N = A^{(-1)}_NJ_1\;;&\;B^{(0)}_N J_2=J_2B_N^{(-1)}\;;\; J_1B^{(0)}_N = B_N^{(-1)}J_1;\\
A^{(-1)}_Ng_N = g'_N\varphi(A^{(0)}_N)\;;&\; B^{(-1)}_Ng'_N = g_N\varphi(B^{(0)}_N).
\end{align*}
\end{remark}

\begin{remark}
[Reduction to a matrix problem]
\label{rem:lifting_perfect}
Now, lifting the mod-$p^{N-m}$ reduction of $(f_N,\hat{f}_N)$ to an isogeny $(f,\hat{f})$ out of $\mathcal{V}$ becomes equivalent to finding matrices $A^{(i)},B^{(i)}\in \mathrm{Mat}_{h\times h}(W(R))$ and $g'\in \GL_h(W(R))$ with the following properties:
\begin{enumerate}
   \item $A^{(i)}B^{(i)} = B^{(i)}A^{(i)} = p^nI_h$;
   \item $A^{(0)} J_2=J_2A^{(-1)}\;;\; J_1A^{(0)} = A^{(-1)}J_1\;;\;B^{(0)} J_2=J_2B^{(-1)}\;;\; J_1B^{(0)} = B^{(-1)}J_1$;
   \item $A^{(-1)}g = g'\varphi(A^{(0)})\;;\; B^{(-1)}g' = g\varphi(B^{(0)})$;
   \item $(A^{(-1)},B^{(-1)},A^{(0)},B^{(0)},g')\equiv (A^{(-1)}_N,B^{(-1)}_N,A^{(0)},B^{(0)}_N,g'_N)\pmod{p^{N-m}}$.
\end{enumerate}
Note that, since $J_1$ and $J_2$ are invertible after inverting $p$ and $W(R)$ is $p$-torsion free, $A^{(-1)}$ determines the remaining three matrices, and, therefore, $A^{(-1)}$ and $g$ together determine $g'$ as well.
\end{remark}

\begin{remark}
\label{rem:canonicity_perfect}
In the situation of the previous remark, if $(\tilde{A}^{(i)},\tilde{B}^{(i)},\tilde{g}')$ is another tuple with the same properties, then it gives an isomorphic lift if there exists (a necessarily unique) $\alpha\in \GL_h(W(R))$ with the following properties:
\begin{itemize}
   \item $J_2\alpha J_2^{-1} = J_1^{-1}\alpha J_1\in \GL_h(W(R))$;
   \item $\alpha\equiv I_h\pmod{p^{N-m}}$;
   \item $\tilde{A}^{(-1)} = \alpha A^{(-1)}$;
   \item $\tilde{g}' = \alpha^{-1}g'\varphi(J_2\alpha J_2^{-1})$.
\end{itemize}
\end{remark}

\begin{construction}
[A scheme of matrices]
The previous remarks make it clear that it will be profitable to study the geometry of the scheme $\mathcal{M}(h,n)$ over $\Int_p$ parameterizing pairs of $h\times h$-matrices $(A,B)$ with $AB = BA = p^n I_h$. This is an affine scheme; let $S(h,n)$ be the $\Int_p$-algebra such that $\mathcal{M}(h,n) = \Spec S(h,n)$. The scheme also has an action of $\GL_h$ given by
\[
g\cdot(A,B) = (gA,Bg^{-1}).
\]
Over $\Rat_p$, projection onto the first coordinate gives a $\GL_h$-equivariant isomorphism $\mathcal{M}(h,n)_{\Rat_p}\xrightarrow{\simeq}\GL_{h,\Rat_p}$. Consider the group action map
\[
\GL_h\times \mathcal{M}(h,n)\xrightarrow{(g,(A,B))\mapsto ((A,B),(gA,Bg^{-1}))}\mathcal{M}(h,n)\times \mathcal{M}(h,n).
\]
If $\GL_h = \Spec T$, then this corresponds to a map $S(h,n)\otimes_{\Int_p}S(h,n)\to T\otimes_{\Int_p}S(h,n)$ that is an isomorphism after inverting $p$. In particular, there exists an integer $t(h,n)\ge 0$ such that we have
\[
p^{t(h,n)}(T\otimes_{\Int_p}S(h,n))\subset \mathrm{im}(S(h,n)\otimes_{\Int_p}S(h,n)).
\]
\end{construction}

\begin{lemma}
\label{lem:t(h,n)_props}
Suppose that $S$ is a flat $\Int_p$-algebra, and that we are given $(A_1,B_1),(A_2,B_2)\in \mathcal{M}(h,n)(S)$, $m\ge t(h,n)$, and $g\in \GL_h(S)$ such that
\[
(gA_1,B_1g^{-1})\equiv (A_2,B_2)\pmod{p^m}.
\]
Then there exists (a necesarily unique) $\tilde{g}\in \GL_h(S)$ such that $(\tilde{g}A_1,B_1\tilde{g}^{-1}) = (A_2,B_2)$.
\end{lemma}
\begin{proof}
We can interpret $((A_1,B_1),(A_2,B_2))$ as a map of $\Int_p$-algebras $\gamma:S(h,n)\otimes_{\Int_p}S(h,n)\to S$, and we want to know that this map can be extended to a map $T\otimes_{\Int_p}S(h,n)\to S$. More precisely, we want to know that the composition
\[
\gamma':(T\otimes_{\Int_p}S(h,n))[1/p]\xrightarrow{\simeq}(S(h,n)\otimes_{\Int_p}S(h,n))[1/p]\to S[1/p]
\]
carries $T\otimes_{\Int_p}S(h,n)$ into $S$. By the definition of $t(h,n)$, for any element $x\in T\otimes_{\Int_p}S(h,n)$, we have $\gamma'(p^{t{h,n}}x)\in S$. Therefore, it is enough to know that we have $\gamma'(p^{t(h,n)}x)\in p^{t(h,n)}S$ for all such $x$. But, by our hypothesis of the existence of $g$, there exists a map $\delta:T\otimes_{\Int_p}S(h,n)\to S/p^mS$ such that we have a commuting diagram
\[
\begin{diagram}
S(h,n)\otimes_{\Int_p}S(h,n)&\rTo^\gamma&S\\
\dTo&&\dTo\\
T\otimes_{\Int_p}S(h,n)&\rTo_{\delta}&S/p^mS.
\end{diagram}
\]
This tells us that we have 
\[
\gamma'(p^{t(h,n)}x)\equiv \delta(p^{t(h,n)}x)\equiv p^{t(h,n)}\delta(x)\pmod{p^m}.
\]
Since $m\ge t(h,n)$, this completes the proof.
\end{proof}

\begin{construction}
[Controlling $p$-torsion in the automorphism groups]
For any pair $(A,B)\in \mathcal{M}(h,n)(C)$, let $\mathcal{I}(A,B)\to \Spec C$ be the affine group scheme obtained as the stabilizer in $\GL_{h,C}$ of $(A,B)$: this organizes into a relative group scheme $\mathcal{I}\to \mathcal{M}(h,n)$ whose restriction over the generic fiber is the trivial group scheme. Suppose that $\mathcal{I} = \Spec D$, and let $D\to S(h,n)$ be the augmentation map associated with the identity section. Set $J = \ker(D\to S(h,n))$; then, since $J[1/p]=0$, there exists an integer $c(h,n)\ge 1$ such that $p^{c(h,n)}J = 0$
\end{construction}   

\begin{lemma}
\label{lem:c(h,n)_props}
Let $m(h,n) = \max\{t(h,n),c(h,n)\}$. Let the hypotheses and notation be as in Lemma~\ref{lem:t(h,n)_props}, but assume in addition that $m\ge m(h,n)+1$. Then there exists $\tilde{g}\in \GL_h(S)$ such that $(\tilde{g}A_1,B_1g^{-1}) = (A_2,B_2)$, and moreover we have
\[
\tilde{g}\equiv g\pmod{p^{m-m(h,n)}}.
\]
\end{lemma}
\begin{proof}
The existence of $\tilde{g}$ is clear since $m(h,n)\ge t(h,n)$. The element $g^{-1}\tilde{g}$ belongs to $\mathcal{I}(A,B)(S/p^mS)$, and so is associated with a $\Int_p$-algebra map $D \to S/p^mS$. Since $p^{c(h,n)}J = 0$, this map must map $J$ into $p^{m-c(h,n)}S/p^mS$. Therefore, the composition $D\to S/p^mS\to S/p^{m-m(h,n)}$ kills $J$, and hence factors through the augmentation map. In other words, $g^{-1}\tilde{g}$ is congruent to the identity mod-$p^{m-M(h,n)}$.
\end{proof} 

\begin{lemma}
\label{lem:r(h,n)_props}
Suppose that $S$ is a $p$-complete flat $\Int_p$-algebra and that, for $m\ge n+1$, we have $A,B\in \mathrm{Mat}_{h\times h}(S)$ such that $AB\equiv p^nI_h\pmod{m}$. Then there exists $(A',B')\in \mathcal{M}(h,n)(S)$ such that 
\[
(A,B)\equiv (A',B')\pmod{p^{m-n}}
\]
\end{lemma} 
\begin{proof}
By hypothesis, there exists an $h\times h$-matrix $C$ such that $AB = p^n(I_h+p^{m-n}C)$. Now, $A' = A$ and $B'=B(I_h+p^{m-n}C)^{-1}$ does the job.
\end{proof}

\begin{proof}
[Proof of Proposition~\ref{prop:translate_isogenies_to_maps_of_stacks}]
Let $m(h,n)$ be as in Lemma~\ref{lem:c(h,n)_props}, and set $M(h,n) = n+2m(h,n)$. We will show that this does the job. 

By Proposition~\ref{prop:isogenies_to_perfect_case}, we can assume that $R$ is perfect. Suppose that we are in the banal situation of Remark~\ref{rem:perfect_rings_bases}. If $N\ge M(h,n)+1$, then by Lemma~\ref{lem:r(h,n)_props}, the reduction $(A^{(-1)}_{N-n},B^{(-1)}_{N-n})$ of $(A^{(-1)}_N,B^{(-1)}_N)$ mod-$p^{N-n}$ can be lifted to a pair $(A^{(-1)},B^{(-1)})\in \mathcal{M}(h,n)(W(R))$. 

Suppose that, with respect to the direct sum decomposition $W(R)^h =W(R)^{h-d}\oplus W(R)^d$, we have
\[
A^{(i)}_N = \begin{pmatrix}
A^{(i)}_{N,1}&A^{(i)}_{N,2}\\
A^{(i)}_{N,3}&A^{(i)}_{N,4}
\end{pmatrix} \;;\; B^{(i)}_N = \begin{pmatrix}
B^{(i)}_{N,1}&B^{(i)}_{N,2}\\
B^{(i)}_{N,3}&B^{(i)}_{N,4}
\end{pmatrix} 
\]
Then the identities $A_N^{(0)}J_2 = J_2A_N^{(-1)}$ and $B_N^{(0)}J_2 = J_2B_N^{(-1)}$ tell us that $A^{(0)}_N$ and $B^{(0)}_N$ are essentially determined by $A^{(-1)}_N$ and $B^{(-1)}_N$, except that there is an ambiguity up to $p$-torsion in the choice of $A^{(0)}_{N,2}$ and $B^{(0)}_{N,2}$ such that $pA^{(0)}_{N,2} = A^{(-1)}_{N,2}$ and $pB^{(0)}_{N,2} = B^{(-1)}_{N,2}$. In particular, all such choices agree mod-$p^{N-1}$.

This shows that we have
\[
A^{(0)} = J_2A^{(-1)}J_2^{-1},B^{(0)} = J_2B^{(-1)}J_2^{-1}\in \mathrm{Mat}_{h\times h}(W(R)),
\]
and also that their images mod-$p^{N-n}$ agree with those of $A^{(0)}_N$ and $B^{(0)}_N$.

Now, consider the pairs
\[
(A^{(-1)},B^{(-1)}), (\varphi(A^{(0)}),\varphi(B^{(0)}))\in \mathcal{M}(h,n)(W(R)).
\]
If $\tilde{g}'_N\in \GL_h(W(R))$ is a lift of $g_N$, then we have
\[
(\tilde{g}'_NA^{(-1)},B^{(-1)}(\tilde{g}'_N)^{-1}) \equiv (\varphi(A^{(0)}),\varphi(B^{(0)})\pmod{p^{N-n}}.
\]
Therefore, Lemma~\ref{lem:c(h,n)_props} gives us a unique $g'\in \GL_h(W(R))$ such that
\begin{itemize}
   \item $g'\equiv \tilde{g}'_N\pmod{p^{N-n-m(h,n)}}$;
   \item $(g'A^{(-1)},B^{(-1)}(g')^{-1}) = (\varphi(A^{(0)}),\varphi(B^{(0)}))$.
\end{itemize}
By Remark~\ref{rem:lifting_perfect}, we find that we have constructed a lift $(f,\hat{f})$ of the reduction of $(f_N,\hat{f}_N)$ mod-$p^{N-n-m(h,n)}.$

If $(\tilde{A}^{(-1)},\tilde{B}^{(-1)},\tilde{A}^{(0)},\tilde{B}^{(0)},\tilde{g}')$ is another such lifting datum, then  applying Lemma~\ref{lem:c(h,n)_props} once again gives us unique elements $\alpha^{(-1)},\alpha^{(0)}\in \GL_h(W(R))$ such that, for $i=1,2$,
\begin{itemize}
   \item $(\alpha^{(i)}A^{(i)},B^{(i)}(\alpha^{(i)})^{-1}) = (\tilde{A}^{(i)},\tilde{B}^{(i)})$;
   \item $\alpha^{(i)}\equiv 1\pmod{p^{N-M(h,n)}}$.
\end{itemize}
The uniqueness ensures that $\alpha^{(0)} = J_2\alpha^{(-1)}J_2^{-1}$, and so we see, using Remark~\ref{rem:canonicity_perfect}, that $\alpha = \alpha^{(-1)}$ gives us the unique isomorphism between the lifts determined by these data, lifting the identity mod-$p^{N-M(h,n)}$.  

Let 
\[
\mathcal{X}^{\mathrm{ban}}_{\mathcal{V}}(n)(R)\subset \mathcal{X}_{\mathcal{V}}(n)(R)\;;\; \mathcal{X}^{\mathrm{ban}}_{m,\mathcal{V}}(n)(R)\subset \mathcal{X}_{m,\mathcal{V}}(n)(R)
\]
be the subgroupoids spanned by the banal objects. Then in the preceding paragraphs we have constructed a canonical commuting diagram
\[
\begin{diagram}
\mathcal{X}^{\mathrm{ban}}_{\mathcal{V}}(n)(R)&\rTo&\mathcal{X}^{\mathrm{ban}}_{N,\mathcal{V}}(n)(R)&\rTo&\mathcal{X}^{\mathrm{ban}}_{\mathcal{V}}(n)(R)\\
&\rdTo&\dTo&\ldTo\\
&&\mathcal{X}^{\mathrm{ban}}_{N-M(n,h),\mathcal{V}}(n)(R)
\end{diagram}
\]
where the composition of the horizontal arrows is the identity.

Since every $F$-gauge is banal Zariski locally on $\Spec R$, we can now complete the proof using Zariski descent.
\end{proof}

\section{Classifying finite flat group schemes}
\label{sec:ffg}

In this section, $R$ will be a discrete object in $\mathrm{CRing}^{p\text{-comp}}$.

\subsection{The main theorem}

Let $\mathrm{FFG}(R)$ be the category of finite locally free commutative $p$-power torsion group schemes over $R$, and let $\mathrm{FFG}_n(R)$ be the subcategory spanned by the $p^n$-torsion objects. These are exact additive categories. Also, write $\mathsf{P}^{\mathrm{syn}}_{\{0,1\}}(R)$ for the $\infty$-subcategory of $\mathrm{Perf}^{[-1,0]}_{\{0,1\}}(R^{\mathrm{syn}})$ spanned by the objects that can be lifted to $\Int/p^n\Int$-modules for some $n\ge 1$: By Corollary~\ref{cor:syntomic_raynaud}, this is a discrete exact additive subcategory that is a union of its subcategories $\mathsf{P}_{n,\{0,1\}}^{\mathrm{syn}}(R)$. Here, we will prove the main theorem of this paper:

\begin{theorem}
\label{thm:main}
There is a natural exact equivalence of categories
\[
\mathcal{G}:\mathsf{P}^{\mathrm{syn}}_{\{0,1\}}(R)\xrightarrow{\simeq} \mathrm{FFG}(R)
\]
carrying $\mathsf{P}^{\mathrm{syn}}_{n,\{0,1\}}(R)$ onto $\mathrm{FFG}_n(R)$, functorial in $R\in \mathrm{CRing}^{p\text{-comp}}_{\heartsuit}$, compatible with arbitrary base-change, satisfying fpqc descent, and compatible also with the equivalence in Theorem~\ref{thm:dieudonne} for $n$-truncated Barsotti-Tate groups. For each $\mathcal{M}\in \mathsf{P}^{\mathrm{syn}}_{\{0,1\}}(R)$ there exists a canonical Cartier duality isomorphism\footnote{The dual on the left hand side is in the category of perfect complexes over $R^{\mathrm{syn}}$.}
   \[
    \mathcal{G}(\mathcal{M}^\vee\{1\}[1])\xrightarrow{\simeq}\mathcal{G}(\mathcal{M})^*
   \]
\end{theorem}

\begin{remark}
[Local finite presentation of $\mathsf{P}^{\mathrm{syn}}_{\{0,1\}}$]
\label{rem:p_syn_finitely_presented}
   The theorem implies that the classical truncation of the formal prestack $R\mapsto  \mathsf{P}^{\mathrm{syn}}_{\{0,1\}}(R)$ is represented by a locally finitely presented $p$-adic formal Artin stack. Independently of the theorem, one can show directly that, even as a derived prestack, it is represented by a locally finitely presented quasi-smooth derived formal Artin stack over $\Spf \Int_p$: It admits as cotangent complex the perfect complex $(\gr^{-1}_{\mathrm{Hdg}}M_{\mathrm{taut}})^\vee\otimes \Fil^0_{\mathrm{Hdg}}M_{\mathrm{taut}}$, which has Tor amplitude in $[-1,1]$. Moreover, by Theorem~\ref{thm:HTwts01_representable}, it is a filtered colimit of the stacks $\mathsf{P}^{\mathrm{syn}}_{n,\{0,1\}}$, which are themselves inverse limits of the locally finitely presented derived formal algebraic stacks
   \[
   \mathsf{P}^{\mathrm{syn},(m)}_{n,\{0,1\}}:  C\mapsto\Mod{\Int/p^n\Int}\bigl(\mathrm{Perf}^{[-1,0]}_{\{0,1\}}(R^{\mathrm{syn}}\otimes\Int/p^m\Int)\bigr).
   \]
   The key point now is to establish the local finite presentation of these inverse limits. For this, one proves an analogue of Proposition~\ref{prop:translate_isogenies_to_maps_of_stacks}, showing that for any $n,s\ge 1$ there exist integers $N\ge m$ such that, for any discrete $\Int/p^s\Int$-algebra $C$, we have a canonical commuting diagram
   \begin{align*}
\begin{diagram}
\mathsf{P}^{\mathrm{syn}}_{n,\{0,1\}}(C)&\rTo&\mathsf{P}^{\mathrm{syn},(N)}_{n,\{0,1\}}(C)&\rTo&\mathsf{P}^{\mathrm{syn}}_{n,\{0,1\}}(C)\\
&\rdTo&\dTo&\ldTo\\
&&\mathsf{P}^{\mathrm{syn},(m)}_{n,\{0,1\}}(C).
\end{diagram}
\end{align*}
The line of argument is very similar to that of the proposition (though one category level higher) and involves the use of deformation theory to reduce to the case of perfect rings where one can work concretely with Witt vectors.
\end{remark}

The rest of this section is devoted to the proof of the theorem.

\subsection{The functor $\mathcal{G}$}
We can assume that $R$ is $p$-nilpotent. The functor is defined using syntomic cohomology: Given $\mathcal{M}\in \mathsf{P}^{\mathrm{syn}}_{n,\{0,1\}}(R)$, we take $\Gamma_{\mathrm{syn}}(\mathcal{M})$ to be the functor on $\mathrm{CRing}_{R/}$ given by
\[
\Gamma_{\mathrm{syn}}(\mathcal{M}):C \mapsto \tau^{\leq 0}R\Gamma\left(C^{\mathrm{syn}},\mathcal{M}\vert_{C^{\mathrm{syn}}}\right).
\]

This is a quasisyntomic sheaf over $R$, and to check that it is represented by a finite flat group scheme, we can work pro-\'etale locally. Therefore, by Theorem~\ref{thm:syntomic_raynaud}, we can assume that $\mathcal{M}$ is the cofiber of a map of vector bundle $F$-gauges $\mathcal{V}^{-1}\xrightarrow{f}\mathcal{V}^0$. Now, Theorem~\ref{thm:dieudonne} gives us $p$-divisible groups $\mathcal{H}^{-1}$ and $\mathcal{H}^0$ associated with $\mathcal{V}^{-1}$ and $\mathcal{V}^0$ and a map $\alpha:\mathcal{H}^{-1}\to \mathcal{H}^0$.

\begin{lemma}
\label{lem:gamma_syn_ffg}
$\alpha$ is an isogeny of $p$-divisible groups and the classical truncation $\mathcal{G}(\mathcal{M})$ of $\Gamma_{\mathrm{syn}}(\mathcal{M})$ is represented by the finite flat group scheme $\ker(\alpha)$.
\end{lemma}
\begin{proof}
Note that we have a map $\hat{f}:\mathcal{V}^{0}\to \mathcal{V}^{-1}$ with $f\circ \hat{f}$ and $\hat{f}\circ f$ both equal to multiplication by $p^n$. This yields a map $\hat{\alpha}:\mathcal{H}^0\to \mathcal{H}^{-1}$ with $\alpha\circ \hat{\alpha}$ and $\hat{\alpha}\circ \alpha$ equal to multiplication by $p^n$. In particular, $\alpha$ is surjective.

It remains to verify that $\ker \alpha$ is finite flat over $R$ and that it is isomorphic to $\mathcal{G}(\mathcal{M})$. The first actually follows from the existence of the map $\hat{\alpha}$, since it shows that $\ker \alpha$ is a quotient of the finite flat group scheme $\mathcal{H}^0[p^n]$\footnote{This fact doesn't seem to have a ready reference in the literature, but here is an argument conveyed to us by Alex Youcis:  Suppose that $\mathcal{G}_1\to \mathcal{G}_2$ is a map of finitely presented affine group schemes that is a quotient map for the fppf topology, with $\mathcal{G}_1$ finite flat over the base. We want to know that $\mathcal{G}_2$ is also finite flat. We reduce to the case of a Noetherian base. Then the fiber-by-fiber criterion for flatness tells us that $\mathcal{G}_1$ is faithfully flat, and hence finite flat, over $\mathcal{G}_2$. This, combined with the finite flatness of $\mathcal{G}_1$ over the base, then finishes the proof.}but we can give a simultaneous proof of both facts by studying $\mathcal{G}(\mathcal{M})$.

Theorem~\ref{thm:perfect_f-gauges_repble} tells us that $\Gamma_{\mathrm{syn}}(\mathcal{M})$ is an inverse limit of the locally finitely presented derived algebraic stacks
\[
\Gamma_{\mathrm{syn},m}(\mathcal{M})\defn \Gamma_{\mathrm{syn}}(\mathcal{M}\otimes\Int/p^m\Int).
\]

Moreover, if $\Fil^\bullet_{\mathrm{Hdg}}M$ is the perfect complex over $R$ with Tor amplitude in $[-1,0]$ obtained by pulling $\mathcal{M}$ back along $x^{\mathcal{N}}_{\dR,R}$, $\Gamma_{\mathrm{syn}}(\mathcal{M})$ admits a tangent complex over $R$ given by the pullback of $\gr^{-1}_{\mathrm{Hdg}}M[-1]$, which of course is perfect with Tor amplitude in $[0,1]$, and also has virtual rank $0$.

Now, since $\mathcal{M}$ is isomorphic to a direct summand of $\mathcal{M}\otimes\Int/p^n\Int$, we see that $\Gamma_{\mathrm{syn}}(\mathcal{M})$ is also locally finitely presented.

Combining the three previous paragraphs now shows that $\Gamma_{\mathrm{syn}}(\mathcal{M})$ is represented by a quasi-smooth derived algebraic stack over $R$ of virtual codimension $0$.

On the other hand, for every $m\ge 1$, we have an isomorphism
\[
\Gamma_{\mathrm{syn},m}(\mathcal{M})_{\mathrm{cl}}\xrightarrow{\simeq}\cofib(\mathcal{H}^{-1}[p^m]\to \mathcal{H}^0[p^m]),
\]
and taking inverse limits yields an isomorphism $\mathcal{G}(\mathcal{M})\xrightarrow{\simeq}\ker F$. 

This shows that the classical truncation of $\Gamma_{\mathrm{syn}}(\mathcal{M})$ takes discrete values; thus $\Gamma_{\mathrm{syn}}(\mathcal{M})$ is a quasi-smooth derived algebraic space over $R$ of virtual rank $0$. It also shows that, on algebraically closed fields, it takes \emph{finite} values. Therefore, by~\cite[Lemma 11.3.4]{gmm}, it is in fact a quasi-finite flat lci scheme over $R$, and, since it is a closed subscheme of $\mathcal{G}_1[p^k]$ for $k$ sufficiently large, it is in fact finite flat.

This verifies all the assertions of the lemma.
\end{proof}

\begin{remark}
   [Cartier duality pairing]
\label{rem:cartier_duality_pairing}
 We have a canonical pairing
 \[
    \Gamma_{\mathrm{syn}}(\mathcal{M}^\vee[1]\{1\})\times \Gamma_{\mathrm{syn}}(\mathcal{M}) \to \Gamma_{\mathrm{syn}}(\mathcal{M}^\vee\otimes \mathcal{M}[1]\{1\})\to \Gamma_{\mathrm{syn}}(\mathcal{O}^{\mathrm{syn}}\{1\}[1]).
 \]
 Here, $\mathcal{O}^{\mathrm{syn}}$ is the structure sheaf of $R^{\mathrm{syn}}$, and we have used the tautological pairing $\mathcal{M}^\vee\otimes \mathcal{M} \to \mathcal{O}^{\mathrm{syn}}$. By Proposition~\ref{prop:syntomic_and_fppf} and~\cite[Proposition 11.4.1]{gmm}, we have a canonical isomorphism of formal stacks
 \[
    \Gamma_{\mathrm{syn}}(\mathcal{O}^{\mathrm{syn}}\{1\}[1])\xrightarrow{\simeq}\varprojlim_m B\mup[p^m],
 \]
 where the transition maps are induced by the $p$-power map $\mup[p^{n+1}]\to \mup[p^n]$. If $\mathcal{M}$ is in $\mathsf{P}^{\mathrm{syn}}_{n,\{0,1\}}(R)$, then this factors canonically through the fiber of the $p^n$-power map
 \[
    \fib(\varprojlim_m B\mup[p^m]\xrightarrow{p^n}\varprojlim_m B\mup[p^m])\xrightarrow{\simeq}\varprojlim_m \mup[p^m]\otimes^{\mathbb{L}}_{\Int}\Int/p^n\Int\xrightarrow{\simeq}\mup[p^n].
 \]
 Explicitly, the inverse to the composition of these isomorphisms takes a section $x\in \mup[p^n](C)$ to the compatible family of $\mup[p^{m-n}]$-torsors (for $m\ge n$) parameterizing sections $x_m\in \mup[p^m](C)$ such that $x_m^{p^{m-n}} = x$. 

 In sum, for $\mathcal{M}\in \mathsf{P}^{\mathrm{syn}}_{n,\{0,1\}}(R)$ we have a canonical pairing
 \begin{align}\label{eqn:cartier_duality_canonical_pairing}
 \mathcal{G}(\mathcal{M}^\vee[1]\{1\})\times \mathcal{G}(\mathcal{M})\to \mup[p^n].
 \end{align}
\end{remark}

\subsection{Full faithfulness and essential surjectivity}

Next, we show:
\begin{lemma}
\label{lem:full_faithfulness}
The functor $\mathcal{G}$ is fully faithful.
\end{lemma} 
\begin{proof}
We want to show that the map
\[
\Hom(\mathcal{M}_1,\mathcal{M}_2)\to \Hom(\mathcal{G}(\mathcal{M}_1),\mathcal{G}(\mathcal{M}_2))
\]
is an isomorphism for $\mathcal{M}_1,\mathcal{M}_2\in \mathsf{P}^{\mathrm{syn}}_{n,\{0,1\}}(R)$.

By pro-\'etale descent and Theorem~\ref{thm:syntomic_raynaud}, we can assume that we have, for $i=1,2$, cofiber sequences
\[
\mathcal{V}^{-1}_i\xrightarrow{f_i} \mathcal{V}^{0}_i\to \mathcal{M}_i
\]
with $\mathcal{V}^?_i$ vector bundle $F$-gauges over $R$. 

Theorem~\ref{thm:dieudonne} gives us maps of $p$-divisible groups $\mathcal{H}_i^{-1}\to \mathcal{H}^0_i$ for $i=1,2$ associated functorially with $f_i$. 

Set $\mathcal{G}_i = \mathcal{G}(\mathcal{M}_i)$. We first claim that there are canonical isomorphisms
\begin{align*}
\Hom(\mathcal{M}_1,\mathcal{V}^?_2[1])&\xrightarrow{\simeq}\Hom(\mathcal{G}_1,\mathcal{H}^?_2)
\end{align*}
for $?=-1,0$.

For this, note that we have a diagram with exact rows
\[
\begin{diagram}
0&\rTo&\Hom(\mathcal{V}_1^0,\mathcal{V}_2^?)&\rTo& \Hom(\mathcal{V}_1^{-1},\mathcal{V}_2^?)&\rTo& \Hom(\mathcal{M}_1,\mathcal{V}_2^?[1])&\rTo& \Hom(\mathcal{V}_1^0,\mathcal{V}_2^?[1])&\rTo& \Hom(\mathcal{V}_1^{-1},\mathcal{V}_2^?[1])\\
&&\dTo&&\dTo&&\dTo&&\dTo&&\dTo\\
0&\rTo& \Hom(\mathcal{H}_1^0,\mathcal{H}_2^?)&\rTo& \Hom(\mathcal{H}_1^{-1},\mathcal{H}_2^?)&\rTo& \Hom(\mathcal{G}_1,\mathcal{H}_2^?)&\rTo& \Ext(\mathcal{H}_1^0,\mathcal{H}_2^?)&\rTo& \Ext(\mathcal{H}_1^{-1},\mathcal{H}_2^?).
\end{diagram}
\]

The vertical maps are all obtained via the functor $\Gamma_{\mathrm{syn}}$, and the three vertical arrows on the right also use Proposition~\ref{prop:syntomic_and_fppf}. All arrows except the middle one are known to be isomorphisms: the first two from Theorem~\ref{thm:dieudonne} and the last two from Theorem~\ref{thm:exactness}. From this and the five lemma it follows that the middle arrow is also an isomorphism.

Now, consider another diagram
\[
\begin{diagram}
0&\rTo&\Hom(\mathcal{M}_1,\mathcal{M}_2)&\rTo&\Hom(\mathcal{M}_1,\mathcal{V}^0_2[1])&\rTo&\Hom(\mathcal{M}_1,\mathcal{V}^{-1}_2[1])\\
&&\dTo&&\dTo&&\dTo\\
0&\rTo&\Hom(\mathcal{G}_1,\mathcal{G}_2)&\rTo&\Hom(\mathcal{G}_1,\mathcal{H}^0_2)&\rTo&\Hom(\mathcal{G}_1,\mathcal{H}^{-1}_2)
\end{diagram}
\]
Once again, the maps here are induced by $\Gamma_{\mathrm{syn}}$. The bottom row is exact, and we have just seen that the right two vertical arrows are isomorphisms. It is now enough to know that the top row is also exact, which certainly would follow if we knew $\Hom(\mathcal{M}_1,\mathcal{V}_2^{-1}) = 0$. This can be checked just as in the proof of Lemma~\ref{lem:enough_to_check_mod_n}: This is the space of sections of $\mathcal{F}= \mathcal{M}_1^\vee\otimes\mathcal{V}_2^{-1}$, which is a perfect $F$-gauge of Hodge-Tate weights $\{0,1\}$ and Tor amplitude $[0,1]$. We reduce to the case where $R$ is a perfect $\Field_p$-algebra, where the conclusion follows since $\mathcal{V}_2^{-1}$ is a vector bundle over the $\Int_p$-flat stack $R^{\mathrm{syn}}$.
\end{proof}

Let us now verify essential surjectivity. By the full faithfulness shown above, it suffices to check that any finite flat group scheme is \'etale locally in the image of the functor. By Raynaud's theorem~\cite[Th\'eor\`eme 3.1.1]{bbm:cris_ii}, we can assume that the group scheme is the kernel of an isogeny between $p$-divisible groups. But now we can simply invoke Theorem~\ref{thm:dieudonne}, and combine it with Lemma~\ref{lem:gamma_syn_ffg}, to see that such a group scheme is in the image of $\mathcal{G}$.

\subsection{Exactness}

Let us turn to the exactness of the equivalence. As in the proof of Theorem~\ref{thm:exactness}, for any $\mathcal{M}_1, \mathcal{M}_2$ in $\mathsf{P}^{\mathrm{syn}}_{n,\{0,1\}}(R)$ with $G_i = \mathcal{G}(\mathcal{M}_i)$ for $i=1,2$, we have a canonical map 
\begin{align}\label{eqn:ext_to_ext}
\Hom(\mathcal{M}_1,\mathcal{M}_2[1])\to \Ext(G_1,G_2)
\end{align}
and it remains to prove:

\begin{lemma}\label{lem:ffg_ext_bijection}
The map~\eqref{eqn:ext_to_ext} is a bijection.
\end{lemma}
\begin{proof}
By \'etale descent and Corollary~\ref{cor:bt_in_the_middle}, we can assume that there exist $\mathcal{V},\mathcal{W}\in \mathrm{Vect}^{\mathrm{syn}}_{n,\{0,1\}}(R)$ and fiber sequences 
\[
\mathcal{M}_2\to \mathcal{V}\to \mathcal{M}'_2\;;\; \mathcal{M}'_2\to \mathcal{W}\to \mathcal{M}_2
\]
in $\mathsf{P}^{\mathrm{syn}}_{n,\{0,1\}}(R)$. Set $H = \mathcal{G}(\mathcal{V})$ and $G'_2 = \mathcal{G}(\mathcal{M}'_2)$. Then we get the following commuting diagram with exact rows:
\[
\begin{diagram}
\Hom(\mathcal{M}_1,\mathcal{V})&\rTo& \Hom(\mathcal{M}_1,\mathcal{M}'_2)&\rTo& \Hom(\mathcal{M}_1,\mathcal{M}_2[1])&\rTo& \Hom(\mathcal{M}_1,\mathcal{V}[1])&\rTo&\Hom(\mathcal{M}_1,\mathcal{M}'_2[1])\\
\dTo&&\dTo&&\dTo&&\dTo&&\dTo\\
\Hom(G_1,H)&\rTo& \Hom(G_1,G'_2)&\rTo& \Ext(G_1,G_2)&\rTo& \Ext(G_1,H)&\rTo&\Ext(G_1,G'_2).
\end{diagram}
\]
The two vertical maps on the left are isomorphisms. Suppose that we knew that the fourth vertical arrow is an isomorphism. This would tell us that the third arrow is injective. If we knew further that the fifth arrow is also injective, then we could also conclude that the third arrow is in fact an isomorphism. This injectivity would however be implied by the same argument applied with the fiber sequence $\mathcal{M}'_2\to \mathcal{W}\to \mathcal{M}_2$. In summary, it is enough to know that~\eqref{eqn:ext_to_ext} is an isomorphism when $\mathcal{M}_2\in \mathrm{Vect}^{\mathrm{syn}}_{n,\{0,1\}}(R)$. 

A similar argument in the other variable reduces us to the case where $\mathcal{M}_1$ is also in $\mathrm{Vect}^{\mathrm{syn}}_{n,\{0,1\}}(R)$\footnote{One can also obtain this reduction via Cartier duality.},  and so we are now done by Theorem~\ref{thm:exactness}. 
\end{proof}

At this point, we have essentially completed the proof of Theorem~\ref{thm:main}. The remaining assertions to be checked are the compatibility with Cartier duality, arbitrary base-change and fpqc descent. By construction, the equivalence is compatible with quasisyntomic descent. Therefore, we can use Theorem~\ref{thm:syntomic_raynaud} (and Remark~\ref{rem:cartier_duality_pairing}) to reduce the verification to the case of truncated Barsotti-Tate group schemes, where it is already known by Theorem~\ref{thm:dieudonne}.

\section{Some cohomological complements}
\label{sec:complements_to_the_classification_result}

In this section, we record the relationship between the cohomological realizations of an object in $\mathrm{FFG}_n(R)$ with those of its associated $F$-gauge. We follow this up with a couple of applications to the fppf cohomology of such group schemes, including a strengthening of a result of Bragg-Olsson on the representability of the relative fppf cohomology of a finite flat group scheme in proper smooth families and a proof (following Bhatt-Lurie) of the $p$-primary part of purity results of \v{C}esnavi\v{c}ius-Scholze.

\subsection{Cohomological realizations}

Suppose that $R$ is $p$-complete and discrete and that we have $G\in \mathrm{FFG}(R)$ corresponding to $\mathcal{M}\in \mathsf{P}^{\mathrm{syn}}_{\{0,1\}}(R)$. The results of this subsection are immediate from the proof of Theorem~\ref{thm:main} and the corresponding facts for truncated BT groups.

\begin{proposition}
   [Flat cohomology via syntomic cohomology]
\label{prop:flat_via_syntomic_ffg}
There is a canonical isomorphism
\[
   R\Gamma_{\mathrm{fppf}}(\Spec R,G)\xrightarrow{\simeq}R\Gamma(R^{\mathrm{syn}},\mathcal{M}).
\]
In fact, for every $m\ge 1$ there is a canonical isomorphism of smooth Artin $m$-stacks
\begin{align*}
B^mG \xrightarrow{\simeq}\Gamma_{\mathrm{syn}}(\mathcal{M}[m]).
\end{align*}
\end{proposition}

\begin{remark}
[Fppf cohomology over animated commutative rings]
\label{rem:fppf_animated}
For $G\in \mathrm{FFG}(R)$, one can define the fppf cohomology $R\Gamma_{\mathrm{fppf}}(\Spec C,G)$ for any $p$-nilpotent animated commutative $R$-algebra $C$; see~\cite[\S 5.2]{Cesnavicius2019-pn}. For any $n\ge 0$, the values of $\tau^{\le n}R\Gamma_{\mathrm{fppf}}(\Spec C,G)[n]$ are canonically isomorphic to $(B^nG)(C)$. Therefore, the conclusion of Proposition~\ref{prop:flat_via_syntomic_ffg} shows that we have a canonical isomorphism
\[
R\Gamma_{\mathrm{fppf}}(\Spec C,G)\xrightarrow{\simeq}R\Gamma(C^{\mathrm{syn}},\mathcal{M}).
\]
\end{remark}

\begin{remark}
[Crystalline realizations]
 The notation here will be as in Remark~\ref{rem:crystalline_realization_bt}. By~\cite[Definition 3.1.5]{bbm:cris_ii}, we have the Dieudonn\'e \emph{complex} $\Delta(G^*)$, given as the truncated internal RHom sheaf $\tau^{\le 1}\underline{\mathrm{RHom}}(G^*,\mathcal{O}^{\mathrm{crys}})$ in the big crystalline topos of $\Spf R$ equipped with the fppf topology. This is a perfect complex of crystals with Tor amplitude in $[0,1]$ and admits a map from the complex $\Fil^0_{\mathrm{Hdg}}\Delta(G^*) \defn \tau^{\le 1}\underline{\mathrm{RHom}}(G^*,\mathcal{I}^{\mathrm{crys}})$. For any divided power thickening $C'\twoheadrightarrow C$ of $p$-nilpotent $R$-algebras, this gives us a map of perfect complexes
\[
   \Fil^0\Delta(G^*)(C')\to \Delta(G^*)(C').
\]
 When evaluated on the trivial thickening $R\xrightarrow{\mathrm{id}}R$, we obtain a cofiber sequence (see~\cite[Proposition 3.2.10]{bbm:cris_ii})
 \[
    \Lie(G^*)^\vee[-1]\simeq (\Fil^0\Delta(G^*))(R)\to \Delta(G^*)(R)\to \Lie(G).
 \]
On the other hand, just as in Remark~\ref{rem:crystalline_realization_bt}, associated with $\mathcal{M}$ we have a perfect complex of crystals $\Delta(\mathcal{M})$ equipped with a filtration $\Fil^\bullet_{\mathrm{Hdg}}\Delta(\mathcal{M})$.
\end{remark}

\begin{proposition}
   [Comparison with crystalline realization]
\label{prop:crystalline_comp_ffg}
There is a canonical isomorphism of maps of perfect complexes of sheaves in the crystalline site
\[
   (\Fil^0_{\mathrm{Hdg}}\Delta(\mathcal{M})\to \Delta(\mathcal{M}))\xrightarrow{\simeq} (\Fil^0_{\mathrm{Hdg}}\Delta(G^*)[1]\to \Delta(G^*)[1]).
\]
In particular, if $\Fil^\bullet_{\mathrm{Hdg}}M$ is the filtered perfect complex over $R$ obtained from $\mathcal{M}$ via pullback along $x^{\mathcal{N}}_{\dR}$, then there is a canonical isomorphism
\[
\gr^{-1}_{\mathrm{Hdg}}M\xrightarrow{\simeq}\Lie(G)[1]
\]
\end{proposition}

\begin{remark}
   [Grothendieck-Messing for finite flat group schemes]
Combining Proposition~\ref{prop:crystalline_comp_ffg} with Theorem~\ref{thm:HTwts01_representable} yields a Grothendieck-Messing theory for finite flat group schemes. More precisely, the groupoid of lifts of $G\in \mathrm{FFG}(R)$ along a nilpotent divided power thickening $R'\twoheadrightarrow R$ are in canonical equivalence with the groupoid of filtered perfect complex structures on the perfect complex $\Delta(G)(R')$ lifting the Hodge-filtered complex $\Fil^\bullet_{\mathrm{Hdg}}\Delta(G)(R)$. This is a special case of a result of Faltings~~\cite[Theorem 17]{Faltings2002-mp}.
\end{remark}

\begin{corollary}
   [The Mazur-Roberts carpet]
\label{cor:mazur_roberts_ffg}
Suppose that $(C'\twoheadrightarrow C,\gamma)$ is a nilpotent divided power extension of $p$-nilpotent animated commutative $R$-algebras with $I = \ker(C'\twoheadrightarrow C)$. Then we have a canonical fiber sequence
\[
R\Gamma_{\mathrm{fppf}}(\Spec C',G)\to R\Gamma_{\mathrm{fppf}}(\Spec C,G)\to \Lie(G)\otimes_RI[1]
\]
\end{corollary}

\begin{remark}
\label{rem:p-complete_fppf_limit}
    Suppose that $C$ is a derived $p$-complete animated commutative $R$-algebra. By a standard limiting argument, using the integrability of the iterated classifying stacks $B^mG$, we have
    \[
     R\Gamma(\Spec C,G) \xrightarrow{\simeq}\varprojlim_m R\Gamma(\Spec C/{}^{\mathbb{L}}p^m,G).
    \]
    If $p>2$, then we can combine this with Corollary~\ref{cor:mazur_roberts_ffg} applied via a limit argument to the pro-nilpotent divided power extension $R\twoheadrightarrow R/{}^{\mathbb{L}}p$ to obtain a  canonical fiber sequence
   \[
      R\Gamma_{\mathrm{fppf}}(\Spec C,G)\to R\Gamma_{\mathrm{fppf}}(\Spec C/{}^{\mathbb{L}}p,G)\to \Lie(G)[1].
   \]
   Concretely, if $C$ is $p$-completely flat over $\Int_p$, then the map $H^i(\Spec C,G)\to H^i(\Spec C/pC,G)$ is an isomorphism for $i\ge 2$, and we have a long exact sequence
   \[
    0\to t_G\to G(C)\to G(C/pC)\to v_G\to H^1_{\mathrm{fppf}}(\Spec C,G)\to H^1_{\mathrm{fppf}}(\Spec C/pC,G)\to 0.
   \]
\end{remark}

\begin{proposition}
   [\'Etale comparison]
\label{prop:etale_comparison}
There is a canonical isomorphism $T_{\et}(\mathcal{M})\xrightarrow{\simeq}G^{\ad}_\eta$ as well as a comparison isomorphism
\[
R\Gamma_{\mathrm{qsyn}}(\Spf R,\mathcal{M}_{-}[1/\mathcal{I}])^{\varphi = \mathrm{id}}\xrightarrow{\simeq}R\Gamma_{\et}(\Spf(R)^{\ad}_\eta,G^{\ad}_\eta).
\]
The notation here is as in Remark~\ref{rem:etale_realization_cohom}.
\end{proposition}

\subsection{Representability of smooth proper pushforwards}
The main theorem of this subsection generalizes a result of Bragg-Olsson~\cite[Theorem 1.8]{bragg2021representability}, which, in the formulation here, deals with the special case where $X$ and $S$ are algebraic spaces over a field of characteristic $p$ and $\mathcal{G}$ is a height $1$ finite flat $p$-torsion group scheme. In \emph{loc. cit.}, the authors are also able to prove some weaker results, but also under weaker hypotheses: They assume $\pi$ to be projective, but not necessarily smooth, and they prove representability generically over the base.
\begin{theorem}
\label{thm:bragg_olsson_smooth}
Let $\pi:\mathfrak{X}\to \mathfrak{S}$ be a proper smooth map of $p$-adic formal algebraic spaces. Suppose that we have $\mathcal{G}\in \mathrm{FFG}(\mathfrak{X})$. Then: 
\begin{enumerate}
   \item For all $n\ge 0$, the fppf sheaf $\hat{\mathcal{R}}^n \defn (\tau^{\leq n}R\pi_* \mathcal{G})[n]$ is represented by a finitely presented Artin $n$-stack over $\mathfrak{S}$.
   \item If $\hat{\mathcal{R}}^i$ is flat over $\mathfrak{S}$ for all $i<n$, then the fppf sheaf $\hat{\mathcal{H}}^n \defn R^n\pi_*\mathcal{G}$ is represented by a formally finitely presented formal algebraic space over $\mathfrak{S}$ and the natural map $\hat{\mathcal{R}}^n \to \hat{\mathcal{H}}^n$ is faithfully flat.
\end{enumerate}
\end{theorem}  
\begin{proof}
By \'etale descent, we can assume that $\mathfrak{S} = \Spec R$ for some $R\in \mathrm{CRing}^{p-\mathrm{nilp}}_{\heartsuit}$. Let $\mathcal{M}\in \mathsf{P}^{\mathrm{syn}}_{n,\{0,1\}}(\mathfrak{X}^{\mathrm{syn}})$ be the $F$-gauge associated with $\mathcal{G}$ via Theorem~\ref{thm:main}. 

Now, by Proposition~\ref{prop:pushforwards_proper_smooth} below, for some $m\ge 1$, $\mathcal{F} = R\pi^{\mathrm{syn}}_* \mathcal{M}$ is an $\Int/p^m\Int$-module in perfect $F$-gauges over $R$ of Hodge-Tate weights $\leq 1$ and Tor amplitude $[-1,d]$. In particular, for every $n\ge 0$, Theorem~\ref{thm:perfect_f-gauges_repble} tells us that $\Gamma_{\mathrm{syn}}(\mathcal{F}[n]/{}^{\mathbb{L}}p^m)$ is represented by a finitely presented derived Artin $(n+1)$-stack over $\mathfrak{S}$. Since $\mathcal{F}[n]/{}^{\mathbb{L}}p^m \simeq \mathcal{F}[n]\oplus \mathcal{F}[n+1]$, one deduces from this that $\Gamma_{\mathrm{syn}}(\mathcal{F}[n])$ is also represented by a finitely presented derived Artin $n$-stack over $\mathfrak{S}$.

By Proposition~\ref{prop:flat_via_syntomic_ffg}, we have, for any classical $R$-algebra $C$,
\[
(R\pi_* \mathcal{G})(C) \simeq R\Gamma\left(C^{\mathrm{syn}},\mathcal{F}\vert_{C^{\mathrm{syn}}}\right).
\]

In particular, we find that, for all $n\ge 0$, $\hat{\mathcal{R}}^n$ is the classical truncation of $\Gamma_{\mathrm{syn}}(\mathcal{F}[n])$ and so is a locally finitely presented Artin $n$-stack over $\mathfrak{S}$.

The second part of the theorem follows from~\cite[Theorem 5.2]{bragg2021representability}.
\end{proof}

\begin{remark}
   [Shape of the stack]
Suppose that $R$ is an $\Field_p$-algebra with the resolution property (so that every perfect complex can be represented as a complex of vector bundles). Then one can use Remark~\ref{rem:f-gauges_repble_close_look} to see that, when the theorem shows that $R^n\pi_* \mathcal{G}$ is representable, it also shows that it is contained in the category of stacks over $S$ obtained by taking the smallest Serre subcategory of the category of fppf sheaves of abelian groups containing by vector bundles and height $1$ finite flat $p$-torsion group schemes over $S$.
\end{remark}

\begin{corollary}
    Let $\pi:X\to S$ be a proper smooth map of algebraic spaces and suppose that we have $\mathcal{G}\in \mathrm{FFG}(S)$. Suppose also that the fppf sheaf $R^1\pi_*\pi^*\mathcal{G}$ is represented by a finitely presen{}ted algebraic space that is flat over $S$.\footnote{In fact, the representability is automatic. It is the flatness that has to be taken as a hypothesis.} Then the fppf sheaf $R^2\pi_*\pi^*\mathcal{G}$ is also represented by a finitely presented algebraic space over $S$.
\end{corollary}
\begin{proof}
    By Noetherian approximation and \'etale descent, we can assume that $S$ is an affine scheme of finite type over $\Int$.
    
    Set $\mathcal{H}^i = R^i\pi_*\pi^*\mathcal{G}$ and $\mathcal{R}^i = (\tau^{\le i}R\pi_*\pi^*\mathcal{G})[i]$. We want to show that $\mathcal{H}^2$ is represented by a finitely presented algebraic space over $S$. 
    
    Quite generally, the restriction of $\mathcal{H}^i$ over $S[1/p]$ is represented by an \'etale algebraic space: Indeed, it is a constructible sheaf by~\cite[Exp. XVI, Th\'eor\`eme 1]{SGA4}, and so we conclude using~\cite[Exp. IX, Proposition 2.7]{SGA4}. 

    Now, since $\pi$ is proper smooth, one finds that $\mathcal{H}^0 = \pi_*\pi^*\mathcal{G}$ is once again finite flat over $S$: It is the Weil restriction of $\mathcal{G}$ from the finite \'etale scheme over $S$ corresponding to the quasicoherent sheaf $\pi_*\Reg{X}$. This, combined with our hypothesis and (2) of Theorem~\ref{thm:bragg_olsson_smooth}, tells us that the $p$-adic formal completion $\mathfrak{H}^2$ of $\mathcal{H}^2$ (that is, its restriction to $p$-nilpotent schemes over $S$) is represented by a formally finitely presented formal algebraic space over the $p$-adic completion $\mathfrak{S}$ of $S$.

    Let $\mathfrak{H}^{2,\mathrm{ad}}_\eta$ be the adic generic fiber of $\mathfrak{H}^2$. We claim that the natural map from $\mathfrak{H}^{2,\mathrm{ad}}_\eta$ to the adic space $\mathcal{H}^2[1/p]^{\ad}$ associated with $\mathcal{H}^2[1/p]$ is \'etale: this is because both source and target are \'etale as adic algebraic spaces over $S[1/p]^{\ad}$.

    By~\cite[Theorem 2.25]{achinger_youcis}, we now find that there exists a unique algebraic space $\tilde{\mathcal{H}}^2$ over $S$ restricting to $\mathcal{H}^2[1/p]$ over $S[1/p]$ and to $\mathfrak{H}^2$ over $\mathfrak{S}$ with gluing data given by the \'etale map from the previous paragraph.

    To finish, we must know that $\tilde{\mathcal{H}}^2$ does indeed represent $\mathcal{H}^2$. By results of \v{C}esnavi\v{c}ius-Scholze (see Proposition~\ref{prop:beauville_laszlo} below), for any affine scheme $\Spec R\to S$ over $S$ with bounded $p$-power torsion, and with $\hat{R}$ the classical $p$-completion of $R$, the following square is Cartesian\footnote{One also needs a Mittag-Leffler argument for the identification of $\mathcal{H}^2(\hat{R})$ with $\mathfrak{H}^2(\hat{R})$: This uses the fact that $\mathcal{H}^1$ is of finite type over $S$.}
    \[
    \begin{diagram}
        \mathcal{H}^2(R)&\rTo&\mathcal{H}^2(\hat{R})\simeq \mathfrak{H}^2(\hat{R})\\
        \dTo&&\dTo\\
        \mathcal{H}^2(R[1/p]&\rTo&\mathcal{H}^2(\hat{R}[1/p]).
    \end{diagram}
    \]
    Therefore, the full faithfulness part of~\cite[Theorem 2.25]{achinger_youcis} applied to the algebraic spaces $\Spec R$ and $\tilde{\mathcal{H}}^2$ now shows that we have $\mathcal{H}^2(R)\simeq \tilde{\mathcal{H}}^2(R)$.
\end{proof}

\begin{remark}
   [Derived algebraization]
In fact, if one uses the Artin-Lurie representability theorem~\cite[Theorem 7.1.6]{lurie_thesis}, then it is possible to show that the complex $(\tau^{[1,n]}R\pi_*\pi^*\mathcal{G})[n]$ is represented by an Artin $n$-stack over $S$. 
\end{remark}

\begin{remark}
[Cohomology of the multiplicative group]
When $n=1$, the corollary tells us that $R^1\pi_*\mup[p^m]$ is represented by a finitely presented algebraic space over $S$, and when this algebraic space is flat over $S$, we find that $R^2\pi_*\mup[p^m]$ is also represented by a finitely presented algebraic space over $S$. For instance, this is the case when $S = \Spec \kappa$ is the spectrum of a field.

In general, however, it is only the \emph{complex} $(\tau^{[1,2]}R\pi_*\mup[p^m])[2]$ that can be expected to have good representability properties.
\end{remark}

\begin{proposition}
\label{prop:pushforwards_proper_smooth}
Let $\pi:\mathfrak{X}\to \mathfrak{S}$ be a proper smooth map of $p$-adic formal algebraic spaces of relative dimension $d$. If $\mathcal{M}$ is a perfect $F$-gauge over $X$ of Hodge-Tate weights in $[n,m]$ and Tor amplitude $[a,b]$, then $R\pi^{\mathrm{syn}}_*\mathcal{M}$ is a perfect $F$-gauge over $\mathfrak{S}$ with Hodge-Tate weights in $[n-d,m]$ and Tor amplitude $[a,b+2d]$.
\end{proposition}
\begin{proof}
This can be deduced from results of Guo-Li~\cite{guo2023frobenius}. However, since these are not stated in the generality we require, we sketch a proof here for the convenience of the reader. 

It is of course enough to prove it after replacing `syn' with `$\mathcal{N}$'. One can reduce to the case where $\mathfrak{S} = \Spec R$ is in $\mathrm{CRing}^{p\text{-nilp}}$. Now, by quasisyntomic descent we can assume that $R$ is semiperfectoid. In particular, $R^{\mathcal{N}}$ is canonically isomorphic to the formal Rees stack $\Rees(\Fil^\bullet_{\mathcal{N}}\Prism_R)$. This means that perfect complexes over $R^{\mathcal{N}}$ are equivalent to filtered perfect complexes over the filtered animated commutative ring $\Fil^\bullet_{\mathcal{N}}\Prism_R$.\footnote{We are using the fact that such objects are automatically derived $(p,I_R)$-complete.}

Suppose that $\Spec A\to \mathfrak{X}$ is an affine quasisyntomic cover with $A$ semiperfectoid, and consider the corresponding simplicial scheme $\Spec A^{(\bullet)}$ where
\[
\Spec A^{(i)} = \underbrace{\Spec S\times_{X}\Spec S\times\cdots\times_{X}\Spec S}_{i\text{-times}}.
\]
Then $A^{\mathcal{N}}\to \mathfrak{X}^{\mathcal{N}}$ is a flat cover, and we have 
\[
A^{(i),\mathcal{N}}\simeq \underbrace{A^{\mathcal{N}}\times_{X^{\mathcal{N}}}A^{\mathcal{N}}\times\cdots\times_{X^{\mathcal{N}}}A^{\mathcal{N}}}_{i\text{-times}}.
\]

This shows that $R\pi^{\mathcal{N}}_* \mathcal{M}$ corresponds to the filtered complex $\mathrm{Tot}\left(\Fil^\bigstar_{\mathcal{N}}\mathcal{M}^{(\bullet)}\right)$, where $\Fil^\bigstar_{\mathcal{N}}\mathcal{M}^{(\bullet)}$ is the filtered perfect complex over $\Fil^\bigstar_{\mathcal{N}}A^{(\bullet)}$. To show that this is perfect with Tor amplitude in $[a,b+2d]$, it is enough to know that the associated graded $\mathrm{Tot}\left(\gr^\bigstar_{\mathcal{N}}\mathcal{M}^{(\bullet)}\right)$ is a graded perfect complex over $\gr^\bigstar_{\mathcal{N}}\Prism_R$ with Tor amplitude in $[a,b+2d]$.

Let $\gr^\heartsuit_{\mathrm{Hdg}}M$ be the graded perfect complex over $\mathfrak{X}$ obtained by pulling $\mathcal{M}$ back along the de Rham point, and let $\gr^{\heartsuit}_{\mathrm{Hdg}}M^{(\bullet)}$ be the cosimplicial graded perfect complex over $A^{(\bullet)}$ obtained by via restriction along $\Spec A^{(\bullet)}\to \mathfrak{X}$.

Then $\gr^\bigstar_{\mathcal{N}} \mathcal{M}^{(\bullet)}$ admits a canonical finite `weight' filtration whose $i$-th graded piece admits a canonical isomorphism
\[
\gr^i_{\mathrm{wt}}\gr^\bigstar_{\mathcal{N}}\mathcal{M}^{(\bullet)}\simeq \gr^i_{\mathrm{Hdg}}M^{(\bullet)}\otimes_{A^{(\bullet)}}\gr^\bigstar_{\mathcal{N}}\Prism_{A^{(\bullet)}}(i),
\]
where $(i)$ denotes an $i$-shift in grading.\footnote{Recall our convention that the $i$-th associated graded piece of a filtered module is in graded degree $-i$.}

This shows that $\mathrm{Tot}\left(\gr^\bigstar_{\mathcal{N}}\mathcal{M}^{(\bullet)}\right)$ inherits a finite filtration with associated graded pieces
\[
\mathrm{Tot}\left(\gr^i_{\mathrm{Hdg}}M^{(\bullet)}\otimes_{A^{(\bullet)}}\gr^\bigstar_{\mathcal{N}}\Prism_{A^{(\bullet)}}(i)\right)
\]

Now, the condition on Hodge-Tate weights implies that these pieces are non-zero only $i\in [-m,-n]$. This reduces us to the following

\begin{lemma}
For any perfect complex $M$ over $\mathfrak{X}$ restricting to a graded perfect complex $M^{(\bullet)}$ over $A^{(\bullet)}$ with Tor amplitude in $[a,b]$, the graded complex
\[
\mathrm{Tot}\left(M^{(\bullet)}\otimes_{A^{(\bullet)}}\gr^\bigstar_{\mathcal{N}}\Prism_{A^{(\bullet)}}\right)
\]
is graded perfect over $\gr^\bigstar_{\mathcal{N}}\Prism_R$ with Tor amplitude $[a,b+2d]$, and its graded base-change along $\gr^\bigstar_{\mathcal{N}}\Prism_R\to R$, where $R$ is trivially graded, is supported in graded degrees $[0,d]$.
\end{lemma}
\begin{proof}
Choose a map $R_0\to R$ with $R_0$ perfectoid along with a generator $\xi\in \Fil^1_{\mathcal{N}}\Prism_{R_0}$. For $i\ge 0$, this yields isomorphisms
\[
\gr^i_{\mathcal{N}}\Prism_{A^{(\bullet)}}\xrightarrow{\simeq}\gr^i_{\mathcal{N}}\varphi^*\Prism_{A^{(\bullet)}/R_0}\xrightarrow{\simeq}\Fil^{\mathrm{conj}}_i\overline{\Prism}_{A^{(\bullet)}/R_0},
\]
and we have $\gr^{\mathrm{conj}}_i\overline{\Prism}_{A^{(\bullet)}/R_0}\simeq \wedge^i\mathbb{L}_{A^{(\bullet)}/R_0}[-i]$.

Therefore, $\mathrm{Tot}\left(M^{(\bullet)}\otimes_{A^{(\bullet)}}\gr^\bigstar_{\mathcal{N}}\Prism_{A^{(\bullet)}}\right)$ corresponds to a decreasingly filtered complex over $\Fil^{\mathrm{conj}}_\bigstar \overline{\Prism}_{A^{(\bullet)}/R_0}$, and considering associated gradeds reduces us to knowing that the graded complex
\[
\mathrm{Tot}\left(M^{(\bullet)}\otimes_{A^{(\bullet)}}\gr^{\mathrm{conj}}_\bigstar\overline{\Prism}_{A^{(\bullet)}/R_0}\right)\simeq R\Gamma(\mathfrak{X},M\otimes_{\Reg{\mathfrak{X}}}\wedge^{\bigstar}\mathbb{L}_{\mathfrak{X}/R_0}[-\bigstar])
\]
is graded perfect over $\gr^{\mathrm{conj}}_\bigstar\overline{\Prism}_{R/R_0}$ with Tor amplitude in $[a,b+2d]$, and that its graded base-change over $R$ is supported in graded degrees $[0,d]$. 

Now, $\wedge^i\mathbb{L}_{\mathfrak{X}/R_0}$ is canonically filtered with graded pieces isomorphic to $\wedge^k\mathbb{L}_{\mathfrak{X}/R}\otimes \wedge^l\mathbb{L}_{R/R_0}$, for $k+l = i$. 

One can upgrade this to knowing that $\wedge^\bigstar \mathbb{L}_{\mathfrak{X}/R_0}[-\bigstar]$, as a complex over $\mathcal{X} = \mathfrak{X}\times_{\Spec R}\Spec(\gr^{\mathrm{conj}}_\bigstar\overline\Prism_{R/R_0})/\Gm$, is filtered with graded pieces isomorphic to $\wedge^k\mathbb{L}_{\mathfrak{X}/R}[-k]\otimes_R\Reg{\mathcal{X}}(-k)$. Now we finally use our assumption that $\mathfrak{X}$ is smooth over $R$ of relative dimension $d$, which tells us that each of these graded pieces is a shifted vector bundle in degree $k$ and vanishes if $k>d$.

Therefore, we are now reduced to knowing that the relative cohomology over $\Spec(\gr^{\mathrm{conj}}_\bigstar\overline{\Prism}_{R/R_0})/\Gm$ of the restriction of a perfect complex $M$ of Tor amplitude $[a,b]$ over the product $\mathfrak{X}\times_{\Spec R}\Spec(\gr^{\mathrm{conj}}_\bigstar\overline\Prism_{R/R_0})/\Gm$ is represented by a graded perfect complex with Tor amplitude in $[a,b+d]$. This is of course a standard fact about the coherent cohomology of proper morphisms of relative dimension $d$.
\end{proof}

\end{proof}

\subsection{Purity of fppf cohomology}
\label{subsec:purity}

\begin{definition}
   Let $X$ be a scheme and $Z\subset X$ a constructible closed subset with complement $U = X\backslash Z$. For any fppf sheaf of abelian groups $G$ over $X$, we set
   \[
      R\Gamma_Z(X,G) = \fib(R\Gamma_{\mathrm{fppf}}(X,G)\to R\Gamma_{\mathrm{fppf}}(U,G)).
   \]
\end{definition}

\begin{proposition}
[Excision]
\label{prop:beauville_laszlo}
Suppose that $R$ is a  discrete, not necessarily $p$-complete, ring with derived $p$-completion $\hat{R}$, and that we have $G\in \mathrm{FFG}(R)$. Then the square
\[
\begin{diagram}
   R\Gamma_{\mathrm{fppf}}(\Spec R,G)&\rTo&R\Gamma_{\mathrm{fppf}}(\Spec\hat{R},G)\\
   \dTo&&\dTo\\
    R\Gamma_{\et}(\Spec R[1/p],G)&\rTo&R\Gamma_{\et}(\Spec\hat{R}[1/p],G)
\end{diagram}
\]
is Cartesian and moreover the natural map
\[
R\Gamma_{\mathrm{fppf}}(\Spec\hat{R},G)\to R\Gamma_{\mathrm{fppf}}(\Spf\hat{R},G)
\]
is an isomorphism. 
\end{proposition}
\begin{proof}
   The last assertion about algebraic vs. formal fppf cohomology was already observed in Remark~\ref{rem:p-complete_fppf_limit}
   
   Using \'etale descent for fppf cohomology, we can reduce to the case where $R$ is $p$-Henselian, where the result follows from results in~\cite{Cesnavicius2019-pn}. Indeed, with the $p$-Henselian condition, the bottom arrow in the diagram is an isomorphism by~\cite[Remark 2.2.6]{Cesnavicius2019-pn}. Therefore, we have to check that the top arrow is an isomorphism, which follows from the first paragraph and~\cite[Theorem 5.3.5]{Cesnavicius2019-pn}.
\end{proof}

\begin{corollary}
   \label{cor:beauville_laszlo}
In the situation of Proposition~\ref{prop:beauville_laszlo}, let $\mathcal{M}\in \mathsf{P}_{\{0,1\}}^{\mathrm{syn}}(\hat{R})$ be the perfect $F$-gauge over $\hat{R}$ associated with $G$ via Theorem~\ref{thm:main}. Then we have a Cartesian square
\[
 \begin{diagram}
   R\Gamma_{\mathrm{fppf}}(\Spec R,G)&\rTo&R\Gamma(\hat{R}^{\mathrm{syn}},\mathcal{M})\\
   \dTo&&\dTo\\
    R\Gamma_{\et}(\Spec R[1/p],G)&\rTo&R\Gamma_{\mathrm{qsyn}}(\Spf\hat{R},\mathcal{M}_{-}[1/\mathcal{I}])^{\varphi = \mathrm{id}}.
\end{diagram}  
\]
\end{corollary}
\begin{proof}
   This follows from Propositions~\ref{prop:flat_via_syntomic_ffg},~\ref{prop:etale_comparison} and~\ref{prop:beauville_laszlo}, combined with the isomorphism
   \[
    R\Gamma(\Spf(\hat{R})^{\ad}_\eta,G^{\ad}_\eta) \xrightarrow{\simeq}  R\Gamma_{\et}(\Spec\hat{R}[1/p],G).
   \]
   This isomorphism follows from~\cite[Corollary 3.2.2]{Huber1996-bg} and finite Galois descent.
\end{proof}

\begin{remark}
[Fpqc descent for fppf cohomology]
\label{rem:fpqc_descent_for_fppf_cohomology}
One can now give a slightly streamlined proof of fpqc descent for fppf cohomology~\cite[Theorem 5.5.2]{Cesnavicius2019-pn}. Instead of reducing to the case of perfect $\Field_p$-algebras, we can directly appeal to Remark~\ref{rem:fpqc_descent_for_syntomic_cohomology} and Proposition \ref{prop:flat_via_syntomic_ffg} instead.
\end{remark}

\begin{definition}
   A derived scheme $Y$ over $\Field_p$ is \defnword{quasisyntomic} if $\mathbb{L}_{Y/\Field_p}$ has Tor amplitude in $[-1,0]$ as an object of $D(Y)$. A scheme $X$ is $p$-\defnword{quasisyntomic} if its structure sheaf has bounded $p$-power torsion and if its derived base-change over $\Field_p$ is quasisyntomic. This implies that  $X$ be covered by affine opens $\Spec R$ such that the the derived $p$-completion $\hat{R}$ of $R$ is in fact the classical $p$-completion and such the $p$-completed absolute cotangent complex $\mathbb{L}_{\hat{R}}$ has Tor amplitude in $[-1,0]$ as an object in $D(R)$. 
\end{definition}

The next result is a direct generalization of~\cite[Corollary 8.5.7]{bhatt2022absolute} (from which we borrow the terminology and notation) and implies Theorem~\ref{introthm:cesnavicius-scholze}.

\begin{theorem}
[Purity]
\label{thm:purity}
Let $X$ be a $p$-quasisyntomic qcqs scheme, and let $Z\subset X$ be a constructible closed subset contained in $X_{(p=0)}$. Suppose that there exists an integer $d\ge 0$ such that, for every affine open $V\subset X$, we have $R\Gamma_{Z\cap V}(V,\Reg{V})\in D^{\ge d}$. Then, for any $p$-power torsion commutative finite locally free commutative group scheme $G$ over $X$, we have $R\Gamma_Z(X,G)\in D^{\ge d}$.
\end{theorem}
\begin{proof}
It is enough to consider the case where $X = \Spec R$ is affine. Since $U = X\backslash Z$ is a quasicompact open with $U[1/p]=\Spec R[1/p]$, using Zariski descent, we find from Corollary~\ref{cor:beauville_laszlo} a Cartesian square
\[
 \begin{diagram}
   R\Gamma_{\mathrm{fppf}}(U,G)&\rTo&R\Gamma(\hat{U}^{\mathrm{syn}},\mathcal{M})\\
   \dTo&&\dTo\\
    R\Gamma_{\et}(\Spec R[1/p],G)&\rTo&R\Gamma_{\mathrm{qsyn}}(\hat{U},\mathcal{M}_{-}[1/\mathcal{I}])^{\varphi = \mathrm{id}}.
\end{diagram}  
\]
Here $\hat{U}$ is the (classical) $p$-adic formal scheme obtained by taking the completion of $U$ along $U_{(p=0)}$. Therefore, we see that $R\Gamma_Z(X,G)$ is the total limit of the diagram
\[
    \begin{diagram}
   R\Gamma(\hat{R}^{\mathrm{syn}},\mathcal{M})&\rTo&R\Gamma(\hat{U}^{\mathrm{syn}},\mathcal{M})\\
   \dTo&&\dTo\\
    R\Gamma_{\mathrm{qsyn}}(\Spf\hat{R},\mathcal{M}_{-}[1/\mathcal{I}])^{\varphi = \mathrm{id}}&\rTo&R\Gamma_{\mathrm{qsyn}}(\hat{U},\mathcal{M}_{-}[1/\mathcal{I}])^{\varphi = \mathrm{id}}.
\end{diagram} 
\] 

By Remark~\ref{rem:qsynt_site}, one now finds that $R\Gamma_Z(X,G)$ is obtained by applying the functor 
\[
R\Gamma_Z(\Spf \hat{R},\_) = \fib(R\Gamma_{\mathrm{qsyn}}(\Spf \hat{R},\_)\to R\Gamma_{\mathrm{qsyn}}(\hat{U},\_))
\]
to the diagram of quasisyntomic sheaves
\[
\begin{diagram}
   \Fil^0 \mathcal{M}_{-}&\rTo^{\varphi_0-\mathrm{can}}& \mathcal{M}_{-}\\
   \dTo&&\dTo\\
   (\Fil^0 \mathcal{M}_{-})[1/\mathcal{I}]&\rTo_{\varphi_0 - \mathrm{can}}& \mathcal{M}_{-}[1/\mathcal{I}]
\end{diagram}
\]
and then taking the total limit.

Now, following the argument in~\cite[Corollary 8.5.7]{bhatt2022absolute}, we can reduce to the case where $R$ is a classically $p$-complete algebra over $R_0 = \Int_p[\zeta_{p^\infty}]^{\wedge}_p$, and all sheaves above can be viewed as modules over the initial (perfect) prism $(A,I)$ for $R_0$. Here, in the notation of \emph{loc. cit.}, we have $I = ([p]_q)$ where $[p]_q\equiv (q-1)^p \pmod{p}$, and we finally find an isomorphism
\[
R\Gamma_Z(X,G) \xrightarrow{\simeq}\fib\left(R\Gamma_Z(\Spf R,R\Gamma_{(q-1)}(\Fil^0 \mathcal{M}_{-} ))\xrightarrow{\varphi_0 - \mathrm{can}}R\Gamma_Z(\Spf R,R\Gamma_{(q-1)}(\mathcal{M}_{-} ))\right)
\]

Now, we have
\[
 \gr^{-1}_{\mathrm{Hdg}}M[-1] \simeq \fib(\Fil^0 \mathcal{M}_{-}\xrightarrow{\mathrm{can}} \mathcal{M}_{-}),
\]
where we are viewing $\gr^{-1}_{\mathrm{Hdg}}M[-1]$ as a perfect complex over $\Spf R$ with Tor amplitude in $[0,1]$ and with cohomology killed by a power of $p$. Our hypothesis shows that $R\Gamma_Z(\Spf R,\gr^{-1}_{\mathrm{Hdg}}M[-1])$ is in $D^{\ge d}$. 

The rest of the proof is essentially the same as that of~\cite[Corollary 8.5.7]{bhatt2022absolute}: it only remains to know that $R\Gamma_Z(\Spf R,R\Gamma_{(q-1)}(\mathcal{M}_{-} ))$ is in $D^{\ge d}$. Using the fiber sequence 
\[
 \mathcal{M}_{-}\xrightarrow{[p]_q} \mathcal{M}_{-}\to \overline{\mathcal{M}}_{-}
\]
it suffices to know that $R\Gamma_Z(\Spf R,\overline{\mathcal{M}}_{-}[-1])$ is in $D^{\ge d}$. As in the proof of Proposition~\ref{prop:pushforwards_proper_smooth}, this latter object admits an exhaustive increasing filtration with associated gradeds of the form
\[
 R\Gamma_Z(\Spf R,N\otimes_R\wedge^i\mathbb{L}_{R/R_0}[-i])
\]
where $N$ is a perfect complex over $\Spf R$ with Tor amplitude $[0,1]$ and cohomology killed by a power of $p$. Since $R$ is $p$-quasisyntomic, $N\otimes_R\wedge^i\mathbb{L}_{R/R_0}[-i]$ is a $p$-power torsion object in $D^{\ge 0}(R)$, and so the theorem follows.
\end{proof}

\begin{corollary}
   \label{cor:cesnavicius_scholze}
Theorem~\ref{introthm:cesnavicius-scholze} holds.
\end{corollary}
\begin{proof}
   Via \'etale localization~\cite[Lemma 7.1.1]{Cesnavicius2019-pn}, we can reduce to the case where $X$ is lci of dimension $\leq d$. This means that it is $p$-quasisyntomic and the vanishing condition for local cohomology in Theorem~\ref{thm:purity} holds with $d$ as given here.
\end{proof}

\section{Frames and divided Dieudonn\'e complexes}
\label{sec:divided_dieudonne}

Here we show that our classification of finite flat $p$-power torsion commutative group schemes can be translated into terms very close to those appearing in~\cite{MR4530092} but over quite general frames. We then apply this to recover various known classifications of such group schemes, and also give some extensions of such results.

\subsection{Frames}

\begin{definition}
An (animated) \defnword{frame} for $R\in \mathrm{CRing}^{p\text{-comp}}$ is a tuple $\underline{A} = (A\twoheadrightarrow \overline{A},A\twoheadrightarrow R,\varphi,\overline{\varphi})$, where:
\begin{enumerate}
   \item $A\twoheadrightarrow\overline{A}$ is a surjective map of animated commutative rings with fiber $I\to A$ given by a generalized Cartier divisor such that $A$ is derived $(p,I)$-complete;
   \item $A\twoheadrightarrow R$ is a surjection of animated commutative rings with fiber $\Fil^1A$;
   \item $\varphi:A\to A$ is a `na\"ive' Frobenius lift in the sense that it is an endomorphism of animated commutative rings such that the induced endomorphism of $\pi_0(A)/p\pi_0(A)$ is Frobenius;
   \item $\overline{\varphi}:R\to \overline{A}$ is map of animated commutative rings;
\end{enumerate}
along with a commutative diagram
\begin{align}\label{eqn:comm_diagram_frame}
\begin{diagram}
A&\rTo^\varphi&A\\
\dOnto&&\dOnto\\
R&\rTo_{\overline{\varphi}}&\overline{A}.
\end{diagram}
\end{align}
We will write $\varphi_1:\Fil^1A\to I$ for the $\varphi$-linear map induced on the fibers of the vertical arrows. 
\end{definition}

\begin{remark}[Maps between frames]
\label{rem:frame_maps}
Frames organize into an $\infty$-category in a natural way where maps 
\[
f:\underline{A} \to \underline{A}' = (A'\twoheadrightarrow{A}',A'\twoheadrightarrow R',\varphi',\overline{\varphi}')
\]
are maps between the corresponding commutative diagrams~\eqref{eqn:comm_diagram_frame} such that the induced map $A'\otimes_A\overline{A}\to \overline{A}'$ is an equivalence (so that $A'\otimes_AI \xrightarrow{\simeq}I'$).
\end{remark}

\begin{remark}
The notion of a frame was introduced by Zink~\cite{Zink2001-dg}, but the one we use here is different: Even when $A$ is discrete, we do not require that $I$ be generated by $p$, nor do we require $\Fil^1A$ to be endowed with divided powers. Our notion is very closely related to that of a generalized frame as defined by Lau~\cite[\S 11]{MR2679066}.
\end{remark}

\begin{remark}
In \cite[\S 5]{gmm}, one finds a somewhat different definition of an animated frame, which is a generalization of the notion of a higher frame due to Lau~\cite{MR4355476}. Every frame in the sense of~\cite{gmm} gives rise to a frame in the sense employed here. This more refined notion will show up briefly in \S~\ref{subsec:specializing_to_the_breuil_kisin_case}.
\end{remark}

\begin{definition}
[Prismatic frames]
A frame $\underline{A}$ is \defnword{prismatic} if the following conditions hold:
\begin{enumerate}
   \item The pair $(A,I)$ is a(n animated) prism.\footnote{See~\cite[\S 2]{bhatt2022prismatization}.}
   \item The endomorphism $\varphi:A\to A$ is the one obtained from the underlying animated $\delta$-ring structure  on $A$: In particular, it is a lift of the Frobenius endomorphism of $A/{}^{\mathbb{L}}p$.
\end{enumerate}
\end{definition}

\begin{remark}
   [\'Etale lifts for frames]
\label{rem:etale_lifts_frames}
If $\underline{A}$ is a frame for $R$ and $f:R\to R'$ is a $p$-completely \'etale map, then there is a canonical frame $\underline{A}'$ for $R'$ and a map $\underline{A}\to \underline{A}'$ lifting $f$ such that the underlying map $A\to A'$ is $(p,I)$-completely \'etale. If $\underline{A}$ is prismatic, then so is $\underline{A}'$, and we have an underlying map of prisms $(A,I)\to (A',I')$. For an explanation, see for instance the discussion in~\cite[Proposition 5.4.25]{gmm}.
\end{remark}

\begin{example}[$p$-adic frames]
\label{ex:crystalline_frame}
Suppose that $A$ is an animated $\delta$-ring; then $A$ is equipped with a canonical \emph{crystalline} prism structure where $I \to A$ is isomorphic to $A\xrightarrow{p}A$.  We will call a prismatic frame $p$\defnword{-adic} if its underlying prism is of this form.
\end{example}

\begin{example}
[Prismatic frames for semiperfectoid rings]
Suppose that $R\in \mathrm{CRing}^{p\text{-comp}}$ is semiperfectoid. Then its absolute prismatic cohomology $(\Prism_R,I_R)$---which is the initial prism for $R$---underlies a prismatic frame for $R$ with $\overline{\Prism}_R$ its absolute Hodge-Tate cohomology. This frame is $p$-adic if and only if $R$ is an $\Field_p$-algebra.
\end{example}  

\begin{example}
[The Witt frame]
\label{example:witt_frame}
Suppose that we have $R\in \mathrm{CRing}^{p\text{-comp}}$. Then the animated ring of Witt vectors $W(R)$ underlies a $p$-adic frame for $R$ with $\overline{W(R)} = W(R)/{}^{\mathbb{L}}p$ with the map $R\to \overline{W(R)}$ classifying $x_{\dR}:\Spf R \to R^\Prism$: That is, it is the canonical map through which the composition $W(R)\xrightarrow{F}W(R)\to\overline{W(R)}$ factors.
\end{example}

\begin{definition}
[Laminations]
A \defnword{lamination} for a prismatic frame $\underline{A}$ for $R$ is an extension of the canonical map of $\delta$-rings $A\to W(R)$ lifting $A\to R$ to a map of frames $\underline{A}\to \underline{W(R)}$. Giving such an extension is equivalent to giving an isomorphism of Cartier-Witt divisors
\[
(I\otimes_AW(R)\to W(R))\xrightarrow{\simeq}(W(R)\xrightarrow{p}W(R))
\]
as points of $R^{\Prism}(R)$. If $\underline{A}$ is equipped with a lamination, then we we will say that it is \defnword{laminated}. 
\end{definition}

\begin{remark}
\label{rem:laminated_via_map}
   If $\underline{A}\to \underline{A}'$ is a map of prismatic frames, then any lamination for $\underline{A}$ induces one for $\underline{A}'$.
\end{remark}

\begin{remark}
   \label{rem:lamination_criterion}
Suppose that $\underline{A}$ is a prismatic frame for $R$ and that we have an isomorphism of $R$-algebras
\[
   \overline{A}\otimes_AW(R)\simeq W(R)/{}^{\mathbb{L}}(I\otimes_AW(R))\xrightarrow{\simeq}W(R)/{}^{\mathbb{L}}p
\]
Then the left hand side is in particular an $\Field_p$-algebra, and we immediately obtain a canonical isomorphism of Cartier-Witt divisors $(W(R)\xrightarrow{p}W(R))\simeq (I\otimes_AW(R)\to W(R))$; see~\cite[Example 5.1.9]{bhatt_lectures}. From this, we find that an isomorphism as above equips $\underline{A}$ with a lamination.
\end{remark}

\begin{remark}
[Laminations and the prismatization]
\label{rem:laminations_and_prismatization}
Suppose that we have a prism $(A,I,\overline{\varphi}:R\to \overline{A})$ in the absolute prismatic site for $R$. Extending this to the datum of a prismatic frame $\underline{A}$ for $R$ amounts to giving a factoring of $A\xrightarrow{\varphi}A\to \overline{A}$ through $\overline{\varphi}$ via a surjective map $A\to R$. Remark~\ref{rem:prisms_and_prismatization} gives us a map $\Spf A \to R^\Prism$, and we can consider the composition
\[
   \Spf R \to \Spf A \to R^\Prism,
\]
which classifies the Cartier-Witt divisor $(I\otimes_AW(R)\to W(R))$ equipped with the structure map
\[
   R\xrightarrow{\bar{\varphi}}\overline{A}\to \overline{A}\otimes_AW(R)
\]
When $\underline{A}$ is laminated, then this map is isomorphic to the canonical map $R\to \overline{W(R)}$ classifying the de Rham point $x_{\mathrm{dR}}:\Spf R \to R^\Prism$. Conversely, if we have such an isomorphism of maps, then we obtain an isomorphism of $R$-algebras
\[
W(R)/{}^{\mathbb{L}}(I\otimes_AW(R))\xrightarrow{\simeq} W(R)/{}^{\mathbb{L}}p,
\]
which endows $\underline{A}$ with a lamination by Remark~\ref{rem:lamination_criterion}.
\end{remark}

\begin{remark}
[Laminations for frames for semiperfectoids]
\label{rem:semiperf_lamination}
If $R$ is semiperfectoid, then Remark~\ref{rem:laminations_and_prismatization} admits the following interpretation: Given $(A,I,\overline{\varphi}:R\to \overline{A})$, since $(\Prism_R,I_R)$ is the initial prism for $R$, we obtain a canonical map $(\Prism_R,I_R)\to (A,I)$. Extending the given tuple to the datum of a laminated prismatic frame $\underline{A}$ is now equivalent to giving a map of frames $\underline{\Prism}_R\to \underline{A}$ lifting the map $(\Prism_R,I_R,R\to \overline{\Prism}_R)\to (A,I,R\to \overline{A})$. In particular, we see that the category of laminated prismatic frames for $R$ admits an initial object, $\underline{\Prism}_R$, and a final object, $\underline{W(R)}$.
\end{remark}

\begin{example}
   [Frames for polynomial algebras]
\label{ex:frames_for_polynomial_algebras}
Suppose that $\underline{A}$ is a laminated prismatic frame for $R$. If $S = R[x_1,\ldots,x_n]^{\wedge}_p$ is a $p$-completed polynomial algebra over $R$, then $B = A[x_1,\ldots,x_n]^{\wedge}_p$, equipped with any Frobenius lift extending that on $A$, is uniquely endowed with the structure of a laminated prismatic frame $\underline{B}$ for $S$ equipped with a map of laminated frames $\underline{A}\to \underline{B}$ lifting $R\to S$. This applies in particular to the situation where $R$ is semiperfectoid and $\underline{A} = \underline{\Prism}_R$.
\end{example}

\begin{remark}
[Base-change for laminated frames]
\label{rem:base_change_laminated}
Suppose that $\underline{A}$ is a laminated prismatic frame for $R$ and that $S$ is a $p$-complete $R$-algebra such that $S^\Prism\to R^\Prism$ is relatively formally affine. By Remark~\ref{rem:base_change_prisms}, we have an object $(B,J,S\to \overline{B})$ in the absolute prismatic site for $S$ such that $\Spf B\to S^\Prism$ is isomorphic to $\Spf A\times_{R^\Prism}S^\Prism$. Since $\underline{A}$ is laminated, the de Rham point $\Spf R\to R^\Prism$ factors through $\Spf A$, and hence the corresponding point for $S$ factors as
\[
\Spf S \to \Spf R\times_{R^\Prism}S^\Prism\simeq \Spf B \to S^\Prism,
\]
implying that we have a map $B\to S$ extending $A\to R$. We now claim that the map $B\xrightarrow{\varphi}B\to \overline{B}$ factors through $B\to S$ and lifts $(A,I,R\to \overline{A})\to (B,J,S\to \overline{B})$ to a map of frames $\underline{A}\to \underline{B}$. The claim comes down to two observations:
\begin{itemize}
   \item For any $\overline{B}$-algebra $C$, since $I\otimes_AW(C)\to W(C)$ is in the Hodge-Tate locus, its base-change under $F:W(C)\to W(C)$ is canonically isomorphic to $W(C)\xrightarrow{p}W(C)$~\cite[Proposition 3.6.6]{bhatt2022absolute}.
   \item In this situation, the map 
   \[
      W(C)/{}^{\mathbb{L}}(I\otimes_AW(C))\to W(C)/{}^{\mathbb{L}}p
   \]
   induced by $F$ factors through the quotient $W(C)\twoheadrightarrow C$. Therefore, the structure map $S\to W(C)/{}^{\mathbb{L}}(I\otimes_AW(C))\to W(C)/{}^{\mathbb{L}}p$ factors through a map $S\to C$.
\end{itemize}
\end{remark}

\begin{lemma}
\label{lem:syntomic_frame_criterion}
Suppose that we have $(A,I,R\to \overline{A})$ in the prismatic site for $R$, and suppose that we have $R\to R_{\infty}$ as in Remark~\ref{rem:quasisyntomic_base_change}. Then in the notation of \emph{loc. cit.}, the following data are equivalent:
\begin{enumerate}
   \item A cosimplicial prismatic frame $\underline{A}^{(\bullet)}_{\infty}$ for $R^{\otimes_R(\bullet+1)}_\infty$ with underlying prism $(A_\infty^{(\bullet)},I_\infty^{(\otimes)})$ such that the map 
   \[
   (\Prism_{R^{\otimes_R(\bullet +1)}_\infty},I_{R^{\otimes_R(\bullet +1)}_\infty})\to (A^{(\bullet)}_\infty,I^{(\bullet)}_\infty)
   \]
   of cosimplicial prisms in the prismatic site for $R_{\infty}^{\otimes_R( \bullet +1)}$ lifts to a map $\underline{\Prism}_{R^{\otimes_R(\bullet +1)}_\infty}\to \underline{A}^{(\bullet)}_\infty$ of cosimplicial frames for $R^{\otimes_R(\bullet +1)}_\infty$.
   \item The datum of a laminated frame $\underline{A}$ for $R$ with underlying tuple $(A,I,R\to \overline{A})$.
\end{enumerate}
\end{lemma}
\begin{proof}
The implication (1)$\Rightarrow$(2) is immediate from Remark~\ref{rem:semiperf_lamination} and $p$-completely faithfully flat descent, while the other implication follows from Remark~\ref{rem:base_change_laminated}.
\end{proof}

\subsection{Divided Dieudonn\'e complexes}
\label{subsec:divided_dieudonne_modules}

From now on, all our frames will be prismatic unless otherwise noted.

\begin{definition}[Divided Dieudonn\'e complexes]
\label{defn:adm_dieudonne}
Let $\underline{A}$ be a frame for $R$. A \defnword{divided Dieudonn\'e complex} over $R$ \defnword{with respect to $\underline{A}$} is a tuple $\underline{\mathsf{M}} = (\mathsf{M},\Fil^0M\to M,\mathsf{M}\xrightarrow{\Psi_{\mathsf{M}}}\varphi^*\mathsf{M},\xi)$ such that:
\begin{enumerate}
   \item $\mathsf{M}\in \mathrm{Perf}(A)$ is a perfect complex over $A$;
   \item $\Fil^0M \to M \defn R\otimes_{A}\mathsf{M}$ is a map of perfect complexes over $R$;
   \item $\Psi_{\mathsf{M}}:\mathsf{M}\to \varphi^*\mathsf{M}$ is a map of perfect complexes over $A$;
   \item $\xi$ is an isomorphism of perfect complexes over $A$ sitting in a diagram
   \[
    \begin{diagram}
    \varphi^*\mathsf{M}&\rTo&\cofib(\Psi_{\mathsf{M}})\\
    &\rdTo&\dTo^{\simeq}_{\xi}\\
    &&\overline{A}\otimes_{\overline{\varphi},R}\gr^{-1}M.
    \end{diagram}
   \] 
\end{enumerate}
Here, $\gr^{-1}M = M/\Fil^0M$ and the diagonal map is obtained as the composition
\[
\varphi^*\mathsf{M}\to \overline{A}\otimes_A\varphi^*\mathsf{M}\xrightarrow{\simeq}\overline{A}\otimes_{\overline{\varphi},R}M\to \overline{A}\otimes_{\overline{\varphi},R}\gr^{-1}M.
\]

These can be organized into a stable $\infty$-category in a natural way, which we denote by $\mathrm{DDC}_{\underline{A}}(R)$.

A divided Dieudonn\'e complex has \defnword{Tor amplitude in $[a,b]$} if $\mathsf{M}$ is in $\mathrm{Perf}^{[a,b]}(A)$ and if $\Fil^0M,\gr^{-1}M$ are both in $\mathrm{Perf}^{[a,b]}(R)$. Write $\mathrm{DDC}^{[a,b]}_{\underline{A}}(R)$ for the subcategory spanned by the objects with Tor amplitude in $[a,b]$.
\end{definition} 

\begin{remark}
    One can give a rigorous definition of the $\infty$-category $\mathrm{DDC}_{\underline{A}}(R)$ as follows: For any stable $\infty$-category $\mathcal{A}$, there are three canonical functors
    \[
    s_{\mathcal{A}},t_{\mathcal{A}},c_{\mathcal{A}}: \mathrm{Fun}([1],\mathcal{A})\to \mathcal{A}
    \]
    of stable $\infty$-categories. Here, $[1]$ is the category consisting of two objects with one non-identity arrow between them, and $\mathrm{Fun}([1],\mathcal{A})$ denotes the $\infty$-category of arrows in $\mathcal{A}$. The functor $s_{\mathcal{A}}$ carries an arrow to its source, $t_{\mathcal{A}}$ carries it to its target, and $c_{\mathcal{A}}$ carries it to its cofiber. If $\mathcal{A} = \mathrm{Perf}(C)$ for some animated commutative ring $C$, we will write ${C}$ for the subscript instead of $\mathrm{Perf}(C)$. 

    Consider the stable $\infty$-categories
    \[
    \mathcal{C} = \mathrm{Fun}([1],\mathrm{Perf}(A)) \times \mathrm{Fun}([1],\mathrm{Perf}(R))\;;\; \mathcal{D} = \mathrm{Perf}(A)\times \mathrm{Perf}(A)\times  \mathrm{Perf}(A)\times \mathrm{Perf}(A)\times \mathrm{Perf}(R)\times \mathrm{Perf}(R);
    \]
    \[
    \mathcal{E} = \mathrm{Perf}(A) \times \mathrm{Perf}(A)\times  \mathrm{Perf}(R).
    \]
    There are functors of stable $\infty$-categories
    \[
    \tau = (s_A\circ \mathrm{pr}_1,t_A\circ \mathrm{pr}_1,c_A\circ \mathrm{pr}_1,\overline{\varphi}^*\circ c_R\circ \mathrm{pr}_2,t_R\circ \mathrm{pr}_2,c_R\circ \mathrm{pr}_2): \mathcal{C} \to \mathcal{D};
    \]
    \[
    \alpha = (\varphi^*\circ \mathrm{pr}_1,\mathrm{pr}_1,\Delta_A\circ \mathrm{pr}_2,\Delta_R\circ \mathrm{pr}_3).
    \]
    Here $\Delta_A$ (resp. $\Delta_R$) is the diagonal for $\mathrm{Perf}(A)$ (resp. $\mathrm{Perf}(R)$). Also, we have written $\overline{\varphi}^*$ for the composition of pullback along $\overline{\varphi}:R\to \overline{A}$ composed with the forgetful functor from $\mathrm{Perf}(\overline{A})$ to $\mathrm{Perf}(A)$.

    We now have a Cartesian diagram of stable $\infty$-categories
    \[
    \begin{diagram}
        \mathrm{DDC}_{\underline{A}}(R)&\rTo&\mathcal{C}\\
        \dTo&&\dTo_{\tau}\\
        \mathcal{E}&\rTo_{\alpha}&\mathcal{D}
    \end{diagram}
    \]
\end{remark}

\begin{remark}
[Functoriality]
\label{rem:funct_adm_dieudonne}
Suppose that $f:\underline{A}\to \underline{A}'$ is a map of frames. Then there is a canonical base-change functor $f^*:\mathrm{DDC}_{\underline{A}}(R)\to \mathrm{DDC}_{\underline{A}'}(R')$ lifting the natural base-change functors from $\mathrm{Perf}(A)$ to $\mathrm{Perf}(A')$ and $\mathrm{Perf}(R)$ to $\mathrm{Perf}(R')$. 
\end{remark}

\begin{remark}[Alternative description]
\label{rem:pseudo_verschiebung}
In the situation of Definition~\ref{defn:adm_dieudonne}, there is a canonical map
\[
\sigma_{\mathsf{M}}:I\otimes_A\varphi^*\mathsf{M}\to \mathsf{M}
\]
sitting in a diagram
\[
\begin{diagram}
I\otimes_A\varphi^*\mathsf{M}&\rTo&\varphi^*\mathsf{M}&\rTo&\overline{A}\otimes_A\varphi^*\mathsf{M}\\
\dTo^{\sigma_{\mathsf{M}}}&&\dEquals&&\dTo\\
\mathsf{M}&\rTo^{\Psi_{\mathsf{M}}}&\varphi^*\mathsf{M}&\rTo&\overline{A}\otimes_{\overline{\varphi},R}\gr^{-1}M
\end{diagram}
\]
where the rows are fiber sequences. In particular, note that we have an equivalence
\begin{align}\label{eqn:cofib_zeta}
\cofib(\sigma_{\mathsf{M}})\xrightarrow{\simeq}\overline{A}\otimes_{\overline{\varphi},R}\Fil^0M
\end{align}
Now, set $\mathsf{N} = \mathsf{M}\{-1\}$. We then have a canonical isomorphism $ I\otimes_A\varphi^*\mathsf{M}\{-1\}\xrightarrow{\simeq}\varphi^*\mathsf{N}$, and so $\sigma_{\mathsf{M}}\{-1\}$ yields a map
\[
   \Phi_{\mathsf{N}}:\varphi^* \mathsf{N}\to \mathsf{N}
\]
equipped with an isomorphism $\cofib(\Phi_{\mathsf{N}})\xrightarrow{\simeq}\overline{A}\otimes_{\overline{\varphi},R}\Fil^1N$, where $\Fil^1N \defn \Fil^0M\{-1\}$. One can alternatively define divided Dieudonn\'e complexes in terms of these data instead.
\end{remark}

\begin{remark}[Breuil-Kisin twist]
Since $\underline{A}$ is assumed to be prismatic, we have a canonical invertible $A$-module $A\{1\}$, the \defnword{Breuil-Kisin twist}, equipped with a canonical isomorphism $u:I^{-1}\otimes_AA\{1\}\xrightarrow{\simeq}\varphi^*A\{1\}$. For any $A$-module $\mathsf{M}$ and $n\in \Int$, we set $\mathsf{M}\{n\}\defn \mathsf{M}\otimes_AA\{1\}^{\otimes_A n}$.
\end{remark}

\begin{definition}[Cartier duality]
\label{defn:adm_dieudonne_cartier_duality}
Suppose that $\underline{\mathsf{M}}$ is in $\mathrm{DDC}_{\underline{A}}(R)$. Then the (unnormalized) \defnword{Cartier dual} $\underline{\mathsf{M}}^*$ is given by the tuple with
\begin{itemize}
   \item $\mathsf{M}^* = \mathsf{M}^\vee\{1\}$;
   \item $\Fil^0M^* = (M/\Fil^0M)^\vee\{1\}\to M^* = M^\vee\{1\}$, so that $\gr^{-1}M^* \simeq (\Fil^0M)^\vee\{1\}$;
   \item $\Psi_{\mathsf{M}^*}$ given by the composition
   \[
    \mathsf{M}^*= \mathsf{M}^\vee\{1\}\xrightarrow{\sigma_{\mathsf{M}}^\vee\{1\}}\varphi^*\mathsf{M}^\vee\otimes_A(I^{-1}\otimes_AA\{1\})\xrightarrow[\simeq]{1\otimes u}\varphi^*\mathsf{M}^\vee\otimes_A\varphi^*A\{1\}\xrightarrow{\simeq}\varphi^*\mathsf{M}^*;
   \]
   \item The isomorphism $\xi^*$ is obtained after twisting from~\eqref{eqn:cofib_zeta}.
\end{itemize}
\end{definition}

\begin{remark}[Rank one divided complexes]
   \label{rem:unit_and_twist}
   We have the canonical divided complex given by the tuple $(A,R\xrightarrow{\mathrm{id}} R,A\xrightarrow{\mathrm{id}}A,0)$, which we will denote by $\underline{\mathbf{1}}$. Its Cartier dual $\underline{\mathbf{1}}^*$ is given by the tuple 
   \[
   (A\{1\},0\to R\{1\},A\{1\}\to I^{-1}\otimes A\{1\}\xrightarrow[\simeq]{u}\varphi^*A\{1\},\xi^*)
   \]
   where $\xi^*$ is the composition 
   \[
   I^{-1}/A\otimes_AA\{1\}\xrightarrow{\simeq}I^{-1}\otimes_A\overline{A}\{1\}\xrightarrow{\simeq}\overline{A}\otimes_{A}\varphi^*A\{1\}\xrightarrow{\simeq}\overline{A}\otimes_RR\{1\}.
   \]
\end{remark}

\begin{remark}
   [Maps from $\underline{\mathbf{1}}$]
\label{rem:maps_from_1}
If $\underline{\mathsf{M}}$ is a divided Dieudonn\'e complex over $\underline{A}$, set $\Fil^0\mathsf{M} = \fib(\mathsf{M}\to \gr^{-1}M)$. Then we obtain a map $v_{\mathsf{M}}:\varphi^*\Fil^0\mathsf{M}\to \mathsf{M}$ sitting in a diagram
\[
\begin{diagram}
  \varphi^*\Fil^0\mathsf{M}&\rTo&\varphi^*\mathsf{M}&\rTo&(A\otimes_{\varphi,A}R)\otimes_R\gr^{-1}M\\
  \dTo^{v_{\mathsf{M}}}&&\dEquals&&\dTo\\
  \mathsf{M}&\rTo_{\Psi_{\mathsf{M}}}&\varphi^*\mathsf{M}&\rTo&\overline{A}\otimes_{\overline{\varphi},R}\gr^{-1}M
\end{diagram}
\]
where both rows are fiber sequences. Unwinding definitions, $\mathrm{RHom}_{\underline{A}}(\underline{\mathbf{1}},\underline{\mathsf{M}})$ (computed in the stable $\infty$-category $\mathrm{DDC}_{\underline{A}}(R)$) is given by
\[
  \mathrm{RHom}_{\underline{A}}(\underline{\mathbf{1}},\underline{\mathsf{M}}) \simeq \fib(\Fil^0\mathsf{M}\xrightarrow{\mathrm{can} - v_{\mathsf{M}}\circ \varphi^*}\mathsf{M}),
\]
where $\mathrm{can}:\Fil^0\mathsf{M}\to \mathsf{M}$ is the natural map.
\end{remark}

\begin{example}
   [Maps from $\underline{\mathbf{1}}$ to $\underline{\mathbf{1}}^*$]
   \label{ex:map_from_unit_bk_twist}
Let us specialize the previous remark to the case $\underline{\mathsf{M}} = \underline{\mathbf{1}}^*$. Here, the map $v_{\mathsf{M}}\circ \varphi^*$ is the composition
\[
(\Fil^1A)\{1\}\xrightarrow{\varphi_1\{1\}} I\otimes_A\varphi^*A\{1\} \simeq A\{1\}
\]
where the first map is the Breuil-Kisin twist of the semilinear divided Frobenius map $\varphi_1:\Fil^1A \to I$. We will denote this composition by $\varphi_1\{1\}$ as well. Then we find:
\[
   \mathrm{RHom}_{\underline{A}}(\underline{\mathbf{1}},\underline{\mathbf{1}}^*) \simeq \fib((\Fil^1A)\{1\}\xrightarrow{\mathrm{can} - \varphi_1\{1\}}A\{1\}). 
\]
\end{example}

\begin{remark}
[Windows over frames]
\label{rem:windows}
Suppose that $A$ and $R$ are discrete and $(I\to A)\simeq (A\xrightarrow{\theta}A)$ for some element $\theta\in A$. Suppose that $\underline{\mathsf{M}}$ is in $\mathrm{DDC}^{[0,0]}_{\underline{A}}(R)$: this means that $\mathsf{M}$ is finite locally free over $A$ and that $\gr^{-1}M$ and $\Fil^0M$ are finite locally free over $R$. Let the pair $(\mathsf{N},\Phi_{\mathsf{N}})$ be as in Remark~\ref{rem:pseudo_verschiebung}, and set $\mathsf{N}^1 = \Fil^0\mathsf{M}\{-1\}$, so that $\mathsf{N}^1\subset \mathsf{N}$ is a submodule with $\mathsf{N}/\mathsf{N}^1$ finite locally free over $R$. The map $v_{\mathsf{M}}$ from Remark~\ref{rem:maps_from_1} yields a map $\Phi^1_{\mathsf{N}}:\varphi^*\mathsf{N}^1\to \mathsf{M}$. Our hypotheses imply (see~\cite[Remark 5.2]{lau2018divided}) that there is a \emph{normal decomposition} $\mathsf{N} = \mathsf{L}\oplus \mathsf{T}$ with
\[
\mathsf{N}^1 = \mathsf{L}\oplus (\Fil^1A\otimes_A \mathsf{T})\subset \mathsf{L}\oplus \mathsf{T} = \mathsf{N},
\]
and also that the assignment
\[
   \varphi^* \mathsf{L}\oplus \varphi^* \mathsf{T} \xrightarrow{(l,t)\mapsto (\Phi^1_{\mathsf{N}}(l),\Phi_{\mathsf{N}}(t))}\mathsf{L}\oplus \mathsf{T}
\]
is an isomorphism of $A$-modules. Therefore, if write $F:\mathsf{N}\to \mathsf{N}$ and $F_1: \mathsf{N}^1\to \mathsf{N}$ for the $\varphi$-semilinear maps underlying $\Phi_{\mathsf{N}}$ and $\Phi^1_{\mathsf{N}}$, then we find that,  in the terminology of~\cite[\S 2]{MR2679066}, $(\mathsf{N},\mathsf{N}^1,F,F_1)$ is a \emph{window} over the frame $(A,\Fil^1A,R,\varphi,\varphi_1)$. It is not difficult to show that this is in fact giving us an equivalence of categories between $\mathrm{DDC}^{[0,0]}_{\underline{A}}(R)$ and the relevant category of windows.
\end{remark}

\begin{example}
   [Displays]
\label{ex:displays}
When $\underline{A} = \underline{W(R)}$ is the Witt vector frame from Example~\ref{example:witt_frame}, then $\mathrm{DDC}^{[0,0]}_{\underline{A}}(R)$ is just the category of $3n$-displays as defined by Zink~\cite{Zink2002-rx}; see~\cite[p. 4, Example]{MR2679066}.
\end{example}

\begin{remark}[Comparison with Ansch\"utz-Le Bras]
\label{rem:anschutz_lebras_comp}
If in the situation of Remark~\ref{rem:windows}, $\theta$ is a non-zero divisor and $\Fil^1A = \varphi^{-1}(\theta A)$, then the datum of the pair $(\mathsf{N},F)$, along with certain admissibility conditions, determines the rest of the tuple; see the discussion in~\cite[Proposition 5.7.6]{gmm}, and also~\cite[Lemma 2.1.16]{Cais2017-bf}. This applies in particular when $A = \Prism_R$ for $R$ qrsp, and so we see that, in this case, $\mathrm{DDC}^{[0,0]}_{\underline{\Prism}_R}(R)$ is the category of admissible prismatic Dieudonn\'e modules appearing in Ansch\"utz-Le Bras~\cite{MR4530092}.
\end{remark}

\subsection{Descent for divided Dieudonn\'e complexes}
\label{subsec:descent_for_divided_dieudonne_complexes}

In this subsection, we enumerate various useful descent properties enjoyed by the categories of divided Dieudonn\'e complexes. As a consequence, we obtain a certain integrability property (Proposition~\ref{prop:adm_dieu_integrable}) that will prove helpful for the classification results obtained later in the section.

\begin{remark}
[\'Etale descent for divided Dieudonn\'e complexes]
\label{rem:adm_dieu_etale_descent}
Suppose that $R\to R'$ is a $p$-completely \'etale and faithfully flat map. Then Remark~\ref{rem:etale_lifts_frames} tells us that the corresponding cosimplicial $R$-algebra $R^{',\otimes_R(\bullet +1)}$ (the tensor product here is $p$-completed) lifts to a cosimplicial frame $\underline{A}^{',(\bullet+1)}$ with $A^{',(m)}$ $(p,I)$-completely \'etale and faithfully flat\footnote{To see the faithful flatness, note that $\pi_0(A/{}^{\mathbb{L}}(p,I))$ is a nilpotent thickening of $\pi_0(R/{}^{\mathbb{L}}(p,I))$.} over $A$ for all $m\ge 1$. In fact, one sees that $A^{',(m)} = A^{',\otimes_A m}$, where we are using $(p,I)$-completed tensor products. Using this and $(p,I)$-complete faithfully flat descent for perfect complexes, one finds that the natural functor
\[
   \mathrm{DDC}_{\underline{A}}(R)\to \mathrm{Tot}\left(\mathrm{DDC}_{\underline{A}^{',\otimes_A(\bullet +1)}}(R^{',\otimes_R(\bullet+1)})\right)
\]
is an equivalence of stable $\infty$-categories.
\end{remark}

\begin{remark}
[Quasisyntomic descent for divided Dieudonn\'e complexes]
\label{rem:adm_dieu_qsynt_descent}
Completely analogously, in the situation of Lemma~\ref{lem:syntomic_frame_criterion}, the natural functor
\[
   \mathrm{DDC}_{\underline{A}}(R)\to \mathrm{Tot}\left(\mathrm{DDC}_{\underline{A}_\infty^{\otimes_A(\bullet +1)}}(R^{\otimes_R(\bullet+1)}_\infty)\right)
\]
is an equivalence of stable $\infty$-categories.
\end{remark}

\begin{remark}
\label{rem:adm_dieu_closed_immersion}
Suppose that we have a map $\Int_p[x]^{\wedge}_p\xrightarrow{x\mapsto t}R$, and that $S = R/{}^{\mathbb{L}}(t)$. By Remarks~\ref{rem:base_change_prisms_closed_immersions} and~\ref{rem:base_change_laminated}, base-change along $S^\Prism\to R^\Prism$ of $\underline{A}$ yields a map of frames $\underline{A}\to \underline{B}$ lifting $R\to S$, and we obtain in this way a cosimplicial frame $\underline{B}^{(\bullet)}$ for $S^{\otimes_R(\bullet+1)}$ with $B^{(m)}=B^{\otimes_A(m+1)}$, and a functor
\begin{align}\label{eqn:adm_dieu_base_change_closed_immersion}
  \mathrm{DDC}_{\underline{A}}(R)\to \mathrm{Tot}\left(\mathrm{DDC}_{\underline{B}^{(\bullet)}}(S^{\otimes_R(\bullet+1)})\right) 
\end{align}
\end{remark}

\begin{lemma}
[Derived descent for divided Dieudonn\'e complexes]
\label{lem:adm_dieu_derived_descent}
Suppose that $R$ is $(t)$-adically derived complete. Then the functor~\eqref{eqn:adm_dieu_base_change_closed_immersion} is an equivalence.
\end{lemma}
\begin{proof}
It is enough to know that the functors
\begin{align*}
\mathrm{Perf}(A)\to \mathrm{Tot}\left(\mathrm{Perf}(B^{(\bullet)})\right)\;;\;\mathrm{Perf}(R)\to \mathrm{Tot}\left(\mathrm{Perf}(S^{\otimes_R(\bullet+1)})\right)
\end{align*}
are equivalences.

Since $R$ is $(t)$-complete, that the second is an equivalence follows from~\cite[Proposition 3.1.5,Corollary 3.1.4]{MR4560539}. The same reasoning shows that the functors
\[
\mathrm{Perf}(A)\to \mathrm{Tot}\left(\mathrm{Perf}(\overline{A}^{\otimes_A(\bullet+1)})\right)\;;\; \mathrm{Perf}(B^{(m)})\to \mathrm{Tot}\left(\mathrm{Perf}(\overline{B}^{(m),\otimes_{B^{(m)}}(\bullet+1)})\right)
\]
are equivalences. Therefore, to finish it is enough to know that, for each $k\ge 0$, the functor
\[
   \mathrm{Perf}(\overline{A}^{\otimes_A(k+1)})\to \mathrm{Tot}\left(\mathrm{Perf}(\overline{B}^{(\bullet),\otimes_{B^{(\bullet)}}(k+1)})\right)
\]
is an equivalence. In other words, we want to know that the map $\overline{A}^{\otimes_A(k+1)}\to \overline{B}^{\otimes_B(k+1)} $ of $p$-complete rings satisfies descent for perfect complexes.

Via the map $\Spf R \xrightarrow{t}\Ga$ yielding $R^\Prism\xrightarrow{t^\Prism}\Ga^\Prism$, we obtain a map $\Spf A\to R^\Prism\to \Ga^\Prism$. Unwinding definitions, we have
\[
\Spf \overline{A}^{\otimes_A(k+1)}\simeq \Spf A\times_{\Ga^\Prism}\underbrace{\Ga^{\mathrm{HT}}\times_{\Ga^\Prism}\times\cdots\times_{\Ga^\Prism}\Ga^{\mathrm{HT}}}_{k+1}= \Spf A\times_{\Ga^\Prism}\Ga^{\mathrm{HT},\times (k+1)}
\]
and
\[
\Spf \overline{B}^{\otimes_B(k+1)}\simeq \Spf A\times_{\Ga^\Prism}\underbrace{\Int_p^{\mathrm{HT}}\times_{\Ga^\Prism}\times\cdots\times_{\Ga^\Prism}\Int_p^{\mathrm{HT}}}_{k+1} = \Spf A\times_{\Ga^\Prism}\Int_p^{\mathrm{HT},\times (k+1)}
\]
where we are using the map $\Int_p^{\mathrm{HT}}\to \Ga^{\mathrm{HT}}\to \Ga^\Prism$ induced by the zero section of $\Ga$. We are also using the superscript $\times (k+1)$ to denote the $(k+1)$-fold self-fiber product over $\Ga^\Prism$. 

Now, by~\cite[Example 9.1]{bhatt2022prismatization}, we have a map $\Ga^{\mathrm{HT}}\to \Ga$ equipped with a section $\Ga\to \Ga^{\mathrm{HT}}$ arising from the $\delta$-structure on $\Ga$ given by the Frobenius lift $x\mapsto x^p$, and yielding an isomorphism of $\Ga$-stacks 
\[
   \Ga^{\mathrm{HT}}\xrightarrow{\simeq}B(\Ga^\sharp\rtimes \Gm^\sharp)\times_{\Spf \Int_p}\Ga.
\]
On the other hand, we know from Remark~\ref{rem:hodge-tate_locus} that we have a canonical isomorphism
\[
   \Int_p^{\mathrm{HT}}\xrightarrow{\simeq}B\Gm^\sharp.
\]
Via these isomorphisms the map $\Int_p^{\mathrm{HT}}\to \Ga^{\mathrm{HT}}$ given by the zero section is just the composition
\[
   \Int_p^{\mathrm{HT}} \simeq B\Gm^\sharp\to \Ga^{\mathrm{HT}}\times_{\Ga,0}\Spf\Int_p\simeq B(\Ga^\sharp\rtimes \Gm^\sharp) \to B(\Ga^\sharp\rtimes \Gm^\sharp)\times_{\Spf \Int_p}\Ga\simeq \Ga^{\mathrm{HT}}.
\]
where the first map is from the obvious homomorphism of group schemes and the second map is from the zero section of $\Ga$.

Since the first map is $p$-completely faithfully flat, it is enough to know that
\begin{align}\label{eqn:final_descent_map}
  \Spf A\times_{\Ga^\Prism}\left(\Ga^{\mathrm{HT}}\times_{\Ga}\Spf\Int_p\right)^{\times(k+1)}\to \Spf A\times_{\Ga^\Prism}\Ga^{\mathrm{HT},\times(k+1)}
\end{align}
satisfies descent for perfect complexes. Unwinding definitions, one sees that the composition
\[
   \Spf\overline{A} \simeq \Spf A\times_{\Ga^\Prism}\Ga^{\mathrm{HT}}\to \Ga^{\mathrm{HT}}\to \Ga
\]
corresponds to the image $u=\overline{\varphi}(t)$ of $t$ under $\overline{\varphi}:R\to \overline{A}$. If, for $j=1,\ldots,k+1$, $u_j\in \overline{A}^{\otimes_A(k+1)}$ is the section given by $u$ in the $j$-th coordinates and $1$s elsewhere, then one sees that the map~\eqref{eqn:final_descent_map} is given by
\[
   \Spf(\overline{A}^{\otimes_A(k+1)}/{}^{\mathbb{L}}(u_1,\ldots,u_{k+1}))\to \Spf(\overline{A}^{\otimes_A(k+1)}).
\]
This satisfies descent for perfect complexes by~\cite{MR4560539} again since $\overline{A}^{\otimes_A(k+1)}$ is $(u_1,\ldots,u_{k+1})$-complete. Indeed, since $\overline{A}^{\otimes_A(k+1)}$ is a nilpotent thickening of $\overline{A}$, it is enough to know that $\overline{A}$ is $(u)$-complete. This follows from our hypothesis on $R$ and the fact that the Frobenius endomorphism of $\overline{A}/{}^{\mathbb{L}}p$ factors through $R/{}^{\mathbb{L}}p$.
\end{proof}

\begin{remark}
   \label{rem:adm_dieu_inverse_limit}
In the setting of Remark~\ref{rem:adm_dieu_closed_immersion}, for $k\ge 1$, set $S_k = R/{}^{\mathbb{L}}(t^k)$. Then we obtain an inverse system of frames
\[
   \underline{A} \to \cdots \to \underline{B}_{k+1}\to \underline{B}_k\to \cdots \to \underline{B}_1 = \underline{B}
\]
lifting $R\to \cdots \to S_{k+1}\to S_k \to \cdots \to S_1 = S$. This gives us a functor
\begin{align}\label{eqn:adm_dieu_completion_along_divisor}
  \mathrm{DDC}_{\underline{A}}(R)\to \varprojlim_k \mathrm{DDC}_{\underline{B}_k}(S_k).
\end{align}
\end{remark}

\begin{proposition}
[Integrability for divided Dieudonn\'e complexes]
\label{prop:adm_dieu_integrable}
Suppose that $R$ is $(t)$-adically derived complete. Then the functor~\eqref{eqn:adm_dieu_completion_along_divisor} is an equivalence.
\end{proposition}
\begin{proof}
This can be deduced from Lemma~\ref{lem:adm_dieu_derived_descent} in somewhat standard fashion:  The lemma reduces us to the case where $R/^{\mathbb{L}}(t^m)\simeq R\oplus R[1]$ for all $m$ sufficiently large and so the maps $\underline{B}_{m+1}\to \underline{B}_m$ factor through $\underline{A}\to \underline{B}_m$ for all such $m$. The proposition follows easily from this.
\end{proof}

\subsection{From $F$-gauges to divided Dieudonn\'e complexes}

For the rest of this section, $R$ will be discrete. The next result follows from~\cite[Proposition 5.7.3]{gmm} and Theorem~\ref{thm:main}.
\begin{proposition}
\label{prop:ffg_classification_semiperf}
For any semiperfectoid $R$, there are canonical exact equivalences
\[
\mathrm{Mod}_{\Int/p^n\Int}\left(\mathrm{DDC}^{[-1,0]}_{\underline{\Prism}_R}(R)\right)\xleftarrow{\simeq} \mathsf{P}^{\mathrm{syn}}_{n,\{0,1\}}(R)\xrightarrow{\simeq}\mathrm{FFG}_n(R)
\]
that are compatible with Cartier duality in that the composition carries $\mathsf{M}^*[1]$ to the Cartier dual of the image of $\mathsf{M}$.
\end{proposition}

\begin{remark}
\label{rem:btn_under_equivalence}
Within $\mathrm{Mod}_{\Int/p^n\Int}\left(\mathrm{DDC}^{[-1,0]}_{\underline{\Prism}_R}(R)\right)$ we have the subcategory spanned by objects where the underlying perfect complex $\mathcal{M}$ arises from a vector bundle over $R^{\Prism}\otimes\Int/p^n\Int$, and where $\gr^{-1}M$ and $\Fil^0M$ are also obtained from vector bundles over $\Spec R/{}^{\mathbb{L}}p^n$. This maps isomorphically onto the subcategory $\mathrm{BT}_n(R)\subset \mathrm{FFG}_n(R)$ via the equivalence above.
\end{remark}

We can globalize our definitions now as follows.
\begin{definition}[Prismatic divided Dieudonn\'e complexes]
\label{def:prismatic_adm_dieu}
For any $R$, a \defnword{prismatic divided Dieudonn\'e complex for $R$} is a tuple 
\[
\underline{\mathcal{M}} = (\mathcal{M},\Fil^0_{\mathrm{Hdg}}M\to M, \mathcal{M}\xrightarrow{\Psi_{\mathcal{M}}}\varphi^*\mathcal{M},\xi)
\]
such that:
\begin{enumerate}
   \item $\mathcal{M}\in \mathrm{Perf}(R^{\Prism})$ is a perfect complex over $R^\Prism$;
   \item $\Fil^0_{\mathrm{Hdg}}M \to M \defn x_{\dR,R}^*\mathcal{M}$ is a map of perfect complexes over $\Spf R$;
   \item $\Psi_{\mathcal{M}}:\mathcal{M}\to \varphi^*\mathcal{M}$ is a map of perfect complexes over $R^{\Prism}$;
   \item $\xi$ is an isomorphism of perfect complexes sitting in a diagram
   \[
    \begin{diagram}
    \varphi^*\mathcal{M}&\rTo&\cofib(\Psi_{\mathcal{M}})\\
    &\rdTo&\dTo^{\simeq}_{\xi}\\
    &&\iota^{\mathrm{HT}}_*\overline{\varphi}^*\gr^{-1}_{\mathrm{Hdg}}M.
    \end{diagram}
   \] 
\end{enumerate}
Here, $\gr^{-1}_{\mathrm{Hdg}}M = M/\Fil^0_{\mathrm{Hdg}}M$ and the diagonal map is obtained as the composition
\[
\varphi^*\mathcal{M}\to \iota^{\mathrm{HT}}_*\iota^{\mathrm{HT},*}\varphi^*\mathcal{M}\xrightarrow{\simeq}\iota^{\mathrm{HT}}_*\overline{\varphi}^*M\to \overline{\iota}^{\mathrm{HT}}_*\overline{\varphi}^*\gr^{-1}_{\mathrm{Hdg}}M.
\]
Here, $\iota^{\mathrm{HT}}:R^{\mathrm{HT}}\to R^\Prism$ is the closed immersion from the Hodge-Tate locus.

These objects organize into an $\infty$-category $\mathrm{DDC}_{\Prism}(R)$, and we can define the subcategories $\mathrm{DDC}^{[a,b]}_{\Prism}(R)$ as above.
\end{definition}

\begin{corollary}
\label{cor:ffg_classification}
There are canonical exact equivalences
\[
\mathrm{Mod}_{\Int/p^n\Int}\left(\mathrm{DDC}^{[-1,0]}_{\Prism}(R)\right)\xleftarrow{\simeq} \mathsf{P}^{\mathrm{syn}}_{n,\{0,1\}}(R)\xrightarrow{\simeq}\mathrm{FFG}_n(R).
\]
\end{corollary}
\begin{proof}
Combine Proposition~\ref{prop:ffg_classification_semiperf} with quasisyntomic descent (Remark~\ref{rem:adm_dieu_qsynt_descent}).
\end{proof}

\begin{remark}
   [Relationship with the absolute prismatic site]
Following~\cite[Proposition 8.15]{bhatt2022prismatization} and its proof, one sees that giving a perfect complex over $R^\Prism$ is the same as giving a perfect complex of prismatic $F$-crystals with respect to the \emph{animated} absolute prismatic site of $R$.  In some cases, one is able to replace the prismatization $R^\Prism$ by its classical truncation, and hence the category appearing on the left of the equivalence in Corollary~\ref{cor:ffg_classification} can be described in terms of the \emph{classical} absolute prismatic site of $R$. When $R/pR$ is an $F$-finite and $F$-nilpotent $\Field_p$-algebra, we verify this in Theorem~\ref{thm:classical_prismatization_enough} below, and we recover in this way the classification of $p$-divisible groups by Lau in terms of divided Dieudonn\'e crystals over such $\Field_p$-algebras~\cite{lau2018divided}.
\end{remark}

\begin{corollary}
\label{cor:ffg_frame_realization}
Suppose that $\underline{A}$ is a laminated prismatic frame for $R$. Then we have canonical functors
\[
\mathrm{FFG}_n(R) \xleftarrow{\simeq}\mathsf{P}^{\mathrm{syn}}_{n,[0,1]}(R)\rightarrow \Mod{\Int/p^n\Int}\left(\mathrm{DDC}^{[-1,0]}_{\underline{A}}(R)\right)
\]
\end{corollary} 

\begin{remark}
From Remark~\ref{rem:btn_under_equivalence}, we see that the subcategory $\mathrm{BT}_n(R)$ gets mapped to the subcategory of $\Mod{\Int/p^n\Int}\left(\mathrm{DDC}^{[-1,0]}_{\underline{A}}(R)\right)$ spanned by the objects $\underline{\mathsf{M}}$ where $\mathsf{M}$ is a finite locally free module over $A/{}^{\mathbb{L}}p^n$, and where $\gr^{-1}M$ and $\Fil^0M$ are finite locally free over $R/{}^{\mathbb{L}}p^n$.
\end{remark}

\begin{remark}
[Compatibility with descent]
The resulting functor 
\[
\mathrm{FFG}_n(R) \to \Mod{\Int/p^n\Int}\left(\mathrm{DDC}^{[-1,0]}_{\underline{A}}(R)\right)
\]
is compatible with $p$-completely \'etale (see Remark~\ref{rem:adm_dieu_etale_descent}) and quasisyntomic descent (see Remark~\ref{rem:adm_dieu_qsynt_descent}).
\end{remark}

\subsection{Breuil-Kisin frames}

We now specialize to a particular kind of frame that finds its origins in the classification theorems of Breuil and Kisin for finite flat group schemes over $p$-adic rings of integers.

\begin{remark}
[Frames from prisms]
Suppose that $(A,I')$ is a prism. Then we obtain a frame $\underline{A}$ for $R = A/^{\mathbb{L}}I'$ as follows: We set $(\Fil^1A\to A) = (I'\to A)$ and take $I = \varphi^*I'\to A$. The map $\overline{\varphi}$ is the associated map of `quotient' animated commutative rings obtained from the Frobenius lift $\varphi:A\to A$. 
\end{remark}

\begin{definition}
[Breuil-Kisin frames]
A \defnword{Breui-Kisin frame} for $R$ is a frame $\underline{A}$ obtained from a prism $(A,I')$ such that $A$ is $p$-completely flat, so that $I'\subset A$ is an ideal and $R = A/I'$ is discrete. By construction, such a frame is automatically prismatic.
\end{definition}

\begin{remark}[Laminations for Breuil-Kisin frames]
\label{rem:bk_frame_laminations}
As explained in~\cite[Example 5.4.16]{gmm}, each Breuil-Kisin frame admits a canonical lamination that is functorial for maps between such frames.
\end{remark}

\begin{example}[Lifts]
\label{ex:bk_frame_p-adic}
Any $p$-completely flat $\delta$-ring $A$ yields a $p$-adic Breuil-Kisin frame for $R = A/pA$. This is what Lau simply calls a \emph{lift} of $R$ in~\cite[\S 1.2]{MR3867290}.
\end{example}

\begin{example}[The original Breuil-Kisin frame]
\label{ex:bk_frame}
The seminal example is $\Sig = W(\kappa)\pow{u}$, where $\kappa$ is a perfect field, equipped with the Frobenius lift satisfying $\varphi(u) = u^p$, and $I' = (E(u))$ for an Eisenstein polynomial $E(u)\in W(\kappa)[u]$. More generally, as in~\cites{vasiu_zink,MR2679066}, we can take $\Sig = W(\kappa)\pow{u_1,\ldots,u_m}$ equipped with the Frobenius lift $u_i\mapsto u_i^p$, and $E\in \Sig$ is a power series such that $I' = (E)$ equips $\Sig$ with the structure of a transversal prism. In this situation, $R =\Sig/(E)$ is always regular local.
\end{example}

\begin{remark}
   [Sections of the Breuil-Kisin twist]
\label{rem:sections_bk_bk-frame}
Specializing Remark~\ref{ex:map_from_unit_bk_twist} to the Breuil-Kisin case, one sees that we have
\[
   \mathrm{RHom}_{\underline{A}}(\mathbf{1},\mathbf{1}^*)\simeq \fib(I'\{1\}\xrightarrow{\mathrm{can}-u}A\{1\}),
\]
where $u:I'\{1\}\to A\{1\}$ is the semilinear map underlying the isomorphism
\[
   \varphi^*(I'\{1\})\simeq \varphi^*I' \otimes_A\varphi^*A\{1\}\simeq I\otimes_A\varphi^*A\{1\}\simeq A\{1\}.
\]
If $A$ is $p$-adic, then this further simplifes to
\[
   \mathrm{RHom}_{\underline{A}}(\mathbf{1},\mathbf{1}^*)\simeq \fib(A\xrightarrow{p-\varphi}A).
\]
\end{remark}

\begin{definition}
   [Breuil-Kisin windows]
\label{def:bk_windows}
A \defnword{$p^n$-torsion Breuil-Kisin window} over a Breuil-Kisin frame $\underline{A}$ is a triple $(\mathsf{N},F_{\mathsf{N}},V_{\mathsf{N}})$ where:
\begin{enumerate}
   \item $\mathsf{N}$ is a $p^n$-torsion $A$-module of projective dimension $1$;
   \item $F_{\mathsf{N}}:\varphi^*\mathsf{N}\to \mathsf{N}$ and $V_{\mathsf{N}}:I'\otimes_A\mathsf{N}\to \varphi^*\mathsf{N}$ are $A$-linear maps such that the compositions
   \[
    V_{\mathsf{N}}\circ(1\otimes F_{\mathsf{N}}):I'\otimes_A\varphi^*\mathsf{N}\to \varphi^*\mathsf{N}\;;\; F_{\mathsf{N}}\circ V_{\mathsf{N}}:I'\otimes_A\mathsf{N}\to \mathsf{N}
   \]
   are the maps induced by the inclusion $I'\subset A$.
\end{enumerate}
Write $\mathrm{BK}_{\underline{A},n}(R)$ for the category of such triples.
\end{definition}

\begin{remark}
\label{rem:bk_cartier_duality}
There is a Cartier duality involution on $\mathrm{BK}_{\underline{A},n}(R)$ determined as follows. View $\mathsf{N}$ as a perfect complex with Tor amplitude $[-1,0]$ and set $\mathsf{N}^* = H^0(\mathsf{N}^\vee[-1]\{1\})$, where $\mathsf{N}^\vee$ is the $A$-linear derived dual. Since $A$ is $p$-torsion free, we see that $\mathsf{N}^*$ is once again a $p^n$-torsion $A$-module of projective dimension $1$. Moreover, we can take 
\[
F_{\mathsf{N}^*} = H^0(V_{\mathsf{N}}^\vee[-1]\{1\})\;;\; V_{\mathsf{N}^*} = H^0(F_{\mathsf{N}}^\vee[-1]\{1\}).
\]
\end{remark}

\begin{proposition}
\label{prop:BK_gauge}
If $\underline{A}$ is of Breuil-Kisin type, then there is a canonical exact equivalence 
\[
\Mod{\Int/p^n\Int}(\mathrm{DDC}^{[-1,0]}_{\underline{A}}(R))\xrightarrow{\simeq}\mathrm{BK}_{\underline{A},n}(R)
\]
compatible with Cartier duality.
\end{proposition}
\begin{proof}
Let $J$ be the Breuil-Kisin twist associated with the prism structure $(A,I')$ so that we have a canonical isomorphism $I'\otimes_A \varphi^*J \xrightarrow{\simeq}J$, or equivalently an isomorphism $I'\otimes_AJ^\vee \xrightarrow{\simeq}\varphi^*J^\vee$.

Since $A$ is $p$-completely flat, the $\infty$-category of perfect complexes over $A$ with Tor amplitude in $[-1,0]$ and with cohomology killed by a power of $p$ is equivalent via the functor $\mathsf{M}\mapsto H^0(\mathsf{M})$ to the (discrete) category of $p$-power torsion $A$-modules of projective dimension $1$. This equivalence is compatible with arbitrary base-change along maps to $p$-completely flat rings. In particular, we see that if $\mathsf{N}$ is a $p$-power torsion $A$-module of projective dimension $1$, then the derived base-change $\varphi^*\mathsf{N}$ is once again a $p$-power torsion $A$-module of projective dimension $1$. 

With this in hand, suppose that we have $\underline{\mathsf{M}}$ in $\Mod{\Int/p^n\Int}(\mathrm{DDC}^{[-1,0]}_{\underline{A}}(R))$. The fiber sequence
\[
\mathsf{M}\to \varphi^*\mathsf{M}\to \overline{A}\otimes_R\gr^{-1}M
\]
shows that there is an equivalence $\mathsf{M}\xrightarrow{\simeq}\varphi^*\Fil^0\mathsf{M}$ giving us isomorphisms
\begin{align}
\label{eqn:bk_twisted_isom}
I'\otimes_A(J^\vee\otimes_A\mathsf{M})\xrightarrow{\simeq}(I'\otimes_AJ^\vee)\otimes_A\varphi^*\Fil^0\mathsf{M}\xrightarrow{\simeq}\varphi^*(J^\vee\otimes_A\Fil^0\mathsf{M}).
\end{align}

Observe now that $\Fil^0\mathsf{M}$ is perfect with Tor amplitude in $[-1,0]$ with cohomology killed by $p^n$: This follows by contemplating the diagram
\[
\begin{diagram}
I'\otimes_A\mathsf{M}&\rTo&\mathsf{M}&\rTo&M\\
\dTo&&\dEquals&&\dTo\\
\Fil^0\mathsf{M}&\rTo&\mathsf{M}&\rTo&\gr^{-1}M\\
\end{diagram}
\]
of perfect complexes over $A$ where both rows are exact, and observing that the cofiber of the left vertical map is $\Fil^0M$, which has Tor amplitude $[-2,0]$ as a complex over $A$.
 
Set $\mathsf{N} = H^0(J^\vee\otimes_A\Fil^0\mathsf{M})$: This is a $p^n$-torsion $A$-module of projective dimension $1$ that is equipped via~\eqref{eqn:bk_twisted_isom} with an isomorphism 
\[
I'\otimes_AH^0(J^\vee\otimes_A\mathsf{M})\xrightarrow{\simeq}\varphi^*\mathsf{N}. 
\]
In particular, the tautological map $\Fil^0\mathsf{M}\to \mathsf{M}$ yields a map $V_{\mathsf{N}}:I'\otimes_A\mathsf{N}\to \varphi^*\mathsf{N}$, and the map $I'\otimes_A\mathsf{M}\to \Fil^0\mathsf{M}$ appearing in the left vertical column in the diagram above yields a map $F_{\mathsf{N}}:\varphi^*\mathsf{N}\to \mathsf{N}$. 

One now checks that $(\mathsf{N},F_{\mathsf{N}},V_{\mathsf{N}})$ is an object in $\mathrm{BK}_{\underline{A},n}(R)$ giving us the functor asserted by the proposition. The construction of the inverse can be obtained by tracing the path backwards via the equivalence from the second paragraph of the proof. The main additional observation is that, given such a tuple in $\mathrm{BK}_{\underline{A},n}(R)$, we have fiber sequences
\[
\cofib(V_{\mathsf{N}})\to R\otimes_A\mathsf{N}\to \cofib(F_{\mathsf{N}})\;;\; \cofib(F_{\mathsf{N}})\to R\otimes_A\varphi^*\mathsf{N}\to \cofib(V_{\mathsf{N}}),
\]
which show that $\cofib(V_{\mathsf{N}})$ and $\cofib(F_{\mathsf{N}})$, which \emph{a priori} have Tor amplitude in $[-2,0]$ over $R$, are actually perfect over $R$ with Tor amplitude in $[-1,0]$.

We leave the verification of compatibility with exact sequences and duality to the reader.
\end{proof}

\begin{remark}
\label{rem:bk_syntomic}
Corollary~\ref{cor:ffg_frame_realization}, combined with Remark~\ref{rem:bk_frame_laminations} and Proposition~\ref{prop:BK_gauge} gives us functors
\[
\mathrm{FFG}_n(R)\xleftarrow{\simeq}\mathsf{P}^{\mathrm{syn}}_{n,\{0,1\}}(R)\rightarrow \mathrm{BK}_{\underline{A},n}(R).
\]
The subcategory $\mathrm{BT}_n(R)$ is mapped to the subcategory of $\mathrm{BK}_{\underline{A},n}(R)$ spanned by the objects where the underlying module $\mathsf{N}$ is finite locally free over $A/p^nA$.
\end{remark}

\begin{example}
[Classification over perfectoid rings]
\label{ex:perfect_prisms}
If $(A,I')$ is a perfect prism---equivalently, if $R$ is perfectoid---then $A = \Prism_R$ is the prismatic cohomology of $R$, and so Propositions~\ref{prop:ffg_classification_semiperf} and~\ref{prop:BK_gauge} together give us equivalences
\[
\mathrm{FFG}_n(R)\xleftarrow{\simeq}\mathsf{P}^{\mathrm{syn}}_{n,\{0,1\}}(R)\xrightarrow{\simeq} \mathrm{BK}_{\underline{\Prism}_R,n}(R).
\]
This recovers the equivalence of~\cite[Theorem 5.4]{MR4530092} (see also~\cite[Theorem 10.12]{MR3867290} for the case $p>3$).
\end{example}

Applying Proposition~\ref{prop:BK_gauge} and Remark~\ref{rem:bk_syntomic} to the case $I' = (p)$ yields the following result, which can also be deduced using crystalline Dieudonn\'e theory; see~\cite[\S 3]{MR3867290}.
\begin{proposition}
\label{prop:frobenius_lift_enough}
Suppose that $A$ is a $p$-completely flat $\delta$-ring such that $R=A/pA$. Then there is a canonical functor  
\[
   \mathrm{FFG}_n(R)\to \mathrm{BK}_{\underline{A},n}(R),
\]
where the right hand side is equivalent to the category of triples $(\mathsf{N},F_{\mathsf{N}},V_{\mathsf{N}})$ where  $\mathsf{N} = \mathsf{N}(G)$ is a $p^n$-torsion $R$-module of projective dimension $1$ and $F_{\mathsf{N}}:\varphi^*\mathsf{N}\to \mathsf{N}$ and $V_{\mathsf{N}}:\mathsf{N}\to \varphi^*\mathsf{N}$ are $A$-linear maps such that $F_{\mathsf{N}}\circ V_{\mathsf{N}}$ and $V_{\mathsf{N}}\circ F_{\mathsf{N}}$ are both equal to multiplication-by-$p$.
\end{proposition}

\begin{remark}
\label{rem:perfect_case}
If $R$ is perfect, so that $A = W(R) = \Prism_R$, then this assignment is an equivalence by Example~\ref{ex:perfect_prisms}, and we recover the classification by Dieudonn\'e-Manin (when $R$ a field), Berthelot (when $R$ is a valuation ring), Gabber and Lau (for general perfect $R$); see~\cite[\S 6]{MR2983008}.
\end{remark}

On the other end of the crystalline situation in Proposition~\ref{prop:frobenius_lift_enough}, we have the following transversal situation:
\begin{proposition}
\label{prop:bk_frames_flat}
Suppose that $R$ is $p$-completely flat---equivalently that $I'$ is locally generated by an element that is a non-zero divisor mod-$p$. Then $\mathrm{BK}_{\underline{A},n}(R)$ is equivalent to the category of pairs $(\mathsf{N},F_{\mathsf{N}})$ where:
\begin{enumerate}
   \item $\mathsf{N}$ is a $p$-power torsion $A$-module of projective dimension $1$;
   \item $F_{\mathsf{N}}:\varphi^*\mathsf{N}\to \mathsf{N}$ is an injective map whose cokernel is killed by $I'$.
\end{enumerate}
In particular, we can functorially associate with every finite flat group scheme $G$ over $R$ a pair $(\mathsf{N}(G),\varphi_{\mathsf{N}(G)})$ with these properties.
\end{proposition}
\begin{proof}
Suppose that $\mathsf{N}$ is a $p^n$-torsion $A$-module of projective dimension $1$. By exhibiting $\mathsf{N}$ as the cokernel of a map $\mathsf{Q}'\to \mathsf{Q}$ between finite flat $A$-modules, one sees that it can be seen as a submodule of $\mathsf{Q}'/p^n\mathsf{Q}'$. This shows that $\mathsf{N}$ is $I'$-torsionfree, and so we have $I'\otimes_A\mathsf{N} = I'\mathsf{N}$. 

Now, given a tuple $(\mathsf{N},F_{\mathsf{N}},V_{\mathsf{N}})$ in $\mathrm{BK}_{\underline{A},n}(R)$, as observed in the proof of Proposition~\ref{prop:BK_gauge}, $\cofib(F_{\mathsf{N}})$ is perfect over $R$ with Tor amplitude in $[-1,0]$, and its cohomology is of course $p^n$-torsion. Since $R$ is $p$-completely flat, we see that it is quasi-isomorphic to its zeroth cohomology. Therefore, $F_{\mathsf{N}}$ is injective with cokernel killed by $I'$. Moreover, we have inclusions
\[
I'\otimes_A\mathsf{N} \simeq I'\mathsf{N}\subset \mathrm{im}(F_{\mathsf{N}})\simeq\varphi^*\mathsf{N}\subset \mathsf{N},
\]
the first of which is isomorphic to the map $V_{\mathsf{N}}$. Conversely, given such an injective map $F_{\mathsf{N}}$, $V_{\mathsf{N}}$ is determined by the inclusion $I'\mathsf{N}\subset \mathrm{im}(F_{\mathsf{N}})$. This shows the desired equivalence.
\end{proof}

\begin{remark}
  The pairs $(\mathsf{N},F_{\mathsf{N}})$ in Proposition~\ref{prop:bk_frames_flat} are a general version of the notion of a \defnword{Breuil window} introduced by Vasiu-Zink in~\cite{vasiu_zink}.
\end{remark}

\begin{example}
[Equivalence for $p$-completely flat perfectoid rings]
\label{ex:p-comp_flat_perfectoid}
If $R$ is a $p$-completely flat perfectoid ring, then combining Example~\ref{ex:perfect_prisms} with Proposition~\ref{prop:bk_frames_flat} tells us that $\mathrm{FFG}_n(R)$ is equivalent to the category of pairs $(\mathsf{N},F_{\mathsf{N}})$ over $\Prism_R$ as given in the proposition above.
\end{example}

\subsection{Nilpotent divided complexes and connected finite flat group schemes}
\label{sub:nilpotent_admissible_modules_and_connected_finite_flat_group_schemes}

It has long been known to experts that essentially any frame can be used to classify \emph{connected} finite flat group schemes; see for instance~\cite[\S 10]{MR2679066} for some instances of this phenomenon. In particular, as observed in Zink's seminal work~\cite{Zink2002-rx}, one can use the Witt frame for this purpose. We codify the underlying principle in this subsection. 

\begin{construction}
   [A canonical operator]
\label{const:operator_for_nilpotence}
Suppose that $\underline{\mathsf{M}}$ is a divided Dieudonn\'e complex over a frame $\underline{A}$ for $R$. Then in the notation of Remarks~\ref{rem:pseudo_verschiebung} and~\ref{rem:maps_from_1}, we have a commuting diagram
\[
   \begin{diagram}
      I\otimes_A\varphi^*\Fil^0\mathsf{M}&\rTo^{1\otimes v_{\mathsf{M}}}&I\otimes_A\mathsf{M}\\
      \dTo&&\dTo\\
      I\otimes_A\varphi^*\mathsf{M}&\rTo_{\sigma_{\mathsf{M}}}&\mathsf{M}
   \end{diagram}
\]
This yields a $\varphi$-semilinear map of $A$-modules $I\otimes_A\gr^{-1}M\to \mathsf{M}/{}^{\mathbb{L}}I$, which gives a $\varphi$-semilinear operator
\[
   \gr^{-1}M\{-1\}/^{\mathbb{L}}(p,I)\to \mathsf{M}\{-1\}/{}^{\mathbb{L}}(p,I)\to \gr^{-1}M\{-1\}/{}^{\mathbb{L}}(p,I)
\]
of $R/{}^{\mathbb{L}}(p,I)$-modules. If $B \defn (R/pR)_{\mathrm{red}}$, then, in turn, we obtain a $\varphi$-semilinear operator of perfect complexes over $B$
\[
 \eta_{\mathsf{M}}: B\otimes_R\gr^{-1}M\{-1\}\to B\otimes_R\gr^{-1}M\{-1\}.
\]
Here, we are using the fact that the image of $I$ in $R/pR$ is killed by Frobenius, so that the map $R\to B$ factors through $R/{}^{\mathbb{L}}(p,I)$.
\end{construction}

\begin{example}
   If $\underline{\mathsf{M}} = \underline{\mathbf{1}}$, then $\gr^{-1}M=0$ and the operator also vanishes. On the other hand, if $\underline{\mathsf{M}} = \underline{\mathbf{1}}^*$, then $M = \gr^{-1}M = R\{1\}$, and $\eta_{\mathbf{1}^*}:B\to B$ is simply the Frobenius endomorphism of $B$.
\end{example}

\begin{definition}
   [Nilpotent divided complexes]
Let $\mathfrak{c}\subset B$ be a finitely generated radical ideal. We will say that $\underline{\mathsf{M}}$ is $\mathfrak{c}$-\defnword{nilpotent} if the base change of the operator $\eta_{\mathsf{M}^*}$ over $B/\mathfrak{c}$ is nilpotent.\footnote{The use of Cartier duality here is to ensure Remark~\ref{rem:connectedness_and_nilpotence} is valid.} Write $\mathrm{DDC}^{\mathfrak{c}\mathhyph\mathrm{nilp}}_{\underline{A}}(R)$ for the $\infty$-subcategory of $\mathrm{DDC}_{\underline{A}}(R)$ spanned by such objects. If $\mathfrak{c} = 0$, we will simply write $\mathrm{DDC}^{\mathrm{nilp}}_{\underline{A}}(R)$ instead.
\end{definition}

\begin{remark}
   [Nilpotence in the Breuil-Kisin case]
\label{rem:nilpotence_bk}
In the situation of Proposition~\ref{prop:BK_gauge}, the triple $(\mathsf{N},F_{\mathsf{N}},V_{\mathsf{N}})$ is associated with a $\mathfrak{c}$-nilpotent divided Dieudonn\'e complex precisely when the map $F_{\mathsf{N}^*}:\varphi^*\mathsf{N}^*\to \mathsf{N}^*$ corresponds to a $\varphi$-semilinear endomorphism of $\mathsf{N}^*$ whose base-change over $B/\mathfrak{c}$ is nilpotent. Write $\mathrm{BK}^{\mathfrak{c}\mathhyph\mathrm{nilp}}_{\underline{A},n}(R)$ for the subcategory of $\mathrm{BK}_{\underline{A},n}(R)$ spanned by such objects. Once again, if $\mathfrak{c} = 0$, we will just write $\mathrm{BK}^{\mathrm{nilp}}_{\underline{A},n}(R)$ instead.
\end{remark}

\begin{remark}
   [Connectedness and nilpotence]
\label{rem:connectedness_and_nilpotence}
Suppose that $\underline{\mathcal{M}}$ is a prismatic divided Dieudonn\'e complex over $R$. Then, via quasisyntomic descent and Construction~\ref{const:operator_for_nilpotence} applied with $\underline{A} = \underline{\Prism}_S$ for semiperfectoid $R$-algebras $S$, we obtain a canonical $\varphi$-semilinear operator $\eta_{\mathcal{M}^*}$ on $B\otimes_R\gr^{-1}_{\mathrm{Hdg}}M^*\{-1\}$. If $\underline{\mathcal{M}}$ is associated with  $G\in \mathrm{FFG}_n(R)$ via Corollary~\ref{cor:ffg_classification}, then this operator is $\mathfrak{c}$-nilpotent precisely when the restriction of $G$ over $B/\mathfrak{c}$ is connected. Indeed, it is enough to see this when $R = \kappa$ is an algebraically closed field, where this translates into the usual criterion in \emph{contravariant} Dieudonn\'e theory involving the nilpotence of the semilinear operator $F$ on the Dieudonn\'e module.
\end{remark}

The next result follows from the methods of~\cite[\S 5.9]{gmm}. 
\begin{proposition}
[Unique lifting principle using nilpotence]
\label{prop:deformation_theory_nilpotent}
Suppose that $\underline{A}'\to \underline{A}$ is a surjective\footnote{By this, we mean that the underlying map $A'\to A$ is surjective.} map of prismatic frames for $R$. Then the natural functor $\mathrm{DDC}^{\mathrm{nilp}}_{\underline{A}'}(R)\to \mathrm{DDC}^{\mathrm{nilp}}_{\underline{A}}(R)$ is an equivalence of stable $\infty$-categories.
\end{proposition}

\begin{remark}
   [Grothendieck-Messing for divided Dieudonn\'e complexes: the nilpotent case]
\label{rem:groth-messing_nilpotent}
To prove the proposition using the ideas in~\cite[\S 5.9]{gmm}, one first needs to formulate a more general assertion. First, observe that, for any frame $\underline{A}$ for $R$, we have a canonical functor
\[
  \mathrm{DDC}_{\underline{A}}(R)\to  \mathrm{Perf}(\Aff^1/\Gm\times \Spf R)
\]
carrying $\underline{\mathsf{M}}$ to $M$ with its two-step filtration $0\to \Fil^0M \to M$. Now, suppose that we have a map $R'\to R$ in $\mathrm{CRing}^{p\mathhyph\mathrm{comp}}$, and a sequence of surjective maps of prismatic frames
  \[
  \underline{A}'\to \underline{A}_1\to \underline{A}
  \]
  satisfying the following conditions:
  \begin{itemize}
     \item $\underline{A}'\to \underline{A}_1$ is a map of frames for $R'$;
     \item $\underline{A}$ is a frame for $R$ and $\underline{A}_1\to \underline{A}$ lifts $R'\to R$;
     \item $(A_1,I_1)\to (A,I)$ is an isomorphism of prisms.
  \end{itemize}
  Then we have a canonical diagram
  \begin{align}\label{eqn:adm_dieu_groth-mess_diagram}
     \begin{diagram}
       \mathrm{DDC}_{\underline{A}'}(R')&\rTo&\mathrm{Perf}(\Aff^1/\Gm\times \Spf R')\\
       \dTo&&\dTo\\
       \mathrm{DDC}_{\underline{A}}(R)&\rTo&\mathrm{Perf}(\Aff^1/\Gm\times \Spf R)\times_{\mathrm{Perf}(R)}\mathrm{Perf}(R')
     \end{diagram}
  \end{align}
  Here, $\mathrm{DDC}_{\underline{A}}(R)\to \mathrm{Perf}(R')$ is obtained via the functor $\mathrm{Perf}(A)\simeq \mathrm{Perf}(A_1)\to \mathrm{Perf}(R')$. 

  Then the more refined claim is that, in this situation,~\eqref{eqn:adm_dieu_groth-mess_diagram} is Cartesian when restricted to the subcategories of nilpotent divided Dieudonn\'e complexes.  The idea is to use deformation theory, the nilpotence hypothesis, and various integrability properties of the stack of perfect complexes to reduce first to the case where $A'$ and $A$ are discrete, and next, by replacing $A'$ successively with quotients of the form $A'/K^m$ where $K=\ker(A'\to A)$, to the case where $A'\to A$ is the identity, where it becomes trivial to prove.
\end{remark}

\begin{notation}
  Suppose that $B$ is $\mathfrak{c}$-adically derived complete. Let $\mathrm{FFG}_n^{\mathfrak{c}\mathhyph\mathrm{conn}}(R)$ be the subcategory of $\mathrm{FFG}_n(R)$ spanned by the objects whose restriction over $\Spec B/\mathfrak{c}$ is connected. If $\mathfrak{c} = 0$, we will simply write $\mathrm{FFG}_n^{\mathrm{conn}}(R)$ instead.
\end{notation}

\begin{corollary}
   \label{cor:nilpotent_connected}
Suppose that $B$ is $\mathfrak{c}$-adically derived complete. For any laminated prismatic frame $\underline{A}$ for $R$, the functor in Corollary~\ref{cor:ffg_frame_realization} yields an exact equivalence
\[
\mathrm{FFG}_n^{\mathfrak{c}\mathhyph\mathrm{conn}}(R)\xrightarrow{\simeq}\Mod{\Int/p^n\Int}\left(\mathrm{DDC}^{[-1,0],\mathfrak{c}\mathhyph\mathrm{nilp}}_{\underline{A}}(R)\right).
\]
In particular, if $\underline{A}$ is a Breuil-Kisin frame for $R$, then we obtain an exact equivalence
\[
   \mathrm{FFG}_n^{\mathfrak{c}\mathhyph\mathrm{conn}}(R)\xrightarrow{\simeq}\mathrm{BK}^{\mathfrak{c}\mathhyph\mathrm{nilp}}_{\underline{A},n}(R).
\]
\end{corollary}
\begin{proof}
   First consider the case where $\mathfrak{c}\subset B$ is nilpotent. Here, we can use quasisyntomic descent to reduce to the case of $R$ semiperfectoid. We can then apply Proposition~\ref{prop:deformation_theory_nilpotent} to the surjective maps of frames $\underline{\Prism}_R\to \underline{W(R)}$ and $\underline{A}\to \underline{W(R)}$, and then appeal to Proposition~\ref{prop:ffg_classification_semiperf} and Remark~\ref{rem:connectedness_and_nilpotence}.

   For the general case, choose a map $\Int_p[t_1,\ldots,t_m]^{\wedge}_p\to R$ such that the images of $t_1,\ldots,t_m$ in $B$ generate $\mathfrak{c}$. For all $k\ge 1$, set $\tilde{R}_k \defn R/{}^{\mathbb{L}}(t_1^k,\ldots,t_m^k)$. By Remarks~\ref{rem:base_change_laminated} and~\ref{rem:base_change_prisms_closed_immersions}, we obtain canonical maps of frames $\underline{A}\to \underline{\tilde{A}}_k$ lifting $R\to \tilde{R}_k$. The argument from the first paragraph shows that we have a canonical equivalence
   \[
     \Mod{\Int/p^n\Int}\left(\mathrm{DDC}^{[-1,0],\mathrm{nilp}}_{\underline{\Prism}_{\tilde{R}_k}}(\tilde{R}_k)\right)\xrightarrow{\simeq}\Mod{\Int/p^n\Int}\left(\mathrm{DDC}^{[-1,0],\mathrm{nilp}}_{\underline{\tilde{A}}_k}(\tilde{R}_k)\right).
   \]
   Taking the limit over $k$ gives us equivalences:
   \begin{align*}
    \mathrm{FFG}_n^{\mathfrak{c}\mathhyph\mathrm{conn}}(R)&\xrightarrow{\simeq}\Mod{\Int/p^n\Int}\left(\mathrm{DDC}^{[-1,0],\mathfrak{c}\mathhyph\mathrm{nilp}}_{\Prism}(R)\right)\\
    &\xrightarrow{\simeq}\varprojlim_k\Mod{\Int/p^n\Int}\left(\mathrm{DDC}^{[-1,0],\mathrm{nilp}}_{\Prism}(\tilde{R}_k)\right)\\
    &\xrightarrow{\simeq}\varprojlim_k\Mod{\Int/p^n\Int}\left(\mathrm{DDC}^{[-1,0],\mathrm{nilp}}_{\underline{\tilde{A}}_k}(\tilde{R}_k)\right).
   \end{align*}
   Here, the first equivalence is from Remark~\ref{rem:connectedness_and_nilpotence}, while the second is from Proposition~\ref{cor:ffg_classification}, and the equivalence
   \[
      \mathsf{P}^{\mathrm{syn}}_{n,\{0,1\}}(R)\xrightarrow{\simeq}\varprojlim_k\mathsf{P}^{\mathrm{syn}}_{n,\{0,1\}}(\tilde{R}_k),
   \]
   which holds by the derived $(t_1,\ldots,t_m)$-completeness of $R$ and the fact that it is true if we replace $\tilde{R}_k$ with $\pi_0(\tilde{R}_k)$ instead. Indeed, in the latter case, we are simply saying that the category of finite locally free $p$-power torsion group schemes over $R$ is obtained as the inverse limit over $k$ of that over $\pi_0(\tilde{R}_k)$.

   To finish, it is enough to know that the natural functor
   \[
     \mathrm{DDC}_{\underline{A}}(R)\to \varprojlim_k\mathrm{DDC}_{\underline{\tilde{A}}_k}(\tilde{R}_k)
   \]
   is an equivalence. By induction on $m$, we can reduce to the case $m=1$, where this follows from Proposition~\ref{prop:adm_dieu_integrable}
\end{proof}

\begin{remark}
\label{rem:lau_zink_lifts}
   When $R$ is an $\Field_p$-algebra and $\underline{A}$ is a $p$-adic Breuil-Kisin frame for $R$ (referred to as a \emph{lift} of $R$ in~\cite[\S 1.2]{MR3867290}), the last result above is already known by work of Zink and Lau; see Remark 3.2 and the proof of Corollary 10.4 in ~\cite{MR3867290}. In the general Breuil-Kisin case, it appears to be new, though a particular instance of it can be found in \cite[\S 10.5]{MR2679066}.
\end{remark}

\begin{remark}
\label{rem:lau_zink_connected}
   A particular consequence of the corollary is an equivalence of categories between connected $p$-divisible groups and nilpotent displays over $R$. This is a theorem of Lau~\cite[Theorem C]{MR2983008}.
\end{remark}

\begin{example}
   [Classification of connected group schemes over fields]
If $\kappa$ is any characteristic $p$ field and $\mathcal{O}$ is a Cohen ring for $\kappa$ equipped with a Frobenius lift $\varphi:\mathcal{O}\to \mathcal{O}$, this gives a $p$-adic Breuil-Kisin frame $\underline{\mathcal{O}}$ for $\kappa$ and we obtain an equivalence
\[
   \mathrm{FFG}_n^{\mathrm{conn}}(\kappa)\xrightarrow{\simeq}\mathrm{BK}^{\mathrm{nilp}}_{\underline{\mathcal{O}},n}(\kappa).
\]
\end{example}

\begin{example}
   [Classification of group schemes with connected special fibers over complete local rings]
\label{ex:connected_complete_local_ring}
If $R$ is a complete local ring with maximal ideal $\mathfrak{m}$ and $\underline{A}$ is any frame for $R$, we obtain an equivalence
\[
  \mathrm{FFG}_n^{\mathfrak{m}\mathhyph\mathrm{conn}}(R)\xrightarrow{\simeq}\Mod{\Int/p^n\Int}\left(\mathrm{DDC}^{[-1,0],\mathfrak{m}\mathhyph\mathrm{nilp}}_{\underline{A}}(R)\right) 
\]
In particular, if $\underline{A}$ is a Breuil-Kisin frame, then we have an equivalence
\[
   \mathrm{FFG}_n^{\mathfrak{m}\mathhyph\mathrm{conn}}(R)\xrightarrow{\simeq}\mathrm{BK}_{\underline{A},n}^{\mathfrak{m}\mathhyph\mathrm{nilp}}(R).
\]
When $R$ is an $\Field_p$-algebra, one can use this to recover a result of de Jong~\cite[Theorem 9.3]{MR1235021}. When $R$ is regular of mized characteristic, this recovers a classification due to Lau~\cite[Theorem 10.7]{MR2679066}; see also~\cite[\S 2.7]{Cais2017-bf}.
\end{example}

\begin{example}
   [A non-regular example]
\label{ex:non_regular}
Note that the previous example has no regularity constraints whatsoever. For instance, if $\kappa$ is a perfect field, and $\Reg{K}$ is a totally ramified extension of $W(\kappa)$ generated by a uniformizer with Eisenstein polynomial $E(u)\in W(\kappa)[u]$, then we can take $R = \Reg{K}\pow{x,y}/(x^2,xy,y^2)$ with the frame $\underline{A}$ associated with the prism $(A,I') = (W(\kappa)\pow{u,x,y}/(x^2,xy,y^2),(E(u)))$ equipped with the Frobenius lift $u\mapsto u^p,x\mapsto x^p,y\mapsto y^p$. In fact, one can also replace $R$ with its quotients by powers of $\pi$ by also modifying $A$ appropriately by taking quotients by powers of $u$.
\end{example}

\begin{example}
   [Classification of connected group schemes over polynomial rings with semiperfectoid coefficients]
If $R = R_0[x_1,\ldots,x_n]^{\wedge}_p$ with $R_0$ semiperfectoid, then Example~\ref{ex:frames_for_polynomial_algebras} combines with Corollary~\ref{cor:nilpotent_connected} to give us an equivalence
\[
   \mathrm{FFG}_n^{\mathrm{conn}}(R)\xrightarrow{\simeq}\Mod{\Int/p^n\Int}\left(\mathrm{DDC}^{[-1,0],\mathrm{nilp}}_{\underline{A}}(R)\right).
\]
where $\underline{A}$ is a frame for $R$ with underlying ring $A = \Prism_{R_0}[x_1,\ldots,x_n]^{\wedge}_p$.
\end{example}

\subsection{A nilpotence criterion for equivalence}
\label{sub:a_criterion_for_equivalence}

One can improve Proposition~\ref{prop:deformation_theory_nilpotent} to be valid for the full category of divided complexes by moving the locus of nilpotence from the modules to the map of frames. We will explore this phenomenon in this subsection. $R$ will be discrete here.

\begin{construction}
\label{const:varphi_1}
Suppose that $\underline{A}'\to \underline{A}$ is a map of prismatic frames for $R$. If  we set
\[
K = \fib(A'\to A)\simeq \fib(\Fil^1A'\to \Fil^1A), 
\]
then the twisted divided Frobenius maps $(\Fil^1A')\{1\}\to A'\{1\}$ and $(\Fil^1A)\{1\}\to A\{1\}$ (see Example~\ref{ex:map_from_unit_bk_twist})  induce a semilinear endomorphism $\dot{\varphi}_1:K\{1\}\to K\{1\}$.
\end{construction}

\begin{remark}
[Sections of the Breuil-Kisin twist]
\label{rem:bk_twist_sections_deform}
   Following Example~\ref{ex:map_from_unit_bk_twist}, we see that we have
   \begin{align*}
   \fib\left(\mathrm{RHom}_{\underline{A}'}(\underline{\mathbf{1}},\underline{\mathbf{1}}^*)\to \mathrm{RHom}_{\underline{A}}(\underline{\mathbf{1}},\underline{\mathbf{1}}^*)\right)
   &\xrightarrow{\simeq}\fib(K\{1\}\xrightarrow{\mathrm{id}-\dot{\varphi}_1}K\{1\}).
   \end{align*}
   Therefore when $\mathrm{id}-\dot{\varphi}_1$ is an equivalence, we find that the map
   \[
      \mathrm{RHom}_{\underline{A}'}(\underline{\mathbf{1}},\underline{\mathbf{1}}^*)\to \mathrm{RHom}_{\underline{A}}(\underline{\mathbf{1}},\underline{\mathbf{1}}^*)
   \]
   is an equivalence. 
\end{remark}

The methods of~\cite[\S 5.9]{gmm} allow one to generalize the above remark, yielding the following analogue of Proposition~\ref{prop:deformation_theory_nilpotent}. It can be viewed as an animated refinement of a classical result of Zink and Lau; see especially~\cite[Theorem 3.2]{MR2679066} or~\cite[Proposition 5.6]{lau2018divided}.
\begin{proposition}
[Unique lifting principle using nilpotence of divided Frobenius]
\label{prop:deformation_theory}
Suppose that $\underline{A}'\to \underline{A}$ is a surjective map of prismatic frames for $R$ such that the endomorphism $\dot{\varphi}_1:K\{1\} \to K\{1\}$ is topologically locally nilpotent with respect to the $(p,I)$-adic topology. Then the natural functor $\mathrm{DDC}_{\underline{A}'}(R)\to \mathrm{DDC}_{\underline{A}}(R)$ is an exact equivalence.
\end{proposition}

\begin{remark}
   [Grothendieck-Messing for divided Dieudonn\'e complexes]
\label{rem:groth-messing_K_nilpotent}
The methods used actually show that, in the context of Remark~\ref{rem:groth-messing_nilpotent}, if we take the endomorphism $\dot{\varphi}_1:K\{1\}\to K\{1\}$ obtained from the map $\underline{A}\to \underline{A}_1$, then the diagram~\eqref{eqn:adm_dieu_groth-mess_diagram} is Cartesian whenever $\dot{\varphi}_1$ is topologically locally nilpotent.
\end{remark}

\begin{remark}
   [Grothendieck-Messing for nilpotent divided power extensions]
\label{rem:groth-messing_nilpotent_PD}
One situation in which some of the conditions of Remark~\ref{rem:groth-messing_K_nilpotent} hold is when $(R'\twoheadrightarrow R,\gamma)$ is a nilpotent divided power extension of semiperfectoid rings and we take $\underline{A}' = \underline{\Prism}_{R'}$ and $\underline{A} = \underline{\Prism}_R$. The divided powers yield a lift $A\to R'$ giving us the intermediate frame $\underline{A}_1$ for $R'$, and the associated endomorphism $\dot{\varphi}_1:K\{1\} \to K\{1\}$ is topologically locally nilpotent; see~\cite[Proposition 6.13.1]{gmm}. Indeed, this is essentially how the deformation theoretic content of Theorem~\ref{thm:HTwts01_representable} is shown. We should note however that the map $\Prism_{R'}\to \Prism_R$ is not in general surjective, so Proposition~\ref{prop:deformation_theory} does not directly apply.
\end{remark}

\begin{construction}
\label{const:map_of_frames_prismatic}
Suppose that $\underline{A}$ is a laminated prismatic frame for a semiperfectoid ring $R$. Then we have a canonical map of frames $\underline{\Prism}_R\to \underline{A}$ for $R$, and so Construction~\ref{const:varphi_1} gives a canonical operator
\[
\dot{\varphi}_1^{\underline{A}}:K^{A}\{1\}\to K^{A}\{1\},
\]
where $K^{A} = \fib(\Prism_R\to A)$. 
\end{construction}

\begin{corollary}
\label{cor:semiperfectoid_annealed}
Suppose that $R$ is semiperfectoid and $\underline{A}$ is a laminated prismatic frame for $R$ such that the map $\Prism_R\to A$ is surjective and such that the endomorphism $\dot{\varphi}^{\underline{A}}_1$ is topologically locally nilpotent. Then we have an exact equivalence
\[
\mathrm{FFG}_n(R)\xrightarrow{\simeq}\Mod{\Int/p^n\Int}(\mathrm{DDC}^{[-1,0]}_{\underline{A}}(R)).
\]
\end{corollary}
\begin{proof}
   Immediate from Propositions~\ref{prop:ffg_classification_semiperf} and~\ref{prop:deformation_theory}.
\end{proof}

We will now globalize Corollary~\ref{cor:semiperfectoid_annealed}. From now on $R$ will be an arbitrary $p$-complete discrete ring.

\begin{lemma}
\label{lem:semiperf_cover_surjective}
Let $\underline{A}$ be a laminated prismatic frame for $R$. Then the following are equivalent:
\begin{enumerate}
   \item For some quasisyntomic cover $R\to R_\infty$ as in Construction~\ref{rem:quasisyntomic_base_change}, the map of cosimplicial frames $\underline{\Prism}_{R_\infty^{\otimes_R(\bullet +1)}}\to \underline{A}^{(\bullet)}_\infty$ from Lemma~\ref{lem:syntomic_frame_criterion} is surjective.
   \item For any quasisyntomic cover $R\to R_\infty$ as in Construction~\ref{rem:quasisyntomic_base_change}, the map of cosimplicial frames $\underline{\Prism}_{R_\infty^{\otimes_R(\bullet +1)}}\to \underline{A}^{(\bullet)}_\infty$ from Lemma~\ref{lem:syntomic_frame_criterion} is surjective.
   \item The associated map $\Spf A\to R^\Prism$ is a closed immersion.
\end{enumerate}
\end{lemma} 
\begin{proof}
The equivalence of the first two assertions comes down to the following assertion: If $\underline{A}$ is a laminated prismatic frame for $R$, and, if $R\to R'$ is a quasisyntomic cover of semiperfectoid rings such that the base-change\footnote{The tensor products are derived and $(p,I)$-completed.} $\Prism_{R'}\to \Prism_{R'}\otimes_{\Prism_{R'}}A$ is surjective, then the map $\Prism_R\to A$ is surjective. This follows from the fact that the map $\Prism_{R}\to \Prism_{R'}$ is $(p,I)$-completely faithfully flat; see Proposition~\ref{prop:qsyn_to_flat_covers}. The equivalence of these with the last assertion is now clear from quasisyntomic descent for closed immersions.
\end{proof}

\begin{remark}
[Surjectivity criterion]
\label{rem:surjectivity_criterion}
Suppose that $\underline{A}$ is a frame for $R$ such that the natural map $A\to A/{}^{\mathbb{L}}(p,I)$ factors through a map $R\to A/{}^{\mathbb{L}}(p,I)$. Then, for any map of frames $\underline{A}'\to \underline{A}$ for $R$, the composition $A'\to A/{}^{\mathbb{L}}(p,I)$ is surjective, and hence, by $(p,I)$-completeness, the map $A'\to A$ is automatically surjective. This condition holds for Breuil-Kisin frames, and also for any frame that is obtained via quasisyntomic base-change from a Breuil-Kisin frame via Remark~\ref{rem:base_change_laminated}. In particular, the equivalent conditions of Lemma~\ref{lem:semiperf_cover_surjective} are always valid when $\underline{A}$ is a Breuil-Kisin frame.
\end{remark}

\begin{construction}
\label{const:map_of_frames_prismatic_general}
Suppose that $\underline{A}$ is a laminated prismatic frame for $R$. For each map $R\to S$ with $S$ semiperfectoid, we obtain the base-change $\underline{A}_S$ of $\underline{A}$ from Remark~\ref{rem:base_change_laminated}, and a map of frames $\underline{\Prism}_S\to \underline{A}_S$, along with an operator $\dot{\varphi}^{\underline{A}_S}_1:K^A_S\{1\}\to K^A_S\{1\}$, where $K^A_S = \fib(\Prism_S\to A_S)$. In this way, we obtain quasisyntomic sheaves $A_{-}$, $K^A_{-}$ of $\Prism_{-}$-modules over $\Spf R$, along with an operator $\dot{\varphi}^{\underline{A}}_1:K^A_{-}\{1\}\to K^A_{-}\{1\}$, whose values on semiperfectoid quasisyntomic $R$-algebras $S$ are given as just described.
\end{construction} 

\begin{remark}
Suppose that we have an ideal $\mathfrak{c}\subset B\defn (R/pR)_{\mathrm{red}}$. We obtain a map of quasisyntomic sheaves (of derived rings) $\Prism_{-}\to \Reg{-}^{\mathfrak{c}}$ whose values on semiperfectoid $S$ is given by the composition
\[
\Prism_S \to S \to S\otimes_RB/\mathfrak{c}.
\]
\end{remark}

\begin{corollary}
\label{cor:general_annealed}
Suppose that $\underline{A}$ is a laminated prismatic frame for $R$ satisfying the equivalent conditions of Lemma~\ref{lem:semiperf_cover_surjective}. Suppose also that there is some finitely generated ideal $\mathfrak{c}\subset B$ such that $B$ is derived $\mathfrak{c}$-complete and such that the base-change of $\dot{\varphi}^{\underline{A}}_1$ along $\Prism_{-}\to \Reg{-}^{\mathfrak{c}}$ is locally nilpotent. Then we have an exact equivalence
\[
\mathrm{FFG}_n(R)\xrightarrow{\simeq}\Mod{\Int/p^n\Int}(\mathrm{DDC}_{\underline{A}}(R)).
\]
In particular, if $\underline{A}$ is of Breuil-Kisin type, we obtain an equivalence
\[
   \mathrm{FFG}_n(R)\xrightarrow{\simeq}\mathrm{BK}_{\underline{A},n}(R).
\]
\end{corollary}
\begin{proof}
  First, suppose that $\mathfrak{c}$ is nilpotent.  Then the result follows from Corollary~\ref{cor:semiperfectoid_annealed} and quasisyntomic descent. In general, we can argue as in the proof of Corollary~\ref{cor:nilpotent_connected} to reduce to this case.
\end{proof}

\begin{remark}
The proof of the corollary shows something stronger: Combined with Proposition~\ref{prop:syntomic_and_fppf}, it implies that the fppf cohomology of any $G\in \mathrm{FFG}_n(R)$ associated with $\underline{\mathsf{M}}\in \Mod{\Int/p^n\Int}(\mathrm{DDC}_{\underline{A}}(R))$ is isomorphic to $\mathrm{RHom}_{\underline{A}}(\underline{\mathbf{1}},\underline{\mathsf{M}})$.
\end{remark}

\begin{remark}
   [Classification via Grothendieck-Messing]
\label{rem:classification_using_groth-messing}
Suppose that $R'\twoheadrightarrow R$ is nilpotent divided power extension. In this case, there is extension $x_{\dR,R'}:\Spf R'\to R^{\Prism}$ of $x_{\dR}$ that is obtained from the divided power structure; see~\cite[Lemma 6.8.1]{gmm}. Suppose that $\underline{A}$ is a laminated prismatic frame for $R$ equipped with a lift $\underline{A}\to R'$ such that the associated composition $\Spf R'\to \Spf A\to R^\Prism$ is isomorphic to $x_{\dR,R'}$. Suppose also that $\underline{A}$ satisfies the hypotheses of Corollary~\ref{cor:general_annealed}. Write $\mathrm{DDC}_{\underline{A},\mathrm{tors}}^{[-1,0]}(R)$ for the subcategory of $\mathrm{DDC}^{[-1,0]}_{\underline{A}}(R)$ spanned by the $\Int/p^n\Int$-module objects for all $n\ge 1$. Then from Theorem~\ref{thm:HTwts01_representable} and Corollaries~\ref{cor:ffg_classification} and~\ref{cor:general_annealed}, we obtain a canonical Cartesian diagram\footnote{Strictly speaking, we only obtain a Cartesian diagram of the underlying groupoids in this way, but one can upgrade it to a Cartesian diagram of categories with little difficulty.}
\[
   \begin{diagram}
      \mathrm{FFG}(R')&\rTo& \mathrm{Perf}(\Aff^1/\Gm\times \Spf R')\\
      \dTo&&\dTo\\
      \mathrm{DDC}_{\underline{A},\mathrm{tors}}^{[-1,0]}(R)&\rTo&\mathrm{Perf}(\Aff^1/\Gm\times \Spf R)\times_{\mathrm{Perf}(R)}\mathrm{Perf}(R').
   \end{diagram}
\]
\end{remark}

\begin{remark}
\label{rem:groth-messing_low_ramification}
Suppose that $\mathcal{O}$ is a complete mixed characteristic DVR of absolute ramification index $e<p-1$, uniformizer $\pi$ and residue field $\kappa$. Then, for each $n\ge 1$, and any $p$-complete $\mathcal{O}$-algebra $S$, $S/(\pi^n)\to S/(\pi)$ is a nilpotent divided power extension. If $\underline{A}$ is a frame for $S/(\pi)$ satisfying the hypotheses of Remark~\ref{rem:classification_using_groth-messing}, we get a Cartesian diagram as above for each $n$ with $R' = S/(\pi^n)$ and $R = S/(\pi)$. Taking the limit over $n$ now gives us a Cartesian diagram
\[
   \begin{diagram}
      \mathrm{FFG}(S)&\rTo& \mathrm{Perf}(\Aff^1/\Gm\times \Spf S)\\
      \dTo&&\dTo\\
      \mathrm{DDC}_{\underline{A},\mathrm{tors}}^{[-1,0]}(S/(\pi))&\rTo&\mathrm{Perf}(\Aff^1/\Gm\times \Spec S/(\pi))\times_{\mathrm{Perf}(S/(\pi))}\mathrm{Perf}(S).
   \end{diagram}
\]
\end{remark}

\begin{example}
[Honda systems]
\label{ex:honda_systems}
Let us apply Remark~\ref{rem:groth-messing_low_ramification} to the case where $S = \mathcal{O}$ and where $\kappa$ is perfect. Here, we can take $\underline{A} = \underline{\Prism}_{\kappa}$, and so we obtain a Cartesian diagram
\[
 \begin{diagram}
      \mathrm{FFG}(\mathcal{O})&\rTo& \mathrm{Perf}(\Aff^1/\Gm\times \Spf \mathcal{O})\\
      \dTo&&\dTo\\
      \mathrm{BK}_{\underline{\Prism}_{\kappa},\mathrm{tors}}(\kappa)&\rTo&\mathrm{Perf}(\Aff^1/\Gm\times \Spec \kappa)\times_{\mathrm{Perf}(\kappa)}\mathrm{Perf}(\mathcal{O}).
   \end{diagram}
\]
Here, $\mathrm{BK}_{\underline{\Prism}_{\kappa},\mathrm{tors}}(\kappa)$ is the `union' of the categories $\mathrm{BK}_{\underline{\Prism}_\kappa,n}(\kappa)$ for $n\ge 1$. In this way, we obtain a description of $\mathrm{FFG}(\mathcal{O})$ in terms of a triple $(\mathsf{M},F_{\mathsf{M}},V_{\mathsf{M}})$ with $\mathsf{M}$ a $p$-power torsion $W(\kappa)$-module, along with a filtration on $\mathcal{O}\otimes_{W(\kappa)}\mathsf{M}$ satisfying some additional compatibilities. A full unwinding of this yields the notion of a \emph{finite Honda system} as given in~\cite[Definition 2.6]{MR1727133}, and the Cartesian diagram recovers the first part of Theorem 3.6 of \emph{loc. cit.} It is actually possible to recover the full theorem from the methods here: One has to observe that in Remark~\ref{rem:classification_using_groth-messing}, if one drops the nilpotence hypothesis on the divided powers, the diagram there is still Cartesian when restricted to $\mathrm{FFG}^{\mathrm{conn}}(R')$ and to the subcategory of nilpotent objects in $\mathrm{DDC}_{\underline{A},\mathrm{tors}}^{[-1,0]}(R')$.
\end{example}

\subsection{Specializing to the Breuil-Kisin case}
\label{subsec:specializing_to_the_breuil_kisin_case}

In this subsection, $\underline{A}$ will be a Breuil-Kisin frame for a discrete ring $R$. Here, by Remark~\ref{rem:surjectivity_criterion}, the equivalent conditions of Lemma~\ref{lem:semiperf_cover_surjective} always hold. Therefore, the only obstruction to the functor from Remark~\ref{rem:bk_syntomic} being an equivalence is the nilpotence condition on $\dot{\varphi}_1^{\underline{A}}$ from Corollary~\ref{cor:general_annealed}. We will see that this can be understood in somewhat more concrete terms via the cotangent complex.

\begin{remark}
   [The Rees stack associated with a Breuil-Kisin frame]
\label{rem:rees_bk_frame}
Just as in Remark~\ref{rem:nygaard_rees_stack}, we have the $(p,I')$-complete formal Rees stack $\Rees(\Fil^\bullet_{I'}A)$ associated with the $I'$-adic filtration on $A$, and the Frobenius lift on $A$ extends to a map of Rees stacks $\Rees(\Fil^\bullet_IA)\to \Rees(\Fil^\bullet_{I'}A).$. We obtain two maps $\tau,\sigma:\Spf A\to \Rees(\Fil^\bullet_{I'}A)$, where $\tau$ is the pullback of $\Gm/\Gm\to \Aff^1/\Gm$ and $\sigma$ is obtained from the filtered Frobenius lift and the isomorphism
\[
   \Spf A \xrightarrow{\simeq}\Rees(\Fil^\bullet_{I,\pm}A).
\]
Note also that we have a canonical map $x_{\underline{A}}:\Aff^1/\Gm\times \Spf R\to \Rees(\Fil^\bullet A)$ associated with the map $\Fil^\bullet A\to \Fil^\bullet_{\mathrm{triv}}R$ of filtered rings. Here, $\Fil^\bullet_{\mathrm{triv}}R$ is the trivial filtration on $R$ supported in graded degree $0$. 
\end{remark}

\begin{proposition}
   [Mapping the Rees stack to the filtered prismatization]
\label{prop:mapping_rees_to_filt_prismatization}
There is a canonical map of formal stacks $\iota^{\mathcal{N}}_{\underline{A}}:\Rees(\Fil^\bullet A)\to R^\mathcal{N}$ such that
\[
\iota^{\mathcal{N}}_{\underline{A}}\circ \tau = j_{\dR}\circ \iota_{(A,I)}\;;\;\iota^{\mathcal{N}}_{\underline{A}}\circ \sigma = j_{\mathrm{HT}}\circ \iota_{(A,I)}\;;\; \iota^{\mathcal{N}}_{\underline{A}}\circ x_{\underline{A}} = x^{\mathcal{N}}_{\dR}.
\]
\end{proposition}
\begin{proof}
   See~\cite[Example 6.10.5]{gmm}.
\end{proof}

\begin{remark}
   [Base-change of filtered frames]
\label{rem:base_change_filtered}
Combining Proposition~\ref{prop:mapping_rees_to_filt_prismatization} with Proposition~\ref{prop:nygaard_qsynt_descent} shows that for a quasisyntomic cover $R\to S$, the base-change $ \Rees(\Fil^\bullet A)\times_{R^{\mathcal{N}}}S^{\mathcal{N}}$ is of the form $\Rees(\Fil^\bullet A_S)$ for a $(p,I)$-complete filtered animated commutative ring $\Fil^\bullet A_S$ with underlying ring $A_S$. In fact, $\Rees(\Fil^\bullet A_S)$ is $(p,I)$-completely flat over $\Rees(\Fil^\bullet A)$ and is therefore is a \emph{classical} formal stack. Moreover, the associated graded stack 
\[
\Rees(\Fil^\bullet A_S)_{(t=0)} = \Spf\left(\bigoplus_{m\leq 0}\gr^{-m}A_S\right)/\Gm
\]
is $p$-completely flat over 
\[
\Rees(\Fil^\bullet A)_{(t=0)} = \Spf\left(\bigoplus_{m\leq 0}I^{',m}/I^{',m+1}\right)/\Gm
\]
and is hence also classical. All of this means that $\Fil^\bullet A_S$ is in fact a classical filtered ring: the modules $\Fil^mA_S$ are \emph{submodules} of $A_S$. Moreover, $\Fil^1A_S\subset A_S$ agrees with the submodule given by the frame structure $\underline{A}_S$ on $A_S$. 
\end{remark}

\begin{remark}
   [Factoring $\dot{\varphi}_1$]
\label{rem:factoring_varphi_1}
With the notation as above, the Frobenius lift on $A_S$ extends to a map of filtered rings $\Fil^\bullet A_S\to \Fil^\bullet_{I}A_S$ that in filtered degree $1$ is the divided Frobenius map $\varphi_1:\Fil^1A_S \to IA_S$. This is deduced easily from the construction and the fact that it is true with $\Fil^\bullet A_S$ replaced with $\Fil^\bullet A$ or $\Fil^\bullet_{\mathcal{N}}\Prism_S$. This means that the composition
\[
   \Fil^1A_S\{1\}\xrightarrow{\varphi_1\{1\}}A_S\{1\} \to A_S\{1\}/^{\mathbb{L}}(p,I)
\]
factors through $\gr^1A_S\{1\}$, and similarly the corresponding composition for $\Fil^\bullet_{\mathcal{N}}\Prism_S$ factors through $\gr^1_{\mathcal{N}}\Prism_S$. This implies that the composition
\[
   K^A_S\{1\}\xrightarrow{\dot{\varphi}^{\underline{A}_S}_1}K^A_S\{1\}\to K^A_S\{1\}/{}^{\mathbb{L}}(p,I_S)
\]
factors through $\gr^1K^A_S\{1\}$ where
\[
   \gr^1K^A_S = \fib(\gr^1_{\mathcal{N}}\Prism_S\to \gr^1A_S).
\]
Using this, one finds that the base-change over $B$ of $\dot{\varphi}^{\underline{A}_S}_1\circ \dot{\varphi}^{\underline{A}_S}_1$ factors through a $\varphi$-semilinear map
\begin{align}\label{eqn:base_change_over_B_gr1_factorization}
   B\otimes_R\gr^1K^A_S\{1\}\to B\otimes_R\gr^1K^A_S\{1\}.
\end{align}
In particular, $\dot{\varphi}^{\underline{A}}_1$ satisfies the condition in Corollary~\ref{cor:general_annealed} for the ideal $\mathfrak{c}$ if and only the base-change of~\eqref{eqn:base_change_over_B_gr1_factorization} over $B/\mathfrak{c}$ is nilpotent for a quasisyntomic cover $R\to S$.
\end{remark}

\begin{remark}
   [Relationship with the cotangent complex]
\label{rem:gr1_cotangent_complex}
Suppose that $S$ is a $p$-complete $R$-algebra. Then, by~\cite[Corollary 5.5.18]{bhatt2022absolute}, we have a canonical isomorphism of $S$-modules $\gr^1_{\mathcal{N}}\Prism_S\{1\}\xrightarrow{\simeq}\mathbb{L}_{S}[-1]$. This is obtained via a comparison map between Nygaard filtered abslute prismatic cohomology and Hodge filtered $p$-completed derived de Rham cohomology. In particular, we see that we have
\[
  R\Gamma_{\mathrm{qsyn}}(\Spf R,\gr^1K^A_{-})\xrightarrow{\simeq}\fib(\gr^1_{\mathcal{N}}\Prism_R\to \gr^1A)\simeq \fib(\mathbb{L}_{R}[-1]\to I'/I^{',2}).
\]
Unwinding definitions shows that the map on the right is up to sign obtained from the rotation of the canonical fiber sequence
\[
   I'/I^{',2}\to R\otimes_A\mathbb{L}_A\to \mathbb{L}_R.
\]
In other words, we have a canonical isomorphism
\begin{align}\label{eqn:Rgamma_gr^1K}
   R\Gamma_{\mathrm{qsyn}}(\Spf R,\gr^1K^A_{-})\xrightarrow{\simeq}(R\otimes_A\mathbb{L}_A)[-1].
\end{align}
\end{remark}

\begin{remark}
[The key operator via the cotangent complex]
\label{rem:divided_frob_cotangent complex}
Via~\eqref{eqn:Rgamma_gr^1K}, the operator~\eqref{eqn:base_change_over_B_cotangent_desc}, viewed as a map of quasisyntomic sheaves over $\Spf R$, yields by taking global sections a $\varphi$-semilinear endomorphism 
\begin{align}\label{eqn:base_change_over_B_cotangent_desc}
B\otimes_A\mathbb{L}_A[-1]\to B\otimes_A\mathbb{L}_A[-1]. 
\end{align}
This can be described quite easily. Note that the differential $d\varphi:\mathbb{L}_A\to \mathbb{L}_A$ of the Frobenius lift of $A$ admits a canonical factorization
\[
   \mathbb{L}_A\xrightarrow{`d\varphi/p'}\mathbb{L}_A\xrightarrow{p}\mathbb{L}_A
\]
Now~\eqref{eqn:base_change_over_B_cotangent_desc} is just the base-change of $`d\varphi/p'$. 
\end{remark}

\begin{remark}
   [The semiperfectoid case]
\label{rem:cotangent_complex_semiperfectoid}
Suppose that $R$ is semiperfectoid. Then combining Remarks~\ref{rem:factoring_varphi_1},~\ref{rem:gr1_cotangent_complex} and~\ref{rem:divided_frob_cotangent complex} shows that $\dot{\varphi}^{\underline{A}}_1$ is topologically locally nilpotent if and only if the divided Frobenius endomorphism of $\mathbb{L}_A$ is topologically locally nilpotent.
\end{remark}

\begin{remark}
[Frames for non-discrete $R$]
\label{rem:frames_for_non-discrete_R}
All of the preceding discussion---except for the last two sentences of Remark~\ref{rem:base_change_filtered}---is still valid if we assume that $\underline{A}$ is a $p$-adic prismatic frame with $A$ discrete and $(\Fil^1A \to A)\simeq (A\xrightarrow{p}A)$ \emph{without} requiring that $A$ be $p$-torsion free. In this case, $\underline{A}$ is a frame for the not necessarily discrete $\Field_p$-algebra $A/{}^{\mathbb{L}}p$.
\end{remark}

\begin{proposition}
   \label{prop:semiperf_bk_case}
Suppose that $R$ is semiperfectoid, and suppose that the operator~\eqref{eqn:base_change_over_B_cotangent_desc} is locally nilpotent. Then the functor $\mathrm{FFG}_n(R)\to \mathrm{BK}_{\underline{A},n}(R)$ is an equivalence.
\end{proposition}
\begin{proof}
   Follows from Corollary~\ref{cor:semiperfectoid_annealed} and Remark~\ref{rem:cotangent_complex_semiperfectoid}.
\end{proof}

\begin{remark}
[Nilpotence in terms of $\delta$-structure: degree $0$]
   \label{rem:cotangent_complex_semiperfectoid_deg_0}
If $R$ is semiperfectoid and we choose a surjection $R_0\to R$ with $R_0$ perfectoid, then we have $\mathbb{L}_A = \mathbb{L}_{A/\Prism_{R_0}}$, and so, if $K = \ker(\Prism_{R_0}\to A)$, then $H^{0}(\mathbb{L}_A[-1]) = \tau^{\ge 0}\mathbb{L}_A[-1] = K/K^2$. The operator on $H^0(\mathbb{L}_A[-1])=K/K^2$ is inherited from the endomorphism $\overline{\delta}$ of $K/K^2$ induced by the function $\delta:K\to K$ arising from the compatibility of the map $\Prism_{R_0}\to A$ with $\delta$-structures. Therefore, the nilpotence condition in cohomological degree $0$ reduces to asking for this operator to be topologically locally nilpotent. We will see in Corollary~\ref{cor:semiperf_bk_case} below that this is already sufficient to yield the equivalence in Proposition~\ref{prop:semiperf_bk_case} under some finiteness conditions.
\end{remark}

\subsection{The classical prismatization suffices}
\label{subsec:the_classical_prismatization_suffices}

In this subsection, we will see that the classical prismatization---and hence the classical absolute prismatic site---is sufficient for classifying finite flat group schemes in many cases. This amounts to a reinterpretation of results of Lau~\cites{MR3867290,lau2018divided}, which we recover here via a slightly different argument.

\begin{definition}
[Classical prismatic divided Dieudonn\'e complexes]
   A \defnword{classical} prismatic divided Dieudonn\'e complex for $R$ is a tuple $\underline{\mathcal{M}}_{\mathrm{cl}}$ satisfying the same conditions as that of a prismatic divided Dieudonn\'e complex (Definition~\ref{def:prismatic_adm_dieu}), except that we replace $R^{\Prism}$ with its classical truncation $R^{\Prism}_{\mathrm{cl}}$. Write $\mathrm{DDC}_{\Prism_{\mathrm{cl}}}(R)$ for the $\infty$-category of such tuples.
\end{definition}

\begin{remark}
[The semiperfectoid case]
\label{rem:semiperfectoid_classical_ddc}
When $R$ is semiperfectoid, we have
\[
\mathrm{DDC}_{\Prism_{\mathrm{cl}}}(R)\simeq \mathrm{DDC}_{\underline{\Prism}_{R,\mathrm{cl}}}(R),
\]
where $\underline{\Prism}_{R,\mathrm{cl}}$ is the frame obtained by taking the classical truncation of $\underline{\Prism}_R$: More precisely, we replace $\Prism_R$ with its $(p,I_R)$-adically completed classical truncation\footnote{Here, we are taking the underived usual inverse limit giving the classical $(p,I_R)$-adic completion.}
\[
\Prism_{R,\mathrm{cl}} = \varprojlim_n \pi_0(\Prism_R)/(p,I_R)^n\pi_0(\Prism_R),
\]
and $\overline{\Prism}_R$ with $\overline{\Prism}_R\otimes_{\Prism_R}\Prism_{R,\mathrm{cl}}$. Note that, since $\pi_0(\Prism_R)$ is $(p,I_R)$-adically derived complete, $\Prism_{R,\mathrm{cl}}$ is also its maximal separated quotient; see~\cite[\href{https://stacks.math.columbia.edu/tag/091T}{Proposition 091T}]{stacks-project}.
\end{remark}

\begin{remark}
[Description via quasisyntomic descent]
\label{rem:classical_absolute_prismatic_site}
Via quasisyntomic descent, one finds that $\mathrm{DDC}_{\Prism_{\mathrm{cl}}}(R)$ can also be described in terms seen in the introduction. Namely, it is equivalent to the $\infty$-category of tuples
\[
(\mathcal{M}_{-}\xrightarrow{\Psi_{\mathcal{M}}}\varphi^* \mathcal{M}_{-},\Fil^0_{\mathrm{Hdg}}M\to M,\xi)
\]
where:
\begin{enumerate}
   \item $\mathcal{M}_{-}$ is a perfect complex over $\Prism_{-,\mathrm{cl}}$ with cohomology killed by $p^n$;
   \item $M$ is the perfect complex over $R$ corresponding to the base-change of $\mathcal{M}_{-}$ along $\Prism_{-,\mathrm{cl}}\to \mathcal{O}$ and $\Fil^0_{\mathrm{Hdg}}M\to M$ is a map of perfect complexes;
   \item $\Psi_{\mathcal{M}}:\mathcal{M}_{-}\to \varphi^* \mathcal{M}_{-}$ is a map of complexes over $\Prism_{-,\mathrm{cl}}$ whose cofiber is equipped with an isomorphism $\xi$ to $ \overline{\Prism}_{-,\mathrm{cl}}\otimes_R\gr^{-1}_{\mathrm{Hdg}}M$.
\end{enumerate}
\end{remark}

\begin{remark}
[Comparison with divided Dieudonn\'e crystals]
\label{rem:comp_with_divided_dieu}
When $R$ is an $\Field_p$-algebra, a classical prismatic divided Dieudonn\'e complex where the underlying perfect complex $\mathcal{M}$ is a vector bundle over $R^\Prism$ is the same as a \emph{divided Dieudonn\'e crystal} as defined in~\cite{lau2018divided}. This follows from Remarks~\ref{rem:semiperfectoid_classical_ddc} and~\ref{rem:windows}.
\end{remark}

We will need the following definitions appearing in~\cite{lau2018divided}.
\begin{definition}
   [$F$-finiteness and $F$-nilpotence]
An $\Field_p$-algebra $R$ is $F$-\defnword{finite} if the Frobenius endomorphism $\varphi:R\to R$ is of finite type (equivalently, is finite). It is $F$-\defnword{nilpotent} if $\ker\varphi$ is a nilpotent ideal.
\end{definition}

\begin{remark}
   [$F$-nilpotence for semiperfect rings]
\label{rem:fnilp_semiperf}
Suppose that $R$ is semiperfect with inverse perfection $R^\flat = \varprojlim_\varphi R$, and set $J = \ker(R^\flat\to R)$. Then $R$ is $F$-nilpotent if and only if there exists $r\ge 1$ such that $J^{p^r}\subset \varphi(J)$.
\end{remark}

\begin{theorem}
\label{thm:classical_prismatization_enough}
Suppose that one of the following holds:
\begin{enumerate}
   \item $R$ is $p$-quasisyntomic;
   \item $R/pR$ is $F$-finite and $F$-nilpotent.
\end{enumerate}
Then the functor
\[
 \mathrm{FFG}_n(R)\to \Mod{\Int/p^n\Int}\left( \mathrm{DDC}^{[-1,0]}_{\Prism_{\mathrm{cl}}}(R)\right)
\]
is an exact equivalence.
\end{theorem}

\begin{proof}
When $R$ is $p$-quasisyntomic, $R^{\Prism}$ is already classical~\cite[Corollary 8.13]{bhatt2022prismatization}, so the result follows trivially from Corollary~\ref{cor:ffg_classification}. This refines the main result of~\cite{Mondal2024-cy}.

So we will assume from here on that $R/pR$ is $F$-finite and $F$-nilpotent. To begin, via Lemma~\ref{lem:adm_dieu_derived_descent}, we can reduce to the case where $R$ is $p$-nilpotent.

 Next, using quasisyntomic descent, we reduce to the case where $R$ is semiperfectoid with $R/pR$ $F$-nilpotent. Here, we are using the fact that any $F$-finite and $F$-nilpotent $\Field_p$-algebra admits a quasisyntomic cover $R\to R_\infty$ obtained by adjoining all $p$-power roots of a finite set of elements of $R$ so that $R_\infty^{\otimes_R m}$ is $F$-nilpotent for all $m$: See~\cite[Lemma 2.6]{lau2018divided}. We can now conclude using Proposition~\ref{prop:F-nilp-nilpotent} below.
\end{proof}

\begin{remark}
  Case (2) of the theorem yields in particular an equivalence
\[
   \mathrm{BT}(R) \xrightarrow{\simeq}\mathrm{DDC}_{\Prism_{\mathrm{cl}}}^{[0,0]}(R)
\]
in the situation where $R/pR$ is $F$-finite and $F$-nilpotent. When $R$ is an $\Field_p$-algebra, Remark~\ref{rem:comp_with_divided_dieu} shows that this simply a rephrasing of the main result of~\cite{lau2018divided}.
\end{remark}

\begin{proposition}
   \label{prop:F-nilp-nilpotent}
Suppose that $R$ is semiperfectoid and that $R/pR$ is $F$-nilpotent. Then the map of frames $\zeta_R:\underline{\Prism}_R\to \underline{\Prism}_{R,\mathrm{cl}}$ satisfies the conditions of Corollary~\ref{cor:semiperfectoid_annealed}.
\end{proposition}

The rest of this subsection will be devoted to the proof of this proposition, but first we make a few observations beginning with the following corollary.

\begin{corollary}
   \label{cor:semiperf_bk_case}
Suppose that $R$ is semiperfectoid with $R/pR$ $F$-nilpotent, and that $\underline{A}$ is a Breuil-Kisin frame for $R$. Suppose that the divided Frobenius endomorphism of $\mathbb{L}_A$ induces a topologically locally nilpotent endomorphism of $H^{-1}(\mathbb{L}_A)$. Then we have an exact equivalence
\[
   \mathrm{FFG}_n(R)\xrightarrow{\simeq}\mathrm{BK}_{\underline{A},n}(R).
\]{}
\end{corollary}
\begin{proof}
   By Propositions~\ref{prop:F-nilp-nilpotent} and~\ref{prop:deformation_theory}, it is enough to know that the operator $\dot{\varphi}^{\underline{A}}_{1,\mathrm{cl}}$ on $K^A_{\mathrm{cl}}\{1\} = \ker(\Prism_{R,\mathrm{cl}}\to A)\{1\}$ is topologically locally nilpotent. Since $\Prism_{R,\mathrm{cl}}$ is the classical $(p,I_R)$-adic completion of $H^0(\Prism_R)$, $K^A_{\mathrm{cl}}$ is the classical $(p,I_R)$-adic completion of $H^0(K^A)$. Therefore, it is enough to know that $H^0(\dot{\varphi}^{\underline{A}}_{1})$ is locally nilpotent mod-$(p,I_R)$. We now conclude using Remark~\ref{rem:cotangent_complex_semiperfectoid}.
\end{proof}

\begin{remark}
Proposition~\ref{prop:F-nilp-nilpotent} actually yields a partial answer to a question raised in~\cite[Remark 8.7.6]{gmm}. It shows that, in the notation of \emph{loc. cit.}, for discrete inputs $R$ such that $R/pR$ is $F$-nilpotent and $F$-finite, the $\infty$-groupoid $\Gamma_{\mathrm{syn}}(\mathcal{X})(R)$ can be computed using only the classical truncation of $R^{\mathrm{syn}}$. In particular, it implies that the values on such inputs of the smooth formal Artin stacks $\BT{n}$ from Section 9 of \emph{loc. cit.} can also be computed via the classical syntomification. 
\end{remark}

\begin{remark}
   [Reduction to the case of semiperfect algebras]
\label{rem:reduction_to_Fp_algebras_F-nilp-nilpotent}
By the discussion in Remark~\ref{rem:groth-messing_nilpotent_PD}, if $R'\twoheadrightarrow R$ is a square-zero extension equipped with trivial divided powers, then $\zeta_{R'}$ satisfies these conditions whenever $\zeta_R$ does so. This allows us to reduce the proof of Proposition~\ref{prop:F-nilp-nilpotent} to the case where $R$ is itself a semiperfect $F$-nilpotent $\Field_p$-algebra.
\end{remark}

The key for the proof of Proposition~\ref{prop:F-nilp-nilpotent} is a construction due to Lau:
\begin{lemma}
   \label{lem:straight_weak_nilpotence}
Suppose that $R^\flat$ is a perfect $\Field_p$-algebra and $J\subset R^\flat$ an ideal. For each $m\ge 0$, let $K_m\subset W(R^\flat)$ be the subset consisting of elements of the form $(a_0,a_1,\ldots)$ in Witt coordinates with $a_i\in J^{p^{i+m}}$. Set $R_m = R^\flat/J_m$. Suppose that $R \defn R_0$ is $F$-nilpotent. Then, for each $m\ge 0$:
\begin{enumerate}
   \item $K_m$ is a $p$-adically closed $\delta$-ideal in $W(R^\flat)$;
   \item The $\delta$-ring $A_m = W(R^\flat)/K_m$ underlies a laminated $p$-adic prismatic frame $\underline{A}_m$ for $R_m$ with $\Fil^1A_m = pA_m$;
   \item $A_{m+1}$ underlies a laminated $p$-adic prismatic frame $\underline{A}'_{m+1}$ for $R_m$ with
   \[
      \Fil^1A'_{m+1} = (K_m+pW(R^\flat))/K_{m+1}\subset A_{m+1};
   \]
   \item The ideal $\Fil^1A'_{m+1}\subset A_{m+1}$ is equipped with divided powers compatible with the canonical divided power structure on $pA_{m+1}$, and restricting to the trivial divided power structure on the ideal $K_m/K_{m+1}$.
\end{enumerate}
\end{lemma}
\begin{proof}
All of this follows from~\cite[\S 6]{lau2018divided}. We give some details here for the benefit of the reader. We will prove (1) for $K\defn K_0$; the proof for general $m$ is the same. First, note that $K\subset W(R^\flat)$ is indeed a $p$-adically closed ideal and that $F(K)\subset K$; see~\cite[Lemma 7.6]{MR3867290}. Also, if $x\in W(R^\flat)$ is such that $V(x)\in K$ then in fact $x\in K$. We will make repeated use of this. 

Let us now make note of the following identities in $W(R^\flat)$ for $x\in W(R^\flat)$ and $k\ge 1$:
\begin{itemize}
   \item $V^k(x)^p = V^k(p^{k(p-1)}x^p)$: This follows by repeatedly using the identity $V^k(x)y = V^k(xF^k(y))$.
   \item For any $x\in W(R^\flat)$ and any $k\ge 1$, we have
   \[
      \delta(V^k(x)) = V^{k-1}(x) - p^{k(p-1)-1}V^k(x^p).
   \]
   This follows from the previous identity and the equalities
  \[
   V^k(x)^p +p\delta(V^k(x)) = F(V^k(x)) = pV^{k-1}(x)
  \]
\end{itemize}

If $a_k\in J^{p^k}$ for $k\ge 1$, then $y = V^k([a_k])\in K$ and we have
\[
\delta(y) = V^{k-1}([a_k]) - p^{k(p-1)-1}V^k([a_k^p]) = (1-p^{k(p-1)}) V^{k-1}([a_k])
\]
Since $V^{k-1}([a_k])\in K$, we see that $\delta(y)$ belongs to $K$.

Now, by Remark~\ref{rem:fnilp_semiperf}, there exists $r\ge 1$ such that $J^{p^r}\subset \varphi(J)$. Suppose that $a\in K\cap p^rW(R^\flat)$; then the hypothesis $J^{p^r}\subset \varphi(J)$ implies that $J^{p^{r+s}}\subset \varphi^s(J)$ for all $s\ge 0$, and so we see that $a = V^rF(b)$ for some $b\in K$. In particular, we have
\[
   \delta(a) = V^{r-1}(F(b)) - p^{r(p-1)-1}V^r(F(b^p))\in K.
\]

Since every element of $K$ can be written in the form $[a_0]+V[a_1]+\cdots+V[a_{r-1}]+a$ for $a_k,a$ as above, assertion (1) now follows.

Let's move on to (2): Knowing that $A_m$ underlies a $p$-adic prismatic frame for $R_m$ with $\Fil^1A_m = pA_m$ amounts to seeing that $\varphi(A_m[p]) = 0$, where $A_m[p]\subset A_m$ is the $p$-torsion. This follows from~\cite[Lemma 7.6]{MR3867290}. That the frame is laminated amounts to the observation that the composition $W(R^\flat)\xrightarrow{F}W(R^\flat)\to \overline{W(R_m)}$ factors through $A_m\to R_m\to \overline{W(R_m)}$. 

Assertions (3) and (4) are shown during the proof of~\cite[Proposition 6.4]{lau2018divided}.
\end{proof}

\begin{remark}
   \label{rem:straight_weak_lifts}
 With the notation of Lemma~\ref{lem:straight_weak_nilpotence},  $A\defn W(R^\flat)/K_0$ underlies what Lau calls a \emph{straight weak lift} of $R$~\cite[Definition 7.3]{MR3867290}, meaning in our terminology that it underlies the $p$-adic prismatic frame $\underline{A}$ for $R$. Note that $\underline{A}$ need not be of Breuil-Kisin type, since it might have $p$-torsion. 
\end{remark}

\begin{proof}
[Proof of Proposition~\ref{prop:F-nilp-nilpotent}]
By Remark~\ref{rem:reduction_to_Fp_algebras_F-nilp-nilpotent}, we can assume that $R$ is semiperfect. Let $A = W(R^\flat)/K_0$ be the straight weak lift of $R$ from Remark~\ref{rem:straight_weak_lifts}. This underlies two laminated $p$-adic prismatic frames: We have the frame $\underline{A}$ for $R$ from the remark, but we also have the frame $\underline{\tilde{A}}$ for $\tilde{R} = A/{}^{\mathbb{L}}p$ with $\overline{A} = \tilde{R}$ and $\overline{\varphi}$ the Frobenius endomorphism of $\tilde{R}$. On the other hand, via the map $\Prism_R\to A\to \tilde{R}$, we can also equip $\Prism_R$ with the structure of a frame $\underline{\tilde{\Prism}}_R$ for $\tilde{R}$. In fact, this arises via a divided power structure on the map $\tilde{R}\to R$ (see Remark~\ref{rem:groth-messing_nilpotent_PD}), which we will explain below in more detail. 

All of this gives us the following diagram of frames:
\[
   \begin{diagram}
      \underline{\Prism}_{\tilde{R}}&\rTo&\underline{\tilde{\Prism}}_R&\rTo&\underline{\Prism}_R&\rTo&\underline{\Prism}_{R,\mathrm{cl}}\\
      \dTo&&\dTo&&\dTo&\ldTo\\
      \underline{\tilde{A}}&\rEquals&\underline{\tilde{A}}&\rTo&\underline{A}
   \end{diagram}
\]
If 
\[
K^A_{\mathrm{cl}} = \fib(\Prism_{R,\mathrm{cl}}\to A) = \ker(\Prism_{R,\mathrm{cl}}\to A), 
\]
then the associated semilinear operator on $K^A_{\mathrm{cl}}\{1\}$ is topologically nilpotent by~\cite[Proposition 6.3]{lau2018divided}.

To complete the proof now, it is sufficient to know the following things:
\begin{enumerate}
   \item The frame $\underline{\tilde{A}}$ satisfies the conditions of Corollary~\ref{cor:semiperfectoid_annealed}.
   \item The semilinear operator on $\fib(\Prism_{\tilde{R}}\to \Prism_R)\{1\}$ induced by the top left map of frames is topologically locally nilpotent.
\end{enumerate}
Let us consider (2) first. By Remark~\ref{rem:groth-messing_nilpotent_PD}, this would follow if we knew that it is associated with a nilpotent divided power structure on $\tilde{R}=A/{}^{\mathbb{L}}p\to A/pA=R$. This is a consequence of Lemma~\ref{lem:nilpotent_divided_powers_mod_p} below. When $p=2$, we also need the additional observation that $c^2 = \varphi(c) = 0$ for all $c\in A[2]$ (part of the condition of being a weak lift).

For (1), we will use the $\delta$-rings $A_m$ constructed in Lemma~\ref{lem:straight_weak_nilpotence}, and consider the associated laminated prismatic frames $\underline{\tilde{A}}_m$ for $\tilde{R}_m = A_m/{}^{\mathbb{L}}p$. 

By Remarks~\ref{rem:frames_for_non-discrete_R} and~\ref{rem:cotangent_complex_semiperfectoid},  $\dot{\varphi}^{\tilde{A}_m}_1$ is topologically locally nilpotent if and only if the divided Frobenius endomorphism of $\mathbb{L}_{A_m}$ is topologically locally nilpotent. 

By (3) and (4) of \emph{loc. cit.}, and by repeated application of Remark~\ref{rem:groth-messing_nilpotent_PD} along with Remark~\ref{rem:cotangent_complex_semiperfectoid} again, we find that the divided Frobenius endomorphism induces a nilpotent operator on $\fib(\mathbb{L}_{A_m}\to \mathbb{L}_A)/{}^{\mathbb{L}}p$. 

Since $J^{p^r}\subset \varphi(J)$, the map $A_r\to A$ factors through the Frobenius lift of $A$. In particular, $\mathbb{L}_{A_r}/{}^{\mathbb{L}}p\to \mathbb{L}_A/{}^{\mathbb{L}}p$ is nullhomotopic.

Combining the last three paragraphs now shows that $\dot{\varphi}^{\tilde{A}_m}_1$ is topologically nilpotent for all $m\ge 0$ and completes the proof of the proposition.
\end{proof}

\begin{lemma}
   \label{lem:nilpotent_divided_powers_mod_p}
Suppose that $A$ is a discrete ring, and if $p=2$ suppose also that there exists $m\ge 1$ such that, for all $c\in A[2]$, we have $c^{2^m} = 0$. Then the divided power structure on the square-zero extension $A/{}^{\mathbb{L}}p\to A/pA$---obtained from the canonical $p$-adic divided power structures on $A\to A/{}^{\mathbb{L}}p$ and $A\to A/pA$---is nilpotent.
\end{lemma}
\begin{proof}
   When $p>2$ this is of course immediate since the canonical divided power structure on $\Int_p\to \Field_p$ is pro-nilpotent. In general, we can view $A[p][1]$ as the space associated with the groupoid $pA/A$ via the nerve construction. Here, $A$ acts on $pA$ via $pt\cdot a = p(t+a)$. From this optic, the divided power operator $\gamma_m$ corresponds to the functor carrying $pt$ to $\gamma_m(pt)=\frac{p^{m-1}}{m!}pt^m$ and an arrow $pt \xrightarrow{a}p(t+a)$ to $\gamma_m(pt)\xrightarrow{p^{m-1}\sum_{i=1}^m\frac{t^i}{i!}\frac{a^{m-i}}{(m-i)!}}\gamma_m(p(t+a))$. When $p>2$, the operator $\gamma_p$ is already nullhomotopic with the nullhomotopy given by
   \begin{align*}
    pA &\to pA \times A\\
    t' = pt &\mapsto (\gamma_p(pt), -\frac{p^{p-1}}{p!}t^p)
   \end{align*}
   via the observation that the second entry is independent of the choice of $t\in A$ such that $t' = pt$. When $p=2$, then our additional hypothesis tells us that we can provide a nullhomotopy for $\gamma_{2^m}$ via
    \begin{align*}
    2A &\to 2A \times A\\
    t' = 2t &\mapsto (\gamma_{2^m}(2t), -\frac{2^{2^m-1}}{(2^m)!}t^{2^m}).
   \end{align*}
\end{proof}

\begin{remark}
[Classification in terms of weak lifts]
\label{rem:weak_bk_semiperfectoid}
The argument above actually shows the following: Let $R$ be a semiperfect $\Field_p$-algebra. Suppose that $A$ is a $\delta$-ring such that $A/pA = R$ and such that the Frobenius lift $\varphi:A\to A$ kills $A[p]$ (in other words, $A$ is a \emph{weak} lift of $R$). If $\underline{A}$ is the corresponding frame for $R$ with $\Fil^1A = pA$, and if the divided Frobenius on $\mathbb{L}_A$ is topologically locally nilpotent, then we have an exact equivalence of categories
\[
   \mathrm{FFG}_n(R)\xrightarrow{\simeq}\Mod{\Int/p^n\Int}(\mathrm{DDC}_{\underline{A}}^{[-1,0]}(R)).
\]
\end{remark}

\subsection{Group schemes with constant \'etale rank}
\label{subsec:constant_rank}

Let us fix a Breuil-Kisin frame $\underline{A}$ for $R$. 

\begin{construction}
 \label{const:CA_construction}
As in Construction~\ref{const:map_of_frames_prismatic_general}, the assignment $S\mapsto K^A_{S,\mathrm{cl}}$ on $R_{\mathrm{qsyn}}$ yields a quasisyntomic sheaf $K^A_{-,\mathrm{cl}}$ of modules over $\Prism_{-,\mathrm{cl}}$, the sheafification of $S\mapsto \Prism_{S,\mathrm{cl}}$, and a semilinear endomorphism 
\[
\dot{\varphi}^{\underline{A}}_{\mathrm{cl}}:K^A_{-,\mathrm{cl}}\{1\}\to K^A_{-,\mathrm{cl}}\{1\}.
\]

Set
\begin{align*}
C^A &= \fib( R\Gamma_{\mathrm{qsyn}}(\Spf R,K^A_{-}) \xrightarrow{\mathrm{id} - \dot{\varphi}^{\underline{A}}_1} R\Gamma_{\mathrm{qsyn}}(\Spf R,K^A_{-}));\\
C^A_{\mathrm{cl}} &= \fib( R\Gamma_{\mathrm{qsyn}}(\Spf R,K^A_{-,\mathrm{cl}}) \xrightarrow{\mathrm{id} - \dot{\varphi}^{\underline{A}}_{1,\mathrm{cl}}} R\Gamma_{\mathrm{qsyn}}(\Spf R,K^A_{-,\mathrm{cl}})).
\end{align*}  
\end{construction}

\begin{remark}
\label{rem:CA_isomorphism}
Note that the map
\[
\fib(\Prism_{-}\to \Prism_{-,\mathrm{cl}})\to \fib(K^A_{-}\to K^A_{-,\mathrm{cl}})
\]
is an isomorphism. Write $N_{-}$ for the target. If $R/pR$ is $F$-finite and $F$-nilpotent, then the quasisyntomic site over $R$ admits a basis consisting of semiperfectoids $S$ with $S/pS$ $F$-nilpotent~\cite[Lemma 2.6]{lau2018divided}, and Proposition~\ref{prop:F-nilp-nilpotent} tells us that the endomorphism $\dot{\varphi}^{\underline{A}}_{1,N}:N_{-}\{1\}\to N_{-}\{1\}$ induced from $\dot{\varphi}^{\underline{A}}_1$ and $\dot{\varphi}^{\underline{A}}_{1,\mathrm{cl}}$ is topologically locally nilpotent. This shows that the natural map $C^A\to C^A_{\mathrm{cl}}$ is an isomorphism. 
\end{remark}

\begin{remark}
[Homomorphisms and extensions between $\Int/p\Int$ and $\mup$]
\label{rem:ordinary_nilpotence_condition}
By Remark~\ref{rem:sections_bk_bk-frame}, we have
\[
\mathrm{RHom}_{\underline{A}}(\mathbf{1},\mathbf{1}^*/{}^{\mathbb{L}}p)\simeq \fib(I'/pI'\{1\} \xrightarrow{\mathrm{can}-u}A/pA\{1\})
\]
In particular, we have $\Ext^i_{\underline{A}}(\mathbf{1},\mathbf{1}^*/{}^{\mathbb{L}}p)=0$ for $i\neq 0,1$. Combining Proposition~\ref{prop:syntomic_and_fppf}, Proposition~\ref{prop:ffg_classification_semiperf} and Remark~\ref{rem:bk_twist_sections_deform} with quasisyntomic descent, we now obtain a long exact sequence
\[
\begin{tikzcd}[column sep=2.8em, row sep=3.5em] 
0 \arrow[r] & H^0(C^A/{}^{\mathbb{L}}p) \arrow[r] & \mup(R) \arrow[r] & \Hom_{\underline{A}}(\underline{\mathbf{1}},\underline{\mathbf{1}}^*/{}^{\mathbb{L}}p) \arrow[dll, sloped, pos=0.5,
                                                                        out=-45, 
                                                                        in=135,  
                                                                        looseness=1.3 
                                                                        ] \\
& H^1(C^A/{}^{\mathbb{L}}p) \arrow[r] & H^1_{\mathrm{fppf}}(\Spec R,\mup) \arrow[r] & \Ext^1_{\underline{A}}(\underline{\mathbf{1}},\underline{\mathbf{1}}^*/{}^{\mathbb{L}}p) \arrow[dll, sloped, pos=0.5,
                                                                        out=-45, 
                                                                        in=135,  
                                                                        looseness=1.3 
                                                                        ] \\
& H^2(C^A/{}^{\mathbb{L}}p)\arrow[r] & H^2_{\mathrm{fppf}}(\Spec R,\mup) \arrow[r]  &0
\end{tikzcd}
\]
By Remark~\ref{rem:CA_isomorphism}, if $R/pR$ is $F$-finite and $F$-nilpotent, we can replace $C^A$ with $C^A_{\mathrm{cl}}$ here.
\end{remark}

\begin{definition}
   \label{defn:stable_rank}
Suppose that $\kappa$ is an algebraically closed field in characteristic $p$ and that $F:M\to M$ is a $\varphi$-semilinear operator on a finite dimensional $\kappa$-vector space $M$. The \defnword{stable rank} of $F$ is the common rank of $F^m$ for all $m$ sufficiently large.
\end{definition}

\begin{definition}
   \label{defn:etale_rank}
An object $\underline{\mathsf{N}}$ in $\mathrm{BK}_{\underline{A},n}(R)$ has \defnword{constant \'etale rank} along a constructible closed subset $Z\subset \Spec R/pR$ if the operator $F_{\mathsf{N}^*}:\varphi^*\mathsf{N}^*\to \mathsf{N}^*$ has locally constant stable rank along $Z$ in the following sense: For every connected component $Z^\circ\subset Z$, there exists $r\ge 0$ such that for any algebraically closed point $x\in Z^\circ(\kappa)$, the base-change of $F_{\mathsf{N}^*}$ over $\kappa$ along $x$ has stable rank $r$. If $Z = V(\mathfrak{c})$ for a finitely generated ideal $\mathfrak{c}\subset (R/pR)_{\mathrm{red}}$, write 
\[
   \mathrm{BK}^{\mathfrak{c}\mathhyph\mathrm{const}}_{\underline{A},n}(R)\subset \mathrm{BK}_{\underline{A},n}(R)
\]
for the subcategory spanned by such objects.
\end{definition}

\begin{lemma}
   \label{lem:connected_etale_sequence}
Suppose that $\mathfrak{c}\subset B = (R/pR)_{\mathrm{red}}$ is a finitely generated ideal such that $B$ is $\mathfrak{c}$-adically derived complete. Then every object $\underline{\mathsf{N}}$ in $\mathrm{BK}^{\mathfrak{c}\mathhyph\mathrm{const}}_{\underline{A},n}(R)$ sits in a canonical short exact sequence
\[
  0\to \underline{\mathsf{N}}^{\mathrm{nilp}}\to \underline{\mathsf{N}}\to \underline{\mathsf{N}}^{\et}\to 0
\]
where $\underline{\mathsf{N}}^{\mathrm{nilp}}$ is $\mathfrak{c}$-nilpotent, and where $\underline{\mathsf{N}}^{\et}$ is \emph{\'etale} in the sense that $F_{\mathsf{N}^{\et,*}}$ is an isomorphism.
\end{lemma}
\begin{proof}
This follows from standard arguments; see for instance~\cite[(4.2.3)]{dejong:formal_rigid}.  
\end{proof}

\begin{proposition}
   \label{prop:constant_rank}
Suppose that $\mathfrak{c}\subset B = (R/pR)_{\mathrm{red}}$ is a finitely generated ideal such that $B$ is $\mathfrak{c}$-adically derived complete. Let 
\[
   \mathrm{FFG}^{\mathfrak{c}\mathhyph\mathrm{const}}_n(R)\subset \mathrm{FFG}_n(R)
\]
be the full subcategory spanned by the objects with locally constant \'etale rank along $V(\mathfrak{c})\subset \Spec R/pR$. Suppose that $C^A$ is nullhomotopic or that $R/pR$ is $F$-nilpotent and $F$-finite with $C^A_{\mathrm{cl}}$ nullhomotopic. Then the functor from Remark~\ref{rem:bk_syntomic} induces an equivalence of categories
\[
   \mathrm{FFG}^{\mathfrak{c}\mathhyph\mathrm{const}}_n(R)\xrightarrow{\simeq}\mathrm{BK}^{\mathfrak{c}\mathhyph\mathrm{const}}_{\underline{A},n}(R)
\]
\end{proposition}
\begin{proof}
 By reducing to the case where $R = \kappa$ is an algebraically closed field, one sees that the functor maps $\mathrm{FFG}^{\mathfrak{c}\mathhyph\mathrm{const}}_n(R)$ to $\mathrm{BK}^{\mathfrak{c}\mathhyph\mathrm{const}}_{\underline{A},n}(R)$. It remains to see that this is an equivalence.

  Using Remark~\ref{rem:ordinary_nilpotence_condition}, finite \'etale descent and d\'evissage, one finds that the functor is an equivalence on `ordinary' objects: On the right hand side, these are objects that are finite \'etale locally isomorphic to extensions of $\mathbf{1}/{}^{\mathbb{L}}p^m$ by $\mathbf{1}^*/{}^{\mathbb{L}}p^n$ for $m,n\ge 0$, and on the left they are extensions of \'etale group schemes by multiplicative ones. Corollary~\ref{cor:nilpotent_connected} also tells us that, for all $n\ge 1$, it restricts to an equivalence 
\[
  \mathrm{FFG}^{\mathrm{c}\mathhyph\mathrm{conn}}_n(R)\to \mathrm{BK}^{\mathfrak{c}\mathhyph\mathrm{nilp}}_{\underline{A},n}(R).
\]
To finish the proof, one can directly follow the argument in the proof of~\cite[Theorem 10.2]{MR1235021}; see also the argument in~\cite[\S 4.2]{dejong:formal_rigid} in the case of a field with finite $p$-basis. This uses Lemma~\ref{lem:connected_etale_sequence} as input. It also needs the observation that a finite flat group scheme with locally constant \'etale rank is an extension of an \'etale group scheme by one with connected fibers, which is a result of Messing~\cite[Lemma II.4.8]{Messing1972-qq}. Finally, the vanishing of the $p$-power torsion in the Brauer group used in de Jong's proof is obtained here via the use of the tail of the long exact sequence from Remark~\ref{rem:ordinary_nilpotence_condition} and the vanishing of $H^2(C^A/{}^{\mathbb{L}}p)$.
\end{proof}

\begin{remark}
   \label{rem:cotangent_complex_criterion_CA_vanishing}
Arguing as in Remark~\ref{rem:cotangent_complex_semiperfectoid}, we find that $C^A$ is nullhomotopic whenever the divided Frobenius operator on $\mathbb{L}_A$ is topologically locally nilpotent. If $R/pR$ is $F$-nilpotent and $F$-finite, then one can argue as in Corollary~\ref{cor:semiperf_bk_case} to see that $C^A_{\mathrm{cl}}$ (and hence $C^A$) is nullhomotopic whenever the induced operator on $\tau^{\ge 0}\mathbb{L}_A$ is topologically locally nilpotent.
\end{remark}

\begin{corollary}
\label{cor:complete_local_CA_vanishing}
Suppose that $R$ is a complete local Noetherian ring (resp. whose residue field has finite $p$-basis) and that $C^A$ (resp. $C^A_{\mathrm{cl}}$) is nullhomotopic. Then there is an exact equivalence
\[
   \mathrm{FFG}_n(R)\xrightarrow{\simeq}\mathrm{BK}_{\underline{A},n}(R).
\]
\end{corollary}
\begin{proof}
   If $Z\subset \Spec R/pR$ is the closed point, then every object in $\mathrm{BK}_{\underline{A},n}(R)$ (resp. in $\mathrm{FFG}_n(R)$) has constant \'etale rank along $Z$. Moreover, if the residue field has finite $p$-basis then $R/pR$ is $F$-finite and $F$-nilpotent.
\end{proof}

\begin{remark}
   [Complete local rings with perfect residue field]
\label{rem:complete_local_perfect_residue}
Suppose that $R$ is a complete local Noetherian ring whose residue field $\kappa$ admits a finite $p$-basis. Then we see from Remark~\ref{rem:cotangent_complex_criterion_CA_vanishing} that $C^A$ is nullhomotopic whenever the divided Frobenius operator on $\tau^{\ge -1}\mathbb{L}_A$ is topologically locally nilpotent. If $R$ is Artin local, it is enough to check that the induced operator on $H^i(\kappa\otimes_A\mathbb{L}_A)$ is nilpotent for $i=-1,0$.
\end{remark}

\begin{example}
   [Regular complete local rings with perfect residue field]
\label{ex:regular_complete_local}
Suppose that $R$ is a regular Noetherian complete local ring of dimension $m$ with perfect residue field $\kappa$. Then there exists a Breuil-Kisin frame $\underline{A}$ with $A = W(\kappa)\pow{x_1,\ldots,x_m}$ with $J = (x_1,\ldots,x_m)$ a $\delta$-ideal. In this case, Corollary~\ref{cor:complete_local_CA_vanishing}, combined with Remark~\ref{rem:complete_local_perfect_residue} can be used to recover the results of Lau from~\cite[\S 6]{lau:displays}, and in particular, gives a different approach to Kisin's classification of finite flat group schemes over complete discrete valuation rings in mixed characteristic with perfect residue field~\cite[\S 2.3]{kisin:f_crystals} in terms of certain $\varphi$-modules over $\Sig = W(\kappa)\pow{u}$. Indeed, if $\mx\subset R$ is the maximal ideal, then, for all $t\ge 1$, $A_t \defn A/J^t$ is a Breuil-Kisin frame for $R/\mx^t$. By Remark~\ref{rem:complete_local_perfect_residue}, $C^{A_t}$ is nullhomotopic whenever the divided Frobenius operator on $H^i(\kappa\otimes_{A_t}\mathbb{L}_{A_t})$ is nilpotent for $i=-1,0$. Furthermore, we have $C^A = \varprojlim_t C^{A_t}$. Now, observe that we have
\[
   \tau^{\ge -1}\mathbb{L}_{A_t}\simeq \cofib(J^t/J^{2t}\xrightarrow{d}A_t\otimes_A\widehat{\Omega}^1_A),
\]
where $\widehat{\Omega}^1_A$ is the $p$-completion of the module of differentials for $A$.
For $t\ge 2$, one can use this to show that
\[
   \tau^{\ge -1}(\kappa\otimes_{A_t}\mathbb{L}_{A_t}) \simeq \cofib(\mx^t/\mx^{t+1}\xrightarrow{0}\mx/\mx^2).
\]
Moreover the endomorphisms divided Frobenius operators on $H^{-1}(\kappa\otimes_{A_t}\mathbb{L}_{A_t}) = \mx^t/\mx^{t+1}$ and $H^0(\kappa\otimes_{A_t}\mathbb{L}_{A_t})=\mx/\mx^2$ for $i=-1,0$  can be identified up to sign with those arising from the $\delta$-structure maps $\delta:J^t\to J^t$ and $\delta:J\to J$. Therefore, we are reduced to knowing that this operator on $\mx/\mx^2$ is nilpotent. This holds for instance for the usual lift $x_i\mapsto x_i^p$ where the operator is in fact trivial.
\end{example}

\subsection{The characteristic $p$ case}
\label{subsec:the_characteristic_p_case}

We now specialize to the case where $R$ is an $\Field_p$-algebra, so that $\underline{A}$ is a $p$-adic Breuil-Kisin frame, where we will see that the criterion of Corollary~\ref{cor:complete_local_CA_vanishing} can be considerably simplified, yielding a generalization of a result of de Jong~\cite{MR1235021}. We first recall results of Bhatt-Lurie~\cite{bhatt2022absolute} and Bragg-Olsson~\cite{bragg2021representability}. The notation here is from~\cite[\S 7]{gmm}.

\begin{remark}
   \label{rem:q1_q2_maps}
   Let $Z^1_{\Prism}$ and $H^1_{\Prism}$ be the sheaves on $R_{\mathrm{qsyn}}$ given by the assignment
\[
   Z^1_{\Prism}:S\mapsto \Fil^{\mathrm{conj}}_1\overline{\Prism}_{S/\Field_p}\times_{\overline{\Prism}_{S/\Field_p}}\Fil^1_{\mathrm{Hdg}}\overline{\Prism}_{S/\Field_p};\; H^1_{\Prism}:S\mapsto \mathbb{L}_{S/\Field_p}[-1].
\]
There are two maps $q_1,q_2:Z^1_{\Prism}\to H^1_{\Prism}$ from the restrictions of the maps
\[
  \Fil^{\mathrm{conj}}_1\overline{\Prism}_{S/\Field_p}\to \gr^{\mathrm{conj}}_1\overline{\Prism}_{S/\Field_p}\xrightarrow{\simeq}\mathbb{L}_{S/\Field_p}[-1]\;;\; \Fil^{1}_{\mathrm{Hdg}}\overline{\Prism}_{S/\Field_p}\to \gr^{1}_{\mathrm{Hdg}}\overline{\Prism}_{S/\Field_p}\xrightarrow{\simeq}\mathbb{L}_{S/\Field_p}[-1]
\]
respectively.
\end{remark}

\begin{remark}
   \label{rem:alphap_cotangent_complex}
Since the restriction of the canonical map $\overline{\Prism}_{S/\Field_p}\to S$ to $\Fil^{\mathrm{conj}}_0\overline{\Prism}_{S/\Field_p} \simeq S$ is the Frobenius endomorphism of $S$, we obtain isomorphisms
\[
   \fib(q_1)(S)\simeq \Fil^{\mathrm{conj}}_0\overline{\Prism}_{S/\Field_p}\times_{\overline{\Prism}_{S/\Field_p}}\Fil^1_{\mathrm{Hdg}}\overline{\Prism}_{S/\Field_p}\simeq \fib(S\xrightarrow{\varphi}S)\simeq \alphap(S).
\]
In particular, we have $R\Gamma_{\mathrm{qsyn}}(\Spec R,\fib(q_1))\simeq R\Gamma_{\mathrm{fppf}}(\Spec R,\alphap)$, and we obtain a canonical fiber sequence
\begin{align}\label{eqn:alphap_fiber_sequence}
R\Gamma_{\mathrm{fppf}}(\Spec R,\alphap)\to \Fil^{\mathrm{conj}}_1\overline{\Prism}_{R/\Field_p}\times_{\overline{\Prism}_{R/\Field_p}}\Fil^1_{\mathrm{Hdg}}\overline{\Prism}_{R/\Field_p}\xrightarrow{R\Gamma(q_1)} R\Gamma_{\mathrm{qsyn}}(\Spec R,H^1_{\Prism})\simeq \mathbb{L}_{R/\Field_p}[-1]
\end{align}
and a canonical map
\begin{align}\label{eqn:alphap_cotangent_complex}
   R\Gamma_{\mathrm{fppf}}(\Spec R,\alphap)\xrightarrow{R\Gamma(q_2)} R\Gamma_{\mathrm{qsyn}}(\Spec R,H^1_{\Prism})\simeq \mathbb{L}_{R/\Field_p}[-1]
\end{align}
\end{remark}

\begin{remark}
 [Splitting the conjugate filtration]
\label{rem:conjugate_splitting}
The existence of the frame $\underline{A}$ gives a splitting of the canonical fiber sequence~\cite[Proposition 3.17]{Bhatt2012-fe}
\[
   \Fil^{\mathrm{conj}}_0\overline{\Prism}_{R/\Field_p}\to \Fil^{\mathrm{conj}}_1\overline{\Prism}_{R/\Field_p}\to \gr^{\mathrm{conj}}_1\overline{\Prism}_{R/\Field_p}\simeq \mathbb{L}_{R/\Field_p}[-1].
\]
In fact, the mod-$p^2$ reduction of $(A,\varphi)$ is already sufficient to obtain this splitting. In turn, this gives us a splitting of the fiber sequence~\eqref{eqn:alphap_fiber_sequence} and so a decomposition
\begin{align}\label{eqn:alphap_lifting_decomposition}
 \Fil^{\mathrm{conj}}_1\overline{\Prism}_{R/\Field_p}\times_{\overline{\Prism}_{R/\Field_p}}\Fil^1_{\mathrm{Hdg}}\overline{\Prism}_{R/\Field_p}\simeq R\Gamma_{\mathrm{fppf}}(\Spec R,\alphap)\oplus \mathbb{L}_{R/\Field_p}[-1].
\end{align}
The restriction of $q_2$ to the second summand is the endomorphism $f_{\underline{A}}$ of $\mathbb{L}_{R/\Field_p}[-1]$ induced by the operator $`d\varphi/p'$ on $\mathbb{L}_A$, while the restriction of $q_1$ is of course the identity.
\end{remark}

\begin{remark}
   \label{rem:mup_cotangent_complex}
Much more non-trivially, for any $S$ in $R_{\mathrm{qsyn}}$, there is a canonical isomorphism $\fib(q_1-q_2)(S)\simeq \mup(S)$. It suffices of course to exhibit an isomorphism $\cofib(q_1-q_2)(S)\simeq B\mup(S)$ for all $\Field_p$-algebras $S$. This is a special case of a theorem of Bragg-Olsson~\cite[Theorem 4.8]{bragg2021representability}. The point is that both sides of the purported isomorphism are now left Kan extended from smooth $\Field_p$-algebras and so it suffices to establish a canonical isomorphism for such inputs. This can be done in two ways: 
\begin{itemize}
   \item By directly showing that the natural map $S^\times/(S^\times)^p\xrightarrow{x\mapsto \mathrm{dlog}(x)}\Omega^1_{S/\Field_p}$ on smooth $\Field_p$-algebras sheafifies to an isomorphism $B\mup(S)\xrightarrow{\simeq}\Omega^{1,\mathrm{cl}}_{S/\Field_p}$ where the target is the space of closed forms. This is a special case of a classical result of Artin-Milne~\cite[Proposition 2.4]{artin_milne}.
   \item Using quasisyntomic descent to reduce to the case of a qrsp $\Field_p$-algebra $S$. Here, if $S^\flat\to S$ is the inverse perfection of $S$ with kernel $J$, using the interpretation of $\Prism_S$ as a divided power envelope, Bhatt and Lurie write down a logarithm map~\cite[p. 168, (36)]{bhatt2022absolute}
   \[
   \mup(S)\simeq (1+J)/(1+\varphi(J))\xrightarrow{\overline{\log}}Z^1_{\Prism}(S)\subset \Fil^1_{\mathrm{Hdg}}\overline{\Prism}_S
   \]
   mapping the source isomorphically onto $\ker(q_1-q_2)$. Here the first isomorphism carries $z\in \mu_p(S)$ to the image of $\tilde{z}^p$ for any lift $\tilde{z}\in S^\flat$ of $z$.
\end{itemize}
Therefore, we obtain a canonical fiber sequence
\begin{align}\label{eqn:mup_fiber_sequence}
R\Gamma_{\mathrm{fppf}}(\Spec R,\mup)\to \Fil^{\mathrm{conj}}_1\overline{\Prism}_{R/\Field_p}\times_{\overline{\Prism}_{R/\Field_p}}\Fil^1_{\mathrm{Hdg}}\overline{\Prism}_{R/\Field_p}\xrightarrow{R\Gamma(q_1-q_2)} R\Gamma_{\mathrm{qsyn}}(\Spec R,H^1_{\Prism})\simeq \mathbb{L}_{R/\Field_p}[-1]
\end{align}
and a canonical map
\begin{align}\label{eqn:mup_cotangent_complex}
   R\Gamma_{\mathrm{fppf}}(\Spec R,\mup)\xrightarrow{R\Gamma(q_2)} R\Gamma_{\mathrm{qsyn}}(\Spec R,H^1_{\Prism})\simeq \mathbb{L}_{R/\Field_p}[-1]
\end{align}
\end{remark}

\begin{remark}
   \label{rem:q2_on_mup_alphap}
For $S$ in $R_{\mathrm{syn}}$, the restrictions of $q_2$ to $\alphap(S)$ and $\mup(S)$ admit the following concrete interpretation: If $J = \ker(S^\flat\to S)$, then we have isomorphisms
\[
\alphap(S)\simeq J/\varphi(J)\;;\;\mup(S) \simeq (1+J)/(1+\varphi(J))\;;\; H^0(\mathbb{L}_{S/\Field_p}[-1])\simeq J/J^2.
\]
Via these isomorphisms, the map $\alphap(S)\to H^0(\mathbb{L}_{S/\Field_p}[-1])$ induced by $q_2$ corresponds to the natural surjection $J/\varphi(J)\to J/J^2$, while the map $\mup(S)\to H^0(\mathbb{L}_{S/\Field_p}[-1])$ corresponds to the composition
\[
 (1+J)/(1+\varphi(J))\to (1+J)/(1+J^2)\xrightarrow[\simeq]{1+y\mapsto y}J/J^2.
\]
\end{remark}

\begin{remark}
   \label{rem:q2_mup_factoring}
Note that we have isomorphisms
\[
  J^2/\varphi(J)\xrightarrow{\simeq}\ker(\alphap(S)\to H^0(\mathbb{L}_{S/\Field_p}[-1]))\;;\; (1+J^2)/(1+\varphi(J))\xrightarrow{\simeq}\ker(\mup(S)\to H^0(\mathbb{L}_{S/\Field_p}[-1])).
\]
By the argument in~\cite[p. 169]{bhatt2022absolute}, the restriction of the map $\mup(S)\to Z^1_{\Prism}(S)$ to $(1+J^2)/(1+\varphi(J))$ factors through an isomorphism
\[
   (1+J^2)/(1+\varphi(J))\xrightarrow{\simeq}J^2/\varphi(J)\subset \alphap(S).
\]
Explicitly, this isomorphism is given as follows: Given an element of $J^2$ of the form $x = uv$ for $u,v\in J$, the class of $1-x$ on the left hand side is carried to $-\sum_{d=1}^{p-1}\frac{x^d}{d}$.
\end{remark}

\begin{remark}
   \label{rem:q2_on_mup_alphap_global}
Remark~\ref{rem:q2_on_mup_alphap} can be globalized. Choose a surjection $\tilde{R}\to R$ such that $\mathbb{L}_{\tilde{R}/\Field_p}\simeq \Omega^1_{\tilde{R}/\Field_p}$ is flat over $\tilde{R}$ and such that $\alphap(\tilde{R})=0$; for instance $\tilde{R}$ can be a polynomial algebra over $\Field_p$. If $J = \ker(\tilde{R}\to R)$, we have isomorphisms
\begin{align*}
  (J\cap \varphi(\tilde{R}))/\varphi(J)\simeq \alphap(R)\;&;\; ((1+J)\cap \varphi(R^\times))/(1+\varphi(J))\simeq \mup(R)\;;\; \\
  H^0(\mathbb{L}_{R/\Field_p}[-1])&\simeq \ker(J/J^2\xrightarrow{x\mapsto dx}R\otimes_{\tilde{R}}\Omega^1_{\tilde{R}/\Field_p}).
\end{align*}
Via these isomorphisms, the map from $\alphap(R)$ (resp. $\mup(R)$) to $H^0(\mathbb{L}_{R/\Field_p}[-1])$ is obtained from the natural map $J/\varphi(J)\to J/J^2$ (resp. $(1+J)/(1+\varphi(J))\xrightarrow{1+z\mapsto z}J/J^2$).
\end{remark}

\begin{remark}
   \label{rem:mup_to_alphap}
Combining Remarks~\ref{rem:conjugate_splitting} and~\ref{rem:mup_cotangent_complex}, we obtain a map
\begin{align}\label{eqn:mup_to_alphap}
   R\Gamma_{\mathrm{fppf}}(\Spec R,\mup)\to R\Gamma_{\mathrm{fppf}}(\Spec R,\alphap).
\end{align}
Moreover, by Remark~\ref{rem:sections_bk_bk-frame} and the Artin-Schreier sequence, the right hand side is canonically isomorphic to 
\[
   \fib(R\xrightarrow{\varphi}R)\simeq \mathrm{RHom}_{\underline{A}}(\mathbf{1},\mathbf{1}^*/{}^{\mathbb{L}}p).
\]
One checks that the resulting map
\[
   R\Gamma_{\mathrm{fppf}}(\Spec R,\mup)\to \mathrm{RHom}_{\underline{A}}(\mathbf{1},\mathbf{1}^*/{}^{\mathbb{L}}p)
\]
is precisely the one giving rise to the long exact sequence in Remark~\ref{rem:ordinary_nilpotence_condition}.
\end{remark}

\begin{remark}
[Direct relationship with cotangent complex]
   \label{rem:direct_cotangent_complex}
Let $M_{\Prism}$ be the sheaf on $R_{\mathrm{qsyn}}$ given by
\[
   M_{\Prism}(S) = \fib(q_1)(S)\times_{Z^1_{\Prism}(S)}\fib(q_2)(S).
\]
Then the natural map $M_{\Prism}\to Z^1_{\Prism}$ factors through both $\mup$ and $\alphap$. In turn, these factorings give a commuting diagram whose rows are fiber sequences
\[
\begin{diagram}
  R\Gamma_{\mathrm{qsyn}}(\Spec R, M_{\Prism})&\rTo& R\Gamma_{\mathrm{fppf}}(\Spec R, \mup)&\rTo^{R\Gamma(q_2)}&\mathbb{L}_{R/\Field_p}[-1];\\
  \dEquals&&\dTo^{\eqref{eqn:mup_to_alphap}}&&\dTo_{\mathrm{id}-f_{\underline{A}}}\\
   R\Gamma_{\mathrm{qsyn}}(\Spec R, M_{\Prism})&\rTo& R\Gamma_{\mathrm{fppf}}(\Spec R, \alphap)&\rTo^{R\Gamma(q_2)}&\mathbb{L}_{R/\Field_p}[-1]. 
\end{diagram}
\]
Here, $f_{\underline{A}}$ is the divided Frobenius map as in Remark~\ref{rem:conjugate_splitting}. In fact, one can say more. Let $N_{\Prism}$ be the sheaf on $R_{\mathrm{qsyn}}$ given by
\[
   N_{\Prism}(S) = \ker(\mup(S)\to H^0(\mathbb{L}_{S/\Field_p}[-1])).
\]
Then Remark~\ref{rem:q2_mup_factoring} shows that there is a canonical identification
\[
   N_{\Prism}(S)\simeq \ker(\alphap(S)\to H^0(\mathbb{L}_{S/\Field_p}[-1])),
\]
and a commuting diagram whose rows are again fiber sequences:
\begin{align}\label{eqn:NPrism_diagram}
\begin{diagram}
  R\Gamma_{\mathrm{qsyn}}(\Spec R, N_{\Prism})&\rTo& R\Gamma_{\mathrm{fppf}}(\Spec R, \mup)&\rTo^{\tau^{\ge 0}R\Gamma(q_2)}&\tau^{\ge 0}\mathbb{L}_{R/\Field_p}[-1];\\
  \dEquals&&\dTo^{\eqref{eqn:mup_to_alphap}}&&\dTo_{\mathrm{id}-\tau^{\ge 0}f_{\underline{A}}}\\
   R\Gamma_{\mathrm{qsyn}}(\Spec R, N_{\Prism})&\rTo& R\Gamma_{\mathrm{fppf}}(\Spec R, \alphap)&\rTo^{\tau^{\ge 0}R\Gamma(q_2)}&\tau^{\ge 0}\mathbb{L}_{R/\Field_p}[-1]. 
\end{diagram}
\end{align}
Combining these observations, we find that we have isomorphisms
\begin{align}
\label{eqn:mup_to_alphap_cotangent_complex_direct}
C^A/{}^{\mathbb{L}}p&\xrightarrow{\simeq}\fib\left(\mathbb{L}_{R/\Field_p}[-1]\xrightarrow{\mathrm{id}-f_{\underline{A}}}\mathbb{L}_{R/\Field_p}[-1]\right)\nonumber\\
&\xrightarrow{\simeq}\fib\left(\tau^{\ge 0}\mathbb{L}_{R/\Field_p}[-1]\xrightarrow{\mathrm{id}-\tau^{\ge 0}f_{\underline{A}}}\tau^{\ge 0}\mathbb{L}_{R/\Field_p}[-1]\right).
\end{align}
\end{remark}

\begin{lemma}
   \label{lem:gamma_delta_compatibility}
Let $\alphap(R)\subset R$ be the kernel of the Frobenius endomorphism. For $x\in \alphap(R)$, choose a lift $y\in A$, so that $y^p = pu$ for some $u\in A$. Then:
\begin{enumerate}
   \item The image $\gamma_{\underline{A}}(x)$ of $u$ in $R$ lies in $\alphap(R)$ and is independent of the choice of lift $y$;
   \item The following diagram is commutative
\[
   \begin{diagram}
      \alphap(R)&\rTo&H^{-1}(\mathbb{L}_{R/\Field_p})\\
      \dTo^{-\gamma_{\underline{A}}}&&\dTo_{H^{-1}(f_{\underline{A}})}\\
      \alphap(R)&\rTo&H^{-1}(\mathbb{L}_{R/\Field_p})
   \end{diagram}
\]
Here the horizontal maps are obtained by applying $H^0$ to ~\eqref{eqn:alphap_cotangent_complex}.
\end{enumerate}
\end{lemma}
\begin{proof}
   Assertion (1) follows from~\cite[Sublemma 9.10]{MR1235021}, but we will in any case prove this implicitly in what follows.

   Choose a surjective map $\tilde{A}\to A$ where $\tilde{A}$ is the $p$-completion of a free $\delta$-ring over $\Int_{(p)}$, and let $K =\ker(\tilde{A}\to A)$. Then we have 
   \[
      H^{-1}(\mathbb{L}_{R/\Field_p}) \simeq \ker(K/(K^2+pK)\xrightarrow{d}R\otimes_{\tilde{A}}\widehat{\Omega}^1_{\tilde{A}}),
   \]
   The map $\alphap(R)\to H^{-1}(\mathbb{L}_{R/\Field_p})$ now admits the following description, which can be deduced from Remark~\ref{rem:q2_on_mup_alphap_global}: Given $x\in R$ with $x^p = 0$, we choose a lift $\tilde{y}\in \tilde{A}$ for $x$. This satisfies $\tilde{y}^p = p\tilde{u} + k$ for some $k\in K$. The image of $k$ in $K/(K^2+pK)$ lands in $H^{-1}(\mathbb{L}_{R/\Field_p})$, and is the image of $x$. Applying $\varphi$ to the previous identity gives us
   \[
   (p\tilde{u}+k + p\delta(\tilde{y}))^p = (\tilde{y}^p + p \delta(\tilde{y}))^p = p(\tilde{u}^p + p\delta(\tilde{u})) + \varphi(k)
   \]
   Expanding the left hand side shows that we have
   \[
     p\tilde{u}^p\in k^p - \varphi(k) + p^2\tilde{A} \Rightarrow \tilde{u}^p\in -\delta(k) + p\tilde{A}.
   \]
   Therefore, $\alphap(R)\to H^{-1}(\mathbb{L}_{R/\Field_p})$ carries $\gamma_{\underline{A}}(x)$ to the image of $-\delta(k)$ in $K/K^2$, and this verifies the commutativity of the diagram in (2).
\end{proof}

\begin{proposition}
[Simplified nilpotence criterion]
   \label{prop:simple_cotangent_vanishing}
Suppose that $R$ is Noetherian, that $\mathfrak{c}\subset R$ is an ideal such that $H^i(\mathbb{L}_{R/\Field_p})$ is $\mathfrak{c}$-adically complete for $i=-1,0$, and that the following conditions hold:
\begin{enumerate}
   \item The operator $f_{\underline{A}}$ induces a $\mathfrak{c}$-adically topologically locally nilpotent endomorphism of $\Omega^1_{R/\Field_p}$.
   \item The operator $\gamma_{\underline{A}}$ is a $\mathfrak{c}$-adically topologically locally nilpotent endomorphism of $\alphap(R)$.
\end{enumerate}
Then $\mathrm{id}-f_{\underline{A}}$ induces an automorphism of $\mathbb{L}_{R/\Field_p}[-1]$. In particular, the conclusion of Proposition~\ref{prop:constant_rank} holds (for the image of $\mathfrak{c}$ in $R_{\mathrm{red}}$).
\end{proposition}
\begin{proof}
  Condition (1) ensures that $\mathrm{id}-H^1(f_{\underline{A}})$ is an automorphism of $\Omega^1_{R/\Field_p}\simeq H^1(\mathbb{L}_{R/\Field_p}[-1])$. Given~\eqref{eqn:mup_to_alphap_cotangent_complex_direct}, it is now enough to know that $\mathrm{id}-H^{0}(f_{\underline{A}})$ is also an automorphism.

  By Remark~\ref{rem:q2_on_mup_alphap_global}, $\mup(R)$ and $\alphap(R)$ have the same image in $H^{-1}(\mathbb{L}_{R/\Field_p}[-1])$. Therefore, condition (2) together with assertion (2) of Lemma~\ref{lem:gamma_delta_compatibility} shows that $H^{-1}(f_{\underline{A}})$ is a $\mathfrak{c}$-adically topologically nilpotent operator on this common image $M$. In the notation of Remark~\ref{rem:direct_cotangent_complex}, set $N = R\Gamma_{\mathrm{syn}}(\Spec R,N_\Prism)$. Then we obtain the followng diagram with exact rows:
 \[
  \begin{diagram}
     0&\rTo&M&\rTo&H^{-1}(\mathbb{L}_{R/\Field_p})&\rTo^{\delta_1}&H^1(N)\\
     &&\dTo^{\simeq}&&\dTo^{\mathrm{id}-H^{-1}(f_{\underline{A}})}&&\dEquals\\
     0&\rTo&M&\rTo&H^{-1}(\mathbb{L}_{R/\Field_p})&\rTo_{\delta_2}&H^1(N)
  \end{diagram}
\]
where $\delta_1$ is the boundary map associated with the top row of~\eqref{eqn:NPrism_diagram} while $\delta_2$ is associated with the bottom row, and where the left vertical arrow is an isomorphism. To finish, we need to know that the image in $H^1(N)$ of $\delta_1$ is equal to that of $\delta_2$, which can be deduced from the last part of Remark~\ref{rem:q2_mup_factoring}.
\end{proof}

\begin{remark}
[Complete local $\Field_p$-algebras]
   \label{rem:complete_local_rings}
Any complete local Noetherian ring with finite $p$-basis satisfies the unnumbered conditions of Proposition~\ref{prop:simple_cotangent_vanishing} with respect to its maximal ideal. Therefore, if the conditions (1) and (2) hold, then we can conclude by Corollary~\ref{cor:complete_local_CA_vanishing} that the functor $\mathrm{FFG}_n(R)\to \mathrm{BK}_{\underline{A},n}(R)$ is an exact equivalence.
\end{remark}

\begin{example}
   [A result of de Jong]
\label{ex:complete_local_dejong}
Suppose that $R$ is a complete local Noetherian $\Field_p$-algebra with perfect residue field $\kappa$ and maximal ideal $\mathfrak{m}$. In this case, $A$ is also complete local. We have
\[
H^0(\kappa\otimes_A\mathbb{L}_A) = \kappa\otimes_RH^0(\mathbb{L}_{R/\kappa})\simeq \kappa\otimes_R\Omega^1_{R/\kappa}\xrightarrow[\simeq]{1\otimes dr \mapsto r(\mathrm{mod}~\mx^2)} \mx/\mx^2.
\]
The endomorphism $\overline{u}:\mx/\mx^2\to \mx/\mx^2$ induced by $f_{\underline{A}}$ can be described as follows: By the argument in~\cite[Lemma 8.2]{MR1235021}, there exists a $\delta$-ideal $I\subset A$ such that $A/I\simeq W(\kappa)$. Now, unwinding definitions, one sees that $\overline{u}$ is given by 
\[
\overline{u}(r(\mathrm{mod}~\mx^2)) = \delta(\tilde{r})(\mathrm{mod}~(p+I^2))
\]
where $\tilde{r}\in I$ is any lift of $r\in \mx$. Condition (1) of Proposition~\ref{prop:simple_cotangent_vanishing} is equivalent to asking for this operator to be nilpotent. Condition (2) amounts to asking for the operator on $\alphap(R)/\mx \alphap(R)$ induced by $\gamma_{\underline{A}}$ to be nilpotent. Combined with Remark~\ref{rem:complete_local_rings}, this recovers the Theorem from the introduction to~\cite{MR1235021}. A slightly finer analysis of the proof of Proposition~\ref{prop:simple_cotangent_vanishing} tells us that condition (2) shows that $\mup(R)\to \alphap(R)$ is an isomorphism, while condition (1) implies that $H^1_{\mathrm{fppf}}(\Spec R,\mup)\to H^1_{\mathrm{fppf}}(\Spec R,\alphap)$ is an isomorphism; compare with~\cite[Lemma 10.1]{MR1235021}. 
\end{example}

\printbibliography

\end{document}